\documentclass[aos]{imsart}

\RequirePackage{amsthm,amsmath,amsfonts,amssymb}

\RequirePackage[OT1]{fontenc}
\RequirePackage{graphicx,color,amssymb,amsmath,amsthm,mathrsfs}
\RequirePackage[authoryear]{natbib}
\RequirePackage[colorlinks,citecolor=blue,urlcolor=blue, hypertexnames=false]{hyperref}
\RequirePackage{tikz}
\RequirePackage{bbm}
\usepackage{caption} 
\usepackage{multibib}

\newcites{SM}{References}

\startlocaldefs
\numberwithin{equation}{section}
\theoremstyle{plain}
\newtheorem{thm}{Theorem}[section]
\newtheorem{lemma}[thm]{Lemma}
\newtheorem{proposition}[thm]{Proposition}
\newtheorem{cor}[thm]{Corollary}

\theoremstyle{definition}
\newtheorem{remark}[thm]{Remark}

\def\R{\mathbb{R}}

\def\N{\mathbb{N}}

\def\1{\mathbbm{1}}

\def\E{\mathbb{E}}
\def\Pr{\mathbb{P}}

\endlocaldefs

\begin{document}

\begin{frontmatter}
\title{Self-organized regime switching in null-recurrent dynamics}
\runtitle{Self-organized regime switching in null-recurrent dynamics}

\begin{aug}
\author[A]{\fnms{Johannes}~\snm{Brutsche}\ead[label=e1]{johannes.brutsche@stochastik.uni-freiburg.de}},
\author[A]{\fnms{Sebastian}~\snm{Hahn}\ead[label=e2]{sebastian.hahn@stochastik.uni-freiburg.de}}
\and
\author[A]{\fnms{Angelika}~\snm{Rohde}\ead[label=e3]{angelika.rohde@stochastik.uni-freiburg.de}}
\address[A]{Mathematical Institute, University of Freiburg\printead[presep={,\ }]{e1}\printead[presep={,\ }]{e2}\printead[presep={,\ }]{e3}}
\end{aug}

\begin{abstract}
Based on discrete observations $X_0,X_{\Delta},\dots, X_{n\Delta}$ for $\Delta=n^{-\gamma}$ with $\gamma\in [0,1)$ of the null-recurrent dynamic $dX_t = \sigma(X_t)dW_t$ with a Brownian motion $W$ and $\sigma(x)=\alpha\1\{x<\rho\} + \beta\1\{x\geq \rho\}$, we derive rate of convergence and limiting distribution of the profile MLE for $\rho$. This includes low-frequency asymptotics ($\gamma=0$) for which the observations form a null-recurrent Markov chain.
The derived non-standard limit is the argsup over a doubly stochastic drifted Poisson process explicitly involving the local time of oscillating Brownian motion. Its dependence on $\rho$ as well as the unknown volatility levels $\alpha$ and $\beta$ is shown to be continuous w.r.t. the topology of weak convergence, enabling statistical inference. Whereas this limit is independent of the sampling frequency, the profile MLE's rate of convergence equals $n^{-(1+\gamma)/2}$ and is proven to be minimax optimal.
The surprising idea of the proof of the limit theorem is to relate the long-term behavior of the null-recurrent Markov chain to the infill asymptotics on a fixed time interval. Indeed, in the very special case that $(X_t)_{t\geq 0}$ is started in the true parameter $X_0=\rho_0$, the process $(X_t-\rho_0)_{t\geq 0}$ is shown to possess a desirable distributional self-similarity. On basis of the strong Markov property, the artificial constellation of starting in $\rho_0$ is finally overcome by a coupling argument.
\end{abstract}



\end{frontmatter}
\numberwithin{equation}{section}
\addtocontents{toc}{\protect\setcounter{tocdepth}{0}}


\section{Introduction and preliminaries}
While Brownian motion originally describes random motion of particles suspended in a medium, oscillating Brownian motion as introduced in~\cite{Keilson/Wellner} naturally arises in the description of diffusive motion in porous or highly inhomogeneous media. Therefore, it serves as a toy model for change-point phenomena with structural breaks not depending on time but purely on state of the process itself. Apart from natural sciences, self-organized regime switching phenomena are also observed in economics, see for example~\cite{Yuan/Yau} and references therein. Figure~\ref{Fig_2} contains an example of stock prices, illustrating a typical sample path with an abrupt change at the threshold $\rho$ between two different volatility levels.

\vspace{-0.3cm}

\begin{center}
\begin{figure}[h]
\includegraphics[width=12cm, height=2.85cm]{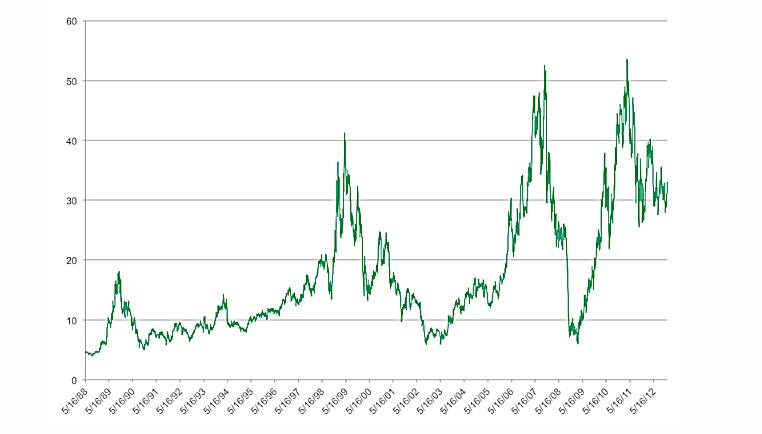}
\caption{\footnotesize{Sotheby's (BID) stock prices, 1988-2012, from \textit{Yahoo!Finance}. The process switches its volatility regime at a threshold $\rho$ of approximately $22$.}}\label{Fig_2}
\end{figure}
\end{center}

In mathematical terms, oscillating Brownian motion is a Markov process $X=(X_t)_{t\geq 0}$ solving the homogeneous stochastic differential equation
\begin{align}\label{eq:SDE_OBM}
dX_t = \sigma_{\rho,\alpha,\beta}(X_t) dW_t, \qquad X_0 =x_0\in\R,
\end{align} 
where $W$ is a standard Brownian motion and for different (but arbitrary) numbers $\alpha,\beta\in\R_{>0}$, the diffusion coefficient is given by 
\begin{align}\label{eq:coefficient_OBM}
\sigma_{\rho,\alpha,\beta}(x) := \begin{cases} \alpha, &\textrm{ if } x< \rho, \\ \beta, &\textrm{ if } x \geq\rho.\end{cases}
\end{align}
As shown in~\cite{LeGall}, the SDE \eqref{eq:SDE_OBM} possesses a unique strong solution. This solution is a strong Markov process, cf.~\cite{Keilson/Wellner}, who also identified the transition density $p_t^{\rho,\alpha,\beta}(x,y)$ from state $x$ to state $y$ within time $t$ for $\rho=0$. It is given for general $\rho$ in~\eqref{eq:transition_density} and one reads off that it is \textit{not} continuous in the parameter $\rho$ of interest and satisfies
\[ p_t^{\rho,\alpha,\beta}(x,y) = p_t^{0,\alpha,\beta}(x-\rho,y-\rho),\]
which will play a major role for this article. Together with the explicit density of the passage time distribution of $X$ in Theorem~$2$ in~\cite{Keilson/Wellner} for $\rho=0$, the process $X$ is seen to be null-recurrent. Especially, in contrast to stationary processes, an invariant probability distribution does not exist and the classical ergodic theorem is not available.

Whereas the threshold $\rho$ is observable via quadratic variation given any finite continuous trajectory $(X_t)_{0\leq t\leq T}$ that crosses $\rho$, its identification is a statistical problem on the basis of discrete observations 
\[X_{k\Delta_n^\gamma}, \ \ k=0,1,\dots, n, \]
with $\Delta_n^\gamma=n^{-\gamma}$ for $\gamma\in [0,1]$. For the so-called infill asymptotics corresponding to $\gamma=1$, the time horizon $n\Delta_n^\gamma$ is bounded, whereas for all $\gamma\in [0,1)$ it tends to infinity (see Figure~\ref{Fig_1}).

\vspace{-0.3cm}

\begin{center}
\begin{figure}[h]
\begin{tikzpicture}
  \draw[->] (0,0) -- (12,0) node[right] {$x$};

  \foreach \x in {0,1,2,3,4,5} {
    \draw (2*\x,0) -- (2*\x,0.1);
    \fill (2*\x,0) circle (2pt);
    \node[below=2pt] at (2*\x,0) {};
  }

  \draw[red,thick] (0.2,-0.1) -- (0.2,0.1) node[above] {$\tfrac{1}{n}$};
  \draw[red,thick] (0.4,-0.1) -- (0.4,0.1) node[above] {$\tfrac{2}{n}$};
  \draw[red,thick] (0.6,-0.1) -- (0.6,0.1) node[above] {};
  \draw[red,thick] (0.8,-0.1) -- (0.8,0.1) node[above] {};
  \draw[red,thick] (1.0,-0.1) -- (1.0,0.1) node[above] {};
  \draw[red,thick] (1.2,-0.1) -- (1.2,0.1) node[above] {};
  \draw[red,thick] (1.4,-0.1) -- (1.4,0.1) node[above] {};
  \draw[red,thick] (1.6,-0.1) -- (1.6,0.1) node[above] {};
  \draw[red,thick] (1.8,-0.1) -- (1.8,0.1) node[above] {};
  \draw[red,thick] (2.0,-0.1) -- (2.0,0.1) node[above] {};
  
  \draw[green,thick] (0.85,0.2) -- (0.85,-0.2) node[below] {$\tfrac{1}{\sqrt{n}}$};
  \draw[green,thick] (1.7,0.2) -- (1.7,-0.2) node[below] {$\tfrac{2}{\sqrt{n}}$};
  \foreach \y in {3,4,5,6,7,8,9,10,11,12,13}{
  	\draw[green,thick] (0.85*\y,0.2) -- (0.85*\y,-0.2) node[below] {};
  }
  
  \foreach \z in {0,1,2,3,4,5} {
    \draw[blue,thick] (2*\z,-0.3) -- (2*\z,0.3) node[above] {\z};
  }
  
\end{tikzpicture}
\caption{\small{An illustration of different observation schemes. The red lines correspond to $\gamma=1$, the infill asymptotics, where no point after $T=1$ is observed. The blue dashes correspond to the low frequency observation scheme $\gamma=0$. The green ones correspond to $\gamma=1/2$, where the observations are more dense with growing $n$, but also the observation window increases.}}\label{Fig_1}
\end{figure}
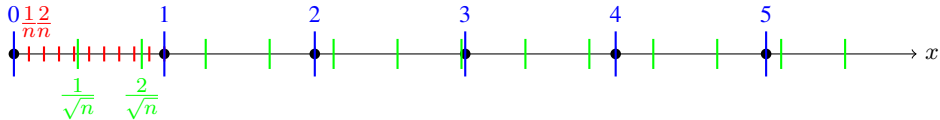
\end{center}

\vspace{-1cm}

Likewise, the only observation scheme for which the recurrence structure is irrelevant is the infill asymptotics (the red scheme in Figure~\ref{Fig_1}). Here, the maximum likelihood estimator (MLE) for $\rho$ has been studied in~\cite{Brutsche/Rohde} when the two volatility levels $\alpha$ and $\beta$ are known to the statistician. All other scenarios $\gamma\in[0,1)$ are still open and content of this article, where we derive rate of convergence and limiting distribution of the profile MLE for $\rho$ as $n\Delta_n^\gamma\rightarrow\infty$ (cf.~Theorem~\ref{thm:weak_convergence}). Compared to the MLE, the profile MLE is even applicable when the volatility levels $\alpha$ and $\beta$ are unknown. Our study includes the particular case that the observations form a fixed null-recurrent Markov chain ($\gamma=0$). Generally speaking, for all $\gamma\in [0,1)$, we are facing two severe challenges at once:
\begin{itemize}
\item the discontinuity of the likelihood function in the parameter of interest and
\item the missing law of large numbers due to the null-recurrence of the dynamics.
\end{itemize}
The more surprising it is that in our case, the long-run asymptotics ($\gamma\in [0,1)$) can be traced back to the infill asymptotics ($\gamma=1$). Indeed, just when $X_0=\rho_0$, the process $(X_t-\rho_0)_{t\geq 0}$ is discovered to possess a distributional self-similarity under the true parameter $\rho_0$. As a consequence, the distribution of the likelihood function can be expressed in terms of the infill observation scheme. By the strong Markov property and a coupling argument this idea is expanded to arbitrary starting points. The design of this proof strategy is the core of this article. Moreover, this route of proof illuminates the occurrence of the local time of an independent oscillating Brownian motion in $\rho_0$ in the intensity of the profile MLE's doubly stochastic Poisson-type limit which is completely incomprehensible from a long-run asymptotic's perspective. Inter alia, this non-standard limit is caused by the discontinuity of the likelihood function in the threshold parameter $\rho$.

We highlight that the limit distribution does not depend on the value of $\gamma$ in the observation scheme, whereas its convergence rate is identified as $n^{-(1+\gamma)/2}$ and proven to be minimax optimal. Note that the derived $\sqrt{n}$-consistency in case $\gamma=0$ does \textit{not} correspond to the classical parametric rate in regular models as the likelihood function is discontinuous in the parameter $\rho$. Instead, it reflects the difficulty inherent to null-recurrent dynamics as compared to stationary time series settings (cf.~\cite{Chan},~\cite{Yuan/Yau}) and is actually fast, see also~\cite{Delattre/Hoffmann}. Finally, we show that the doubly stochastic Poisson-type limit depends continuously on $\alpha,\beta$ and $\rho$ in the topology of weak convergence (cf.~Theorem~\ref{thm:properties_limit}), enabling the construction of asymptotic confidence intervals for the threshold $\rho$.

To the best of our knowledge, the various literature on estimating a threshold in regime switching models in long-term asymptotics rely on ergodic properties so far. The contributions~\cite{Kutoyants}, \cite{Su/Chan_1}, \cite{Su/Chan_2}, and~\cite{Mazzonetto/Pigato} cover inference in ergodic threshold diffusions as the time horizon of the observed trajectories goes to infinity. The typical limit of such change-points is given as the argmax of a drifted two-sided Brownian motion, see for example Theorem~$4$ in~\cite{Su/Chan_1}. This type of limit also occurs in~\cite{Reiss/Strauch/Trottner} in context of change-point estimation for the stochastic heat equation. A limit involving compound Poisson processes arises in~\cite{Chan} in the context of threshold autoregressive time series models. \cite{Yuan/Yau} study an autoregressive model in which the regime-switching is described by a bivariate threshold variable based on minimizing the minimum description length. The respective objective function is not càdlàg and the limiting distribution theory is developed in a space of pure jump processes. In contrast to our model, both time series contributions work under the assumption of existence of an invariant measure.

The article is organized as follows: In Section~\ref{Section:results} we present our main result on the distributional limit of the profile MLE $\hat{\rho}_n^\gamma$ together with the statistical application including important properties of the limit and a simulation study. The proofs extend over Sections~\ref{Section:Preview}-\ref{Section:Proof_sketch_triple} with some technical details deferred to the Supplementary material (containing Appendices A-F).

\section{Main results}\label{Section:results} 
The transition density $p_t^{\rho,\alpha,\beta}(x,y)$ for the process solving~\eqref{eq:SDE_OBM} is explicitly given as
\begin{align}\label{eq:transition_density}
p_t^{\rho,\alpha,\beta}(x,y) = \begin{cases}
\frac{1}{\sqrt{2\pi t}\alpha}\left[ \exp\left( -\frac{(y-x)^2}{2t\alpha^2}\right) - \frac{\alpha-\beta}{\alpha+\beta}\exp\left(-\frac{(y-2\rho+x)^2}{2t\alpha^2}\right)\right] &\textrm{ for } x<\rho, y\leq \rho,\vspace{0.2cm} \\
\frac{1}{\sqrt{2\pi t}\beta}\left[ \exp\left( -\frac{(y-x)^2}{2t\beta^2}\right) + \frac{\alpha-\beta}{\alpha+\beta}\exp\left(-\frac{(y-2\rho+x)^2}{2t\beta^2}\right)\right] &\textrm{ for } x\geq\rho, y>\rho, \vspace{0.2cm} \\
\frac{2}{\alpha+\beta}\frac{\alpha}{\beta}\frac{1}{\sqrt{2\pi t}}\exp\left( -\frac{1}{2t}\left(\frac{y-\rho}{\beta} - \frac{x-\rho}{\alpha}\right)^2\right) & \textrm{ for } x<\rho< y, \vspace{0.2cm} \\
\frac{2}{\alpha+\beta}\frac{\beta}{\alpha}\frac{1}{\sqrt{2\pi t}}\exp\left( -\frac{1}{2t}\left(\frac{y-\rho}{\alpha} - \frac{x-\rho}{\beta}\right)^2\right) & \textrm{ for } y\leq\rho\leq x. \vspace{0.2cm} \\
\end{cases}
\end{align}
Indeed, by Proposition~$4.2$ in \cite{Blanchard/Roeckner/Russo} the transition density is given as the unique solution of the Kolmogorov forward (or Fokker--Planck) equation in a distributional sense, i.e. it satisfies
\begin{align*}
\int_\R p_t^{\rho,\alpha,\beta}(x,y)\varphi(y) dy = \varphi(x) + \int_0^t\int_\R \frac{1}{2} \sigma_{\rho,\alpha,\beta}(y)^2 p_s^{\rho,\alpha,\beta}(x,y)\varphi''(y) dy ds
\end{align*}
for every $t\in [0,T]$ and any $\varphi\in\mathcal{S}(\R)$, where $\mathcal{S}(\R)$ denotes the Schwartz space of rapidly decreasing functions on $\R$. One easily verifies that~\eqref{eq:transition_density} satisfies this equation.
By the Markovian structure, the likelihood function for the observations $(X_{k\Delta_n^\gamma})_{k=1,\dots,n}$ factorizes into a product of transition densities such that the log-likelihood function is given as
\begin{align}\label{eq:likelihood-function}
L_{n,\gamma}^{\alpha,\beta}(X_0, X_{\Delta_n^\gamma}, \dots, X_{n\Delta_n^\gamma};\rho) := \sum_{k=1}^n \log\left(p_{\Delta_n^\gamma}^{\rho,\alpha,\beta}(X_{(k-1)\Delta_n^\gamma},X_{k\Delta_n^\gamma})\right).
\end{align} 
From the explicit transition density in~\eqref{eq:transition_density} it is read off that for continuous functions $f_k$, the likelihood function can be written as 
\[  \sum_{k=0}^{n+1} f_k(\rho,\alpha,\beta)\1_{I_k},\]
with $I_0 =\{-\infty <\rho<X_{0:n}\}$, $I_{n+1}=\{X_{n\Delta_n^\gamma:n}\leq \rho <\infty)\}$ and $I_k=\{X_{(k-1)\Delta_n^\gamma:n}\leq\rho< X_{k\Delta_n^\gamma:n}\}$ for $k=1,\dots, n$, where $(X_{0:n},X_{\Delta_n^\gamma:n},\dots,X_{n\Delta_n^\gamma:n})$ denoting the order statistic of the observations. For given $(\alpha,\beta)$ this function is not continuous, but càdlàg in the parameter $\rho$. Let $f:\R\times \R_{>0}^2\rightarrow\R$ be any function of the form $f(x,y) = \sum_{m=1}^N f_m(x,y)\1_{\{x_{m-1}\leq x< x_m\}}$ with $f_m:\R\times \R_{>0}^2\rightarrow\R$ being continuous. Then we define
\[ \underset{x\in \R\times\R_{>0}^2}{\mathrm{Argsup}\ } f(x) = \left\{(x_1,x_2)\in \R\times\R_{>0}^2:\ \max\left\{\lim_{u\nearrow x_1} f(u,x_2), f(x_1,x_2)\right\} = \sup_{z\in \R\times\R_{>0}^2} f(z)\right\} \]
and write $\arg\sup$ if $\mathrm{Argsup}$ is a singleton. Note that $\mathrm{Argsup}$ is always a subset of $\R^3$ but may be empty. Moreover, the special case of càdlàg functions $g:\R\rightarrow\R$ is part of this definition by considering $f(x_1,x_2,x_3)=g(x_1)$. We finally define the profile MLE $\hat{\rho}_n^\gamma$ to be the first component of some measurably selected representative of
\begin{align}\label{eq:argsup_triple}
\underset{\rho\in\R, \alpha,\beta >0}{\mathrm{Argsup}\ }\ L_{n,\gamma}^{\alpha,\beta}\left(X_0, X_{\Delta_n^\gamma}, \dots, X_{n\Delta_n^\gamma}; \rho\right).
\end{align}
By the consistency result in Proposition~\ref{prop:rate_of_triple_MLE} this set is non-empty with probability tending to one.

In order to properly describe the limit of the profile MLE $\hat{\rho}_n^\gamma$, we need to introduce the limiting process of a suitably rescaled version of the log-likelihood function that corresponds to the likelihood function in~\eqref{eq:likelihood-function}.
For this purpose, we denote by $L_t^z(Y)$ the symmetric local time at point $t$ in $z$ for a continuous local martingale $Y$, i.e. by Corollary~$1.9$ in Chapter~VI of~\cite{Revuz/Yor},
\begin{align}\label{eq:local_time}
L_t^z(Y) = \lim_{\epsilon\searrow 0 } \frac{1}{2\epsilon} \int_0^t \1_{(z-\epsilon,z+\epsilon)}(Y_s) d\langle Y\rangle_s,
\end{align}
with $\langle Y\rangle$ denoting the (predictable) quadratic variation process of $Y$. Now, we define two independent standard Poisson processes $N$ and $N'$ on $\R_{\geq 0}$ and an oscillating Brownian motion $X^{\rho,\alpha,\beta}$ starting in $\rho$ with diffusion coefficient $\sigma_{\rho,\alpha,\beta}$ (i.e. a process satisfying~\eqref{eq:SDE_OBM} with $x_0=\rho$). Based on these three processes, we define $\ell^{\rho,\alpha,\beta} = (\ell^{\rho,\alpha,\beta}(z))_{z\in\R}$ via
\begin{align}\label{eq:def_ell}
\begin{split}
\ell^{\rho,\alpha,\beta}(z) &= \1_{\{z< 0\}} \left( |z| \frac{\beta^2-\alpha^2}{\alpha^2\beta^2} L_1^{\rho}(X^{\rho,\alpha,\beta}) +\log\left(\frac{\alpha^2}{\beta^2}\right) N'\left(\frac{|z|L_1^{\rho}(X^{\rho,\alpha,\beta})}{\alpha^2}-\right)\right)  \\
&\hspace{0.5cm} +  \1_{\{z\geq 0\}} \left( |z| \frac{\alpha^2-\beta^2}{\alpha^2\beta^2} L_1^{\rho}(X^{\rho,\alpha,\beta}) +\log\left(\frac{\beta^2}{\alpha^2}\right)N\left(\frac{|z|L_1^{\rho}(X^{\rho,\alpha,\beta})}{\beta^2}\right)\right).
\end{split}
\end{align} 
We use the left-continuous version $N'(\bullet -)$ to ensure that $\ell^{\rho,\alpha,\beta}$ is càdlàg. For $z<0$, we may write ($z\geq 0$ can be decomposed analogously)
\begin{align*}
\ell^{\rho,\alpha,\beta}(z) &= |z| \left( \frac{\beta^2-\alpha^2}{\alpha^2\beta^2} + \frac{1}{\alpha^2}\log\left(\frac{\alpha^2}{\beta^2}\right)\right) L_1^{\rho}(X^{\rho,\alpha,\beta}) \\
&\hspace{1cm} +\log\left(\frac{\alpha^2}{\beta^2}\right) \left( N'\left(\frac{|z|L_1^{\rho}(X^{\rho,\alpha,\beta})}{\alpha^2}-\right) - \frac{|z|L_1^{\rho}(X^{\rho,\alpha,\beta})}{\alpha^2}\right),
\end{align*}
which is a process with negative drift plus a compensated Poisson process in case $\alpha\neq\beta$, where the first statement follows from the fact that $x\mapsto x-1-x\log(x)$ has a unique maximum in $x=1$ and vanishes in this point. While the limiting drift being a multiple of $|z|$ is locally of the same structure as that of classical estimators for change points in time, the stochastic term is given as a two-sided Poisson process rather than a two-sided Brownian motion for our change in the state space. Recall that the profile MLE $\hat{\rho}_n^\gamma$ is the first component of an element $(\hat{\rho}_n^\gamma,\hat{\alpha}_n^\gamma,\hat{\beta}_n^\gamma)$ of~\eqref{eq:argsup_triple} based on discrete observations from the process in~\eqref{eq:SDE_OBM}.

\begin{thm}\label{thm:weak_convergence}
Let $(\ell^{\rho,\alpha,\beta}(z))_{z\in\R}$ be the process given in~\eqref{eq:def_ell} and $\gamma\in [0,1)$. Then $(\hat{\alpha}_n^\gamma,\hat{\beta}_n^\gamma)$ is $\sqrt{n}$-consistent and 
\begin{align}\label{eq:weak_limit}
n^{\frac{1+\gamma}{2}}\left(\hat{\rho}_n^\gamma-\rho_0\right) \stackrel{\mathcal{D}}{\longrightarrow}\ \underset{z\in\R}{\arg\sup\ }\ell^{\rho_0,\alpha_0,\beta_0}(z). 
\end{align}
\end{thm}

A route of proof is given in three steps in Subsection~\ref{Subsection:Proof_Theorem1} whose details extend over Sections~\ref{Section:StepI} and~\ref{Section:StepII}. Remarkably, the rate of $(\hat{\alpha}_n^\gamma,\hat{\beta}_n^\gamma)$ does not depend on $\gamma$, whereas the rate for $\hat{\rho}_n^\gamma$ interpolates between $n^{-1}$ for $\gamma=1$ (infill asymptotics) and $n^{-1/2}$ for $\gamma=0$ (null-recurrent Markov chain). It follows from Proposition~\ref{prop:minimax_lower} in the supplementary material~\ref{Section:Proof_1A} that this rate is indeed mini\-max. The limiting distribution is the same for all $\gamma$, i.e for all step sizes $\Delta_n^\gamma$ of the observations. It coincides with the limit in Theorem~$1.1$ in~\cite{Brutsche/Rohde} for the infill asymptotics which has been derived in the sense of stable convergence for the MLE on the event $\{L_1^{\rho_0}(X)>0\}$ when $\alpha$ and $\beta$ are known. This is surprising, since the limit behavior for $\gamma\in [0,1)$ depends heuristically on recurrence properties of the observed Markov chain which cannot contribute to the analysis for $\gamma=1$.

As compared to $\gamma=1$, the long-term asymptotics for $\gamma\in[0,1)$ of this article require a completely different proof. For the infill asymptotics, we were able to prove the stronger mode of stable convergence that could be traced back to the uniform stochastic convergence of the semimartingale characteristics of the \textit{sequential} log-likelihood process w.r.t. discretized filtrations. For $\gamma\in [0,1)$, however, this argument is not available any longer and the problem of identifying the joint distribution of the random drift and martingale part of the log-likelihood function arises. Our crucial finding is that in the very special case that $(X_t)_{t\geq 0}$ is started in the true parameter $X_0=\rho_0$, the process $(X_t-\rho_0)_{t\geq 0}$ possesses a desirable distributional self-similarity. This enables to relate the weaker mode of convergence in distribution back to the problem in infill asymptotics, where this interplay is well-understood. Recall that the limit theorem in the infill asymptotics is restricted to the event $\{L_1^{\rho_0}(X)>0\}$, but $\Pr(L_1^{\rho_0}(X)>0)=1$ iff $X_0=\rho_0$. Indeed, the latter equivalence is well-known if $X$ is replaced by a Brownian motion and carries over to OBM by the Dambis--Dubins--Schwarz theorem (cf. Theorem~$19.4$ in~\cite{Kallenberg}). On basis of the strong Markov property, the artificial constellation of starting in $\rho_0$ is finally overcome by a coupling argument. Proving that the increasing number of observations before the coupling time does not affect the limit for $\gamma\in (0,1)$ necessitates a quantitative maximal deviation bound for oscillating Brownian local time (cf.~Lemma~\ref{lemma:probability_A3}).\medskip

\textbf{Statistical implications.}
A priori, the result in Theorem~\ref{thm:weak_convergence} is not directly applicable for statistical purposes as the limit depends on the unknown parameters $(\rho_0,\alpha_0,\beta_0)$. The continuity result in the next theorem paves the way for statistical inference.

\begin{thm}\label{thm:properties_limit}
The limit distribution in~\eqref{eq:weak_limit} is absolutely continuous with respect to Lebesgue measure and depends continuously on $(\rho_0,\alpha_0,\beta_0)$ in the topology of weak convergence.
\end{thm}

A route of proof is given in Subsection~\ref{Subsection:Proof_Theorem2} and the details are provided in Section~\ref{Section:proof_thm2}. It crucially exploits the particular structure of doubly stochastic Poisson process~\eqref{eq:def_ell} which is, conditioned on the local time component, a drifted Poissonian martingale.
This theorem justifies to replace the unknown quantities by the triple MLE $(\hat{\rho}_n^\gamma,\hat{\alpha}_n^\gamma,\hat{\beta}_n^\gamma)$. For $\kappa\in (0,1)$ define $\hat{q}_{\kappa}^n$ to be the $\kappa$-quantile of 
\begin{align}\label{eq:argsup_limit}
\underset{z\in\R}{\arg\sup\ }\ell^{\hat{\rho}_n^\gamma,\hat{\alpha}_n^\gamma, \hat{\beta}_n^\gamma}(z)
\end{align} 
and $q_\kappa$ the $\kappa$-quantile of $\arg\sup_{z\in\R}\ell^{\rho_0,\alpha_0,\beta_0}(z)$. By the properties of the limit given in Theorem~\ref{thm:properties_limit}, Lemma~$21.2$ in~\cite{Vaart} reveals
\[ \hat{q}_\kappa^n \longrightarrow_{\Pr} q_\kappa\]
for any $\kappa\in (0,1)$. In particular, we obtain that
\begin{align}\label{eq:confidence_interval}
\left[ \hat{\rho}_n^\gamma - n^{-\frac{1+\gamma}{2}} \hat{q}_{1-\kappa/2}^n, \hat{\rho}_n^\gamma - n^{-\frac{1+\gamma}{2}} \hat{q}_{\kappa/2}^n\right] 
\end{align} 
is an asymptotic $(1-\kappa)$-confidence interval for $\rho_0$. Table~\ref{Table_prop_rejection} displays the proportion of samples lying in the $(1-\kappa)$-confidence interval~\eqref{eq:confidence_interval}. One can observe that the coverage is already good for small sample sizes $n$, although the quantiles $\hat{q}_\kappa^n$ were estimated on basis of only $2,000$ samples of~\eqref{eq:argsup_limit}.

\vspace{-0.3cm}

\begin{center}
\begin{table}[h]
\begin{tabular}{c|cc}
\textbf{$n$} & \textbf{$\kappa=0.05$} & \textbf{$\kappa = 0.01$} \\
\hline
$500$  & $0.946$ & $0.992$ \\
$1000$ & $0.955$ & $0.995$ 
\end{tabular}\vspace{0.2cm}
\quad 
\begin{tabular}{c|cc}
\textbf{$n$} & \textbf{$\kappa=0.05$} & \textbf{$\kappa = 0.01$} \\
\hline
$500$ & $0.924$ & $0.986$ \\
$1000$ & $0.960$ & $0.996$ 
\end{tabular}\smallskip
\caption{\small{Proportions of samples lying in the $(1-\kappa)$-confidence interval~\eqref{eq:confidence_interval} for different values of $n$ when sampling $N=1000$ paths with parameter $\rho_0=x_0=0$ and $(\alpha,\beta)=(0.5,0.65)$ on a grid of size $1/(100n)$. In the left table we set $\gamma=0.3$ and in the right one $\gamma=0.8$. }}\label{Table_prop_rejection}
\end{table}
\end{center}

\vspace{-1cm}

In order to effectively sample the limit and determine its quantiles, we estimated the local time of the OBM appearing in the limit from simulated discrete sample paths. Here, we used the consistent local time estimator given in Section~$3.1$ of~\cite{Mazzonetto}.

\section{Preview of the proofs of Theorem~\ref{thm:weak_convergence} and Theorem~\ref{thm:properties_limit}}\label{Section:Preview}

In this section we present the proof ideas of the two theorems of Section~\ref{Section:results}, where Subsection~\ref{Subsection:Proof_Theorem1} describes the high-level arguments of the proof of Theorem~\ref{thm:weak_convergence} in three steps and Subsection~\ref{Subsection:Proof_Theorem2} the route of proof of Theorem~\ref{thm:properties_limit}.

\subsection{Proof of Theorem~\ref{thm:weak_convergence}}\label{Subsection:Proof_Theorem1}
The proof is conducted along the following three steps.

\smallskip
\noindent
\textbf{Step I.} We assume $\alpha_0$ and $\beta_0$ to be known and analyze the MLE $\hat{\rho}_n^\gamma(\alpha_0,\beta_0)$ under this assumption. Note that $\hat{\rho}_n^\gamma(\alpha_0,\beta_0)=\rho_0+\hat{\theta}_n^\gamma$, where $\hat{\theta}_n^\gamma$ is a well-defined representative of $\mathrm{Argsup}_{\theta\in\R} \ell_{n,\gamma}^{\alpha_0,\beta_0}(\theta)$, where
\begin{align}\label{eq:def_ell_rho}
\ell_{n,\gamma}^{\alpha_0,\beta_0}(\theta) = \sum_{k=1}^n \log\left(\frac{p_{\Delta_n^\gamma}^{\rho_0+\theta,\alpha_0,\beta_0}(X_{(k-1)\Delta_n^\gamma},X_{k\Delta_n^\gamma})}{p_{\Delta_n^\gamma}^{\rho_0,\alpha_0,\beta_0}(X_{(k-1)\Delta_n^\gamma},X_{k\Delta_n^\gamma}}\right). 
\end{align}

Our crucial observation is that if in addition the process even starts in the true threshold $\rho_0$, the following self-similarity property of the oscillating Brownian motion holds true.

\begin{lemma}\label{lemma:self_similarity}
Let $X=(X_t)_{t\geq 0}$ be a solution to~\eqref{eq:SDE_OBM} with $\rho=\rho_0$, starting at $x_0=\rho_0$. Then for any $c>0$
\[ \left( X_{t} - \rho_0\right)_{t\geq 0}\ \stackrel{\mathcal{D}}{=} \left( c^{-\frac12}\left( X_{ct}-\rho_0\right)\right)_{t\geq 0}.\]
\end{lemma}

By setting $c=n\Delta_n^\gamma$, this allows to conclude for the $n$-dimensional distributions 
\begin{align}\label{eq:distributional_identity}
\left(X_{k\Delta_n^\gamma} - \rho_0\right)_{k=0,1,\dots, n} \stackrel{\mathcal{D}}{=} \left( n^{\frac{1-\gamma}{2}} (X_{k/n}-\rho_0)\right)_{k=0,1,\dots, n}.
\end{align} 
From this property and the explicit form of the transition density $p_t^{\rho,\alpha,\beta}(x,y)$ we can derive the following result (a proof is given in Section~\ref{Section:StepI})

\begin{lemma}\label{lemma:distributional_identity_normalized_logLikelihood}
Let $x_0=\rho_0$ and $X=(X_t)_{t\geq 0}$ be a corresponding solution to~\eqref{eq:SDE_OBM}. Then,
\[ \left( \ell_{n,\gamma}^{\alpha_0,\beta_0}\left( z n^{-(1+\gamma)/2}\right) \right)_{z\in\R} \stackrel{\mathcal{D}}{=}\ \left( \sum_{k=1}^n \log\left(\frac{p_{1/n}^{\rho_0+z/n,\alpha_0,\beta_0}\left(X_{(k-1)/n},X_{k/n}\right)}{p_{1/n}^{\rho_0,\alpha_0,\beta_0}\left(X_{(k-1)/n},X_{k/n}\right)}\right) \right)_{z\in\R}.\]
\end{lemma}

As by definition
\[ n^{\frac{1+\gamma}{2}}\left( \hat{\rho}_n^\gamma(\alpha_0,\beta_0)-\rho_0\right) \in \underset{z\in\R}{\mathrm{Argsup}\ } \ell_{n,\gamma}^{\alpha_0,\beta_0}\left( z n^{-(1+\gamma)/2}\right), \]
Lemma~\ref{lemma:distributional_identity_normalized_logLikelihood} enables us to transfer results derived for infill asymptotics with fixed time horizon to arbitrary equidistant observation schemes with growing time horizon. Note in particular that in case $\gamma=0$ the resulting Markov chain is null-recurrent. In particular
\begin{align}\label{eq:limit_known_alpha/beta}
n^{\frac{1+\gamma}{2}}\left( \hat{\rho}_n^\gamma(\alpha_0,\beta_0)-\rho_0\right) \stackrel{\mathcal{D}}{\longrightarrow}\ \underset{z\in\R}{\arg\sup\ }\ell^{\rho_0,\alpha_0,\beta_0}(z) 
\end{align} 
is then a consequence of Theorem~$1.1$ in~\cite{Brutsche/Rohde}. Note that for the stable convergence result in the infill asymptotics, conditioning on the event $\{L_1^{\rho_0}(X)>0\}$ was crucial, but $\Pr(L_1^{\rho_0}(X)>0)=1$ in case $X_0=\rho_0$.
It should be noted that no martingale central limit theorem is available for the normalized log-likelihood process $\ell_{n,\gamma}^{\alpha_0,\beta_0}(zn^{(-(1+\gamma)/2})$. In the case of infill asymptotics, however, stable convergence could be established in place of convergence in distribution, which allowed the problem to be reduced to the convergence of the semimartingale characteristics of the normalized log-likelihood process $\ell_{n,\gamma}^{\alpha_0,\beta_0}(z/n)$. The identification of the limit was possible, since given $L_1^{\rho_0}(X)$, the limiting drift turned out to be deterministic and thus could be combined with a weak convergence of the martingale part (compare Slutzky's theorem).

\smallskip
\noindent
\textbf{Step II.} Here, we remove the artificial assumption $x_0=\rho_0$ and establish~\eqref{eq:limit_known_alpha/beta} for observations of an oscillating Brownian motion started in arbitrary $x_0\in\R$. This is done by a coupling argument, where the likelihood functions built from different OBMs, one started in $x_0$, the other in $\rho_0$, are related to each other. To this aim, denote by $X^{x_0}$ and $X^{\rho_0}$ two solutions of~\eqref{eq:SDE_OBM} with $\sigma_{\rho_0,\alpha_0,\beta_0}$ starting in $x_0$ and $\rho_0$, respectively. To define our coupling, introduce the stopping time
\begin{align}\label{eq:coupling_time}
T_{\rho_0}^{x_0}:= \inf\left\{t\geq 0: X_t^{x_0} = X_t^{\rho_0}\right\}.
\end{align} 
As $X^{x_0}$ and $X^{\rho_0}$ are on natural scale (see Theorem~$33.7$ in~\cite{Kallenberg}) and recurrent, the argument given in the first lines on p.$85$ of~\cite{Lindvall} yields $\Pr(T_{\rho_0}^{x_0}<\infty)=1$. We now define the process $\overline{X}^c$ as
\begin{align}\label{eq:coupled_process}
\overline{X}_t^c = X_t^{x_0}\1_{\{t\leq T_{\rho_0}^{x_0}\}} + X_t^{\rho_0}\1_{\{t> T_{\rho_0}^{x_0}\}},
\end{align} 
see Figure~\ref{fig:coupling} for an illustration. By the strong Markov property of the oscillating Brownian motion (see~\cite{Keilson/Wellner} and~\cite{Ito/McKean}, p.169), $\overline{X}^c$ is again an OBM started in $x_0$, so the distribution of $\overline{X}^c$ and $X^{x_0}$ coincide. In particular, the distribution of $n^{(1+\gamma)/2}(\hat{\rho}_n(X^{x_0})-\rho_0)$ and $n^{(1+\gamma)/2}(\hat{\rho}_n(\overline{X}^c)-\rho_0)$ coincide and we show that the latter converges weakly to $\arg\sup_{z\in\R} \ell^{\rho_0,\alpha_0,\beta_0}(z)$. To this aim, we will establish the two following claims:
\begin{itemize}
\item[(1A)] Denote by $\ell_{n,\gamma}^{\alpha_0,\beta_0}(zn^{-(1+\gamma)/2}, \overline{X}^c)$ the random variable $\ell_{n,\gamma}^{\alpha_0,\beta_0}(zn^{-(1+\gamma)/2})$ from~\eqref{eq:def_ell_rho} for the process $\overline{X}^c$. Then for any $K>0$, we have
\[ \left(\ell_{n,\gamma}^{\alpha_0,\beta_0}(zn^{-(1+\gamma)/2}, \overline{X}^c)\right)_{z\in [-K,K]} \stackrel{\mathcal{D}}{\longrightarrow} \left( \ell^{\rho_0,\alpha_0,\beta_0}(z)\right)_{z\in [-K,K]} \]
for the process $\ell$ given in~\eqref{eq:def_ell}.
\item[(2A)] The sequence $(n^{(1+\gamma)/2}(\hat{\rho}_n(\overline{X}^c)-\rho_0))_{n\in\N}$ is tight.
\end{itemize}

For $\gamma=0$, (1A) is a consequence of the fact that only finitely many observations occur before the coupling time $T_{\rho_0}^{x_0}$. The proof is conducted in Section~\ref{Section:StepII}. For $\gamma\in (0,1)$, however, the number of observations before the coupling time $T_{\rho_0}^{x_0}$ is increasing in $n$, see Figure~\ref{fig:coupling}. Their contribution to the limit is still shown to be negligible but this is significantly harder to prove. Indeed, it relies on a decomposition of the log-likelihood function into a drift and martingale term which have to be treated thoroughly. Additionally, tight quantitative maximal deviation bounds for oscillating Brownian local time are derived for this purpose (Lemma~\ref{lemma:probability_A3}).

\begin{center}
\begin{figure}[h]
\includegraphics[width=13cm, height=4.5cm]{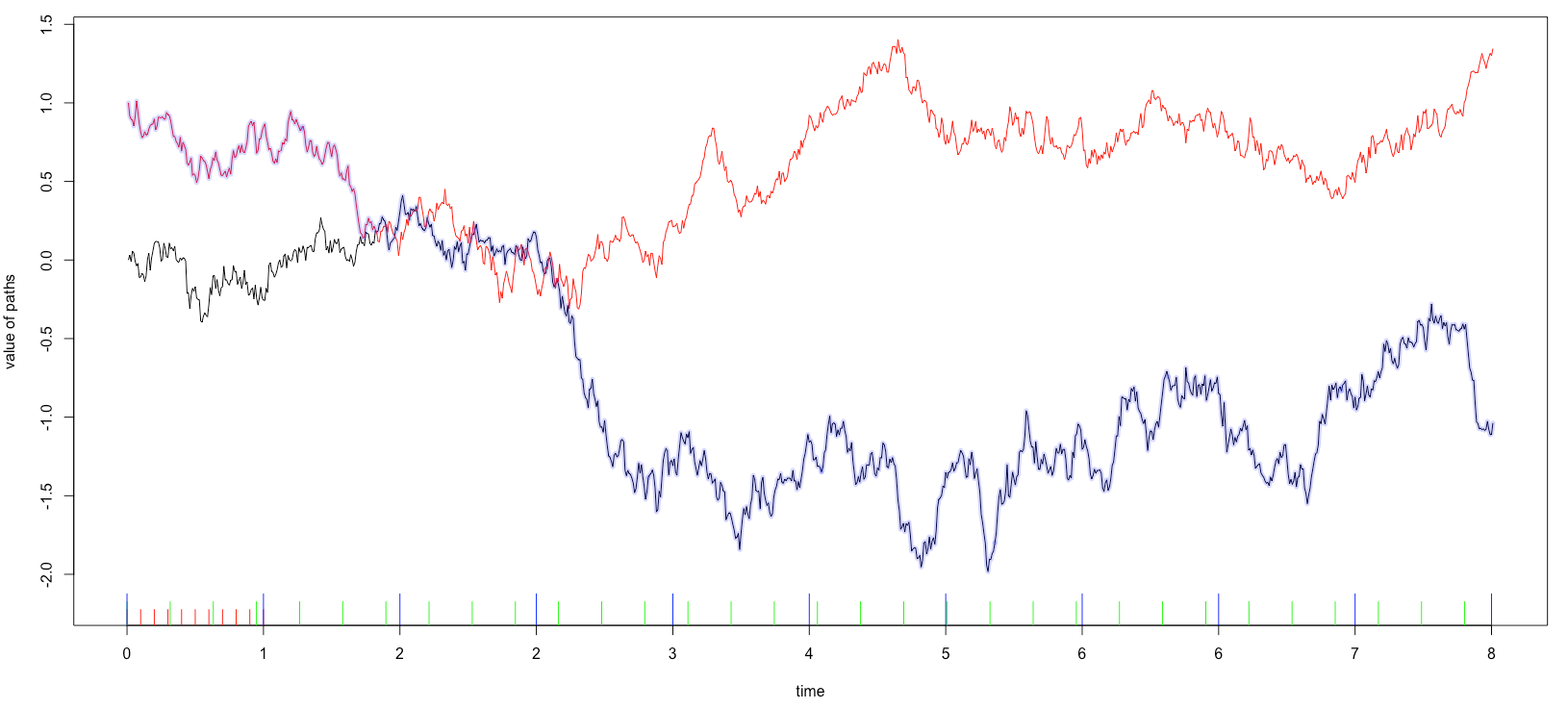}
\caption{\small{Simulation of a path $X^{\rho_0}$ in black for $\rho_0=0$ and $X^{x_0}$ in red with $x_0=1$. The coupled process $\overline{X}^c$ is then given as the blue one. The ticks on the x-axis illustrate the sampling schemes for $\gamma=0$ (blue), $\gamma=1/2$ (green) and $\gamma=1$ (red), compare Figure~\ref{Fig_1}.}}\label{fig:coupling}
\end{figure}
\end{center}

\vspace{-0.3cm}

By (1A) and Lemma~$3$ in Section~$16$ of~\cite{Billingsley_2} we have 
\begin{align*}
\left(\ell_{n,\gamma}^{\alpha_0,\beta_0}(zn^{-(1+\gamma)/2}, \overline{X}^c)\right)_{z\in\R} \stackrel{\mathcal{D}}{\longrightarrow}\ \left( \ell^{\rho_0,\alpha_0,\beta_0}(z)\right)_{z\in\R},
\end{align*} 
where the Skorohod space $\mathcal{D}(\R)$ is endowed with the topology of Skorohod convergence on compact sets. From the tightness in (2A) and Theorem~$3.12$ in~\cite{Ferger} we then conclude
\[ \underset{z\in\R}{\arg\sup\ }\ell_{n,\gamma}^{\alpha_0,\beta_0}(zn^{-(1+\gamma)/2}, \overline{X}^c) \stackrel{\mathcal{D}}{\longrightarrow}\ \underset{z\in\R}{\arg\sup\ }\ell^{\rho_0,\alpha_0,\beta_0}(z) \]
and hence 
\[ n^{\frac{1+\gamma}{2}}\left(\hat{\rho}_n(X^{x_0})-\rho_0\right)\stackrel{\mathcal{D}}{=}n^{\frac{1+\gamma}{2}}\left(\hat{\rho}_n(\overline{X}^c)-\rho_0\right)\stackrel{\mathcal{D}}{\rightarrow} \arg\sup_{z\in\R}\ell^{\rho_0,\alpha_0,\beta_0}(z), \]
which is~\eqref{eq:limit_known_alpha/beta} for arbitrary $x_0\in\R$. A proof of (1A) and (2A) is deferred to Section~\ref{Section:StepII}.\medskip

\smallskip
\noindent
\textbf{Step III.} Remove assumption that $\alpha_0$ and $\beta_0$ are known. As a first step, we derive the rates of convergence of the triple MLE $(\hat{\rho}_n^\gamma, \hat{\alpha}_n^\gamma,\hat{\beta}_n^\gamma)$ that is given as an element of~\eqref{eq:argsup_triple}. Here and subsequently, the subscript $(\rho_0,\alpha_0,\beta_0)$ in the probability $\Pr_{\rho_0,\alpha_0,\beta_0}$ indicates that $X$ solves~\eqref{eq:SDE_OBM} with $\sigma_{\rho_0,\alpha_0,\beta_0}$.

\begin{proposition}\label{prop:rate_of_triple_MLE}
Let $\gamma\in [0,1)$. Then $(\hat{\alpha}_n^\gamma,\hat{\beta}_n^\gamma)$ are $\sqrt{n}$-consistent and $\hat{\rho}_n^\gamma$ is $n^{(1+\gamma)/2}$ consistent, i.e.
\[ \limsup_{K\to\infty}\limsup_{n\to\infty} \Pr_{\rho_0,\alpha_0,\beta_0}\left( \max\left\{ n^{\frac{1+\gamma}{2}} |\hat{\rho}_n^\gamma-\rho_0|, \sqrt{n}|\hat{\alpha}_n^\gamma-\alpha_0|, \sqrt{n}|\hat{\beta}_n^\gamma-\beta_0|\right\}>K \right) =0.\]
On the event $\{L_1^{\rho_0}(X)>0\}$, the result continues to hold for $\gamma=1$.
\end{proposition}

In the high-frequency setting $\gamma=1$, the $\sqrt{n}$-consistency for estimation of $\alpha_0$ and $\beta_0$ was proven for different estimators in~\cite{Lejay/Pigato} when the threshold $\rho_0$ is known to the statistician. 
For the proof of Proposition~\ref{prop:rate_of_triple_MLE}, we introduce the normalized log-likelihood function
\begin{align}\label{eq:log-likelihood_triple}
\ell_{n,\gamma}(\theta_\rho,\theta_\alpha,\theta_\beta) = \sum_{k=1}^n \log\left(\frac{p_{\Delta_n^\gamma}^{\rho_0+\theta_\rho,\alpha_0+\theta_\alpha,\beta_0+\theta_\beta}(X_{(k-1)\Delta_n^\gamma},X_{k\Delta_n^\gamma})}{p_{\Delta_n^\gamma}^{\rho_0,\alpha_0,\beta_0}(X_{(k-1)\Delta_n^\gamma},X_{k\Delta_n^\gamma})}\right)
\end{align} 
that is the three-parameter analogue to $\ell_{n,\gamma}^{\alpha,\beta}(\theta_\rho)$ in~\eqref{eq:def_ell_rho} and use that by definition,
\begin{align*}
&\left(n^{\frac{1+\gamma}{2}} (\hat{\rho}_n^\gamma-\rho_0), \sqrt{n}(\hat{\alpha}_n^\gamma-\alpha_0), \sqrt{n}(\hat{\beta}_n^\gamma-\beta_0) \right)\\
&\hspace{2cm} \in \underset{(\theta_\rho,\theta_\alpha,\theta_\beta)\in\Upsilon_n}{\mathrm{Argsup}\ }\ \ell_{n,\gamma}\left( n^{-(1+\gamma)/2}\theta_\rho, \theta_\alpha/\sqrt{n},\theta_\beta/\sqrt{n}\right), 
\end{align*} 
where $\Upsilon_n := \R\times (-\sqrt{n}\alpha_0,\infty)\times (-\sqrt{n}\beta_0,\infty)$. Here, we first prove consistency of the triple MLE for $\gamma\in (0,1]$, which is outlined in Section~\ref{Section:Proof_sketch_triple} with details provided in Appendix~\ref{Section:triple_details}. By making use of the self-similarity of Lemma~\ref{lemma:self_similarity}, the result then transfers to the null-recurrent case $\gamma=0$ exactly as in the proof of (2A) given in Section~\ref{Section:StepII}.\\

Secondly, we denote by $\Pi_1:\R^3\rightarrow\R$ the projection onto the first coordinate. Let $\epsilon>0$, then by Proposition~\ref{prop:rate_of_triple_MLE} there exists $L>0$ such that with probability larger than $1-\epsilon$,
\begin{align}\label{eq:projection_tripleArgsup}
\begin{split}
&\Pi_1\left( \underset{(\theta_\rho,\theta_\alpha,\theta_\beta)\in\Upsilon_n}{\mathrm{Argsup}\ }\  \ell_{n,\gamma}\left( n^{-(1+\gamma)/2}\theta_\rho, \theta_\alpha/\sqrt{n},\theta_\beta/\sqrt{n}\right) \right)\\
&\hspace{2cm} = \underset{\theta_\rho\in\R}{\mathrm{Argsup}\ }\ \sup_{-L\leq \theta_\alpha,\theta_\beta\leq L}\ell_{n,\gamma}\left( n^{-(1+\gamma)/2}\theta_\rho, \theta_\alpha/\sqrt{n},\theta_\beta\sqrt{n}\right)
\end{split}
\end{align}
 This can be seen as follows: By the special structure of $\ell_{n,\gamma}$ (as described in Section~\ref{Section:results} right after~\eqref{eq:likelihood-function}) the function $\theta_\rho\mapsto \sup_{\theta_\alpha,\theta_\beta\in [-L,L]}\ell_{n,\gamma}$ is càdlàg by Theorem~$14.30$ of~\cite{Aliprantis/Border}, meaning that the second Argsup is well-defined and coincides with the projection of the first one as the supremum over $\theta_\alpha,\theta_\beta \in [-L,L]$ is attained.  In order to proceed, we observe the decomposition
\begin{align}\label{eq:decomposition_joint_ell}
\ell_{n,\gamma}(\theta_\rho,\theta_\alpha,\theta_\beta) = \ell_{n,\gamma}^{\alpha_0,\beta_0}(\theta_\rho) + \ell_{n,\gamma}^2(\theta_\rho,\theta_\alpha,\theta_\beta),
\end{align} 
where $\ell_{n,\gamma}^{\alpha_0,\beta_0}(\theta_\rho)$ is given in~\eqref{eq:def_ell_rho} and
\begin{align*}
\ell_{n,\gamma}^2(\theta_\rho,\theta_\alpha,\theta_\beta) &:= \sum_{k=1}^n \log\left(\frac{p_{\Delta_n^\gamma}^{\rho_0+\theta_\rho,\alpha_0+\theta_\alpha,\beta_0+\theta_\beta}(X_{(k-1)\Delta_n^\gamma},X_{k\Delta_n^\gamma})}{p_{\Delta_n^\gamma}^{\rho_0+\theta_\rho,\alpha_0,\beta_0}(X_{(k-1)\Delta_n^\gamma},X_{k\Delta_n^\gamma})}\right).
\end{align*}
On the right-hand side of~\eqref{eq:decomposition_joint_ell}, we can add and subtract $\ell_{n,\gamma}^2(0,\theta_\alpha,\theta_\beta)$ which obviously does not depend on $\theta_\rho$ anymore, i.e. arrive at
\[ \ell_{n,\gamma}(\theta_\rho,\theta_\alpha,\theta_\beta) = \ell_{n,\gamma}^{\alpha_0,\beta_0}(\theta_\rho) + \ell_{n,\gamma}^2(0,\theta_\alpha,\theta_\beta)+ \ell_{n,\gamma}^3(\theta_\rho,\theta_\alpha,\theta_\beta)\]
for
\[ \ell_{n,\gamma}^3(\theta_\rho,\theta_\alpha,\theta_\beta) := \ell_{n,\gamma}^2(\theta_\rho,\theta_\alpha,\theta_\beta)-\ell_{n,\gamma}^2(0,\theta_\alpha,\theta_\beta)\]
and then obtain from~\eqref{eq:projection_tripleArgsup},
\begin{align*}
&\Pi_1\left( \underset{(\theta_\rho,\theta_\alpha,\theta_\beta)\in\Upsilon_n}{\mathrm{Argsup}\ }\  \ell_{n,\gamma}\left( n^{-(1+\gamma)/2}\theta_\rho, \theta_\alpha/\sqrt{n},\theta_\beta/\sqrt{n}\right) \right)\\
&\hspace{0.5cm} = \underset{\theta_\rho\in\R}{\mathrm{Argsup}\ }\ \left( \ell_{n,\gamma}^{\alpha_0,\beta_0}(n^{-(1+\gamma)/2}\theta_\rho) + \sup_{-L\leq \theta_\alpha,\theta_\beta\leq L}\ell_{n,\gamma}^3(n^{-(1+\gamma)/2}\theta_\rho,\theta_\alpha/\sqrt{n},\theta_\beta/\sqrt{n}) \right).
\end{align*}
This representation motivates the next Proposition which establishes convergence of the function on the left-hand side of which $\mathrm{Argsup}_{\theta_\rho}$ is taken over towards $\ell_{n,\gamma}^{\alpha_0,\beta_0}$. Its proof is deferred to Section~\ref{Section:App_approximation_ell} in the supplementary material. Note that each summand of $\ell_{n,\gamma}^2(\theta_\rho,\theta_\alpha,\theta_\beta)$ is small when $\theta_\alpha,\theta_\beta=\mathcal{O}(n^{-1/2})$ as the density ratio is then approximately one. The main difficulty here is to find a representation of $\ell_{n,\gamma}^2(\theta_\rho,\theta_\alpha,\theta_\beta)$ for $\theta_\alpha,\theta_\beta=\mathcal{O}(n^{-1/2})$ that is explicit enough to bound the difference when evaluated in $\theta_\rho=0$ and $\theta_\rho=zn^{-(1+\gamma)/2}$.

\begin{proposition}\label{prop:approximation_ell}
Let $\gamma\in [0,1)$. Then for every $K>0$,
\[ \sup_{z\in [-K,K]}  \sup_{-L\leq \theta_\alpha,\theta_\beta\leq L} \left|\ell_{n,\gamma}^3\left(z n^{-(1+\gamma)/2}, \theta_\alpha/\sqrt{n},\theta_\beta/\sqrt{n}\right)  \right| \longrightarrow_{\Pr_{\rho_0,\alpha_0,\beta_0}} 0.\]
\end{proposition}

Finally, the claim of Theorem~\ref{thm:weak_convergence} follows by using Proposition~\ref{prop:rate_of_triple_MLE} as an analogue of (2A) and Proposition~\ref{prop:approximation_ell} as an analogue of (1A) and arguing along the lines of Step II.

\subsection{Proof of Theorem~\ref{thm:properties_limit}}\label{Subsection:Proof_Theorem2}
Proving that the limiting distribution in~\eqref{eq:weak_limit} is absolutely continuous with respect to Lebesgue measure requires the observation that the argsup of $\ell^{\rho_0,\alpha_0,\beta_0}$ is only attained at a jump point of this process. Then, by disintegration, the statement follows from the fact that Poisson processes jump at a given point in time with probability zero.

As~\eqref{eq:SDE_OBM} possesses a unique strong solution (cf.~\cite{LeGall}), we assume without loss of generality that all stochastic processes $X^{\rho_n,\alpha_n,\beta_n}$ are living on the same probability space $(\Omega,\mathcal{A},\Pr)$ and are driven by the same Brownian motion $W$. A first ingredient to proving continuity in the topology of weak convergence is a stochastic approximation result for the local time $L_1^{\rho_n,\alpha_n,\beta_n}(X^{\rho_n,\alpha_n,\beta_n})$.

\begin{lemma}\label{lemma:approx_local_time}
Let $(\rho_n)_{n\in\N}$, $(\alpha_n)_{n\in\N}$ and $(\beta_n)_{n\in\N}$ be sequences with $\rho_n\rightarrow\rho$, $\alpha_n\rightarrow\alpha$ and $\beta_n\rightarrow\beta$ for $n\to\infty$. Then
\[ L_1^{\rho_n}(X^{\rho_n,\alpha_n,\beta_n}) \longrightarrow_\Pr L_1^\rho(X^{\rho,\alpha,\beta}).\]
\end{lemma}

The proof of this result is deferred to Section~\ref{Section:proof_thm2} and is based on Tanaka's formula. To proceed, we prove the following two items for sequences as in Lemma~\ref{lemma:approx_local_time}.
\begin{itemize}
\item[(1B)] The sequence $(\arg\sup_{z\in\R}\ell^{\rho_n,\alpha_n,\beta_n}(z))_{n\in\N}$ is tight (where each $\arg\sup$ is unique by Lemma~$5.1$ in~\cite{Brutsche/Rohde}).
\item[(2B)] For each $K>0$ we have
\[ \left(\ell^{\rho_n,\alpha_n,\beta_n}(z)\right)_{z\in [-K,K]} \stackrel{\mathcal{D}}{\longrightarrow} \left(\ell^{\rho,\alpha,\beta}(z)\right)_{z\in [-K,K]}\]
in the Skorohod space $\mathcal{D}([-K,K])$.
\end{itemize}
As in Step II of the proof of Theorem~\ref{thm:weak_convergence}, (1B) and (2B) together reveal
\[ \underset{z\in\R}{\arg\sup\ } \ell^{\rho_n,\alpha_n,\beta_n}(z) \stackrel{\mathcal{D}}{\longrightarrow} \underset{z\in\R}{\arg\sup\ }\ell^{\rho,\alpha,\beta}(z),\]
establishing continuity w.r.t the topology of weak convergence.
Although the items (1B) and (2B) resemble (1A) and (2A), they are statements about very different objects: In (1A) and (2A), the rescaled likelihood function $\ell_n^{\alpha,\beta}$ is treated for fixed parameters $\alpha$ and $\beta$, whereas in (1B) and (2B) properties of the sequence of limiting processes $((\ell^{\rho_n,\alpha_n,\beta_n}(z))_{z\in\R})_{n\in\N}$ indexed in $(\rho_n)_{n\in\N}, (\alpha_n)_{n\in\N}$ and $(\beta_n)_{n\in\N}$ are discussed. Therefore, their proof is completely different to that outlined in Step II in Subsection~\ref{Subsection:Proof_Theorem1} and uses the specific form of $\ell^{\rho_n,\alpha_n,\beta_n}$ as a, conditioned on $L_1^{\rho_n}(X^{\rho_n,\alpha_n,\beta_n})$, compensated Poisson process with deterministic drift. Exploiting this explicit structure, a slicing argument as in Section~$3.2.2$ of~\cite{Vaart/Wellner} together with Doob's maximal inequality for martingales yields tightness of the maximizers, given the sequence of local times $(L_1^{\rho_n}(X^{\rho_n,\alpha_n,\beta_n}))_{n\in\N}$ is bounded away from zero on a set with high probability. We construct this set using $L_1^{\rho}(X^{\rho,\alpha,\beta})$ for the limits $\rho,\alpha,\beta$ of the respective sequence and Lemma~\ref{lemma:approx_local_time}. Details are provided in Section~\ref{Section:proof_thm2}.

For (2B), we abbreviate $L_1^{\rho_n}:=L_1^{\rho_n}(X^{\rho_n,\alpha_n,\beta_n})$ and $L_1^{\rho}:=L_1^{\rho}(X^{\rho,\alpha,\beta})$ and split
\[ \ell^{\rho,\alpha,\beta}(z) = B^{\rho,\alpha,\beta}(z) + M^{\rho,\alpha,\beta}(z),\]
into a drift and martingale part, meaning
\begin{align*}
B^{\rho,\alpha,\beta}(z) &:= \1_{\{z< 0\}}  \left(\frac{\beta^2-\alpha^2}{\alpha^2\beta^2} + \frac{1}{\alpha^2}\log\left(\frac{\alpha^2}{\beta^2}\right) \right)|z|L_1^{\rho} \\
&\hspace{1.5cm} + \1_{\{z\geq 0\}}  \left( \frac{\alpha^2-\beta^2}{\alpha^2\beta^2}+\frac{1}{\beta^2}\log\left(\frac{\beta^2}{\alpha^2}\right)\right) |z|L_1^{\rho}
\end{align*} 
and
\begin{align*}
M^{\rho,\alpha,\beta}(z) &:= \1_{\{z< 0\}}\log\left(\frac{\alpha^2}{\beta^2}\right) \left( N'\left(\frac{|z|L_1^{\rho}}{\alpha^2}-\right) - \frac{|z|L_1^\rho}{\alpha^2}\right) \\
&\hspace{1.5cm} + \1_{\{z\geq 0\}}\log\left(\frac{\beta^2}{\alpha^2}\right) \left( N\left(\frac{|z|L_1^{\rho}}{\beta^2}\right)-\frac{|z|L_1^\rho}{\beta^2}\right).
\end{align*}
With Lemma~\ref{lemma:approx_local_time} it is straightforward to see that
\[ \sup_{z\in [-K,K]}\left| B^{\rho_n,\alpha_n,\beta_n}(z)-B^{\rho,\alpha,\beta}(z)\right| \longrightarrow_\Pr 0\]
and (2B) will follow by proving
\begin{align}\label{eq_proof:2B}
\left(M^{\rho_n,\alpha_n,\beta_n}(z)\right)_{z\in [-K,K]} \stackrel{\mathcal{D}}{\longrightarrow} \left(M^{\rho,\alpha,\beta}(z)\right)_{z\in [-K,K]}
\end{align}  
in the Skorohod space $\mathcal{D}([-K,K])$. As the Poisson processes $N$ and $N'$ are fixed in the definition of $M^{\rho,\alpha,\beta}$, convergence of finite dimensional distributions can be deduced from Lemma~\ref{lemma:approx_local_time} and by proving uniform integrability of $(L_1^{\rho_n})_{n\in\N}$ using Tanaka's formula. Tightness in the Skorohod space is proven by the moment criterion provided in Remark~$(13.14)$ after Theorem~$13.5$ in~\cite{Billingsley_2}. This will mainly be a consequence of the special structure of $M^{\rho,\alpha,\beta}$ being a compensated Poisson process with random intensity. The details are presented in Section~\ref{Section:proof_thm2}.

\section{Details on Part I of the proof of Theorem~\ref{thm:weak_convergence}}\label{Section:StepI}
This Section contains the proof of Lemma~\ref{lemma:self_similarity} and~\ref{lemma:distributional_identity_normalized_logLikelihood}. Both of them will make use of the identities
\begin{align}\label{eq_proof:self_similarity_1}
p_t^{\rho+\theta,\alpha,\beta}(x,y)=p_t^{\theta,\alpha,\beta}(x-\rho,y-\rho)
\end{align}
and 
\begin{align}\label{eq_proof:self_similarity_2}
p_t^{\theta,\alpha,\beta}\left(\sqrt{c}x,\sqrt{c}y\right) = \frac{1}{\sqrt{c}} p_{t/c}^{\theta/\sqrt{c},\alpha,\beta}\left(x,y\right),
\end{align} 
which are valid for all $c,t,\alpha,\beta> 0$ and $x,y,\rho,\theta\in\R$ and can be read off from the explicit form of the transition density given in~\eqref{eq:transition_density}.

\begin{proof}[Proof of Lemma~\ref{lemma:self_similarity}]
Denote $Y_t=c^{-1/2}(X_{ct}-\rho_0)$ and $Z_t = X_t-\rho_0$. For $A\subset\R$ measurable we denote $\sqrt{c}A+\rho_0 := \{\sqrt{c}a+\rho_0| a\in A\}$. Then we find for $s\leq t$,
\begin{align*}
\Pr\left( Y_t\in A| Y_s=y\right) &= \Pr\left( X_{ct} \in \sqrt{c}A + \rho_0 | X_{cs} = \sqrt{c}y+\rho_0\right) & \textrm{(definition)}  \\
&= \int_{\sqrt{c}A + \rho_0} p_{c(t-s)}^{\rho_0,\alpha_0,\beta_0}\left( \sqrt{c}y+\rho_0, z\right) dz & \textrm{(Markov semigroup)}\\
&= \int_A \sqrt{c} p_{c(t-s)}^{\rho_0,\alpha_0,\beta_0}\left( \sqrt{c}y + \rho_0, \sqrt{c}z + \rho_0\right) dz & \textrm{(change of variable)} \\
&= \int_A\sqrt{c} p_{c(t-s)}^{0,\alpha_0,\beta_0}\left( \sqrt{c}y, \sqrt{c}z\right) dz & \eqref{eq_proof:self_similarity_1}\\
&= \int_A p_{t-s}^{0,\alpha_0,\beta_0}(y,z) dz. &\eqref{eq_proof:self_similarity_2} &
\end{align*}
On the other hand, we obtain the transition density for $Z=(Z_t)_{t\geq 0}$ with similar arguments
\begin{align*}
\Pr\left( Z_t\in A| Z_{s}=y\right)  = \int_A p_{t-s}^{\rho_0,\alpha_0,\beta_0}(y+\rho_0,z+\rho_0) dz = \int_A p_{t-s}^{0,\alpha_0,\beta_0}(y,z) dz,
\end{align*}
such that both transition kernels coincide. As $Y_0=Z_0=0$, the distributions of both Markov processes coincide.
\end{proof}

\begin{proof}[Proof of Lemma~\ref{lemma:distributional_identity_normalized_logLikelihood}]
We have the following (distributional) identities:
\begin{align*}
&\ell_{n,\gamma}^{\alpha_0,\beta_0}\left(zn^{-(1+\gamma)/2}\right) & \\
&\hspace{0.3cm} =  \sum_{k=1}^n \log\left(\frac{p_{\Delta_n^\gamma}^{zn^{-(1+\gamma)/2}, \alpha_0,\beta_0}(X_{(k-1)\Delta_n^\gamma}-\rho_0,X_{k\Delta_n^\gamma}-\rho_0)}{p_{\Delta_n^\gamma}^{0,\alpha_0,\beta_0}(X_{(k-1)\Delta_n^\gamma}-\rho_0,X_{k\Delta_n^\gamma}-\rho_0)}\right) & \hspace{-0.5cm}\eqref{eq_proof:self_similarity_1} \\
&\hspace{0.3cm} \stackrel{\mathcal{D}}{=} \sum_{k=1}^n \log\left(\frac{p_{\Delta_n^\gamma}^{zn^{-(1+\gamma)/2}, \alpha_0,\beta_0}\left(n^{(1-\gamma)/2}(X_{(k-1)/n}-\rho_0),n^{(1-\gamma)/2}(X_{k/n}-\rho_0)\right)}{p_{\Delta_n^\gamma}^{0,\alpha_0,\beta_0}\left(n^{(1-\gamma)/2}(X_{(k-1)/n}-\rho_0),n^{(1-\gamma)/2}(X_{k/n}-\rho_0)\right)}\right) & \hspace{-0.5cm}\eqref{eq:distributional_identity}\\
&\hspace{0.3cm} = \sum_{k=1}^n \log\left(\frac{p_{1/n}^{\rho_0+z/n,\alpha_0,\beta_0}\left(X_{(k-1)/n},X_{k/n}\right)}{p_{1/n}^{\rho_0,\alpha_0,\beta_0}\left(X_{(k-1)/n},X_{k/n}\right)}\right). & \hspace{-0.5cm}\eqref{eq_proof:self_similarity_1},\eqref{eq_proof:self_similarity_2}
\end{align*}
Along the same lines, this identity can be shown for any finite vector
\[ \left( \ell_{n,\gamma}^{\alpha_0,\beta_0}\left(z_1 n^{-(1+\gamma)/2}\right), \dots, \ell_{n,\gamma}^{\alpha_0,\beta_0}\left(z_N n^{-(1+\gamma)/2}\right)\right)\] 
and $z_1,\dots, z_N\in\R$. As the law of a process is determined by its finite dimensional marginals, the claim of Lemma~\ref{lemma:distributional_identity_normalized_logLikelihood} follows.
\end{proof}

\section{Details on Step II of the proof of Theorem~\ref{thm:weak_convergence}}\label{Section:StepII}

As outlined in Step II, the general idea for proving (1A) and (2A) is to reduce the assertion to the case of infill asymptotics via the self-similarity established in Step I. The reason for this is that for such observations the result can be derived using stable convergence for semimartingales which is no longer true if $\gamma\in [0,1)$. For each $z\in [-K,K]$ we have the decomposition
\begin{align}\label{eq_proof:prop_start_in_x0_1}
\begin{split}
\ell_{n,\gamma}^{\alpha_0,\beta_0}(zn^{-(1+\gamma)/2}, \overline{X}^c) &= \sum_{k=1}^{\lceil T_{\rho_0}^{x_0}/\Delta_n^\gamma\rceil} \log\left(\frac{p_{\Delta_n^\gamma}^{\rho_0+zn^{-(1+\gamma)/2}, \alpha_0,\beta_0}(\overline{X}_{(k-1)\Delta_n^\gamma}^c,\overline{X}_{k\Delta_n^\gamma}^c)}{p_{\Delta_n^\gamma}^{\rho_0,\alpha_0,\beta_0}(\overline{X}_{(k-1)\Delta_n^\gamma}^c,\overline{X}_{k\Delta_n^\gamma}^c)}\right) \\
&\hspace{0.5cm} + \sum_{k=\lceil T_{\rho_0}^{x_0}/\Delta_n^\gamma\rceil+1}^n \log\left(\frac{p_{\Delta_n^\gamma}^{\rho_0+zn^{-(1+\gamma)/2}, \alpha_0,\beta_0}(\overline{X}_{(k-1)\Delta_n^\gamma}^c,\overline{X}_{k\Delta_n^\gamma}^c)}{p_{\Delta_n^\gamma}^{\rho_0,\alpha_0,\beta_0}(\overline{X}_{(k-1)\Delta_n^\gamma}^c,\overline{X}_{k\Delta_n^\gamma}^c)}\right), 
\end{split} 
\end{align} 
where by the structure of our coupling $\overline{X}^c$ the second summand is given as
\begin{align*}
\sum_{k=\lceil T_{\rho_0}^{x_0}/\Delta_n^\gamma\rceil+1}^n \log\left(\frac{p_{\Delta_n^\gamma}^{\rho_0+zn^{-(1+\gamma)/2}, \alpha_0,\beta_0}(X_{(k-1)\Delta_n^\gamma}^{\rho_0},X_{k\Delta_n^\gamma}^{\rho_0})}{p_{\Delta_n^\gamma}^{\rho_0,\alpha_0,\beta_0}(X_{(k-1)\Delta_n^\gamma}^{\rho_0},X_{k\Delta_n^\gamma}^{\rho_0})}\right). 
\end{align*} 
This means that from~\eqref{eq_proof:prop_start_in_x0_1} we get a decomposition of processes
\begin{align}\label{eq_proof:1A_distributional_identity}
\begin{split}
&\left(\ell_{n,\gamma}^{\alpha_0,\beta_0}(zn^{-(1+\gamma)/2}, \overline{X}^c)\right)_{z\in [-K,K]} - \left(Y\left(z,(\overline{X}_{k\Delta_n^\gamma}^c)_{k=0,\dots,n}\right)\right)_{z\in [-K,K]} \\
&\hspace{2cm} = \left(\ell_{n,\gamma}^{\alpha_0,\beta_0}(zn^{-(1+\gamma)/2}, X^{\rho_0})\right)_{z\in [-K,K]} - \left(Y\left(z,(X_{k\Delta_n^\gamma}^{\rho_0})_{k=0,\dots,n}\right)\right)_{z\in [-K,K]}, 
\end{split}
\end{align} 
where
\[ Y\left(z,(X_{k\Delta_n^\gamma})_{k=0,\dots,n}\right) = \sum_{k=1}^{\lceil T_{\rho_0}^{x_0}/\Delta_n^\gamma\rceil} \log\left(\frac{p_{\Delta_n^\gamma}^{\rho_0+zn^{-(1+\gamma)/2}, \alpha_0,\beta_0}(X_{(k-1)\Delta_n^\gamma},X_{k\Delta_n^\gamma})}{p_{\Delta_n^\gamma}^{\rho_0,\alpha_0,\beta_0}(X_{(k-1)\Delta_n^\gamma},X_{k\Delta_n^\gamma})}\right).\]
The crucial idea for proving (1A) and (2A) is now to start with the decomposition~\eqref{eq_proof:1A_distributional_identity}. Here, we apply the self-similarity argument to the process $(\ell_{n,\gamma}^{\alpha_0,\beta_0}(zn^{-(1+\gamma)/2}, X^{\rho_0}))_{z\in [-K,K]}$, whereas both $Y$-terms are treated as a remainder. For the low-frequency asymptotics $\gamma=0$, this program can directly put into action, since the number $\lceil T_{\rho_0}^{x_0}/\Delta_n^\gamma\rceil$ in the definition of $Y$ is in fact independent of $n$, as well as the time difference of two consecutive $X_{(k-1)\Delta_n^\gamma}$ and $X_{k\Delta_n^\gamma}$. These proofs are given below. For $\gamma\in (0,1)$, the proof is technically more involved and combines the self-similarity and coupling argument with techniques developed in~\cite{Brutsche/Rohde} for the high-frequency setting $\gamma=1$. Whereas the order of numbers of summands is known, the obstacle consists in proving that the summands are (uniformly in $z$) not too large. Here, the exact reasoning is deferred to the proof of (1A) for $\gamma\in (0,1)$ in Appendix~\ref{Section:Proof_1A}. It should be noted that the arguments for the high-frequency case $\gamma=1$ need to be developed separately as the coupling argument reveals a bound of stochastic order $n^{-(1-\gamma)/2}$ on the remainder $Y$. The intuition behind this rate is that the larger $\gamma$, the more observations are made before the coupling time $T_{\rho_0}^{x_0}$. 
The proof of (2A) relies crucially on Lemma~\ref{lemma:probability_A3} which together with its proof is given at the end of this section and is of independent interest. Based on this Lemma, the proof of (2A) is given Section~\ref{Section:Proof_2A}.

\begin{proof}[Proof of (1A) for $\gamma=0$]
First, note that for $\gamma=0$ we have $\Delta_n^\gamma=1$. By Step I, in particular~\eqref{eq:limit_known_alpha/beta}, we know that
\[ \left(\ell_{n,0}^{\alpha_0,\beta_0}(zn^{-1/2}, X^{\rho_0})\right)_{z\in [-K,K]} \stackrel{\mathcal{D}}{\longrightarrow} \left( \ell^{\rho_0,\alpha_0,\beta_0}(z)\right)_{z\in [-K,K]}.\]
The convergence in (1A) then is a consequence from~\eqref{eq_proof:1A_distributional_identity} if we show
\begin{align}\label{eq_proof:1A_vanishing_term}
\sup_{z\in [-K,K]} \left| Y\left(z,Z_{k})_{k=0,\dots,n}\right)\right|\longrightarrow_{\Pr}0\qquad \textrm{ for } Z=\overline{X}^c,X^{\rho_0}.
\end{align} 
In what follows, we will establish~\eqref{eq_proof:1A_vanishing_term} for $Z=\overline{X}^c$ and the supremum over $z\in [0,K]$. The supremum over $[-K,0]$ can be dealt with analogously and the reasoning is then exactly the same for $Z=X^{\rho_0}$. Recall that $T_{\rho_0}^{x_0}$ is independent of $n$ and satisfies $\Pr_{\rho_0,\alpha_0,\beta_0}(T_{\rho_0}^{x_0}<\infty)=1$. Let $\epsilon>0$ be arbitrary and choose a corresponding $T_0\in\N$ such that $\Pr(T_{\rho_0}^{x_0}\leq T_0)>1-\epsilon$. Then we have
\begin{align}\label{eq_proof:1A_bound_Y}
\left| Y\left(z,(\overline{X}_{k})_{k=0,\dots,n}^c\right)\right|\1_{\{T\leq T_0\}} \leq \sum_{k=1}^{T_0} \left| \log\left(\frac{p_{1}^{\rho_0+zn^{-(1+\gamma)/2}, \alpha_0,\beta_0}(\overline{X}_{k-1}^c,\overline{X}_{k}^c)}{p_{1}^{\rho_0,\alpha_0,\beta_0}(\overline{X}_{k-1}^c,\overline{X}_{k}^c)}\right) \right|. 
\end{align} 
We abbreviate $\psi_n^\gamma=n^{-(1+\gamma)/2}$. When considering the different cases $I_{j,k}^{0,z\psi_n^\gamma}$, $j=1,\dots, 9$, that occur in the likelihood ratio~\eqref{eq:def_ell_rho}, namely
\begin{align}\label{def:indicators}
\begin{split}
&I_{1,k}^{\theta',\theta} := \{X_{(k-1)\Delta_n^\gamma}< \rho_0+\theta', X_{k\Delta_n^\gamma} \leq \rho_0+\theta'\},\\
&I_{2,k}^{\theta',\theta} := X_{(k-1)\Delta_n^\gamma}<\rho_0+\theta'< X_{k\Delta_n^\gamma}\leq \rho_0+\theta\},\\
&I_{3,k}^{\theta',\theta} := \{X_{(k-1)\Delta_n^\gamma}<\rho_0+\theta'\leq\rho_0+\theta <X_{k\Delta_n^\gamma}\}, \\
&I_{4,k}^{\theta',\theta} := \{X_{k\Delta_n^\gamma}\leq\rho_0+\theta'\leq X_{(k-1)\Delta_n^\gamma}<\rho_0+\theta \}, \\
&I_{5,k}^{\theta',\theta} :=  \{\rho_0+\theta'\leq X_{(k-1)\Delta_n^\gamma}<\rho_0+\theta, \rho_0+\theta'< X_{k\Delta_n^\gamma}\leq\rho_0+\theta \}, \\
&I_{6,k}^{\theta',\theta} :=  \{ \rho_0+\theta'\leq X_{(k-1)\Delta_n^\gamma} < \rho_0+\theta < X_{k\Delta_n^\gamma} \}, \\
&I_{7,k}^{\theta',\theta} :=  \{ X_{k\Delta_n^\gamma}\leq\rho_0+\theta'\leq \rho_0+\theta\leq X_{(k-1)\Delta_n^\gamma}\}, \\
&I_{8,k}^{\theta',\theta} :=  \{ \rho_0+\theta'< X_{k\Delta_n^\gamma}\leq \rho_0+\theta\leq X_{(k-1)\Delta_n^\gamma}\}, \\
&I_{9,k}^{\theta',\theta} :=  \{ \rho_0+\theta \leq X_{(k-1)\Delta_n^\gamma}, \rho_0+\theta <X_{k\Delta_n^\gamma}\},
\end{split}
\end{align} 
one can observe that for $j=2,8$, the logarithmic term in one summand in~\eqref{eq_proof:1A_bound_Y} behaves approximately like a constant, whereas for all others it is as multiple of $\psi_n^\gamma$. However, the events for $j=2,8$ are unlikely and we exclude them in the following way: as $\overline{X}^c$ has a transition density with respect to Lebesgue measure, we can find a constant $\kappa>0$ small enough such that $\Pr(E_\kappa)\geq 1-\epsilon$ for the event
\[ E_\kappa := \left\{ \overline{X}_{k}^c \notin [\rho_0,\rho_0+\kappa] \quad \textrm{ for all } k=0,1,\dots, T_0\right\}.\]
Then, we let $n$ be large enough such that $K\psi_n^\gamma\leq \kappa$, for which the bound~\eqref{eq_proof:1A_bound_Y} reduces to
\begin{align*}
& \left| Y\left(z,(\overline{X}_{k}^c)_{k=0,\dots,n}\right)\right|\1_{\{T\leq T_0\}}\1_{E_\kappa} \\
&\hspace{0.3cm} \leq \sum_{k=1}^{ T_0} \left| \log\left(\frac{p_{1}^{\rho_0+zn^{-(1+\gamma)/2}, \alpha_0,\beta_0}(\overline{X}_{k-1}^c,\overline{X}_{k}^c)}{p_{1}^{\rho_0,\alpha_0,\beta_0}(\overline{X}_{k-1}^c,\overline{X}_{k}^c)}\right)\right| \left( \1_{I_{1,k}^{0,z\psi_n^\gamma}} + \1_{I_{3,k}^{0,z\psi_n^\gamma}} + \1_{I_{7,k}^{0,z\psi_n^\gamma}} + \1_{I_{9,k}^{0,z\psi_n^\gamma}} \right).
\end{align*}
Note that on $E_{\kappa}$ the terms for $j=4,5,6$ also disappear which can be immediately seen from~\eqref{def:indicators}. This bound is now useful, because for $j=1,3,7,9$ we have
\[ \left| \log\left(\frac{p_{1}^{\rho_0+zn^{-(1+\gamma)/2}, \alpha_0,\beta_0}(\overline{X}_{k-1}^c,\overline{X}_{k}^c)}{p_{1}^{\rho_0,\alpha_0,\beta_0}(\overline{X}_{k-1}^c,\overline{X}_{k}^c)}\right)\right|\1_{I_{j,k}^{0,z\psi_n^\gamma}} \leq z\psi_n^\gamma G_j^{\alpha_0,\beta_0}(\overline{X}_{k-1}^c,\overline{X}_k^c),\]
where for any $a>0$, $\sup_{-a\leq x,y\leq a} G_j^{\alpha_0,\beta_0}(x,y)\leq C_{\alpha_0,\beta_0}(a)$ for $j=1,3,7,9$ and a suitable constant $C_{\alpha_0,\beta_0}(a)>0$, see (D.1), (D.7), (D.14) and (D.17) in~\cite{Brutsche/Rohde}. By the Burkholder-Davis-Gundy inequality, we can choose $a>0$ large enough such that $\Pr(E_a')>1-\epsilon$ for the event
\begin{align}\label{eq_proof:1A_setEa}
E_a' := \left\{ -a \leq \overline{X}_0^c,\overline{X}_1^c,\dots, \overline{X}_{\lceil T_0\rceil}^c \leq a\right\}.
\end{align} 
Hence, we finally obtain the bound
\[ \sup_{z\in [0,K]}\left| Y\left(z,(\overline{X}_{k}^c)_{k=0,\dots,n}\right)\right|\1_{\{T_{\rho_0}^{x_0}\leq T_0\}}\1_{E_\kappa}\1_{E_a'} \leq C_{\alpha_0,\beta_0}(a)K\psi_n^\gamma \longrightarrow 0 \]
and~\eqref{eq_proof:1A_vanishing_term} follows since $\Pr(\{T\leq T_0\}\cap E_\kappa\cap E_a')>1-3\epsilon$.
\end{proof}

\begin{proof}[Proof of (2A) for $\gamma=0$] Let $\epsilon>0$. The identity~\eqref{eq_proof:1A_distributional_identity} gives 
\[ \ell_{n,0}^{\alpha_0,\beta_0}(zn^{-1/2}, \overline{X}^c) = \ell_{n,0}^{\alpha_0,\beta_0}(zn^{-1/2}, X^{\rho_0})  + Y(n,z,T_{\rho_0}^{x_0})\]
with
\[ Y(n,z,T_{\rho_0}^{x_0})= Y\left(z,\overline{X}_{k\Delta_n^\gamma}^c)_{k=0,\dots,n}\right) - Y\left(z,X_{k\Delta_n^\gamma}^{\rho_0})_{k=0,\dots,n}\right).\]
By the explicit form of the transition density in~\eqref{eq:transition_density}, there exist constants $c_1=c_1(\alpha_0,\beta_0,a,K)$ and $c_2=c_2(\alpha_0,\beta_0,a,K)$ such that for any $x,y\in [-a,a]$ and $z\in[-K,K]$, 
\[ c_1\leq  p_1^{zn^{-(1+\gamma)/2},\alpha_0,\beta_0}(x,y) \leq c_2.\]
As a consequence, 
\begin{align*}
\1_{E_a'}\1_{\{T_{\rho_0}^{x_0}\leq T_0\}}\left|Y(n,z,T)\right| \leq T_0 \log(c_2/c_1), 
\end{align*} 
where $E_a'$ is defined in~\eqref{eq_proof:1A_setEa}. We assume that $a$ and $T_0$ are chosen as in the proof of (1A) to ensure $\Pr(\{T_{\rho_0}^{x_0}\leq T_0\}\cap E_a')>1-2\epsilon$. Using $\ell_{n,0}^{\alpha_0,\beta_0}(0)=0$ and $n^{1/2}(\hat{\rho}_n-\rho_0)\in\mathrm{Argsup}_{z\in\R} \ell_{n,0}^{\alpha_0,\beta_0}(zn^{-1/2})$ by definition, we obtain
\begin{align*}
&\Pr\left( \sqrt{n}\left(\hat{\rho}_n(\overline{X}^c) - \rho_0\right) >K\right)\\
&\hspace{1cm} \leq \Pr\left( \sup_{|z|>K} \ell_{n,0}^{\alpha_0,\beta_0}(zn^{-1/2}, \overline{X}^c) \geq 0\right) \\
&\hspace{1cm} \leq \Pr\left( \sup_{|z|>K} \ell_{n,0}^{\alpha_0,\beta_0}(zn^{-1/2}, X^{\rho_0}) \geq - T_0 \log(c_2/c_1)\right) +2\epsilon, \\
&\hspace{1cm} =  \Pr\left( \sup_{|z|>K} \overline{\ell}_n^{\alpha_0,\beta_0}(z/n, X^{\rho_0}) \geq -T_0 \log(c_2/c_1) \right) +2\epsilon,
\end{align*}
where the last step uses the self-similarity of Lemma~\ref{lemma:distributional_identity_normalized_logLikelihood} and the notation
\[ \overline{\ell}_n^{\alpha_0,\beta_0}(\theta, X) =\sum_{k=1}^n \log\left(\frac{p_{1/n}^{\rho_0+\theta,\alpha_0,\beta_0}\left(X_{(k-1)/n},X_{k/n}\right)}{p_{1/n}^{\rho_0,\alpha_0,\beta_0}\left(X_{(k-1)/n},X_{k/n}\right)}\right).\]
At this point, we reached our goal to transfer the problem into the framework of infill asymptotics. The claim now follows by repeating line-by-line the consistency proof given in Section~$3$ in~\cite{Brutsche/Rohde} from (3.7) onwards, replacing there $0$ by $-T_0\log(c_2/c_1)$.
\end{proof}

\begin{lemma}\label{lemma:probability_A3}
Let $\epsilon>0$ and $X$ be a solution to~\eqref{eq:SDE_OBM} with $X_0=x_0$. Then there exists $\xi=\xi_\epsilon>0$ and $\zeta=\zeta_\epsilon>0$ such that 
\[ \liminf_{n\to\infty}\Pr_{\rho_0,\alpha_0,\beta_0}\left(\inf_{y\in [\rho_0-\zeta\sqrt{n\Delta_n^\gamma}, \rho_0+\zeta\sqrt{n\Delta_n^\gamma}]} L_{n\Delta_n^\gamma}^y(X) >\sqrt{n\Delta_n^\gamma} \xi \right)>1-\epsilon.\]
\end{lemma}
\begin{proof}
First, assume $X_0=\rho_0$. By self-similarity of Lemma~\ref{lemma:self_similarity}, definition of the local time in~\eqref{eq:local_time} and $\langle X-\rho_0\rangle_t=\langle X\rangle_t = \int_0^t \sigma_{\rho_0,\alpha_0,\beta_0}(X_s)^2ds= \int_0^t \sigma_{0,\alpha_0,\beta_0}(X_s-\rho_0)^2ds$, we have for any $c>0$ and $y\in\R$,
\begin{align}\label{eq_proof:distributional_identity_local_time}
\begin{split}
L_1^{y}(X) &= L_1^{y-\rho_0}(X-\rho_0) \\
&= \lim_{\epsilon\searrow 0} \frac{1}{2\epsilon} \int_0^1 \1_{(y-\rho_0-\epsilon,y-\rho_0+\epsilon)}(X_t-\rho_0) \sigma_{0,\alpha_0,\beta_0}(X_t-\rho_0)^2 dt \\
&\stackrel{\mathcal{D}}{=} \lim_{\epsilon\searrow 0} \frac{1}{2\epsilon} \int_0^1 \1_{(y-\rho_0-\epsilon,y-\rho_0+\epsilon)}\Big( \frac{1}{\sqrt{c}}(X_{ct}-\rho_0)\Big) \sigma_{0,\alpha_0,\beta_0}\Big(\frac{1}{\sqrt{c}}(X_{ct}-\rho_0)\Big)^2 dt \\
&= \lim_{\epsilon\searrow 0} \frac{1}{2\epsilon} \frac1c \int_0^c \1_{(\sqrt{c}(y-\rho_0-\epsilon),\sqrt{c}(y-\rho_0+\epsilon))}\big( (X_{t}-\rho_0)\big) \sigma_{0,\alpha_0,\beta_0}(X_t-\rho_0)^2 dt \\
&= \frac{1}{\sqrt{c}}  \lim_{\epsilon\searrow 0} \frac{1}{2\epsilon}\int_0^c \1_{(\sqrt{c}(y-\rho_0)-\epsilon,\sqrt{c}(y-\rho_0)+\epsilon)}\big( (X_{t}-\rho_0)\big) \sigma_{0,\alpha_0,\beta_0}(X_t-\rho_0)^2 dt \\
&= \frac{1}{\sqrt{c}} L_c^{\sqrt{c}(y-\rho_0)}(X-\rho_0) = \frac{1}{\sqrt{c}} L_c^{\rho_0+\sqrt{c}(y-\rho_0)}(X).
\end{split}
\end{align}
This reasoning also applies to vectors, i.e. for $y_1,\dots, y_N\in\R$ we find
\[ \left( L_1^{y_1}(X),\dots, L_1^{Y_N}(X)\right) \stackrel{\mathcal{D}}{=} \frac{1}{\sqrt{c}}\left( L_c^{\rho_0+\sqrt{c}(y_1-\rho_0)}(X), \dots, L_c^{\rho_0+\sqrt{c}(y_N-\rho_0)}(X)\right). \]
As the process $L_t^\bullet(X)$ is a continuous process for each $t\geq 0$, we therefore obtain the distributional identity
\[ \left( L_1^y(X)\right)_{y\in\R} \stackrel{\mathcal{D}}{=} \left( \frac{1}{\sqrt{c}} L_c^{\rho_0+\sqrt{c}(y-\rho_0)}(X)\right)_{y\in\R}.\]
Setting $c=n\Delta_n^\gamma$, we then find
\[ \Pr_{\rho_0,\alpha_0,\beta_0}\left(\inf_{y\in [\rho_0-\zeta\sqrt{n\Delta_n^\gamma}, \rho_0+\zeta\sqrt{n\Delta_n^\gamma}]} L_{n\Delta_n^\gamma}^y(X) >\sqrt{n\Delta_n^\gamma} \xi\right) = \Pr\left(\inf_{y\in [\rho_0-\zeta, \rho_0+\zeta]} L_{1}^y(Y) > \xi \right) \] \newpage \noindent
for an OBM $Y$ started in $Y_0=\rho_0$ and with diffusion coefficient $\sigma_{\rho_0,\alpha_0,\beta_0}$. Because $L_1^\cdot(X)$ has a continuous version (as $X$ is a continuous martingale) and every continuous function on a compact set attains its minimum and maximum, Corollary~$29.18$ in~\cite{Kallenberg} about range and support of continuous local martingales reveals
\[ \Pr\left( \left.\inf_{y\in [\rho_0-\zeta, \rho_0+\zeta]} L_1^y(Y) >0 \ \right| E_1 \right) =1\]
for the event $E_1 =\{\overline{Y}_1\geq \rho_0+2\zeta, \underline{Y}_1\leq \rho_0+2\zeta\}$. We now prove that $\Pr(E_1)>1-\epsilon/2$ for $\zeta>0$ small enough. To this aim, define $B_k = \{ L_1^y(Y) >0 \textrm{ for all } y\in [\rho_0-1/k, \rho_0+1/k]\}$. Then $B_{k}\subset B_{k+1}$ for all $k\in\N$ and by continuity of measures from below, we have $\lim_{K\to\infty} \Pr(\bigcup_{k=1}^K B_k) = \Pr(L_1^{\rho_0}(Y)>0)=1$. Consequently, for every $\epsilon>0$ there exists $K_\epsilon$ such that
\[ \Pr\left(\bigcup_{k=1}^{K_\epsilon} B_k\right) = \Pr\left( B_{K_\epsilon}\right) >1-\epsilon/2,\]
and on $B_{K_\epsilon}$ we have $L_1^y(Y)>0$ for all $y\in [\rho_0-1/K_\epsilon, \rho_0+1/K_\epsilon]$ which by Corollary~$29.18$ in~\cite{Kallenberg} gives $\overline{Y}_1\geq\rho_0+1/K_\epsilon$ and $\underline{Y}_1\leq\rho_0-1/K_\epsilon$ on $B_{K_\epsilon}$. Finally, set $\zeta=K_\epsilon^{-1}/2$. 
By continuity of measures from below (analogously to the previous argument), we conclude the existence of $\xi>0$ such that
\[ \Pr\left(\inf_{y\in [\rho_0-\zeta, \rho_0+\zeta]} L_{1}^y(Y) > \xi \right)=\Pr\left(\left. \inf_{y\in [\rho_0-\zeta, \rho_0+\zeta]} L_{1}^y(Y) > \xi\ \right| E_1 \right)\Pr(E_1)>1-\epsilon, \]
giving the claim for for the case that $X$ starts in $\rho_0$.\\
Now, we generalize to $X_0=x_0\in\R$ arbitrary. Define the stopping time $\tau_{\rho_0}=\inf\{t\geq 0\mid X_t=\rho_0\}$. By the strong Markov property of $X$, we have
\[ \left( L_{n\Delta_n}^{y}(X)\right)_{y\in\R} \stackrel{\mathcal{D}}{=} \left( L_{\tau_{\rho_0}}^{y}(X) + L_{n\Delta_n^\gamma -\tau_{\rho_0}}^{y}(Y) \right)_{y\in\R},\]
where $Y$ is an OBM started in $Y_0=\rho_0$. By null-recurrence of $X$, there exists $T_0>0$ such that $\Pr(\tau_{\rho_0}\geq T_0)\leq \epsilon/2$ and together with non-negativity of the local time, we then find for $I_n(\xi,\zeta) = [\rho_0-\zeta\sqrt{n\Delta_n^\gamma}, \rho_0+\zeta\sqrt{n\Delta_n^\gamma}]$,
\[ \Pr_{\rho_0,\alpha_0,\beta_0}\left(\inf_{y\in I_n(\xi,\zeta)} L_{n\Delta_n^\gamma}^y(X) >\sqrt{n\Delta_n^\gamma} \xi \right)\geq \Pr\left(\inf_{y\in I_n(\xi,\zeta)} L_{n\Delta_n^\gamma -T_0}^{y}(Y) > \sqrt{n\Delta_n^\gamma} \xi \right) - \frac{\epsilon}{2}. \]
By the arguments used before, the remaining probability can be lower bounded by $1-\epsilon/2$ for $\xi$ small enough since $n\Delta_n^\gamma\rightarrow\infty$.
\end{proof}

\section{Details of the proof of Theorem~\ref{thm:properties_limit}}\label{Section:proof_thm2}

In this section, we will prove both properties of the limiting distribution in Theorem~\ref{thm:weak_convergence}, i.e. the distribution of $\arg\sup_{z\in\R}\ell^{\rho_0,\alpha_0,\beta_0}(z)$, that are stated in Theorem~\ref{thm:properties_limit}. The existence of a density is shown in Proposition~\ref{prop:density} and the remaining details of Subsection~\ref{Subsection:Proof_Theorem1} are presented thereafter.

\begin{proposition}\label{prop:density}
Let $\rho\in\R$, $\alpha,\beta>0$. Then, the random variable $\arg\sup_{z\in\R}\ell^{\rho,\alpha,\beta}(z)$ admits a density with respect to Lebesgue measure.
\end{proposition}
\begin{proof}
Let $y\in\R$ be fixed. As the Poisson processes $N,N'$ and the oscillating Brownian motion $X^{\rho,\alpha,\beta}$ are independent, we have by disintegration
\begin{align*}
&\Pr\left( \underset{z\in\R}{\arg\sup\ }\ell^{\rho,\alpha,\beta}(z)=y\right) \\
&\hspace{2cm} =  \int_0^\infty \Pr\left(\left. \underset{z\in\R}{\arg\sup\ }\tilde{\ell}^{\alpha,\beta}(cz)=y\ \right|\ L_1^{\rho}(X^{\rho,\alpha,\beta})=c\right) d\Pr^{L_1^\rho(X^{\rho,\alpha,\beta})}(dc),
\end{align*}
where
\begin{align*}
\tilde{\ell}^{\alpha,\beta}(z) &= \1_{\{z< 0\}} \left(\frac{\beta^2-\alpha^2}{\alpha^2\beta^2}  |z| +\log\left(\frac{\alpha^2}{\beta^2}\right) N'\left(\frac{|z|}{\alpha^2}-\right)\right) \\
&\hspace{1cm} +  \1_{\{z\geq 0\}} \left( \frac{\alpha^2-\beta^2}{\alpha^2\beta^2} |z| +\log\left(\frac{\beta^2}{\alpha^2}\right)N\left(\frac{|z|}{\beta^2}\right)\right).
\end{align*}
Because $\tilde{\ell}^{\alpha,\beta}$ is a piecewise linear function, we either find $N'(c|y|/\alpha^2)-N'(c|y|/\alpha^2-)\neq 0$ or $N(c|y|/\beta^2)-N(c|y|/\beta^2-)\neq 0$ for the point $y$ at which is the $\mathrm{argsup}$ is attained. Hence,
\begin{align*}
&\Pr\left(\left. \underset{z\in\R}{\arg\sup\ }\tilde{\ell}^{\alpha,\beta}(cz)=y\ \right|\ L_1^{\rho}(X^{\rho,\alpha,\beta})=c\right) \\
&\hspace{1cm} = \Pr\left( N' \textrm{ jumps at } c|y|/\alpha^2 \textrm{ or } N \textrm{ jumps at } c|y|/\beta^2\right) \\
&\hspace{1cm} \leq \sum_{m=1}^\infty \Pr\left(m\textrm{-th jump of } N'\textrm{ is at } c|y|/\alpha^2 \textrm{ or } m\textrm{-th jump of } N \textrm{ is at }c|y|/\beta^2\right) =0,
\end{align*}
where the last step uses that the time of the $m$-th jump of a Poisson process is gamma distributed and hence has probability zero.
\end{proof}


\begin{proof}[Proof of (1B)]
Recall that for fixed Brownian motion $W$, $X^{\rho_n,\alpha_n,\beta_n}$ is a strong solution solving~\eqref{eq:SDE_OBM} with diffusion coefficient $\sigma_{\rho_n,\alpha_n,\beta_n}$ and $x_0=\rho_n$. First, we introduce the event
\[ E_n := \left\{\frac12 L_1^{\rho}(X^{\rho,\alpha,\beta}) \leq L_1^{\rho_n}(X^{\rho_n,\alpha_n,\beta_n}) \leq \frac32 L_1^{\rho}(X^{\rho,\alpha,\beta})\right\}\]
and abbreviate $L_1^{\rho_n}:=L_1^{\rho_n}(X^{\rho_n,\alpha_n,\beta_n}), L_1^{\rho}:=L_1^{\rho}(X^{\rho,\alpha,\beta})$ in the following. Let $\epsilon>0$ be arbitrary. As $X_0^{\rho,\alpha,\beta}=\rho$, Corollary~$29.18$ in~\cite{Kallenberg} gives $\Pr(L_1^\rho>0)=1$ and by continuity of measures, there exists a constant $\kappa_\epsilon>0$ such that $\Pr(L_1^\rho >\kappa_\epsilon)\geq 1-\epsilon$. Then, 
\begin{align*}
\Pr(E_n^c) \leq \Pr(E_n^c, L_1^\rho>\kappa_\epsilon) + \epsilon \leq \Pr\left( \left| L_1^\rho - L_1^{\rho_n}\right| > \frac{\kappa_\epsilon}{2}\right) + \epsilon
\end{align*}
and by Lemma~\ref{lemma:approx_local_time}, disintegration and Fatou's lemma,
\begin{align*}
&\limsup_{K\to\infty} \limsup_{n\to\infty}\Pr\left( \left|\underset{z\in\R}{\arg\sup\ }\ell^{\rho_n,\alpha_n,\beta_n}(z)\right| > K\right) \\
&\hspace{1cm} \leq \limsup_{K\to\infty} \limsup_{n\to\infty}\Pr\left( \left|\underset{z\in\R}{\arg\sup\ }\ell^{\rho_n,\alpha_n,\beta_n}(z)\right| > K, E_n\right) +\epsilon \\
&\hspace{1cm} \leq \E\left[ \limsup_{K\to\infty}\limsup_{n\to\infty} \Pr\left( \left.\left|\underset{z\in\R}{\arg\sup\ }\ell^{\rho_n,\alpha_n,\beta_n}(z)\right| > K, E_n\right| L_1^{\rho} \right)\right] +\epsilon.
\end{align*}
We now show that the integrand vanishes by proving this for $\mathrm{argsup}_{z\geq 0}$. The general case then follows by an analogous reasoning for $z<0$ and the union bound. For $M=\log_2(K)$, using $\ell^{\rho_n,\alpha_n,\beta_n}(0)=0$,
\begin{align}\label{eq_proof:1B}
\begin{split}
&\Pr\left(\left. \left|\underset{z\geq 0}{\arg\sup\ }\ell^{\rho_n,\alpha_n,\beta_n}(z)\right| > K, E_n\right| L_1^\rho\right) \\
&\hspace{1cm} \leq \Pr\left(\left. \sup_{z>2^M} \ell^{\rho_n,\alpha_n\beta_n}(z) \geq 0,E_n\right| L_1^\rho\right) \\
&\hspace{1cm} = \E\left[ \left. \1_{E_n} \Pr\left(  \left. \sup_{z>2^M} \ell^{\rho_n,\alpha_n\beta_n}(z) \geq 0 \right| L_1^{\rho_n}, L_1^\rho\right) \right| L_1^\rho\right].
\end{split}
\end{align}
Here, the last step uses the tower property of conditional expectation and measurability of $E_n$ with respect to the $\sigma$-algebra generated by $L_1^\rho$ and $L_1^{\rho_n}$. Because the Poisson processes $N,N'$ in the definition of $\ell^{\rho_n,\alpha_n,\beta_n}$ are independent of the Brownian motion $W$ and hence of $L_1^\rho$, we have
\[ \Pr\left(  \left. \sup_{z>2^M} \ell^{\rho_n,\alpha_n\beta_n}(z) \geq 0 \right| L_1^{\rho_n}, L_1^\rho\right)=\Pr\left(  \left. \sup_{z>2^M} \ell^{\rho_n,\alpha_n\beta_n}(z) \geq 0 \right| L_1^{\rho_n}\right).\]
In order to evaluate this term, we observe that
\begin{align*}
\ell^{\rho_n,\alpha_n\beta_n}(z) &=  b_{\alpha_n,\beta_n} L_1^{\rho_n} |z| +  \log\left(\frac{\beta_n^2}{\alpha_n^2}\right)\left( N\left(\frac{zL_1^{\rho_n}}{\beta_n^2}\right) - \frac{zL_1^{\rho_n}}{\beta_n^2}\right)
\end{align*} 
for $z\geq 0$ and $b_{\alpha,\beta} = (\alpha^2-\beta^2)/(\alpha\beta)^2 + \log(\beta^2/\alpha^2)/\beta^2$, which is a decomposition into a deterministic drift and a compensated Poisson process given $L_1^{\rho_n}$. Note that $b_{\alpha,\beta}<0$ for $\alpha\neq\beta$. To continue, we set $S_j=\{2^j<z\leq 2^{j+1}\}$ and recall that on $E_n$ we have $L_1^{\rho_n}>L_1^\rho/2$ and $\Pr(L_1^\rho>0)=1$. Then, applying the union bound and Doob's maximal inequality yields
\begin{align*}
&\1_{E_n}\Pr\left(\left. \sup_{z>2^M} \ell^{\rho_n,\alpha_n\beta_n}(z) \geq 0\right| L_1^{\rho_n}\right) \\
&\hspace{0.5cm} \leq \sum_{j\geq M} \1_{E_n} \Pr\left( \left. \sup_{z\in S_j} \log\left(\frac{\beta_n^2}{\alpha_n^2}\right) \left( N\left(\frac{zL_1^{\rho_n}}{\beta_n^2}\right) - \frac{zL_1^{\rho_n}}{\beta_n^2}\right) \geq -b_{\alpha_n,\beta_n}2^jL_1^{\rho_n}\right| L_1^{\rho_n}\right)\\
&\hspace{0.5cm} \leq   \log\left(\frac{\beta_n^2}{\alpha_n^2}\right) \frac{1}{b_{\alpha_n,\beta_n}^2 (L_1^{\rho_n})^2} \sum_{j\geq M} 2^{-2j} \1_{E_n} \E\left[\left. \left( N\left(\frac{2^{j+1}L_1^{\rho_n}}{\beta_n^2}\right) - \frac{2^{j+1}L_1^{\rho_n}}{\beta_n^2}\right)^2\right| L_1^{\rho_n}\right] \\
&\hspace{0.5cm} =  \1_{E_n}\log\left(\frac{\beta_n^2}{\alpha_n^2}\right) \frac{1}{b_{\alpha_n,\beta_n}^2 \beta_n^2 L_1^{\rho_n}} \sum_{j\geq M} 2^{-2j} 2^{j+1}.
\end{align*}
Consequently, the last term in~\eqref{eq_proof:1B} is bounded as
\begin{align*}
\E\left[ \left. \1_{E_n} \Pr\left(  \left. \sup_{z>2^M} \ell^{\rho_n,\alpha_n\beta_n}(z) \geq 0 \right| L_1^{\rho_n}, L_1^\rho\right) \right| L_1^\rho\right] \leq \log\left(\frac{\beta_n^2}{\alpha_n^2}\right) \frac{4}{b_{\alpha_n,\beta_n}^2\beta_n^2 L_1^\rho} \sum_{j\geq M} 2^{-j}.
\end{align*}
As $\alpha_n\to\alpha$ and $\beta_n\to\beta$, 
\[ \limsup_{n\to\infty} \left(\log\left(\frac{\beta_n^2}{\alpha_n^2}\right) \frac{4}{b_{\alpha_n,\beta_n}^2\beta_n^2 L_1^\rho} \sum_{j\geq M} 2^{-j}\right) \]
converges to zero for $M\to\infty$. Hence, as $M=\log_2(K)$, (1B) follows from this and~\eqref{eq_proof:1B}.
\end{proof}

\begin{proof}[Proof of \eqref{eq_proof:2B}]
By using Tanaka's formula and It\^{o}'s isometry twice, we obtain
\begin{align}\label{eq_proof:2B_second_moment_local_time}
\begin{split}
\E\left[ \left( L_1^{\rho_n}\right)^2\right] &\leq 2\E\left[ |X_1^{\rho_n,\alpha_n,\beta_n}-\rho_n|^2\right] + 2\E\left[ \left(\int_0^1 \mathrm{sgn}((X_s^{\rho_n,\alpha_n,\beta_n}-\rho_n)-) dX_s^{\rho_n,\alpha_n,\beta_n}\right)^2\right] \\
&\leq 2\E\left[\left( \int_0^1 \sigma_{\rho_n,\alpha_n,\beta_n}(X_2^{\rho_n,\alpha_n,\beta_n}) dW_s\right)^2\right] + 2\E\left[ \int_0^1 d\langle X^{\rho_n,\alpha_n,\beta_n}\rangle_s\right] \\
&\leq 2\E\left[ \int_0^1 \sigma_{\rho_n,\alpha_n,\beta_n}^2(X_2^{\rho_n,\alpha_n,\beta_n}) ds\right] + 2\max\{\alpha_n^2,\beta_n^2\} \\
&\leq 4\max\{\alpha_n^2,\beta_n^2\}.
\end{split}
\end{align} 
As $\alpha_n\rightarrow\alpha$ and $\beta_n\rightarrow\beta$, we can conclude that $(L_1^{\rho_n})_{n\in\N}$ is uniformly integrable. We now define the compensated Poisson process
\[ \overline{N}(t) = N_t - t\]
Then, because the Poisson processes $N$ and $N'$ are fixed in the definition of $M^{\rho_n,\alpha_n,\beta_n}$, by the above uniform integrability and Lemma~\ref{lemma:approx_local_time}, 
\begin{align*}
\E\left[\left| \overline{N}\left(\frac{|z|L_1^{\rho_n}}{\beta_n^2}\right) - \overline{N}\left(\frac{|z|L_1^\rho}{\beta^2}\right)\right|\right] &\leq \E\left[\E\left[\left. \left( \overline{N}\left(\frac{|z|L_1^{\rho_n}}{\beta_n^2}\right) - \overline{N}\left(\frac{|z|L_1^\rho}{\beta^2}\right)\right)^2 \right| L_1^{\rho_n},L_1^{\rho}\right] \right]^\frac12 \\
&= |z|\E\left[\left|\frac{L_1^{\rho_n}}{\beta_n^2} - \frac{L_1^\rho}{\beta^2}\right|\right]^\frac12 \stackrel{n\to\infty}{\longrightarrow} 0.
\end{align*} 
A similar argument works for $N'$ such that convergence of finite dimensional distributions in~\eqref{eq_proof:2B} follows, i.e. $(13.11)$ in Theorem~$13.5$ in~\cite{Billingsley_2}. By the same reasoning, condition~$(13.12)$ of this result also holds true. Tightness in the Skorohod space then follows by additionally checking the moment condition in~$(13.14)$ in the same reference, i.e. we prove for $z_1\leq z_2\leq z_3$
\[ \E\left[ \left(M^{\rho_n,\alpha_n,\beta_n}(z_1)-M^{\rho_n,\alpha_n,\beta_n}(z_2)\right)^2 \left(M^{\rho_n,\alpha_n,\beta_n}(z_2)-M^{\rho_n,\alpha_n,\beta_n}(z_3)\right)^2\right] \leq C|z_1-z_3|^2. \]
Once this is done, \eqref{eq_proof:2B} follows by Theorem~$13.5$ in~\cite{Billingsley_2}. In case $0<z_1$, we have
\begin{align*}
&\E\left[ \left(M^{\rho_n,\alpha_n,\beta_n}(z_1)-M^{\rho_n,\alpha_n,\beta_n}(z_2)\right)^2 \left(M^{\rho_n,\alpha_n,\beta_n}(z_2)-M^{\rho_n,\alpha_n,\beta_n}(z_3)\right)^2\right] \\
&\hspace{0.2cm} = \log\left(\frac{\beta_n^2}{\alpha_n^2}\right)^4 \E\left[\left( \overline{N}(z_1 L_1^{\rho_n}\beta_n^{-2}) - \overline{N}(z_2 L_1^{\rho_n}\beta_n^{-2})\right)^2 \left(\overline{N}(z_2 L_1^{\rho_n}\beta_n^{-2})-\overline{N}(z_3 L_1^{\rho_n}\beta_n^{-2})\right)^2 \right] \\
&\hspace{0.2cm} = \log\left(\frac{\beta_n^2}{\alpha_n^2}\right)^4 \E\left[ \E\left[\left.\left( \overline{N}(z_1 L_1^{\rho_n}\beta_n^{-2}) - \overline{N}(z_2 L_1^{\rho_n}\beta_n^{-2})\right)^2\right| L_1^{\rho_n}\right] \right. \\
&\hspace{4cm} \left. \cdot \E\left[\left.\left( \overline{N}(z_2 L_1^{\rho_n}\beta_n^{-2}) - \overline{N}(z_3 L_1^{\rho_n}\beta_n^{-2})\right)^2\right| L_1^{\rho_n}\right] \right] \\
&\hspace{0.2cm} = \log\left(\frac{\beta_n^2}{\alpha_n^2}\right)^4 \beta_n^{-4} (z_2-z_1)(z_3-z_2) \E\left[ \left( L_1^{\rho_n}\right)^2\right] \leq C(\rho,\alpha,\beta) |z_3-z_1|^2,
\end{align*}
where the second step uses conditional independence of the increments given $L_1^{\rho_n}$, the third one that $N$ is a Poisson process given $L_1^{\rho_n}$ and the last estimate is valid sind $\alpha_n\to\alpha$, $\beta_n\to\beta$ and $\rho_n\to\rho$ together with~\eqref{eq_proof:2B_second_moment_local_time}
Now assume $z_1<0$ and $0\leq z_2\leq z_3$. Then using $\overline{N}(0)=\overline{N}'(0)=0$ and independence of $N,N'$ given $L_1^{\rho_n}$,
\begin{align*}
&\E\left[ \left(M^{\rho_n,\alpha_n,\beta_n}(z_1)- M^{\rho_n,\alpha_n,\beta_n}(z_2)\right)^2 \left(M^{\rho_n,\alpha_n,\beta_n}(z_2)-M^{\rho_n,\alpha_n,\beta_n}(z_3)\right)^2\right] \\
&\hspace{0.2cm} = \E\left[ \left( \log\left(\frac{\alpha_n^2}{\beta_n^2}\right)^2\E\left[\left. \overline{N}'(z_1L_1^{\rho_n} \alpha_n^{-2})^2 \right| L_1^{\rho_n}\right] + \log\left(\frac{\beta_n^2}{\alpha_n^2}\right)^2\E\left[\left. \overline{N}(z_2 L_1^{\rho_n} \beta_n^{-2})^2 \right| L_1^{\rho_n}\right] \right) \right. \\
&\hspace{2cm} \left. \cdot \log\left(\frac{\beta_n^2}{\alpha_n^2}\right)^2  \E\left[\left.\left( \overline{N}(z_2 L_1^{\rho_n}\beta_n^{-2}) - \overline{N}(z_3 L_1^{\rho_n}\beta_n^{-2})\right)^2\right| L_1^{\rho_n}\right] \right] \\
&\hspace{0.2cm} = \E\left[ \left(\log\left(\frac{\alpha_n^2}{\beta_n^2}\right)^2 |z_1| L_1^{\rho_n} \alpha_n^{-2} + \log\left(\frac{\beta_n^2}{\alpha_n^2}\right)^2 z_2 L_1^{\rho_n}\beta_n^{-2}\right) \log\left(\frac{\beta_n^2}{\alpha_n^2}\right)^2 (z_3-z_2)L_1^{\rho_n}\beta_n^{-2}\right] \\
&\hspace{0.2cm}\leq C(\rho,\alpha,\beta) |z_3-z_1|^2.
\end{align*}
Here, the last step works the same as before and uses $|z_1|,|z_2|\leq |z_3-z_1|$. All other cases, i.e. $z_1,z_2<0, z_3\geq 0$ and $z_3<0$ follow analogously to the two cases considered.
\end{proof}

\begin{proof}[Proof of Lemma~\ref{lemma:approx_local_time}]
By Tanaka's formula and the triangle inequality,
\begin{align*}
&\left|L_1^{\rho_n}(X^{\rho_n,\alpha_n,\beta_n}) - L_1^\rho(X^{\rho,\alpha,\beta})\right|\\
&\hspace{0.5cm} \leq \left| \int_0^1 \mathrm{sgn}((X_s^{\rho_n,\alpha_n,\beta_n}-\rho_n)-) dX_s^{\rho_n,\alpha_n,\beta_n} - \int_0^1\mathrm{sgn}((X_s^{\rho,\alpha,\beta}-\rho) dX_s^{\rho,\alpha,\beta}\right| \\
&\hspace{1.5cm} + \left| |X_1^{\rho_n,\alpha_n,\beta_n}-\rho_n| - |X_1^{\rho,\alpha,\beta} -\rho|\right|.
\end{align*}
The second summand can be bounded as
\[ \left| |X_1^{\rho_n,\alpha_n,\beta_n}-\rho_n| - |X_1^{\rho,\alpha,\beta} -\rho|\right| \leq \left| X_1^{\rho_n,\alpha_n,\beta_n}-X_1^{\rho,\alpha,\beta}\right| + \left|\rho_n-\rho\right|.\]
Here, the second summand clearly converges to zero, whereas this follows for the first one, since $\sup_{s\leq 1}|X_s^{\rho_n,\alpha_n,\beta_n}-X_s^{\rho,\alpha,\beta}|\rightarrow_\Pr 0$ by Theorem~$1.5$ in~\cite{LeGall}. By Lemma~$18.12$ in~\cite{Kallenberg}, the convergence in probability of the stochastic integrals follows once we establish the convergence $\langle X^{\rho_n,\alpha_n,\beta_n}\rangle_1 \longrightarrow_\Pr \langle X^{\rho,\alpha,\beta}\rangle_1$ of the respective quadratic variations. By definition, we have
\begin{align*}
\left|\langle X^{\rho_n,\alpha_n,\beta_n}\rangle_1-\langle X^{\rho,\alpha,\beta}\rangle_1 \right| &= \left| \int_0^1 \sigma_{\rho_n,\alpha_n,\beta_n}^2(X_s^{\rho_n,\alpha_n,\beta_n}) ds - \int_0^1 \sigma_{\rho,\alpha,\beta}^2 (X_s^{\rho,\alpha,\beta}) ds\right| \\
&\leq \alpha^2\int_0^1\left| \1_{(-\infty,\rho_n)}(X_s^{\rho_n,\alpha_n,\beta_n}) - \1_{(-\infty,\rho)}(X_s^{\rho,\alpha,\beta})\right| ds \\
&\hspace{0.5cm} + \beta^2\int_0^1\left| \1_{[\rho_n,\infty)}(X_s^{\rho_n,\alpha_n,\beta_n}) - \1_{[\rho,\infty)}(X_s^{\rho,\alpha,\beta})\right| ds \\
&\hspace{0.5cm} + \left| \alpha_n^2-\alpha^2\right| + \left|\beta_n^2-\beta^2\right|.
\end{align*}
Clearly, both summands in the last line converge to zero. We now prove this for the first summand on the right-hand side, the second one can be dealt with analogously. To this aim, let $\epsilon>0$ and $n$ large enough such that for $A_n:=\{\sup_{s\leq 1}|X_s^{\rho_n,\alpha_n,\beta_n}-X_s^{\rho,\alpha,\beta}| \leq \epsilon\}$ we have $\Pr(A_n)\geq 1-\epsilon$. On this set, we have
\begin{align*}
&\1_{A_n} \int_0^1\left| \1_{(-\infty,\rho_n)}(X_s^{\rho_n,\alpha_n,\beta_n}) - \1_{(-\infty,\rho)}(X_s^{\rho,\alpha,\beta})\right| ds \\
&\hspace{1cm} \leq \1_{A_n}\int_0^1\left| \1_{(-\infty,\rho_n)}(X_s^{\rho_n,\alpha_n,\beta_n}) - \1_{(-\infty,\rho_n)}(X_s^{\rho,\alpha,\beta})\right| ds \\
&\hspace{2cm} + \1_{A_n}\int_0^1\left| \1_{(-\infty,\rho_n)}(X_s^{\rho,\alpha,\beta}) - \1_{(-\infty,\rho)}(X_s^{\rho,\alpha,\beta})\right| ds \\
&\hspace{1cm} \leq \int_0^1 \1_{[\rho_n-\epsilon,\rho_n+\epsilon]}(X_s^{\rho,\alpha,\beta}) ds + \int_0^1\left| \1_{(-\infty,\rho_n)}(X_s^{\rho,\alpha,\beta}) - \1_{(-\infty,\rho)}(X_s^{\rho,\alpha,\beta})\right| ds \\
&\hspace{1cm} \leq \frac{1}{\min\{\alpha^2,\beta^2\}}\left( 2\epsilon + |\rho_n-\rho|\right) \sup_{\rho-|\rho-\rho_n|\leq y\leq \rho+|\rho-\rho_n|} L_1^y(X^{\rho,\alpha,\beta}).
\end{align*}
Now, the claim follows by tightness of $\sup_{\rho-|\rho-\rho_n|\leq y\leq \rho+|\rho-\rho_n|} L_1^y(X^{\rho,\alpha,\beta})$, which in turn is a consequence of the fact that $L_1^\bullet(X^{\rho,\alpha,\beta})$ is continuous by Corollary~$1.8$ in~\cite{Revuz/Yor} and thus attains its maximum and minimum on compact sets.
\end{proof}

\section{Proof of Proposition~\ref{prop:rate_of_triple_MLE}}\label{Section:Proof_sketch_triple}

For the proof of Proposition~\ref{prop:rate_of_triple_MLE} recall the normalized log-likelihood function $\ell_{n,\gamma}(\theta_\rho,\theta_\alpha,\theta_\beta)$ from~\eqref{eq:log-likelihood_triple} and its decomposition~\eqref{eq:decomposition_joint_ell} into two parts, where the first one equals $\ell_{n,\gamma}^{\alpha_0,\beta_0}(\theta_\rho)$ given in~\eqref{eq:def_ell_rho} and only depends on $\theta_\rho$. Let $\epsilon>0$, the set $E^\gamma$ be defined as
\begin{align}\label{eq:event_Egamma}
E^\gamma = \begin{cases} \{L_1^{\rho_0}(X)>0\}, &\textrm{ if } \gamma=1, \\ \Omega, &\textrm{ else},\end{cases}
\end{align} 
and $B_n$ be an event with $\Pr_{\rho_0,\alpha_0,\beta_0}(B_n)\geq 1-\epsilon$ for $n$ large enough that will be specified later and consists of paths with typical behavior. In the rest of this section, we tacitly assume $(\theta_\rho,\theta_\alpha,\theta_\beta)\in \R\times (-\alpha_0,\infty)\times (-\beta_0,\infty)$. The statement of Proposition~\ref{prop:rate_of_triple_MLE} will follow once we show that
\[ \Pr_{\rho_0,\alpha_0,\beta_0}\left( n^{(1+\gamma)/2}|\hat{\rho}_n-\rho_0| >K_1, \sqrt{n}\max\{|\hat{\alpha}_n-\alpha_0|,|\hat{\beta}_n-\beta_0|\} >K_2,  E^\gamma \right)\]
can be bounded by a multiple of $\epsilon$ for $K_1,K_2$ large enough. As $\ell_{n,\gamma}(0,0,0)=0$, this probability is bounded by
\begin{align*}
&\Pr_{\rho_0,\alpha_0,\beta_0}\left( \sup_{n^{(1+\gamma)/2}|\theta_\rho|>K_1,\sqrt{n}\max\{|\theta_\alpha|, |\theta_\beta|\}>K_2} \ell_{n,\gamma}(\theta_\rho,\theta_\alpha,\theta_\beta)\geq 0, B_n\cap E^\gamma\right) + \epsilon \\
&\hspace{0.5cm} \leq \sum_{j=1}^3 \Pr_{\rho_0,\alpha_0,\beta_0}\left( \sup_{(\theta_\rho,\theta_\alpha,\theta_\beta)\in \mathcal{T}_{n,\gamma}^j(K)} \left( \ell_{n,\gamma}^{\alpha_0,\beta_0}(\theta_\rho) + \ell_{n,\gamma}^2(\theta_\rho,\theta_\alpha,\theta_\beta)\right) \geq 0, B_n\cap E^\gamma\right) +\epsilon,
\end{align*}
where $K=(K_1,K_2,K_3)\in\R_{\geq 0}^3$ and the three index sets for the parameter vector $(\theta_\rho,\theta_\alpha,\theta_\beta)$ are given as
\begin{align*}
\mathcal{T}_{n,\gamma}^1(K) &:= \left\{(\theta_\rho,\theta_\alpha,\theta_\beta): \max\{|\theta_\alpha|,|\theta_\beta|\}\leq K_2/\sqrt{n}, K_1 n^{-(1+\gamma)/2}<|\theta_\rho|<K_3n^{-\gamma/2}\right\} ,\\
\mathcal{T}_{n,\gamma}^2(K) &:= \left\{(\theta_\rho,\theta_\alpha,\theta_\beta): \max\{|\theta_\alpha|,|\theta_\beta|\}> K_2/\sqrt{n}, |\theta_\rho|\leq K_3n^{-\gamma/2}\right\} ,\\
\mathcal{T}_{n,\gamma}^3(K) &:= \left\{(\theta_\rho,\theta_\alpha,\theta_\beta): \theta_\rho > K_3n^{-\gamma/2} \right\}.
\end{align*}
In Subsection~\ref{Subsection:tripleMLE_1} in Appendix~\ref{Section:triple_details} we then prove that $K_3$ can be chosen large enough such that the corresponding probability is bounded by $\epsilon$, then for this $K_3$ we construct the corresponding $K_2(K_3)$ to bound the probability with $\mathcal{T}_{n,\gamma}^2(K)$ by $\epsilon$ and finally for given values $K_2,K_3$ we choose $K_1(K_2,K_3)$ appropriately, i.e. the strategy is
\begin{align*}
\textrm{choose } K_3\ \longrightarrow \textrm{ choose } K_2=K_2(K_3)\ \longrightarrow \textrm{ choose } K_1=K_1(K_2,K_3).
\end{align*}
As can be expected by the definition of these sets, the term $\ell_{n,\gamma}^{\alpha_0,\beta_0}$ will dominate on the set $\mathcal{T}_{n,\gamma}^1$ and $\ell_{n,\gamma}^2$ turns out to be negligible on this set. On $\mathcal{T}_{n,\gamma}^2$ this swaps and $\ell_{n,\gamma}^2$ dominates $\ell_{n,\gamma}^{\alpha_0,\beta_0}$. On the remaining set $\mathcal{T}_{n,\gamma}^3$, however, there exists a subtle interplay of both of them that is described later. A crucial step to formally establish the aforementioned claims is to split both likelihood functions into a dominating term that can be understood reasonably well and a remainder (the form of this decomposition varies with $\mathcal{T}_{n,\gamma}^j$). Similar to the consistency proof under the assumption of known $\alpha_0,\beta_0$ this is done on a set ensuring typical behavior of the paths of $X$. Note that this reasoning applies both for the high-frequency scheme $\gamma=1$ and all intermediate regimes $\gamma\in (0,1)$. In particular, we set
\begin{align}\label{eq_proof:setAn}
\begin{split}
A_n &= \left\{ -\Gamma \sqrt{n\Delta_n^\gamma} < \underline{X}_{n\Delta_n^\gamma},\overline{X}_{n\Delta_n^\gamma}< \Gamma \sqrt{n\Delta_n^\gamma} \right\} \cap \left\{ \sup_{|t-s|<\Delta_n^\gamma} |X_t - X_s| \leq (\Delta_n^\gamma)^{4/9} \right\}\\
&\hspace{1.5cm}  \cap  \left\{\inf_{y\in [\rho_0-\zeta\sqrt{n\Delta_n^\gamma}, \rho_0+\zeta\sqrt{n\Delta_n^\gamma}]} L_{n\Delta_n^\gamma}^y(X) >\sqrt{n\Delta_n^\gamma} \xi\right\},
\end{split}
\end{align}
where the constants $\Gamma,\xi,\zeta$ are chosen such that $\Pr_{\rho_0,\alpha_0,\beta_0}(A_n)>1-\epsilon$ for $n\geq n_0$ large enough. This is possible for $\Gamma$ by the Burkholder-Davis-Gundy inequality and follows for $\xi,\zeta$ by Lemma~\ref{lemma:probability_A3}, whereas it is a consequence of Theorem~$1$ in~\cite{Fischer/Nappo} for the second set in this intersection. On the set $A_n$ we have the following result that allows to estimate the number of $X_{(k-1)\Delta_n^\gamma}$ falling into a given interval.

\begin{lemma}\label{lemma:estimate_counts_interval}
Let $-\zeta\sqrt{n\Delta_n^\gamma}\leq a,b\leq \zeta\sqrt{n\Delta_n^\gamma}$ such that $\rho_0-a\leq \rho_0+b$ and $a+b>3(\Delta_n^\gamma)^{4/9}$. Then,
\[ \1_{A_n} \sum_{k=1}^n \1_{\{\rho_0-a< X_{(k-1)\Delta_n^\gamma}<\rho_0+b\}} \geq \1_{A_n}\frac{\xi(a+b)}{\max\{\alpha_0^2,\beta_0^2\}} \sqrt{\frac{n}{\Delta_n^\gamma}}\]
\end{lemma}
\begin{proof}
As $X_s<\rho_0- (\Delta_n^\gamma)^{4/9}$ for some $s\in [(k-1)\Delta_n^\gamma,k\Delta_n^\gamma]$ implies $X_{(k-1)\Delta_n^\gamma}<\rho_0$ on $A_n$, we obtain using the occupation times formula 
\begin{align*}
\1_{A_n} \sum_{k=1}^n \1_{\{\rho_0-a< X_{(k-1)\Delta_n^\gamma}<\rho_0+b\}} & \geq  \1_{A_n} \frac{1}{\Delta_n^\gamma} \int_0^{n\Delta_n^\gamma} \1_{(\rho_0-a+(\Delta_n^\gamma)^{4/9}, \rho_0+b -(\Delta_n^\gamma)^{4/9}]}(X_s) ds \\
&\geq \frac{1}{\max\{\alpha_0^2,\beta_0^2\}}\1_{A_n} \frac{1}{\Delta_n^\gamma} \int_{\rho_0-a+(\Delta_n^\gamma)^{4/9}}^{\rho_0+b-(\Delta_n^\gamma)^{4/9}} L_{n\Delta_n^\gamma}^y(X) dy \\
&\geq \frac{1}{\max\{\alpha_0^2,\beta_0^2\}}\1_{A_n} \frac{1}{\Delta_n^\gamma} \xi \sqrt{n\Delta_n^\gamma} (a+b -2(\Delta_n^\gamma)^{4/9})
\end{align*}
and the claim follows by the assumption $a+b>3(\Delta_n^\gamma)^{4/9}$. 
\end{proof}

In what follows, we describe the decomposition of both $\ell_{n,\gamma}^{\alpha_0,\beta_0}$ and $\ell_{n,\gamma}^2$ into leading terms and remainders. To this aim, let
\begin{align}\label{def:Theta_j} 
\Theta_{n,\gamma}^1 = \left[ 0, n^{1/4-\gamma/2}\right] \quad \textrm{ and }\quad \Theta_{n,\gamma}^2 = \left( n^{1/4-\gamma/2},\infty\right).
\end{align}
According to straightforward modification of Lemma~\ref{lemma:bound_L_sqrt(n)} for $\gamma\in (0,1)$ and Lemma~$4.1$ in~\cite{Brutsche/Rohde} for $\gamma=1$ (see Remark~\ref{Remark_LK=0}), there exists a sequence $(B_n)_{n\in\N}$ with $A_n\subset B_n$ and $\Pr_{\rho_0,\alpha_0,\beta_0}(B_n^c)<\epsilon$ for which we have
\begin{align}\label{eq:decomposition_ell_1}
\ell_{n,\gamma}^{\alpha_0,\beta_0}(\theta_\rho)\1_{B_n} = \left[\log\left(\frac{\beta_0}{\alpha_0}\right) - \frac12\left(\frac{\beta_0^2}{\alpha_0^2}-1\right)\right]\1_{B_n}  \sum_{k=1}^n \1_{\{\rho_0\leq X_{(k-1)\Delta_n^\gamma} \leq \rho_0+\theta_\rho\}} + R_{n,\gamma}^1(\theta_\rho),
\end{align}
where 
\begin{align}\label{eq_proof:order_remainder_ell1}
 \sup_{\theta_\rho\in\Theta_{n,\gamma}^j} \left| R_{n,\gamma}^1(\theta_\rho)\right]\1_{B_n} \leq C_1 n^{1/2+\gamma(1-j)/6}\1_{B_n}.
\end{align} 
For $\ell_{n,\gamma}^2$, we introduce the decomposition
\begin{align}\label{eq:decomposition_ell_2}
\ell_{n,\gamma}^2(\theta_\rho,\theta_\alpha,\theta_\beta) = H_{n,\gamma}^{\leq}(\theta_\rho,\theta_\alpha) + H_{n,\gamma}^{>}(\theta_\rho,\theta_\beta)+ R_{n,\gamma}^{\leq}(\theta_\rho,\theta_\alpha,\theta_\beta)+ R_{n,\gamma}^{>}(\theta_\rho,\theta_\alpha,\theta_\beta),
\end{align} 
where the leading terms $H_{n,\gamma}^{\leq}, H_{n,\gamma}^{>}$ are given as
\begin{align*}
H_{n,\gamma}^{\leq}(\theta_\rho,\theta_\alpha) &:= \sum_{k=1}^n \left[\log\left(\frac{\alpha_0}{\alpha_0+\theta_\alpha}\right) -\frac12\left(\frac{\alpha_0^2}{(\alpha_0+\theta_\alpha)^2}-1\right)\frac{(X_{k\Delta_n^\gamma}-X_{(k-1)\Delta_n^\gamma})^2}{\alpha_0^2\Delta_n^\gamma}\right]\\
&\hspace{2cm} \cdot \1_{\{X_{(k-1)\Delta_n^\gamma} \leq \rho_0+\theta_\rho\}} \\
H_{n,\gamma}^{>}(\theta_\rho,\theta_\beta) &:= \sum_{k=1}^n \left[\log\left(\frac{\beta_0}{\beta_0+\theta_\beta}\right) -\frac12\left(\frac{\beta_0^2}{(\beta_0+\theta_\beta)^2}-1\right)\frac{(X_{k\Delta_n^\gamma}-X_{(k-1)\Delta_n^\gamma})^2}{\beta_0^2\Delta_n^\gamma}\right] \\
&\hspace{2cm} \cdot\1_{\{X_{(k-1)\Delta_n^\gamma} > \rho_0+\theta_\rho\}} 
\end{align*} 
and the remainder terms $R_{n,\gamma}^{\leq}, R_{n,\gamma}^{>}$ as
\begin{align*}
R_{n,\gamma}^{\leq}(\theta_\rho,\theta_\alpha,\theta_\beta) &:= \sum_{k=1}^n \left( \ell_n^2(\theta_\rho,\theta_\alpha,\theta_\beta) - H_{n,\gamma}^{\leq}(\theta_\rho,\theta_\alpha,\theta_\beta)\right) \1_{\{X_{(k-1)\Delta_n^\gamma\leq \rho_0+\theta_\rho}\}} \\
R_{n,\gamma}^{>}(\theta_\rho,\theta_\alpha,\theta_\beta) &:= \sum_{k=1}^n \left( \ell_n^2(\theta_\rho,\theta_\alpha,\theta_\beta) - H_{n,\gamma}^{>}(\theta_\rho,\theta_\alpha,\theta_\beta)\right) \1_{\{X_{(k-1)\Delta_n^\gamma> \rho_0+\theta_\rho}\}}.
\end{align*}
Heuristically, for the squared increments we have $(X_{k\Delta_n^\gamma}-X_{(k-1)\Delta_n^\gamma})^2\approx \alpha_0^2\Delta_n^\gamma$ in case $X_{(k-1)\Delta_n^\gamma},X_{k\Delta_n^\gamma}\ll\rho_0$ and analogously $(X_{k\Delta_n^\gamma}-X_{(k-1)\Delta_n^\gamma})^2\approx \beta_0^2\Delta_n^\gamma$ for observations $X_{(k-1)\Delta_n^\gamma},X_{k\Delta_n^\gamma}\gg\rho_0$. Hence, for $\theta_\rho\geq 0$, the terms $H_n^{\leq}(\theta_\rho,\theta_\alpha)$ and $H_n^{>}(\theta_\rho,\theta_\beta)$ will often be substituted with their counterparts
\begin{align*}
\overline{H}_{n,\gamma}^{\leq} (\theta_\rho,\theta_\alpha) & :=  \left[\log\left(\frac{\alpha_0}{\alpha_0+\theta_\alpha}\right) -\frac12\left(\frac{\alpha_0^2}{(\alpha_0+\theta_\alpha)^2}-1\right)\right] \sum_{k=1}^n\1_{\{X_{(k-1)\Delta_n^\gamma} \leq \rho_0\}} \\
&\hspace{0.5cm} +  \left[\log\left(\frac{\alpha_0}{\alpha_0+\theta_\alpha}\right) -\frac12\left(\frac{\alpha_0^2}{(\alpha_0+\theta_\alpha)^2}-1\right)\frac{\beta_0^2}{\alpha_0^2}\right] \sum_{k=1}^n\1_{\{\rho_0 <X_{(k-1)\Delta_n^\gamma} \leq \rho_0+\theta_\rho\}}\\
\overline{H}_{n,\gamma}^{>}(\theta_\rho,\theta_\beta)& := \left[\log\left(\frac{\beta_0}{\beta_0+\theta_\beta}\right) -\frac12\left(\frac{\beta_0^2}{(\beta_0+\theta_\beta)^2}-1\right)\right] \sum_{k=1}^n \1_{\{X_{(k-1)\Delta_n^\gamma}> \rho_0+\theta_\rho\}}.
\end{align*} 
Note that the first and third square bracket are non-positive and equal to zero if and only if $\theta_\alpha=0$ and $\theta_\beta=0$, respectively. Moreover, a second order Taylor expansion of the function $f(x)=\log(x)-(x^2-1)/2$ gives $f(x) = -(1-\xi^{-2})(x-1)^2$ for some intermediate value $\xi$ between $1$ and $x$. In case the sums counting the cases where $X_{(k-1)\Delta_n^\gamma}\leq\rho_0$ and that where $X_{(k-1)\Delta_n^\gamma}>\rho_0+\theta_\rho$ are both of order $n$, this means that both leading terms are of constant order for $\theta_\alpha,\theta_\beta = K_2/\sqrt{n}$ and negative. For $\theta_\rho$ not being too small, this order $n$ is given for both summands by Lemma~\ref{lemma:estimate_counts_interval} and the number of $k$ with $\rho_0<X_{(k-1)\Delta_n^\gamma}\leq\rho_0+\theta_\rho$ turns out to be small, which together is the reason why $\ell_{n,\gamma}^2$ is the dominant term on the set $\mathcal{T}_{n,\gamma}^2$. Note in particular, that a negative contribution emerges if either $|\theta_\alpha|>K_2/\sqrt{n}$ or $|\theta_\beta|>K_2/\sqrt{n}$. To formally give the proof in Appendix~\ref{Subsection:tripleMLE_1}, the next result will be used that specifies the error when switching by the different quantities as outlined above. Its proof is deferred to Appendix~\ref{Subsection:tripleMLE_2}.

\begin{lemma}\label{lemma:estimates_tripleMLE}
Let $\epsilon,L>0$, and $g(x) = |\log(x)| + \frac12| x^2-1|$. Then there exists $n_0=n_0(L)\in\N$, a constant $C=C(\epsilon,\alpha_0,\beta_0,L,\Gamma,\xi,\zeta)>0$ and a sequence of sets $(B_n)_{n\in\N}$ with $B_n\subset A_n$ and $\Pr_{\rho_0,\alpha_0,\beta_0}(B_n^c)<\epsilon$ for all $n\in\N$ such that for all $n\geq n_0$:
\begin{itemize}
\item[(a)] For small values of $\theta_\alpha,\theta_\beta$ we have the uniform estimates
\begin{align*}
&\sup_{|\theta_\alpha|\leq L/\sqrt{n}}\sup_{0\leq \theta_\rho\leq Ln^{-\gamma/2} } \left| H_{n,\gamma}^{\leq}(\theta_\rho,\theta_\alpha)\right|\1_{B_n} \leq CL\quad \textrm{ and } \\
&\sup_{|\theta_\beta|\leq L/\sqrt{n}}\sup_{0\leq \theta_\rho\leq Ln^{-\gamma/2} } \left| H_{n,\gamma}^{>}(\theta_\rho,\theta_\beta)\right|\1_{B_n} \leq CL.
\end{align*} 
\item[(b)] For small values of $\theta_\rho$, we have the uniform estimates
\begin{align*}
& \sup_{|\theta_\rho|\leq Ln^{-\gamma/2}}\left| H_{n,\gamma}^{\leq}(\theta_\rho,\theta_\alpha) - H_{n,\gamma}^{\leq}(0,\theta_\alpha)\right|\1_{B_n} \leq CL\sqrt{n} g\left(\frac{\alpha_0}{\alpha_0+\theta_\alpha}\right)   \quad\textrm{ and }\\
& \sup_{|\theta_\rho|\leq Ln^{-\gamma/2}}\left| H_{n,\gamma}^{>}(\theta_\rho,\theta_\beta) - H_{n,\gamma}^{>}(0,\theta_\beta)\right| \1_{B_n} \leq CL\sqrt{n} g\left(\frac{\beta_0}{\beta_0+\theta_\beta}\right).
\end{align*}
\item[(c)] For $\theta_\rho\geq 0$ we have the difference estimates
\begin{align*}
& \left| H_{n,\gamma}^{\leq}(\theta_\rho,\theta_\alpha) - \overline{H}_{n,\gamma}^{\leq}(\theta_\rho,\theta_\alpha)\right|\1_{B_n} \leq C\sqrt{n}g\left(\frac{\alpha_0}{\alpha_0+\theta_\alpha}\right)\quad\textrm{ and } \\
&\left| H_{n,\gamma}^{>}(\theta_\rho,\theta_\beta) - \overline{H}_{n,\gamma}^{>}(\theta_\rho,\theta_\beta)\right|\1_{B_n} \leq C\sqrt{n}g\left(\frac{\beta_0}{\beta_0+\theta_\beta}\right).
\end{align*}
\item[(d)] Let $R>0$ and $d_{\alpha_0,\beta_0}:=\log(2)-\log(1-|\alpha_0-\beta_0|/(\alpha_0+\beta_0))$. For $R_{n,\gamma}^{\leq}$, $R_{n,\gamma}^{>}$ we find function $T_{n,\gamma}^{\leq},T_{n,\gamma}^{>}$ and a constant $C(R)>0$ such that for $\star\in\{\leq,>\}$ and $n\geq n_0(R)$,
\begin{align*}
R_{n,\gamma}^{\star}(\theta_\rho,\theta_\alpha,\theta_\beta) &\leq C(R)\left( g\left(\frac{\alpha_0}{\alpha_0+\theta_\alpha}\right)+g\left(\frac{\beta_0}{\beta_0+\theta_\beta}\right) \right) T_{n,\gamma}^{\star}(\theta_\rho,\theta_\alpha,\theta_\beta) \\
&\hspace{1cm}+ nd_{\alpha_0,\beta_0}\left(\1_{\{\theta_\alpha>R\}} +\1_{\{\theta_\beta>R\}}\right),
\end{align*}
where the functions $T_{n,\gamma}^{\leq},T_{n,\gamma}^{>}$ satisfy the bound
\[  \sup_{\theta_\rho\in \Theta_{n,\gamma}^j} \sup_{\theta_\alpha,\theta_\beta} \left|T_{n,\gamma}^{\star}(\theta_\rho,\theta_\alpha,\theta_\beta)\right| \1_{B_n} \leq C n^{1/2+\gamma(1-j)/6} \1_{B_n}.\]
\end{itemize}
\end{lemma}

It remains to deal with the set $\mathcal{T}_{n,\gamma}^3$, where both $\theta_\rho$ and $\max\{|\theta_\alpha|,\theta_\beta|\}$ deviate from zero. Here, the reasoning outlined above does not apply and longer for two different reasons:
\begin{itemize}
\item[(P1)] The second summand in $\overline{H}_{n,\gamma}^\leq$ is not necessarily negligible any longer as for large values of $\theta_\rho$ the number of $k$ for which $\rho_0<X_{(k-1)\Delta_n^\gamma}\leq\rho_0+\theta_\rho$ is also large and the factor in front is not necessarily negative.
\item[(P2)] In case $\theta_\alpha\approx 0$, we automatically have $|\theta_\beta|>K_2/\sqrt{n}$ and therefore the square bracket in the definition of $\overline{H}_{n,\gamma}^{>}$ is negative. However, for large values of $\theta_\rho$, there might be few $k$ with $X_{(k-1)\Delta_n^\gamma}>\rho_0+\theta_\rho$.
\end{itemize}
Nevertheless, for large $\theta_\rho$, also $\ell_{n,\gamma}^{\alpha_0,\beta_0}$ is quite negative and we try to balance this out with the above mentioned effects of $\ell_{n,\gamma}^2$. To overcome these issues, we combine the expansions~\eqref{eq:decomposition_ell_1} and~\eqref{eq:decomposition_ell_2} to
\begin{align}\label{eq:decomposition_ell_1+2}
\ell_{n,\gamma}^{\alpha_0,\beta_0}(\theta_\rho) + \ell_{n,\gamma}^2(\theta_\rho,\theta_\alpha,\theta_\beta)  &= \overline{H}_{n,\gamma}^{\leq}(0,\theta_\alpha) + \overline{H}_{n,\gamma}^{>}(\theta_\rho,\theta_\beta) + T_{n,\gamma}(\theta_\rho,\theta_\alpha) + R_{n,\gamma}^{all}(\theta_\rho,\theta_\alpha,\theta_\beta),
\end{align}
where
\begin{align*}
T_{n,\gamma}(\theta_\rho,\theta_\alpha) &= \overline{H}_{n,\gamma}^{\leq}(\theta_\rho,\theta_\alpha) - \overline{H}_{n,\gamma}^{\leq}(0,\theta_\alpha) \\
&\hspace{1cm} +\left[\log\left(\frac{\beta_0}{\alpha_0}\right) - \frac12\left(\frac{\beta_0^2}{\alpha_0^2}-1\right)\right] \sup_{\theta_\rho\in\Theta_{n,\gamma}^j} \sum_{k=1}^n \1_{\{\rho_0\leq X_{(k-1)\Delta_n^\gamma} \leq \rho_0+\theta_\rho\}} \\
&= \left[ \log\left(\frac{\beta_0}{\alpha_0}\right) - \frac12\left(\frac{\beta_0^2}{\alpha_0^2}-1\right) + \log\left(\frac{\alpha_0}{\alpha_0+\theta_\alpha}\right) - \frac12\left( \frac{\alpha_0^2}{(\alpha_0+\theta_\alpha)^2}-1\right)\frac{\beta_0^2}{\alpha_0^2}\right] \\
&\hspace{1cm} \cdot\sum_{k=1}^n \1_{\{\rho_0\leq X_{(k-1)\Delta_n^\gamma} \leq \rho_0+\theta_\rho\}}
\end{align*}
and
\begin{align*}
R_{n,\gamma}^{all}(\theta_\rho,\theta_\alpha,\theta_\beta) &= R_{n,\gamma}^1(\theta_\rho) + R_{n,\gamma}^{\leq}(\theta_\rho,\theta_\alpha,\theta_\beta) +R_{n,\gamma}^{>}(\theta_\rho,\theta_\alpha,\theta_\beta)\\
&\hspace{0.8cm} + \left( H_{n,\gamma}^{\leq}(\theta_\rho,\theta_\alpha) - \overline{H}^{\leq}(\theta_\rho,\theta_\alpha) \right)+ \left( H_{n,\gamma}^{>}(\theta_\rho,\theta_\beta) - \overline{H}_{n,\gamma}^{>}(\theta_\rho,\theta_\beta) \right).
\end{align*}
As discussed earlier, both $\overline{H}_{n,\gamma}^{\leq}$ and $\overline{H}_{n,\gamma}^{>}$ are non-positive. Surprisingly, this also turns out to be true for the term $T_{n,\gamma}$ that combines terms of $\ell_{n,\gamma}^{\alpha_0,\beta_0}$ and $\ell_{n,\gamma}^2$, see Lemma~\ref{lemma:function_negative} in Appendix~\ref{Subsection:tripleMLE_3} that establishes
\[ \log\left(\frac{\beta_0}{\alpha_0}\right) - \frac12\left(\frac{\beta_0^2}{\alpha_0^2}-1\right) + \log\left(\frac{\alpha_0}{\alpha_0+\theta_\alpha}\right) - \frac12\left( \frac{\alpha_0^2}{(\alpha_0+\theta_\alpha)^2}-1\right)\frac{\beta_0^2}{\alpha_0^2} \leq 0\]
with equality in case $\theta_\alpha=\beta_0-\alpha_0$. The crucial observation is now that either $\overline{H}_{n,\gamma}^{\leq}$ is close to zero (which happens for $\theta_\alpha\approx 0$), in which case $T_{n,\gamma}$ is significantly negative or the other way around. All details for these heuristics are fully provided in Appendix~\ref{Subsection:tripleMLE_1}. When working these out, the only remaining obstacle is that balancing the leading terms with the remainders is difficult for $\theta_\beta\to -\beta_0$ or $\theta_\beta\to\infty$ and at the same time $\theta_\rho$ being very large (i.e. close to $\sup_{t\leq n\Delta_n^\gamma} X_t$). In this case the bound provided by Lemma~\ref{lemma:estimates_tripleMLE}(d) is no longer useful as $g(\beta_0/(\beta_0+\theta_\beta))$ explodes. However, the problem occurs only for rare events in the sense of atypical path behavior and the probability of these can be controlled reasonably well. This is formally established by substituting the decomposition~\eqref{eq:decomposition_ell_1+2} with that given in Lemma~\ref{lemma:bound_small_theta_beta}. The crucial step in its proof is to construct an injective mapping
\begin{align*}
&\left\{ k: X_{(k-1)\Delta_n^\gamma}\leq \rho_0+\theta_\rho\leq X_{k\Delta_n^\gamma}\right\}\\
&\hspace{0.2cm} \longrightarrow \left\{ k:X_{k\Delta_n^\gamma}<\rho_0+\theta_\rho<X_{(k-1)\Delta_n^\gamma} \right\} \cup \left\{ k:X_{(k-1)\Delta_n^\gamma}\geq \rho_0+\theta_\rho, X_{k\Delta_n^\gamma}>\rho_0+\theta_\rho\right\},
\end{align*}
which turns out to be possible for all observations $X_0,X_{\Delta_n^\gamma},\dots, X_{n\Delta_n^\gamma}$ for which $X_{n\Delta_n^\gamma}$ is bounded awaw from the global maximum of the path $(X_t)_{0\leq t\leq n\Delta_n^\gamma}$. This event is seen to have high probability by means of the Dambis--Dubins--Schwarz theorem. All details are provided in Subsection~\ref{Subsection:tripleMLE_3}.


%
%

\begin{funding}
This work was supported by the DFG Research Unit $5381$, RO $3766/8$-$1$ and CRC$ 1597$, Project-ID $499552394$.
\end{funding}




\bibliographystyle{imsart-nameyear} 
\bibliography{Bibliography}       



\newpage
\setcounter{section}{0}
\setcounter{page}{1}
\renewcommand*\thesection{\Alph{section}}

\begin{frontmatter}

\title{\vspace{-0.3cm} Supplement to "Self-organized regime switching in null-recurrent dynamics"}
\runtitle{Supplement "Self-organized regime switching in null-recurrent dynamics"}

\begin{aug}
\author[A]{\fnms{Johannes}~\snm{Brutsche}}, 
\author[A]{\fnms{Sebastian}~\snm{Hahn}}
\and
\author[A]{\fnms{Angelika}~\snm{Rohde}}
\address[A]{Mathematical Institute, University of Freiburg\printead[presep={,\ }]{e1}\printead[presep={,\ }]{e2}\printead[presep={,\ }]{e3}}
\end{aug}

\end{frontmatter}

\vspace{-0.1cm}
This supplementary material is organized as follows:\\
\vspace{-0.2cm}

{\small
\tableofcontents
\addtocontents{toc}{\protect\setcounter{tocdepth}{3}} 
}

\section{Notation}
Throughout the whole technical supplement, $C_{\alpha,\beta}$ denotes some real and positive constant that only depends on $\alpha$ and $\beta$ but may change from line to line. Any other dependencies are highlighted explicitly, for example by writing $C_{\alpha,\beta}(K)$ in case the constant additionally depends on $K$. Sometimes we want to highlight that the constant $C_{\alpha_0,\beta_0}$ only depends on one of the parameters in which case the other one is omitted. 
We frequently use the notation
\[ \Pr( A,B ) := \Pr(A\cap B).\]
We denote $\Pr_0=\Pr_{\rho_0,\alpha_0,\beta_0}$, where the subscript $(\rho_0,\alpha_0,\beta_0)$ indicates that $X$ solves~\eqref{eq:SDE_OBM} with $\sigma_{\rho_0,\alpha_0,\beta_0}$ and arbitrary starting point $x_0$. In cases where an argument involves two solutions of~\eqref{eq:SDE_OBM} with different starting points, this is explicitly highlighted by writing $X^{x_0}$. In all other cases, an arbitrary starting point is tacitly assumed.

\section{Some preliminary results on OBM}

By a straightforward case-by-case study, one obtains that the transition density \eqref{eq:transition_density} is dominated by a Gaussian density, i.e.
\begin{align}\label{eq:bound_transition_density}
p_t^{\rho,\alpha,\beta}(x,y)\leq \frac{2}{\alpha+\beta}\frac{1}{\sqrt{2\pi t}}\frac{\max\{\alpha,\beta\}}{\min\{\alpha,\beta\}} \exp\left( - \frac{(y-x)^2}{2t\max\{\alpha^2,\beta^2\}}\right). 
\end{align} 
From this, one can easily deduce the following result for an OBM solving~\eqref{eq:SDE_OBM} with $\sigma_{\rho_0,\alpha_0,\beta_0}$ and $X_0=x_0$.

\begin{cor}\label{cor:bound_exp_X^2}
Let $a\in\R$, $b\in\R_{>0}$ and $0\leq l<k$. Then we have
\[ \E_0\left[\left. \exp\left( - \frac{(X_{k\Delta_n^\gamma} - a)^2}{2b\Delta_n^\gamma}\right)\right| X_{l\Delta_n^\gamma}\right] \leq c_{\alpha_0,\beta_0}(b)\frac{1}{\sqrt{k-l}}\]
for some constant $c_{\alpha_0,\beta_0}(b)>0$ not depending on $n$ and $a$.
\end{cor}

Corollary~\ref{cor:bound_exp_X^2} is often used in combination with the inequalities
\[ \sum_{k=1}^n \frac{1}{\sqrt{k}} \leq 2\sqrt{n}, \quad \textrm{ and }\quad \sum_{k=1}^n \sum_{l=1}^{k-1} \frac{1}{\sqrt{k(k-l)}}\leq 4n.\]
Both of them are easily obtained by comparison with corresponding integral and applied without further notice.

\begin{cor}\label{cor:increment_moments}
Let $p\in\N$. Then there exists a constant $C_{\alpha_0,\beta_0}(p)$ such that
\[ \E_0\left[\left. \frac{|X_{k\Delta_n^\gamma}-X_{(k-1)\Delta_n^\gamma}|^p}{(\Delta_n^\gamma)^{p/2}}\right| X_{(k-1)\Delta_n^\gamma}\right] \leq C_{\alpha_0,\beta_0}(p).\]
\end{cor}
\begin{proof}
Using the bound~\eqref{eq:bound_transition_density} directly reveals
\begin{align*}
&\E_0\left[\left. \frac{|X_{k\Delta_n^\gamma}-X_{(k-1)\Delta_n^\gamma}|^p}{(\Delta_n^\gamma)^{p/2}} \right| X_{(k-1)\Delta_n^\gamma}\right]\\
&\hspace{1cm}  \leq C_{\alpha_0,\beta_0} \frac{1}{\sqrt{\Delta_n^\gamma}} \int_\R \frac{|y-X_{(k-1)\Delta_n^\gamma}|^p}{(\Delta_n^\gamma)^{p/2}} \exp\left(-\frac{(y-X_{(k-1)\Delta_n^\gamma})^2}{2\max\{\alpha_0^2,\beta_0^2\}\Delta_n^\gamma}\right) dy  \\
&\hspace{1cm} \leq C_{\alpha_0,\beta_0}\int_\R |y|^p \exp\left(-\frac{y^2}{2\max\{\alpha_0^2,\beta_0^2\}}\right) dy
\end{align*}
and the integral is bounded by a constant depending on $p$.
\end{proof}

\section{Proof of (1A) for $\gamma\in (0,1)$ and minimax lower bound}\label{Section:Proof_1A}

The proof of (1A) for $\gamma\in (0,1)$ will be based on the decomposition
\begin{align}\label{eq:decomposition_drift_martingale}
\ell_{n,\gamma,t}^{\alpha_0,\beta_0}(\theta) = B_{n,,\gamma,t}^{\alpha_0,\beta_0}(\theta) + M_{n,\gamma,t}^{\alpha_0,\beta_0}(\theta),
\end{align} 
of the sequential version 
\[ \ell_{n,\gamma,t}^{\alpha_0,\beta_0}(\theta) = \sum_{k=1}^{\lceil nt\rceil} \log\left(\frac{p_{\Delta_n^\gamma}^{\rho_0+\theta,\alpha_0,\beta_0}(X_{(k-1)\Delta_n^\gamma},X_{k\Delta_n^\gamma})}{p_{\Delta_n^\gamma}^{\rho_0,\alpha_0,\beta_0}(X_{(k-1)\Delta_n^\gamma},X_{k\Delta_n^\gamma}}\right) \]
of the log-likelihood in~\eqref{eq:def_ell_rho}, where
\[ B_{n,\gamma,t}^{\alpha_0,\beta_0}(\theta) := \sum_{k=1}^{\lceil nt\rceil} \E_0\left[\left. \log\left(\frac{p_{\Delta_n^\gamma}^{\rho_0+\theta,\alpha_0,\beta_0}(X_{(k-1)\Delta_n^\gamma},X_{k\Delta_n^\gamma})}{p_{\Delta_n^\gamma}^{\rho_0,\alpha_0,\beta_0}(X_{(k-1)\Delta_n^\gamma},X_{k\Delta_n^\gamma}}\right)\right| X_{(k-1)\Delta_n^\gamma}\right]\]
and $M_{n,\gamma,t}^{\alpha_0,\beta_0}(\theta)=\ell_{n,\gamma,t}^{\alpha_0,\beta_0}(\theta)-B_{n,\gamma,t}^{\alpha_0,\beta_0}(\theta)$. This decomposition is useful, as $M_{n,\gamma,t}^{\alpha_0,\beta_0}(\theta)$ is now a sum of martingale increments. Moreover, we have the following expansion of the drift term $B_{n,\gamma,t}^{\alpha_0,\beta_0}$.

\begin{lemma}\label{lemma:expansion_drift}
Let $|\theta|\leq K\sqrt{\Delta_n^\gamma}$ and $X=(X_t)_{t\geq 0}$ an OBM started in $x_0\in\R$ with diffusion coefficient $\sigma_{\rho_0,\alpha_0,\beta_0}$. Then we have
\[ B_{n,\gamma,t}^{\alpha_0,\beta_0}(\theta) = -|\theta|\left( \1_{\{\theta\geq 0\}}F_{\alpha_0,\beta_0} + \1_{\{\theta<0\}}\tilde{F}_{\alpha_0,\beta_0}\right) \Lambda_{n,\gamma}^{\alpha_0,\beta_0}\left((X_{k\Delta_n^\gamma})_{k=0,\dots,\lceil nt\rceil}\right) + r_{n,\gamma}(t,\theta),\]
where
\begin{align*}
\Lambda_{n,\gamma}^{\alpha_0,\beta_0}\left((X_{k\Delta_n^\gamma})_{k=0,\dots,\lceil nt\rceil}\right) &= \frac{1}{\sqrt{2\pi\Delta_n^\gamma}}\sum_{k=1}^{\lceil nt\rceil} \1_{\{X_{(k-1)\Delta_n^\gamma}<\rho_0\}}\exp\left(-\frac{(X_{(k-1)\Delta_n^\gamma}-\rho_0)^2}{2\Delta_n^\gamma \alpha_0^2}\right) \\
&\hspace{0.5cm} +\frac{1}{\sqrt{2\pi\Delta_n^\gamma}}\sum_{k=1}^{\lceil nt\rceil} \1_{\{X_{(k-1)\Delta_n^\gamma}\geq\rho_0\}}\exp\left(-\frac{(X_{(k-1)\Delta_n^\gamma}-\rho_0)^2}{2\Delta_n^\gamma \beta_0^2}\right).
\end{align*}
\[ F_{\alpha_0,\beta_0} = -\left(\frac{2(\alpha_0-\beta_0)}{\alpha_0\beta_0} + \frac{\alpha_0}{\beta_0}\frac{2}{\alpha_0+\beta_0}\log\left(\frac{\beta_0^2}{\alpha_0^2}\right)\right),\qquad \tilde{F}_{\alpha_0,\beta_0} = F_{\beta_0,\alpha_0}\]
and
\[ \E_{0}\left[ \sup_{|\theta'|\leq |\theta|} \sup_{s\leq t} |r_{n,\gamma}(s,\theta')|\right] \leq C_{\alpha_0,\beta_0}(K) \frac{\theta^2\sqrt{n}}{\Delta_n^\gamma}\sqrt{t}\]
for some constant $C_{\alpha_0,\beta_0}(K)>0$ that is independent of $n$ and $\theta$.
\end{lemma}
\begin{proof}
The proof is completely analogous to that of Proposition~$2.1$ in~\citeSM{App:Brutsche/Rohde} with the generalized step size $\Delta_n^\gamma$ instead of $1/n$.
\end{proof}

\noindent
Secondly, we need a bound on the second moment of the (modified) summands of $\ell_{n,\gamma,t}^{\alpha_0,\beta_0}$. To this aim we define
\begin{align}\label{def:Z_k_other}
Z_k^j(\theta',\theta) := \log\left(\frac{p_{\Delta_n^\gamma}^{\rho_0+\theta,\alpha_0,\beta_0}(X_{(k-1)\Delta_n^\gamma},X_{k\Delta_n^\gamma})}{p_{\Delta_n^\gamma}^{\rho_0+\theta',\alpha_0,\beta_0}(X_{(k-1)\Delta_n^\gamma},X_{k\Delta_n^\gamma})}\right) \1_{I_{j,k}^{\theta',\theta}}
\end{align}
for $j\in\{1,3,4,5,6,7,9\}$, whereas
\begin{align}\label{def:Z_28}
\begin{split}
Z_k^2(\theta',\theta) &:= \left[\log\left(\frac{p_{\Delta_n^\gamma}^{\rho_0+\theta,\alpha_0,\beta_0}(X_{(k-1)\Delta_n^\gamma},X_{k\Delta_n^\gamma})}{p_{\Delta_n^\gamma}^{\rho_0+\theta',\alpha_0,\beta_0}(X_{(k-1)\Delta_n^\gamma},X_{k\Delta_n^\gamma})}\right) - \log\left(\frac{\beta_0^2}{\alpha_0^2}\right)\right] \1_{I_{2,k}^{\theta',\theta}}, \\
Z_k^8(\theta',\theta) &:= \left[ \log\left(\frac{p_{\Delta_n^\gamma}^{\rho_0+\theta,\alpha_0,\beta_0}(X_{(k-1)\Delta_n^\gamma},X_{k\Delta_n^\gamma})}{p_{\Delta_n^\gamma}^{\rho_0+\theta',\alpha_0,\beta_0}(X_{(k-1)\Delta_n^\gamma},X_{k\Delta_n^\gamma})}\right)- \log\left(\frac{\beta_0^2}{\alpha_0^2}\right)\right] \1_{I_{8,k}^{\theta',\theta}},
\end{split}
\end{align} 
where the indicator sets are given in~\eqref{def:indicators}. For these $Z_k^j$, we have the following result.

\begin{lemma}\label{lemma:moment_Z}
Let $K>0$, $-K\sqrt{\Delta_n^\gamma} \leq \theta'<\theta\leq K\sqrt{\Delta_n^\gamma}$ and $X=(X_t)_{t\geq 0}$ an OBM started in $x_0\in\R$ with diffusion coefficient $\sigma_{\rho_0,\alpha_0,\beta_0}$. Then for every $j=1,\dots, 9$ and $1\leq k\leq n$,
\[ \E_0\left[ \left( Z_k^j(\theta',\theta)\right)^2\right] \leq C_{\alpha_0,\beta_0} |\theta'-\theta|^2\frac{1}{\Delta_n^\gamma\sqrt{k}}.\]
\end{lemma}
\begin{proof}
The proof is analogous to that of Proposition~$2.4$ in~\citeSM{App:Brutsche/Rohde}. In each step, the stepsize $1/n$ of this article has to be replaced by $\Delta_n^\gamma$, yielding the desired result.
\end{proof}

With Lemma~\ref{lemma:expansion_drift} and~\ref{lemma:moment_Z} we are now in a position to give the missing proof of (1A) in the intermediate regime $\gamma\in(0,1)$.

\begin{proof}[Proof of (1A) for $\gamma\in(0,1)$]
Recall that the coupling time $T_{\rho_0}^{x_0}$ given in~\eqref{eq:coupling_time} is independent of $n$ with $\Pr(T_{\rho_0}^{x_0}<\infty)=1$. Hence, for any $\epsilon>0$ we can choose $T_0\in\N$ such that $\Pr(T_{\rho_0}^{x_0}\leq T_0)>1-\epsilon$. It then remains to prove that for $\gamma\in (0,1)$,
\begin{align*}
\sup_{z\in [-K,K]}\sum_{k=1}^{\lceil T_0/\Delta_n^\gamma\rceil} \log\left(\frac{p_{\Delta_n^\gamma}^{\rho_0+zn^{-(1+\gamma)/2}, \alpha_0,\beta_0}(\overline{X}_{(k-1)\Delta_n^\gamma}^c,\overline{X}_{k\Delta_n^\gamma}^c)}{p_{\Delta_n^\gamma}^{\rho_0,\alpha_0,\beta_0}(\overline{X}_{(k-1)\Delta_n^\gamma}^c,\overline{X}_{k\Delta_n^\gamma}^c)}\right) \longrightarrow_{\Pr} 0.
\end{align*}
By using exactly the same reasoning, the same holds true for $\overline{X}^c$ replaced by $X^{\rho_0}$. For the proof, we use the drift and martingale decomposition~\eqref{eq:decomposition_drift_martingale}
\begin{align*}
&\sum_{k=1}^{\lceil T_0/\Delta_n^\gamma\rceil} \log\left(\frac{p_{\Delta_n^\gamma}^{\rho_0+zn^{-(1+\gamma)/2}, \alpha_0,\beta_0}(\overline{X}_{(k-1)\Delta_n^\gamma}^c,\overline{X}_{k\Delta_n^\gamma}^c)}{p_{\Delta_n^\gamma}^{\rho_0,\alpha_0,\beta_0}(\overline{X}_{(k-1)\Delta_n^\gamma}^c,\overline{X}_{k\Delta_n^\gamma}^c)}\right) \\
&\hspace{2cm} = B_{n,\gamma,(T_0/\Delta_n^\gamma+1)/n)}^{\alpha_0,\beta_0}\left(zn^{-(1+\gamma)/2}\right) + M_{n,\gamma,(T_0/\Delta_n^\gamma+1)/n}^{\alpha_0,\beta_0}\left(zn^{-(1+\gamma)/2}\right)
\end{align*} 
and prove that the supremum over $z\in [-K,K]$ of both summands converges to zero in probability.\smallskip

\textbf{Drift part. } By Lemma~\ref{lemma:expansion_drift} we directly obtain that
\begin{align*}
&\sup_{z\in [-K,K]} \left| B_{n,\gamma,(T_0/\Delta_n^\gamma+1)/n}^{\alpha_0,\beta_0}\left(zn^{-(1+\gamma)/2}\right)\right|\\
&\hspace{1cm}\leq C_{\alpha_0,\beta_0}Kn^{-(1+\gamma)/2} \frac{1}{\sqrt{\Delta_n^\gamma}}\sum_{k=1}^{\lceil T_0/\Delta_n^\gamma\rceil} \exp\left(-\frac{(\overline{X}^c_{(k-1)\Delta_n^\gamma}-\rho_0)^2}{2\Delta_n^\gamma \max\{\alpha_0^2,\beta_0^2\}}\right)\\
&\hspace{2cm} + \sup_{z\in [-K,K]}\left| r_{n,\gamma}\left( (T_0/\Delta_n^\gamma+1)/n, zn^{-(1+\gamma)/2}\right)\right|.
\end{align*}
Taking expectations on both side gives with Lemma~\ref{lemma:expansion_drift} and Corollary~\ref{cor:bound_exp_X^2}
\begin{align*}
&\E\left[ \sup_{z\in [-K,K]} \left| B_{n,\gamma,(T_0/\Delta_n^\gamma+1)/n}^{\alpha_0,\beta_0}\left(zn^{-(1+\gamma)/2}\right)\right|\right]\\ &\hspace{1cm}\leq C_{\alpha,\beta}Kn^{-(1+\gamma)/2} \frac{1}{\sqrt{\Delta_n^\gamma}}\sum_{k=1}^{\lceil T_0/\Delta_n^\gamma\rceil}\frac{1}{\sqrt{k}} + C_{\alpha_0,\beta_0} \frac{K^2 n^{-(1+\gamma)}}{\Delta_n^\gamma} \sqrt{\frac{T_0}{\Delta_n^\gamma}+1} \\
&\hspace{1cm}\leq C_{\alpha_0,\beta_0}(T_0,K) \left( n^{-(1-\gamma)/2} + n^{-1+\gamma/2}\right) \longrightarrow 0.
\end{align*}
\smallskip

\textbf{Martingale part. } Here, we could split the martingale part $M_{n,t}$ into nine different terms according to the nine regimes $I_{j,k}^{0,\theta}$ given in~\eqref{def:indicators}. Lemma~\ref{lemma:moment_Z} entails that the second moments of individual summands is then different for $j=1,3,4,5,6,7,9$ and $j=2,8$. For this reason, we split $M_{n,\gamma,t}$ into two parts, namely
\[ M_{n,\gamma,(T_0/\Delta_n^\gamma+1)/n}^{\alpha_0,\beta_0}(z\psi_n^\gamma) = M_{n,\gamma}^1(z) + M_{n,\gamma}^2(z),\]
where both $M_{n,\gamma}^1(z)$ and $M_{n,\gamma}^2(z)$ are sums of martingale differences and given as
\begin{align*}
M_{n,\gamma}^1(z) &= \sum_{j=1}^9 \sum_{k=1}^{\lceil T_0/\Delta_n^\gamma\rceil} \left( Z_k^j(0,z \psi_n^\gamma)  - \E\left[ \left. Z_k^j(0,z \psi_n^\gamma)  \right| \overline{X}^c_{(k-1)\Delta_n^\gamma}\right]\right)
\end{align*}
and
\begin{align*}
M_{n,\gamma}^2(z) &= \log\left(\frac{\beta_0^2}{\alpha_0^2}\right)  \sum_{k=1}^{\lceil T_0/\Delta_n^\gamma\rceil} \left( \1_{I_{2,k}^{0,z\psi_n^\gamma}} + \1_{I_{8,k}^{0,z\psi_n^\gamma}} - \E\left[ \left. \1_{I_{2,k}^{0,z\psi_n^\gamma}} + \1_{I_{8,k}^{0,z\psi_n^\gamma}}\right| \overline{X}^c_{(k-1)\Delta_n^\gamma}\right] \right),
\end{align*}
where $Z_k^j$ are given in~\eqref{def:Z_k_other} and~\eqref{def:Z_28} and $I_{2,k}^{0,\theta}$ and $I_{8,k}^{0,\theta}$ are defined in~\eqref{def:indicators}. In what follows, we will prove $\sup_{0\leq z\leq K} |M_{n}^i(z)|\rightarrow_{\Pr} 0$ for $i=1,2$. \smallskip

\begin{itemize}
\item[$\bullet\ \mathbf{M_{n,\gamma}^1}$.] We show stochastic convergence of $\sup_{z\in [0,K]}|M_{n,\gamma}^1(z)|$ to zero by proving
\[ \E\left[\sup_{z\in [0,K]}|M_{n,\gamma}^1(z)|\right] \longrightarrow 0, \]
which is done by chaining and variance bounds. By the special structure of the terms $Z_k^j$ we have for $0\leq z'\leq z\leq K$
\begin{align*}
M_{n,\gamma}^1(z)-M_{n,\gamma}^1(z') &= \sum_{j=1}^9 \sum_{k=1}^{\lceil T_0/\Delta_n^\gamma\rceil} \left( Z_k^j(z'\psi_n,z\psi_n)- \E\left[ \left. Z_k^j(z'\psi_n,z\psi_n) \right| \overline{X}^c_{(k-1)\Delta_n^\gamma}\right] \right).
\end{align*}  
Because the indicators for $Z_k^j,Z_k^l$ are disjoint for $j\neq l$ and the difference $M_{n,\gamma}^1(z)-M_{n,\gamma}^1(z')$ is a sum of martingale increments, iterative conditioning reveals
\begin{align}\label{eq_proof:M1_second_moment_difference}
\begin{split}
&\E\left[ \left( M_{n,\gamma}^1(z)-M_{n,\gamma}^1(z')\right)^2\right] \\
&\hspace{0.5cm} = \sum_{j=1}^9 \sum_{k=1}^{\lceil T_0/\Delta_n^\gamma\rceil} \E\left[ \left( Z_k^j(z'\psi_n^\gamma,z\psi_n^\gamma)- \E\left[ \left. Z_k^j(z'\psi_n^\gamma,z\psi_n^\gamma) \right| \overline{X}^c_{(k-1)\Delta_n^\gamma}\right] \right)^2\right] \\
&\hspace{0.5cm} \leq \sum_{j=1}^9 \sum_{k=1}^{\lceil T_0/\Delta_n^\gamma\rceil} \E\left[  Z_k^j(z'\psi_n^\gamma,z \psi_n^\gamma)^2\right] \\
&\hspace{0.5cm} \leq C_{\alpha_0,\beta_0} |z-z'|^2\psi_n^2 \frac{1}{\Delta_n^\gamma}\sum_{k=1}^{\lceil T_0/\Delta_n^\gamma\rceil} \frac{1}{\sqrt{k}} \leq C_{\alpha_0,\beta_0}(T_0)|z-z'|^2 n^{-(1-\gamma/2)},
\end{split}
\end{align}
where the second last step applied Lemma~\ref{lemma:moment_Z}. Next, we apply a chaining with variance bounds, i.e. with the Young-Orlicz function $\psi(x)=x^2$. To this aim, we define
\[ \rho_n(z,z') := Cn^{-1/2+\gamma/4}|z-z'|,\]
where $C^2$ equals the constant on the right-hand side of~\eqref{eq_proof:M1_second_moment_difference}, ensuring the moment bound $\E[ ( M_{n,\gamma}^1(z)-M_{n,\gamma}^1(z'))^2]\leq \rho_n(z,z')^2$. Denoting 
\[ D(u,\mathcal{T},\rho_n) := \max\{\#\mathcal{T}_0: \mathcal{T}_0\subset\mathcal{T}, \rho_n(t,t')\geq u \textrm{ for different } t,t'\in\mathcal{T}_0\},\]
then yields
\begin{align*}
\E\left[ \sup_{0\leq z\leq K}\left|M_{n,\gamma}^1(z)\right|\right] &\leq \int_0^{CKn^{-1/2+\gamma/4}} \sqrt{ D(u, [0,K],\rho_n)} du \\
&\leq \int_0^{CKn^{-1/2+\gamma/4}} \sqrt{CKn^{-1/2+\gamma/4} u^{-1}} du \\
&= 2CKn^{-1/2+\gamma/4}\longrightarrow 0.
\end{align*}

\smallskip
\item[$\bullet\ \mathbf{M_{n,\gamma}^2}$.] In order to prove $\sup_{0\leq z\leq K}|M_{n,\gamma}(z)^2|\rightarrow_{\Pr}0$, we will establish that its expectation
\[  \E\left[\sup_{z\in [0,K]} \sum_{k=1}^{\lceil T_0/\Delta_n^\gamma\rceil} \left( \1_{I_{2,k}^{0,z\psi_n^\gamma}} + \1_{I_{8,k}^{0,z\psi_n^\gamma}} + \E\left[ \left. \1_{I_{2,k}^{0,z\psi_n^\gamma}} + \1_{I_{8,k}^{0,z\psi_n^\gamma}}\right| \overline{X}^c_{(k-1)\Delta_n^\gamma}\right] \right)\right] \]
converges to zero. In order to bound this term we first use $\1_{I_{2,k}^{0,z\psi_n^\gamma}} \leq \1_{I_{2,k}^{0,K\psi_n^\gamma}}$
which yields
\begin{align*}
&\sup_{z\in [0,K]}  \sum_{k=1}^{\lceil T_0/\Delta_n^\gamma\rceil} \left( \1_{I_{2,k}^{0,z\psi_n^\gamma}} + \E\left[ \left. \1_{I_{2,k}^{0,z\psi_n^\gamma}} \right| \overline{X}^c_{(k-1)\Delta_n^\gamma}\right] \right) \\
&\hspace{1cm} \leq \sup_{z\in [0,K]}  \sum_{k=1}^{\lceil T_0/\Delta_n^\gamma\rceil} \left( \1_{I_{2,k}^{0,K\psi_n^\gamma}} + \E\left[ \left. \1_{I_{2,k}^{0,K\psi_n^\gamma}} \right| \overline{X}^c_{(k-1)\Delta_n^\gamma}\right] \right).
\end{align*}
By the upper bound~\eqref{cor:bound_exp_X^2} on the transition density we obtain the moment bound
\begin{align*}
\E\left[ \1_{I_{2,k}^{0,K\psi_n^\gamma}} \right] &\leq C_{\alpha_0,\beta_0} \E\left[ \1_{\{X_{(k-1)\Delta_n^\gamma}<\rho_0\}}\int_{\rho_0}^{\rho_0+K\psi_n^\gamma} \exp\left(-\frac{(\overline{X}^c_{(k-1)\Delta_n^\gamma}-y)^2}{2\max\{\alpha_0^2,\beta_0^2\}\Delta_n^\gamma}\right) dy \right] \\
&\leq C_{\alpha_0,\beta_0}K \psi_n^\gamma \E\left[ \exp\left(-\frac{(\overline{X}^c_{(k-1)\Delta_n^\gamma}-\rho_0)^2}{2\max\{\alpha_0^2,\beta_0^2\}\Delta_n^\gamma}\right)\right]  \leq \frac{C_{\alpha_0,\beta_0}K\psi_n^\gamma}{\sqrt{k}},
\end{align*}
which then yields
\begin{align*}
&\E\left[ \sup_{z\in [0,K]}  \sum_{k=1}^{\lceil T_0/\Delta_n^\gamma\rceil} \left( \1_{I_{2,k}^{0,z\psi_n^\gamma}} + \E\left[ \left. \1_{I_{2,k}^{0,z\psi_n^\gamma}} \right| \overline{X}^c_{(k-1)\Delta_n^\gamma}\right] \right) \right] \\
&\hspace{1cm} \leq C_{\alpha_0,\beta_0}K\psi_n^\gamma \sum_{k=1}^{\lceil T_0/\Delta_n^\gamma\rceil} \frac{1}{\sqrt{k}} \leq C_{\alpha_0,\beta_0}Kn^{-1/2}\rightarrow 0.
\end{align*}
Next, we use the inequality
\[ \sup_{z\in [0,K]}\1_{I_{8,k}^{0,z\psi_n^\gamma}} \leq \1_{\{\rho_0<\overline{X}^c_{k\Delta_n^\gamma}\leq K\psi_n^\gamma\}}\1_{\{\overline{X}^c_{(k-1)\Delta_n^\gamma}\geq \rho_0\}}.\]
Then again by applying Corollary~\eqref{cor:bound_exp_X^2},
\begin{align*}
&\E\left[\1_{\{\rho_0<\overline{X}^c_{k\Delta_n^\gamma}\leq K\psi_n^\gamma\}}\1_{\{\overline{X}^c_{(k-1)\Delta_n^\gamma}\geq \rho_0\}}\right]\\
&\hspace{0.5cm} \leq C_{\alpha_0,\beta_0}\E\left[\1_{\{\overline{X}^c_{(k-1)\Delta_n^\gamma}\geq \rho_0\}} \int_{\rho_0}^{\rho_0+K\psi_n^\gamma} \exp\left(-\frac{(\overline{X}^c_{(k-1)\Delta_n^\gamma}-y)^2}{2\max\{\alpha_0^2,\beta_0^2\}\Delta_n^\gamma}\right) dy \right] \\
&\hspace{0.5cm} \leq C_{\alpha_0,\beta_0} K\psi_n^\gamma \E\left[ \exp\left(-\frac{(\overline{X}^c_{(k-1)\Delta_n^\gamma}-\rho_0-K\psi_n^\gamma)^2}{2\max\{\alpha_0^2,\beta_0^2\}\Delta_n^\gamma}\right)\right] \\
&\hspace{1cm} + C_{\alpha_0,\beta_0} \E\left[ \1_{\{\rho_0\leq \overline{X}^c_{(k-1)\Delta_n^\gamma}\leq \rho_0+K\psi_n^\gamma\}}\right] \\
&\hspace{0.5cm} \leq \frac{C_{\alpha_0,\beta_0}K\psi_n^\gamma}{\sqrt{k}} + C_{\alpha_0,\beta_0}\int_{\rho_0}^{\rho_0+K\psi_n^\gamma} \frac{1}{\sqrt{k}}\exp\left(-\frac{(y-x_0)^2}{2k\Delta_n^\gamma\max\{\alpha_0^2,\beta_0^2\}}\right) dy\\
&\hspace{0.5cm} \leq \frac{C_{\alpha_0,\beta_0}K\psi_n^\gamma}{\sqrt{k}}.
\end{align*}
Consequently, we also obtain
\begin{align*}
&\E\left[ \sup_{z\in [0,K]}  \sum_{k=1}^{\lceil T_0/\Delta_n^\gamma\rceil} \left( \1_{I_{8,k}^{0,z\psi_n^\gamma}} + \E\left[ \left. \1_{I_{8,k}^{0,z\psi_n^\gamma}} \right| \overline{X}^c_{(k-1)\Delta_n^\gamma}\right] \right) \right] \\
&\hspace{1cm} \leq C_{\alpha_0,\beta_0}K\psi_n^\gamma \sum_{k=1}^{\lceil T_0/\Delta_n^\gamma\rceil} \frac{1}{\sqrt{k}} \leq C_{\alpha_0,\beta_0}Kn^{-1/2}\rightarrow 0
\end{align*}
and the proof is complete.
\end{itemize}
\end{proof}

In the next result, we establish that the rates in Theorem~\ref{thm:weak_convergence} and Proposition~\ref{prop:rate_of_triple_MLE} are indeed minimax optimal. Note that the proof is quite short which is due to the fact that it makes use of the drift expansion provided by Lemma~\ref{lemma:expansion_drift} which itself is a highly non-trivial result.

\begin{proposition}\label{prop:minimax_lower}
Let $\gamma\in [0,1)$. Then the estimation rate $\psi_n^\gamma=n^{-(1+\gamma)/2}$ for $\rho$ is minimax optimal, i.e. for any $s>0$ we have
\[ \liminf_{n\to\infty} \inf_{T_n^\rho} \sup_{\rho} \Pr_{\rho,\alpha,\beta} \left( n^{(1+\gamma)/2}\left| T_n^\rho - \rho\right|>s\right) >0,\]
where the infimum is taken over all estimators $T_n^\rho$ based on $X_0,X_{\Delta_n^\gamma},\dots, X_{n\Delta_n^\gamma}$.
\end{proposition}
\begin{proof}
By Theorem~$2.2$ and $(2.9)$ in~\citeSM{App:Tsybakov} it is enough to prove that the Kullback divergence $D_{KL}(\Pr_{\rho_0,\alpha_0,\beta_0},\Pr_{\rho_0+2s\psi_n^\gamma,\alpha_0,\beta_0})$ remains bounded as $n\to\infty$. Using Lemma~\ref{lemma:expansion_drift} and Corollary~\ref{cor:bound_exp_X^2}, we obtain
\begin{align*}
&D_{KL}\left(\Pr_{\rho_0,\alpha_0,\beta_0},\Pr_{\rho_0+2s\psi_n^\gamma,\alpha_0,\beta_0}\right)\\
&\hspace{0.3cm} = \E_{\rho_0,\alpha_0,\beta_0}\left[ \log\left(\frac{\prod_{k=1}^n p_{\Delta_n^\gamma}^{\rho_0,\alpha_0,\beta_0}(X_{(k-1)\Delta_n^\gamma},X_{(k-1)\Delta_n^\gamma})}{\prod_{k=1}^n p_{\Delta_n^\gamma}^{\rho_0+2s\psi_n^\gamma,\alpha_0,\beta_0}(X_{(k-1)\Delta_n^\gamma},X_{(k-1)\Delta_n^\gamma})}\right) \right]\\
&\hspace{0.3cm} = - \E_{\rho_0,\alpha_0,\beta_0}\left[ \sum_{k=1}^n \E_{\rho_0,\alpha_0,\beta_0}\left[\left.\log\left(\frac{p_{\Delta_n^\gamma}^{\rho_0+2s\psi_n^\gamma,\alpha_0,\beta_0}(X_{(k-1)\Delta_n^\gamma},X_{(k-1)\Delta_n^\gamma})}{p_{\Delta_n^\gamma}^{\rho_0,\alpha_0,\beta_0}(X_{(k-1)\Delta_n^\gamma},X_{(k-1)\Delta_n^\gamma})} \right) \right|X_{(k-1)\Delta_n^\gamma}\right]\right] \\
&\hspace{0.3cm} \leq 2s \max\{F_{\alpha_0,\beta_0}, \tilde{F}_{\alpha_0,\beta_0}\} \E_{\rho_0,\alpha_0,\beta_0}\left[ \psi_n^\gamma \Lambda_{n,\gamma}^{\alpha_0,\beta_0}\left( (X_{k\Delta_n^\gamma})_{k=0,1,\dots,n}\right)\right]\\
&\hspace{1.5cm}+ \E_{\rho_0,\alpha_0,\beta_0}\left[ r_{n,\gamma}(1,2s\psi_n^\gamma)\right] \\
&\hspace{0.3cm} \leq C_{\alpha_0,\beta_0}s \frac{\psi_n^\gamma}{\sqrt{\Delta_n^\gamma}} \sum_{k=1}^n \E_{\rho_0,\alpha_0,\beta_0}\left[ \exp\left(-\frac{(X_{(k-1)\Delta_n^\gamma}-\rho_0)^2}{2\max\{\alpha_0^2,\beta_0^2\}\Delta_n^\gamma}\right)\right] + C_{\alpha_0,\beta_0}s^2\frac{(\psi_n^\gamma)^2 \sqrt{n}}{\Delta_n^\gamma} \\
&\hspace{0.3cm} \leq C_{\alpha_0,\beta_0}(s) \left( \frac{1}{\sqrt{n}}\sum_{k=1}^n \frac{1}{\sqrt{k}} + \frac{1}{\sqrt{n}}\right),
\end{align*}
which remains bounded as $n\to\infty$.
\end{proof}

\section{Proof of (2A) for $\gamma\in (0,1)$}\label{Section:Proof_2A}

In this section, the tightness in (2A) is proven for the case $\gamma\in (0,1)$. As compared to~\citeSM{App:Brutsche/Rohde} which covers the high-frequency regime $\gamma=1$, the innovation consists of describing typical path properties of $(X_t)_{0\leq t\leq n\Delta_n^\gamma}$ via self-similarity and coupling that establishes a connection between the high-frequency regime and that for all other $\gamma\in (0,1)$. An instructive example of this was already given in the important Lemma~\ref{lemma:probability_A3} in the main paper. In the same spirit, we have the following preliminary result.

\begin{lemma}\label{lemma:probability_A6}
Let $\epsilon>0$, $0\leq L<K$ and $X$ be a solution of~\eqref{eq:SDE_OBM} with $\sigma_{\rho_0,\alpha_0,\beta_0}$ and $X_0=x_0$. Then there exists a constant $\eta=\eta_\epsilon>0$ independent of $L,K$ such that
\[ \liminf_{n\to\infty}\Pr_0\left(  \frac{1}{\sqrt{n}}\sum_{k=1}^n \1_{\{\rho_0+Ln^{-\gamma/2}< X_{(k-1)\Delta_n^\gamma}<\rho_0+Kn^{-\gamma/2}\}} > \left(K-L\right)\beta_0^{-2}\eta \right) >1-\epsilon.\]
\end{lemma}
\begin{proof}
Assume that $X$ is started in $x_0$ and recall the coupling time $T_{\rho_0}^{x_0}$ given in~\eqref{eq:coupling_time}. Then, we have for the coupled process $\overline{X}^c$ defined in~\eqref{eq:coupled_process} the distributional identity
\begin{align}\label{eq_proof:A6}
\begin{split}
& \frac{1}{\sqrt{n}}\sum_{k=1}^n \1_{\{\rho_0+Ln^{-\gamma/2}< X_{(k-1)\Delta_n^\gamma}<\rho_0+Kn^{-\gamma/2}\}} \\
&\hspace{0.5cm} \stackrel{\mathcal{D}}{=}\frac{1}{\sqrt{n}}\sum_{k=1}^{\lceil T(\Delta_n^\gamma)^{-1}\rceil} \1_{\{\rho_0+Ln^{-\gamma/2}< \overline{X}_{(k-1)\Delta_n^\gamma}^c<\rho_0+Kn^{-\gamma/2}\}} -\1_{\{\rho_0+Ln^{-\gamma/2}< X_{(k-1)\Delta_n^\gamma}^{\rho_0}<\rho_0+Kn^{-\gamma/2}\}} \\
&\hspace{1.5cm} +\frac{1}{\sqrt{n}} \sum_{k=1}^{n} \1_{\{\rho_0+Ln^{-\gamma/2}< X_{(k-1)\Delta_n^\gamma}^{\rho_0}<\rho_0+Kn^{-\gamma/2}\}}
\end{split}
\end{align}
As $\Pr(T_{\rho_0}^{x_0}<\infty)=1$, there exists $T_0>0$ such that $\Pr(T_{\rho_0}^{x_0}<T_0)>1-\epsilon/4$. Using the upper bound of the transition density~\eqref{eq:bound_transition_density}, we find 
\begin{align*}
&\E\left[ \frac{1}{\sqrt{n}} \left|\sum_{k=1}^{\lceil T_0(\Delta_n^\gamma)^{-1}\rceil} \hspace{-0.3cm}\1_{\{\rho_0+Ln^{-\gamma/2}< \overline{X}_{(k-1)\Delta_n^\gamma}^c<\rho_0+Kn^{-\gamma/2}\}} -\1_{\{\rho_0+Ln^{-\gamma/2}< X_{(k-1)\Delta_n^\gamma}^{\rho_0}<\rho_0+Kn^{-\gamma/2}\}} \right| \right] \\
&\hspace{0.5cm} \leq \frac{C_{\alpha_0,\beta_0}}{\sqrt{n}}\sum_{k=1}^{\lceil T_0(\Delta_n^\gamma)^{-1}\rceil} \int_{Ln^{-\gamma/2}}^{Kn^{-\gamma/2}} \frac{1}{\sqrt{k\Delta_n^\gamma}} \exp\left(-\frac{\min\{(y-x_0)^2,(y-\rho_0)^2\}}{2\max\{\alpha_0^2,\beta_0^2\}k\Delta_n^\gamma}\right) dy\\
&\hspace{0.5cm} \leq \frac{C_{\alpha_0,\beta_0}}{\sqrt{n}}(K-L) n^{-\gamma/2} (\Delta_n^\gamma)^{-1/2} \sum_{k=1}^{\lceil T_0(\Delta_n^\gamma)^{-1}\rceil} \frac{1}{\sqrt{k}} \leq C_{\alpha_0,\beta_0}(K-L)(T_0+1) \frac{1}{\sqrt{n\Delta_n^\gamma}},
\end{align*}
which converges to zero, i.e. the first summand in~\eqref{eq_proof:A6} converges stochastically to zero and it suffices to prove that
\[ \Pr\left(\frac{1}{\sqrt{n}} \sum_{k=1}^{n} \1_{\{\rho_0+Ln^{-\gamma/2}< X_{(k-1)\Delta_n^\gamma}^{\rho_0}<\rho_0+Kn^{-\gamma/2}\}} > 2(K-L)\beta_0^{-2}\eta\right)>1-\epsilon/4.\] 
Applying the self-similarity of Lemma~\ref{lemma:self_similarity}, we have for an OBM $Y$ started in $Y_0=\rho_0$ with diffusion coefficient $\sigma_{\rho_0,\alpha_0,\beta_0}$ that this is a consequence of
\[ \Pr\left( \frac{1}{\sqrt{n}} \sum_{k=1}^n \1_{\{\rho_0+L/\sqrt{n}< Y_{(k-1)/n}<\rho_0+K/\sqrt{n}\}} > 2(K-L)\beta_0^{-2}\eta\right) >1-\epsilon/4,\]
which follows from Lemma~$3.2$ in~\citeSM{App:Brutsche/Rohde} applied to the function $\1_{[L,K]}$ for $n$ sufficiently large and the fact that for any $\epsilon>0$ we can choose $\eta$ small enough such that $\Pr(L_1^{\rho_0}(Y)>4\eta)>1-\epsilon/4$  as seen in the proof of Lemma~\ref{lemma:probability_A3} by continuity of measures and the fact that $\Pr(L_1^{\rho_0}(Y)>0)=1$. Note in particular that this choice of $\eta$ does not depend on $K$ and $L$.
\end{proof}

We now turn to the proof of (2A). Here, we first introduce three events that describe typical behavior of the process $(X_t)_{0\leq t\leq n\Delta_n^\gamma}$. Let $\epsilon>0$. 

\begin{itemize}
\item Define
\begin{align}\label{eq:set_A1}
A_1(n) := \left\{ -\Gamma \sqrt{n\Delta_n^\gamma} < \inf_{0\leq t\leq n\Delta_n^\gamma}\overline{X}_t^c, \sup_{0\leq t\leq n\Delta_n^\gamma}\overline{X}_t^c < \Gamma \sqrt{n\Delta_n^\gamma} \right\}.
\end{align}
As by Jensen's inequality, the Burkholder-Davis-Gundy inequality and It\^{o}'s isometry,
\[ \E\left[ \sup_{0\leq t\leq n\Delta_n^\gamma} |\overline{X}_t^c|\right] \leq \E\left[\left( \int_0^{n\Delta_n^\gamma} \sigma_{\rho_0,\alpha_0,\beta_0}(\overline{X}_s^c) dW_s\right)^2\right]^\frac12 \leq C_{\alpha_0,\beta_0}\sqrt{n\Delta_n^\gamma},\]
there exists a constant $\Gamma>0$ such that for large enough $n$ we have the bound $\Pr_0(A_1(n)) >1-\epsilon$.

\item We define the event
\begin{align}\label{eq:set_A3}
A_3(n) = \left\{\inf_{y\in [\rho_0-\zeta\sqrt{n\Delta_n^\gamma}, \rho_0+\zeta\sqrt{n\Delta_n^\gamma}]} L_{n\Delta_n^\gamma}^y(\overline{X}^c) >\sqrt{n\Delta_n^\gamma} \xi\right\}. 
\end{align} 
Proving that $\zeta$ and $\xi$ can be chosen such that $\Pr(A_3(n))>1-\epsilon$ is stated in Lemma~\ref{lemma:probability_A3}.

\item Motivated by the Hölder-continuity of each sampling path for any order strictly less than $1/2$, we define the event
\begin{align}\label{eq:set_A4}
A_4(n) := \left\{ \sup_{|t-s|<\Delta_n^\gamma} |\overline{X}_t^c - \overline{X}_s^c| \leq (\Delta_n^\gamma)^{4/9} \right\}.
\end{align}
Applying Markov's inequality and Theorem~$1$ in~\citeSM{App:Fischer/Nappo} yields existence of a constant $\tilde{C}>0$ such that
\begin{align*}
\Pr\left(\sup_{|t-s|<\Delta_n^\gamma} |\overline{X}_t^c - \overline{X}_s^c| > (\Delta_n^\gamma)^{4/9} \right) &\leq (\Delta_n^\gamma)^{-4/9} \E\left[ \sup_{|t-s|<\Delta_n^\gamma} |\overline{X}_t^c - \overline{X}_s^c|\right] \\
& \leq \tilde{C}(\Delta_n^\gamma)^{1/18}\sqrt{\log(2n)} \longrightarrow 0.
\end{align*} 
From this, we deduce that there exists $n_0\in\N$ such that $\Pr(A_4(n))>1-\epsilon$ for all $n\geq n_0$.
\end{itemize}

We omit an event $A_2$ that has no corresponding analogue in the high-frequency proof and is somehow absorbed into $A_3(n)$. This makes it easier to compare the proofs for different sampling schemes. As a first step towards (2A), we prove
\begin{align}\label{eq_proof:sqrt(n)_consistency}
\left| \hat{\rho}_n - \rho_0\right| = \mathcal{O}_{\Pr}\left(n^{-\gamma/2}\right).
\end{align}
To this aim, recall the normalized log-likelihood $\ell_n^{\alpha_0,\beta_0}(\theta)$ given in~\eqref{eq:def_ell_rho} and fix an arbitrary $\epsilon>0$. As $\ell_n^{\alpha_0,\beta_0}(0)=0$ and $n^{(1+\gamma)/2}(\hat{\rho}_n-\rho_0)\in\mathrm{Argsup}_{z\in\R}\ell_n^{\alpha_0,\beta_0}(zn^{-(1+\gamma)/2})$ by definition, the property $|\hat{\rho}_n -\rho_0|>Kn^{-\gamma/2}$ implies $\sup_{|\theta|>Kn^{-\gamma/2}}\ell_n^{\alpha_0,\beta_0}(\theta)\geq 0$. Consequently, we are going to show that
\begin{align*}
&\limsup_{K\to\infty} \limsup_{n\to\infty} \Pr\left( n^{\gamma/2}|\hat{\rho}_n - \rho_0|> K\right) \\
&\hspace{2cm} \leq \limsup_{K\to\infty} \limsup_{n\to\infty}\Pr\left( \sup_{n^{\gamma/2}|\theta|>K} \ell_{n,\gamma}^{\alpha_0,\beta_0}(\theta) \geq 0 \right) < \epsilon,
\end{align*} 
where both cases $\theta<0$ and $\theta\geq 0$ can be dealt with analogously and we only discuss the second one. The crucial observation is that $\ell_n(\theta_\rho)$ can be decomposed into a dominant term that and a remainder that can be bounded independently of $\theta_\rho$. The precise statement is given in Lemma~\ref{lemma:bound_L_sqrt(n)} and requires the definition of
\begin{align*}
N_n^L(\theta_\rho) &:= \sum_{k=1}^n \left[ \log\left(\frac{\beta_0}{\alpha_0}\right)-\frac{(\overline{X}_{k\Delta_n^\gamma}^c-\overline{X}_{(k-1)\Delta_n^\gamma}^c)^2}{2\Delta_n^\gamma}\left(\frac{1}{\alpha_0^2}-\frac{1}{\beta_0^2}\right)\right]\\
&\hspace{1.5cm} \cdot \1_{\{\rho_0+Ln^{-\gamma/2}\leq \overline{X}_{(k-1)\Delta_n^\gamma}^c< \rho_0+\theta_\rho\}},
\end{align*} 
where the parameter $L$ will be specified later. Then, we have the following result:

\begin{lemma}\label{lemma:bound_L_sqrt(n)}
Let $K,L\geq 0$, $\epsilon>0$ and $\Theta_n^1 := [Kn^{-\gamma/2},n^{1/4-\gamma/2}]$, $\Theta_n^2:=(n^{1/4-\gamma/2},\infty)$. Then there exists a sequence of sets $(A_n)_{n\in\N}$ with $\Pr_0(A_n^c)\leq\epsilon$ for $n\geq n_0$, such that for $i=1,2$,
\[ \sup_{\theta_\rho\in\Theta_n^{i}} \ell_{n,\gamma}^{\alpha_0,\beta_0}(\theta_\rho)\1_{A_n} \leq \sup_{\theta_\rho \in\Theta_n^{i}} N_n^L(\theta_\rho)\1_{A_n} + F_n^{i}(K,L),\]
where $F_n^{i}(K,L)\geq 0$ and the moment bounds $\E[F_n^1(K,L)n^{-1/2}] \leq C_{\alpha_0,\beta_0}(L)<\infty$ and $\E[F_n^2(K,L)n^{-1/2-\gamma/6}] \leq C_{\alpha_0,\beta_0}(L)<\infty$ with a constant $C_{\alpha_0,\beta_0}(L)$ independent of $K$ and $n$ hold true.
\end{lemma}
\begin{proof}
The proof can be done analogously to that of Lemma~$3.1$ in~\citeSM{App:Brutsche/Rohde}. The only difference is that all steps have to be conducted for the stepsize $\Delta_n^\gamma$ instead of $1/n$ and the sets $A_1,A_3$ and $A_4$ in the proof have to be replaced by the suitable modified versions given in~\eqref{eq:set_A1}, \eqref{eq:set_A3} and~\eqref{eq:set_A4}. Moreover, the set 
\[ \left\{ \sup_{y\in\R} L_1^y(Y) \leq \xi_u\right\}\]
for an OBM $Y$ started in $\rho_0$ has to be replaced by
\[ A_2(n) := \left\{ \sup_{y\in\R} L_{n\Delta_n^\gamma}^y(\overline{X}^c) \leq \xi_u \sqrt{n\Delta_n^\gamma}\right\}.\]
Note that for any $\epsilon>0$ there exists $\xi_u>0$ such that $\Pr(A_2(n))>1-\epsilon$ for $n$ large enough, which can be deduce by the fact that $\Pr(\sup_{y\in\R} L_1^y(Y) >\xi_u)>1-\epsilon$, analogously as in the proof of Lemma~\ref{lemma:probability_A3}.
\end{proof}

Note that $\Theta_n^1$ and $\Theta_n^2$ in Lemma~\ref{lemma:bound_L_sqrt(n)} are disjoint and $\Theta_n^1\cup\Theta_n^2= \{\theta:\ \theta\geq Kn^{-\gamma/2}\}$. Moreover, the moment bound on $F_n^1$ of order $n^{1/2}$ provided by this result is optimal, whereas the order $n^{1/2+\gamma/6}$ for $F_n^2$ is not optimal, but sufficiently accurate for our purpose. To proceed, recall the sets $A_1(n), A_3(n)$ and $A_4(n)$ given in \eqref{eq:set_A1}-\eqref{eq:set_A4} and let the constants $\xi,\zeta,\Gamma$ be specified in such a way that the bounds of their respective probabilities are given for $\epsilon/6$ (instead of $\epsilon$). Additionally, we introduce the event
\[ A_5(n) := \left\{ |F_n^1(K,L)| \leq C_F \sqrt{n},\ |F_n^2(K,L)| \leq C_F n^{1/2+\gamma/6} \right\},\]
where $C_F>0$ is chosen large enough such that $\Pr(A_5(n))>1-\epsilon/6$ which is possible by Lemma~\ref{lemma:bound_L_sqrt(n)} and Markov's inequality. Last but not least, let $A_6(n)$ be the set from Lemma~\ref{lemma:probability_A6} with corresponding choice of $\eta$ such that $\Pr(A_6(n)^c)<\epsilon/6$. Finally, we introduce the event we will be working on during the proof of~\eqref{eq_proof:sqrt(n)_consistency} as
\begin{align}\label{eq:setA_n}
A(n) = A_1(n) \cap A_3(n)\cap A_4(n) \cap A_5(n) \cap A_6(n)\cap A_7(n),
\end{align} 
where $A_7(n)$ is the set $A_n$ in Lemma~\ref{lemma:bound_L_sqrt(n)} for $\epsilon/6$. Then by the findings above we have for $n\geq n_0$ large enough that
\begin{align*}
\Pr( A(n)^c) \leq \sum_{j=1,3,4,5,6,7} \Pr(A_j(n)^c) \leq \epsilon.
\end{align*}
The proof of~\eqref{eq_proof:sqrt(n)_consistency} then follows by arguing line-by-line as from $(4.6)$ in~\citeSM{App:Brutsche/Rohde} onwards, including Subsection~$4.2$ to extend from $n^{\gamma/2}$- to $n^{(1+\gamma)/2}$-consistency. When proving the analogue of Lemma~$4.3$, the application of Lemma~$3.2$ is is replaced by the same self-similarity and coupling argument used in the proof of Lemma~\ref{lemma:probability_A6}.

\section{Proof of Proposition~\ref{prop:approximation_ell}}\label{Section:App_approximation_ell}

We introduce the sets
\begin{align*}
J_{1,k}(\rho) := \left\{ X_{(k-1)\Delta_n^\gamma} < \rho, X_{k\Delta_n^\gamma} \leq \rho  \right\}, \\
J_{2,k}(\rho) := \left\{ X_{(k-1)\Delta_n^\gamma} \geq \rho, X_{k\Delta_n^\gamma} > \rho\right\}, \\
J_{3,k}(\rho) := \left\{ X_{(k-1)\Delta_n^\gamma} < \rho < X_{k\Delta_n^\gamma} \right\}, \\
J_{4,k}(\rho) := \left\{ X_{k\Delta_n^\gamma} \leq \rho \leq X_{(k-1)\Delta_n^\gamma}  \right\}. 
\end{align*}
We abbreviate $\psi_n^\gamma= n^{-(1+\gamma)/2}$ and set $\alpha_n=\alpha_0+\theta_\alpha/\sqrt{n}$, $\beta_n=\beta_0+\theta_\beta/\sqrt{n}$. Using these events, we can estimate
\begin{align*}
&\sup_{z\in [-K,K]}  \sup_{-L\leq \theta_\alpha,\theta_\beta\leq L} \left|\ell_{n,\gamma}^3\left(z n^{-(1+\gamma)/2}, \theta_\alpha/\sqrt{n},\theta_\beta/\sqrt{n}\right)  \right| \\
&\hspace{3cm} \leq  \sum_{m=1}^4 \sup_{z\in [-K,K]}  \sup_{|\alpha_n-\alpha_0|, |\beta_n-\beta_0|\leq L/\sqrt{n}} \left| T_n^m(z,\alpha_n,\beta_n) \right|
\end{align*}
where
\begin{align*}
T_n^m(z,\alpha_n,\beta_n) &:= \sum_{k=1}^n \left( \log\left(\frac{p_{\Delta_n^\gamma}^{\rho_0+z\psi_n^\gamma,\alpha_n,\beta_n}(X_{(k-1)\Delta_n^\gamma},X_{k\Delta_n^\gamma})}{p_{\Delta_n^\gamma}^{\rho_0+z\psi_n^\gamma,\alpha_0,\beta_0}(X_{(k-1)\Delta_n^\gamma},X_{k\Delta_n^\gamma})}\right) \1_{J_{m,k}(\rho_0+z\psi_n^\gamma)} \right. \\
&\hspace{2.5cm} - \left. \log\left(\frac{p_{\Delta_n^\gamma}^{\rho_0,\alpha_n,\beta_n}(X_{(k-1)\Delta_n^\gamma},X_{k\Delta_n^\gamma})}{p_{\Delta_n^\gamma}^{\rho_0,\alpha_0,\beta_0}(X_{(k-1)\Delta_n^\gamma},X_{k\Delta_n^\gamma})}\right) \1_{J_{m,k}(\rho_0)}  \right)
\end{align*}
We now prove stochastic convergence
\[ \sup_{z\in [-K,K]}  \sup_{|\alpha_n-\alpha_0|, |\beta_n-\beta_0|\leq L/\sqrt{n}} \left| T_n^m(z,\alpha_n,\beta_n) \right| \longrightarrow_{\Pr_0} 0 \]
for $m=1,2,3,4$. Here, we will only provide details for $m=1,3$ as $m=2$ follows along the lines of $m=1$ and $m=4$ is analogous to $m=3$. Moreover, we only treat the supremum over $z\in [0,K]$ as the arguments for $z\in [-K,0]$ are the same.

\medskip
\noindent
$\bullet\ \underline{m=1}:$ By explicitly plugging in the transition density~\eqref{eq:transition_density}, we obtain
\begin{align*}
T_n^1(z,\alpha_n,\beta_n) &= T_n^{1,1}(z,\alpha_n,\beta_n) + T_n^{1,2}(z,\alpha_n,\beta_n) + T_n^{1,3}(z,\alpha_n,\beta_n),
\end{align*}
where
\begin{align*}
T_n^{1,1}(z,\alpha_n,\beta_n) &:= \sum_{k=1}^n \log\left(\frac{\alpha_0}{\alpha_n}\right)\left( \1_{J_{1,k}(\rho_0+z\psi_n^\gamma)} - \1_{J_{1,k}(\rho_0)} \right), \\
T_n^{1,2}(z,\alpha_n,\beta_n) &:= \sum_{k=1}^n \frac{1}{2\Delta_n^\gamma} (X_{k\Delta_n^\gamma}-X_{(k-1)\Delta_n^\gamma})^2 \left( \frac{1}{\alpha_0^2}-\frac{1}{\alpha_n^2}\right)\left( \1_{J_{1,k}(\rho_0+z\psi_n^\gamma)} - \1_{J_{1,k}(\rho_0)} \right), \\
T_n^{1,3}(z,\alpha_n,\beta_n) \\
&\hspace{-2cm} := \sum_{k=1}^n \log\left( \frac{1 - \frac{\alpha_n-\beta_n}{\alpha_n+\beta_n}\exp\left( - \frac{2}{\alpha_n^2\Delta_n^\gamma} ( X_{k\Delta_n^\gamma}-\rho_0-z\psi_n^\gamma)(X_{(k-1)\Delta_n^\gamma}-\rho_0-z\psi_n^\gamma) \right)}{1 - \frac{\alpha_0-\beta_0}{\alpha_0+\beta_0}\exp\left( - \frac{2}{\alpha_0^2\Delta_n^\gamma} ( X_{k\Delta_n^\gamma}-\rho_0-z\psi_n^\gamma)(X_{(k-1)\Delta_n^\gamma}-\rho_0-z\psi_n^\gamma)\right) }\right) \1_{J_{1,k}(\rho_0+z\psi_n^\gamma)} \\
&\hspace{0.5cm} - \sum_{k=1}^n \log\left( \frac{1 - \frac{\alpha_n-\beta_n}{\alpha_n+\beta_n}\exp\left( - \frac{2}{\alpha_n^2\Delta_n^\gamma} ( X_{k\Delta_n^\gamma}-\rho_0)(X_{(k-1)\Delta_n^\gamma}-\rho_0) \right)}{1 - \frac{\alpha_0-\beta_0}{\alpha_0+\beta_0}\exp\left( - \frac{2}{\alpha_0^2\Delta_n^\gamma} ( X_{k\Delta_n^\gamma}-\rho_0)(X_{(k-1)\Delta_n^\gamma}-\rho_0)\right) }\right) \1_{J_{1,k}(\rho_0)}.
\end{align*}
Subsequently, we are going to prove
\[ \sup_{z\in [0,K]}  \sup_{|\alpha_n-\alpha_0|, |\beta_n-\beta_0|\leq L/\sqrt{n}} \left| T_n^{1,j}(z,\alpha_n,\beta_n) \right| \longrightarrow_{\Pr_0} 0 \]
for $j=1,2,3$. For the first two cases, we note that
\begin{align*}
\1_{J_{1,k}(\rho_0+z\psi_n^\gamma)} - \1_{J_{1,k}(\rho_0)} &= \1_{\{X_{(k-1)\Delta_n^\gamma}<\rho_0, \rho_0<X_{k\Delta_n^\gamma}<\rho_0+z\psi_n^\gamma\}}+\1_{\{X_{k\Delta_n^\gamma}\leq\rho_0, \rho_0\leq X_{(k-1)\Delta_n^\gamma}\leq \rho_0+z\psi_n^\gamma\}} 
\end{align*}
such that
\begin{align}\label{eq_proof:Approx_ell_1}
\begin{split}
&\sup_{z\in [0,K]} \left( \1_{J_{1,k}(\rho_0+z\psi_n^\gamma)} - \1_{J_{1,k}(\rho_0)}\right) \\
&\hspace{1cm}\leq \1_{\{\rho_0<X_{k\Delta_n^\gamma}<\rho_0+K\psi_n^\gamma\}}\1_{\{X_{(k-1)\Delta_n^\gamma} \leq \rho_0\}} +\1_{\{\rho_0\leq X_{(k-1)\Delta_n^\gamma}\leq \rho_0+K\psi_n^\gamma\}}. 
\end{split}
\end{align} 
Then, using the upper bound~\eqref{eq:bound_transition_density} on the transition density gives for $k\geq 2$,
\begin{align}\label{eq_proof:bound_expectation_indicator_T1}
\begin{split}
&\E_0\left[\1_{\{\rho_0<X_{(k-1)\Delta_n^\gamma}<\rho_0+K\psi_n^\gamma\}}\right]\\
&\hspace{1cm} \leq C_{\alpha_0,\beta_0}\int_{\rho_0}^{\rho_0+K\psi_n^\gamma} \frac{1}{\sqrt{(k-1)\Delta_n^\gamma}} \exp\left(-\frac{(y-x_0)^2}{2(k-1)\Delta_n^\gamma\max\{\alpha_0^2,\beta_0^2\}}\right) dy  \\
&\hspace{1cm} \leq C_{\alpha_0,\beta_0}K\psi_n^\gamma \frac{1}{\sqrt{k\Delta_n^\gamma}} = C_{\alpha_0,\beta_0}K \frac{1}{\sqrt{kn}}
\end{split}
\end{align} 
and the same bound holds when replacing $X_{(k-1)\Delta_n^\gamma}$ with $X_{k\Delta_n^\gamma}$. 

\begin{itemize}
\item[$\mathbf{-\ T_n^{1,1}}$.] First, we observe that by a Taylor expansion, 
\[ \left|\log(\alpha_0/\alpha_n)\right| \leq C_{\alpha_0}|\alpha_0-\alpha_n| \leq C_{\alpha_0}L/\sqrt{n}. \]
Then, we obtain
\begin{align*}
&\E_0\left[ \sup_{z\in [0,K]}  \sup_{|\alpha_n-\alpha_0|, |\beta_n-\beta_0|\leq L/\sqrt{n}} \left| T_n^{1,1}(z,\alpha_n,\beta_n) \right|\right] \\
&\hspace{1cm} \leq \frac{C_{\alpha_0}L}{\sqrt{n}} \sum_{k=1}^n\E_0\left[ \sup_{z\in [0,K]}\left( \1_{J_{1,k}(\rho_0+z\psi_n^\gamma)} - \1_{J_{1,k}(\rho_0)}\right) \right] \\
&\hspace{1cm} \leq \frac{C_{\alpha_0}L}{\sqrt{n}} \sum_{k=1}^n \E_0\left[\1_{\{\rho_0<X_{k\Delta_n^\gamma}<\rho_0+K\psi_n^\gamma\}} +\1_{\{\rho_0\leq X_{(k-1)\Delta_n^\gamma}\leq \rho_0+K\psi_n^\gamma\}}\right] \\
&\hspace{1cm} \leq C_{\alpha_0,\beta_0}LK \frac1n \sum_{k=1}^n \frac{1}{\sqrt{k}} \longrightarrow 0. 
\end{align*}

\smallskip
\item[$\mathbf{-\ T_n^{1,2}}$.] First of all, we obtain
\begin{align}\label{eq_proof:T12_bound_squared_increment}
\begin{split}
&\E_0\left[\frac{(X_{k\Delta_n^\gamma}-X_{(k-1)\Delta_n^\gamma})^2}{\Delta_n^\gamma}\1_{\{\rho_0\leq X_{(k-1)\Delta_n^\gamma}\leq \rho_0+K\psi_n^\gamma\}}  \right] \\
&\hspace{0.5cm} =\E_0\left[\E_0\left[\left.\frac{(X_{k\Delta_n^\gamma}-X_{(k-1)\Delta_n^\gamma})^2}{\Delta_n^\gamma}\right| X_{(k-1)\Delta_n\gamma}\right] \1_{\{\rho_0\leq X_{(k-1)\Delta_n^\gamma}\leq \rho_0+K\psi_n^\gamma\}} \right] \\
&\hspace{0.5cm}  \leq C_{\alpha_0,\beta_0} \E_0\left[ \1_{\{\rho_0\leq X_{(k-1)\Delta_n^\gamma}\leq \rho_0+K\psi_n^\gamma\}} \right] \leq C_{\alpha_0,\beta_0}K\frac{1}{\sqrt{kn}},
\end{split}
\end{align}
where the last step uses the bound~\eqref{eq_proof:bound_expectation_indicator_T1} and Corollary~\ref{cor:increment_moments}.
Moreover, applying iterative conditioning, the bound~\eqref{eq:bound_transition_density} on the transition density and using boundedness of $x\mapsto x^2\exp(-x^2/c)$ yields
\begin{align}\label{eq_proof:T12_bound_squared_increment_indicator}
\begin{split}
&\E_0\left[\frac{(X_{k\Delta_n^\gamma}-X_{(k-1)\Delta_n^\gamma})^2}{\Delta_n^\gamma}\1_{\{\rho_0\leq X_{k\Delta_n^\gamma}\leq \rho_0+K\psi_n^\gamma\}} \1_{\{X_{(k-1)\Delta_n^\gamma}\leq \rho_0\}} \right] \\
&\hspace{0.5cm} \leq C_{\alpha_0,\beta_0}\E_0\left[ \1_{\{X_{(k-1)\Delta_n^\gamma}\leq \rho_0\}}\frac{1}{\sqrt{\Delta_n^\gamma}}\int_{\rho_0}^{\rho_0+K\psi_n^\gamma} \frac{(y-X_{(k-1)\Delta_n^\gamma})^2}{\Delta_n^\gamma} \right. \\
&\hspace{7cm} \left. \cdot\exp\left(-\frac{(y-X_{(k-1)\Delta_n^\gamma})^2}{2\max\{\alpha_0^2,\beta_0^2\}\Delta_n^\gamma}\right) dy \right] \\
&\hspace{0.5cm} \leq C_{\alpha_0,\beta_0}\E_0\left[ \1_{\{X_{(k-1)\Delta_n^\gamma}\leq \rho_0\}}\frac{1}{\sqrt{\Delta_n^\gamma}}\int_{\rho_0}^{\rho_0+K\psi_n^\gamma} \exp\left(-\frac{(y-X_{(k-1)\Delta_n^\gamma})^2}{4\max\{\alpha_0^2,\beta_0^2\}\Delta_n^\gamma}\right) dy \right] \\
&\hspace{0.5cm} \leq \frac{C_{\alpha_0,\beta_0}}{\sqrt{\Delta_n^\gamma}} K\psi_n^\gamma \E_0\left[ \exp\left(-\frac{(X_{(k-1)\Delta_n^\gamma}-\rho_0)^2}{4\max\{\alpha_0^2,\beta_0^2\}\Delta_n^\gamma}\right) \right] \\
&\hspace{0.5cm}  \leq C_{\alpha_0,\beta_0}K \frac{1}{\sqrt{nk}},
\end{split}
\end{align}
where the last step uses Corollary~\ref{cor:bound_exp_X^2}. Finally, \eqref{eq_proof:Approx_ell_1}, a direct evaluation of $|\alpha_0^{-2}-\alpha_n^{-2}|$ and the above two estimates reveal
\begin{align*}
&\E_0\left[ \sup_{z\in [-K,K]}  \sup_{|\alpha_n-\alpha_0|, |\beta_n-\beta_0|\leq L/\sqrt{n}} \left| T_n^{1,2}(z,\alpha_n,\beta_n) \right|\right] \leq C_{\alpha_0,\beta_0} LK \frac1n \sum_{k=1}^n \frac{1}{\sqrt{k}} \rightarrow 0.
\end{align*}

\smallskip
\item[$\mathbf{-\ T_n^{1,3}}$.] We abbreviate
\begin{align*}
a_N &:= \frac{\alpha_N-\beta_N}{\alpha_N+\beta_N}, \\
b_{N,k}(z) &:= \exp\left(-\frac{2}{\alpha_N^2\Delta_n^\gamma}\left( X_{k\Delta_n^\gamma}-\rho_0-z\psi_n^\gamma\right)\left(X_{(k-1)\Delta_n^\gamma}-\rho_0-z\psi_n^\gamma\right)\right).
\end{align*}
for $N\in\N\cup\{0\}$. Note that $b_{N,k}(z)$ also depends on the sample size $n$ and $\gamma$ which is suppressed in the notation. With these abbreviations, we can split
\[ T_n^{1,3}(z,\alpha_n,\beta_n) = S_n^1(z,\alpha_n,\beta_n) + S_n^2(z,\alpha_n,\beta_n)\]
with
\begin{align}\label{eq_proof:def_S1andS2}
\begin{split}
S_n^1(z,\alpha_n,\beta_n) &:= \sum_{k=1}^n \left[ \log\left(\frac{1-a_n b_{n,k}(z)}{1-a_0 b_{0,k}(z)}\right) - \log\left(\frac{1-a_n b_{n,k}(0)}{1-a_0 b_{0,k}(0)}\right)\right] \1_{J_{1,k}(\rho_0)}, \\
S_n^2(z,\alpha_n,\beta_n) &:= \sum_{k=1}^n  \log\left(\frac{1-a_n b_{n,k}(z)}{1-a_0 b_{0,k}(z)}\right)\left( \1_{J_{1,k}(\rho_0+z\psi_n^\gamma)} - \1_{J_{1,k}(\rho_0)}\right).
\end{split}
\end{align}
We will discuss these terms separately. \smallskip

\noindent
$\mathbf{\blacktriangleright S_n^1(z,\alpha_n,\beta_n)}$. Using standard properties of the logarithm, one finds
\begin{align*}
&\log\left(\frac{1-a_n b_{n,k}(z)}{1-a_0 b_{0,k}(z)}\right) - \log\left(\frac{1-a_n b_{n,k}(0)}{1-a_0 b_{0,k}(0)}\right) = \log\left(\frac{(1-a_n b_{n,k}(z))(1-a_0 b_{0,k}(0))}{(1-a_0 b_{0,k}(z))(1-a_n b_{n,k}(0))}\right)
\end{align*}
On the event $J_{1,k}(\rho_0)$, we have $b_{N,k}(z) \leq 1$, meaning that for $N\in\N\cup\{0\}$,
\[ \left(1-\frac{|\alpha_N-\beta_N|}{\alpha_N+\beta_N}\right)\1_{J_{1,k}(\rho_0)} \leq \left(1- a_N b_{N,k}(z)\right)\1_{J_{1,k}(\rho_0)} \leq \left( 1+\frac{|\alpha_N-\beta_N|}{\alpha_N+\beta_N}\right)\1_{J_{1,k}(\rho_0)}.\]
As $\alpha_n\to\alpha_0$, $\beta_n\rightarrow\beta_0$, we conclude that there exists a constant $C=C(\alpha_0,\beta_0)>0$ such that
\[ C^{-1}\1_{J_{1,k}(\rho_0)}\leq \frac{(1-a_n b_{n,k}(z))(1-a_0 b_{0,k}(0))}{(1-a_0 b_{0,k}(z))(1-a_n b_{n,k}(0))}\1_{J_{1,k}(\rho_0)} \leq C\1_{J_{1,k}(\rho_0)}.\]
As $x/(1+x) \leq \log(1+x)\leq x$ for every $x>-1$, we can therefore conclude that
\begin{align*}
&\left| \log\left(\frac{1-a_n b_{n,k}(z)}{1-a_0 b_{0,k}(z)}\right) - \log\left(\frac{1-a_n b_{n,k}(0)}{1-a_0 b_{0,k}(0)}\right)\right|\1_{J_{1,k}(\rho_0)} \\
&\hspace{1cm} \leq C_{\alpha_0,\beta_0} \1_{J_{1,k}(\rho_0)} \left|a_0b_{0,k}(z) - a_nb_{n,k}(z) + a_nb_{n,k}(0) - a_0 b_{0,k}(0) \right.\\
&\hspace{5.5cm} \left. + a_0a_n( b_{n,k}(z)b_{0,k}(0) - b_{0,k}(z)b_{n,k}(0)) \right| \\
&\hspace{1cm}\leq C_{\alpha_0,\beta_0}\left( S_{n,k}^{1,1}(z,\alpha_n,\beta_n) + S_{n,k}^{1,2}(z,\alpha_n,\beta_n) +S_{n,k}^{1,3}(z,\alpha_n,\beta_n)\right),
\end{align*}
with
\begin{align*}
S_{n,k}^{1,1}(z,\alpha_n,\beta_n) &:= \left|(a_0 - a_n)(b_{0,k}(z) - b_{0,k}(0) \right|\1_{J_{1,k}(\rho_0)},\\
S_{n,k}^{1,2}(z,\alpha_n,\beta_n) &:= \left|a_n \left( b_{0,k}(z)-b_{0,k}(0) - b_{n,k}(z) + b_{n,k}(0)\right) \right|\1_{J_{1,k}(\rho_0)},\\
S_{n,k}^{1,3}(z,\alpha_n,\beta_n) &:= \left| a_0a_n( b_{n,k}(z)b_{0,k}(0) - b_{0,k}(z)b_{n,k}(0)) \right|\1_{J_{1,k}(\rho_0)}.
\end{align*}
\smallskip

\noindent
$\mathbf{-\ S_{n,k}^{1,1}(z,\alpha_n,\beta_n)}$. Using a Taylor expansion, we find
\begin{align}\label{eq_proof:S_n11}
\begin{split}
&\sup_{|\alpha_n-\alpha_0|, |\beta_n-\beta_0|\leq L/\sqrt{n}} \left| a_0-a_n\right|\\
&\hspace{1cm}=\sup_{|\alpha_n-\alpha_0|, |\beta_n-\beta_0|\leq L/\sqrt{n}} \left| \frac{\alpha_0-\beta_0}{\alpha_0+\beta_0} - \frac{\alpha_n-\beta_n}{\alpha_n+\beta_n}\right| \leq C_{\alpha_0,\beta_0}(L) \frac{1}{\sqrt{n}}.
\end{split}
\end{align}
To proceed, we first observe that
\begin{align*}
&\left|b_{0,k}(z)-b_{0,k}(0)\right|\\
&\hspace{0.5cm}= \exp\left(-\frac{2}{\alpha_0^2\Delta_n^\gamma}(X_{(k-1)\Delta_n^\gamma}-\rho_0)(X_{k\Delta_n^\gamma}-\rho_0)\right)\\
&\hspace{1.5cm} \cdot \left| \exp\left(\frac{2}{\alpha_0^2\Delta_n^\gamma} z\psi_n^\gamma (X_{k\Delta_n^\gamma}+ X_{(k-1)\Delta_n^\gamma}-2\rho_0) - \frac{2}{\alpha_0^2\Delta_n^\gamma}z^2 (\psi_n^\gamma)^2\right) -1\right|.
\end{align*}
As $X_{(k-1)\Delta_n^\gamma},X_{k\Delta_n^\gamma}\leq \rho_0$ on $J_{1,k}(\rho_0)$, we obtain using $1-e^{-x}\leq x$,
\begin{align*}
&\sup_{z\in [0,K]} \left| b_{0,k}(z)-b_{0,k}(0)\right|\1_{J_{1,k}(\rho_0)} \\
&\hspace{1cm}\leq C_{\alpha_0}(K)\1_{J_{1,k}(\rho_0)}\exp\left(-\frac{2}{\alpha_0^2\Delta_n^\gamma}(X_{(k-1)\Delta_n^\gamma}-\rho_0)(X_{k\Delta_n^\gamma}-\rho_0)\right)\\
&\hspace{2cm} \cdot \left( n^{-(1-\gamma)/2} \left|X_{k\Delta_n^\gamma}+ X_{(k-1)\Delta_n^\gamma}-2\rho_0 \right| + \frac1n\right).
\end{align*} 
Moreover, the explicit form of the transition density given in~\eqref{eq:transition_density} reveals the bound
\begin{align*}
&\1_{J_{1,k}(\rho_0)} p_{\Delta_n^\gamma}^{\rho_0}(X_{(k-1)\Delta_n^\gamma},X_{k\Delta_n^\gamma})\\
&\hspace{1cm} \leq \1_{J_{1,k}(\rho_0)} \frac{1}{\sqrt{2\pi\alpha^2\Delta_n^\gamma}} \left(1+\frac{|\alpha-\beta|}{\alpha+\beta}\right)\exp\left(-\frac{(X_{k\Delta_n^\gamma}-X_{(k-1)\Delta_n^\gamma})^2}{2\alpha^2\Delta_n^\gamma}\right).
\end{align*} 
Using this explicit bound of the transition density on the event $J_{1,k}(\rho_0)$, the upper bound $(1+x)\exp(-x^2/c)\leq C_c\exp(-x^2/(2c)$ in the second inequality and $(a+b)^2 \geq a^2+b^2$ for $ab\geq 0$ in the third one gives the bound
\begin{align}\label{eq_proof:S_n11_2}
\begin{split}
&\E_0\left[\left. \sup_{z\in [0,K]} \left| b_{0,k}(z)-b_{0,k}(0)\right|\1_{J_{1,k}(\rho_0)}\right| X_{(k-1)\Delta_n^\gamma}\right] \\
& \hspace{0.2cm} \leq C_{\alpha_0,\beta_0}  \int_{-\infty}^{\rho_0} \exp\left(-\frac{2}{\alpha^2\Delta_n^\gamma}(y-\rho_0)(X_{(k-1)\Delta_n^\gamma}-\rho_0)\right) \\
&\hspace{1cm} \cdot \left( n^{-(1-\gamma)/2} \left|y+ X_{(k-1)\Delta_n^\gamma}-2\rho_0 \right| + \frac1n\right)\frac{1}{\sqrt{\Delta_n^\gamma}}\exp\left(-\frac{(y-X_{(k-1)\Delta_n^\gamma})^2}{2\alpha^2\Delta_n^\gamma}\right) dy  \\
& \hspace{0.2cm} = C_{\alpha_0,\beta_0} \int_{-\infty}^{\rho_0}  \left( n^{-(1-\gamma)/2} \left|y+ X_{(k-1)\Delta_n^\gamma}-2\rho_0 \right| + \frac1n\right) \\
& \hspace{5cm}\cdot \frac{1}{\sqrt{\Delta_n^\gamma}}\exp\left(-\frac{(y-2\rho_0+X_{(k-1)\Delta_n^\gamma})^2}{2\alpha^2\Delta_n^\gamma}\right) dy \\
& \hspace{0.2cm} \leq C_{\alpha_0,\beta_0} n^{-(1-\gamma)/2}\sqrt{\Delta_n^\gamma}\int_{-\infty}^{\rho_0} \frac{1}{\sqrt{\Delta_n^\gamma}}\exp\left(-\frac{(y-2\rho_0+X_{(k-1)\Delta_n^\gamma})^2}{4\alpha^2\Delta_n^\gamma}\right) dy \\
&\hspace{0.2cm} \leq C_{\alpha_0,\beta_0} n^{-1/2} \exp\left(-\frac{(X_{(k-1)\Delta_n^\gamma}-\rho_0)^2}{4\alpha^2\Delta_n^\gamma}\right) \int_{-\infty}^{\rho_0} \frac{1}{\sqrt{\Delta_n^\gamma}}\exp\left(-\frac{(y-\rho_0)^2}{4\alpha^2\Delta_n^\gamma}\right) dy \\
&\hspace{0.2cm} \leq C_{\alpha_0,\beta_0} n^{-1/2}\exp\left(-\frac{(X_{(k-1)\Delta_n^\gamma}-\rho_0)^2}{4\alpha^2\Delta_n^\gamma}\right).
\end{split}
\end{align}
Taking expectation and using Corollary~\ref{cor:bound_exp_X^2} and~\eqref{eq_proof:S_n11} finally reveals
\begin{align}\label{eq_proof:Sn_11_final}
&\E_0\left[ \sup_{z\in [0,K]}\sup_{|\alpha_n-\alpha_0|, |\beta_n-\beta_0|\leq L/\sqrt{n}}\left| S_{n,k}^{1,1}(z,\alpha_n,\beta_n)\right| \right] \leq  C_{\alpha_0,\beta_0}(L) \frac1n \frac{1}{\sqrt{k}}.
\end{align}
\smallskip

\noindent
$\mathbf{-\ S_{n,k}^{1,2}(z,\alpha_n,\beta_n)}$. We introduce the abbreviation
\begin{align}\label{eq_proof:Sn_12_4}
d(z) := \frac{2}{\Delta_n^\gamma}\left( X_{k\Delta_n^\gamma}-\rho_0-z\psi_n^\gamma\right)\left( X_{(k-1)\Delta_n^\gamma}-\rho_0-z\psi_n^\gamma\right) \geq 0,
\end{align} 
where the inequality is valid on $J_{1,k}(\rho_0)$ which we are working on. By a second order Taylor expansion,
\begin{align}\label{eq_proof:Sn_12_Taylor}
\begin{split}
&b_{n,k}(z) - b_{0,k}(z)\\
&\hspace{0.5cm} = \exp\left( -\frac{1}{\alpha_n^2} d(z)\right) - \exp\left( - \frac{1}{\alpha_0^2}d(z)\right) \\
&\hspace{0.5cm} = \left( \frac{1}{\alpha_0^2}-\frac{1}{\alpha_n^2}\right) d(z)\exp\left(-\frac{1}{\alpha_0^2} d(z)\right) + \frac12 \exp\left(-\xi(z) d(z)\right) d(z)^2  \left( \frac{1}{\alpha_0^2}-\frac{1}{\alpha_n^2}\right)^2,
\end{split}
\end{align}
where $\xi(z)$ is some intermediate value between $\alpha_n^{-2}$ and $\alpha_0^{-2}$. Subsequently, let $n$ be large enough such that $\alpha_0/2\leq \xi(z)\leq 3/2\alpha_0$ uniformly in $z$. This entails that
\[ \exp\left(-\xi(z) d(z)\right) d(z)^2 \leq \exp\left(-\frac{\alpha_0}{2} d(z)\right)d(z)^2\leq C_{\alpha_0}\exp\left( -\frac{\alpha_0}{4} d(z)\right).\]
To proceed, note that $d(z) \geq d(0)$ on $J_{1,k}(\rho_0)$ such that 
\begin{align}\label{eq_proof:Sn_12_remainder}
\begin{split}
\E_0\left[  \sup_{z\in [0,K]} \1_{J_{1,k}(\rho_0)}\exp\left(-\xi(z) d(z)\right) d(z)^2 \right] &\leq C_{\alpha_0}\E_0\left[\exp\left( -\frac{\alpha_0}{4} d(0)\right) \right] \leq C_{\alpha_0,\beta_0}\frac{1}{\sqrt{k}}
\end{split}
\end{align} 
by an analogous derivation as in~\eqref{eq_proof:S_n11_2} with a subsequent application of Corollary~\ref{cor:bound_exp_X^2}. In total we obtain for the remainder
\[ \E_0\left[  \sup_{z\in [0,K]}\sup_{|\alpha_n-\alpha_0|\leq L/\sqrt{n}}  \1_{J_{1,k}(\rho_0)}\exp\left(-\xi(z) d(z)\right) d(z)^2  \left( \frac{1}{\alpha_0^2}-\frac{1}{\alpha_n^2}\right)^2 \right]  \leq \frac{ C_{\alpha_0,\beta_0}L^2}{n\sqrt{k}}.\]
The same reasoning can be applied to $b_{n,k}(0)-b_{0,k}(0)$. Additionally, we have the estimate $|\alpha_n^{-2}-\alpha_0^{-2}|\leq C_{\alpha_0}L/\sqrt{n}$ for $|\alpha_n-\alpha_0|\leq L/\sqrt{n}$, such that we obtain in total 
\begin{align}\label{eq_proof:Sn_12}
\begin{split}
&\E_0\left[\sup_{z\in [0,K]}\sup_{|\alpha_n-\alpha_0|\leq L/\sqrt{n}} \1_{J_{1,k}(\rho_0)}\left|b_{0,k}(z)-b_{0,k}(0) - b_{n,k}(z) + b_{n,k}(0)\right|\right] \\
&\hspace{0.2cm} \leq \frac{C_{\alpha_0}L}{\sqrt{n}} \E_0\left[ \sup_{z\in [0,K]}\1_{J_{1,k}(\rho_0)}\left|d(z)\exp\left(-\frac{1}{\alpha_0^2} d(z)\right) -d(0)\exp\left(-\frac{1}{\alpha_0^2} d(0)\right)\right|\right]\\
&\hspace{1.5cm} + C_{\alpha_0,\beta_0}(L)\frac1n \frac{1}{\sqrt{k}}. 
\end{split}
\end{align}
To evaluate the remaining expectation, first note that
\begin{align}\label{eq_proof:Sn_12_5}
\begin{split}
&\1_{J_{1,k}(\rho_0)}\left|d(z)\exp\left(-\frac{1}{\alpha_0^2} d(z)\right) -d(0)\exp\left(-\frac{1}{\alpha_0^2} d(0)\right)\right| \\
&\hspace{0.5cm} \leq \1_{J_{1,k}(\rho_0)}\left| d(z)-d(0)\right|\exp\left(-\frac{1}{\alpha_0^2} d(z)\right)  \\
&\hspace{1.5cm} + \1_{J_{1,k}(\rho_0)}d(0)\left|\exp\left(-\frac{1}{\alpha_0^2} d(z)\right)-\exp\left(-\frac{1}{\alpha_0^2} d(0)\right)\right| \\
&\hspace{0.5cm} \leq \1_{J_{1,k}(\rho_0)}\left| d(z)-d(0)\right|\exp\left(-\frac{1}{\alpha_0^2} d(0)\right) + \1_{J_{1,k}(\rho_0)}d(0)\left|b_{0,k}(z)-b_{0,k}(0)\right|,
\end{split}
\end{align}
where the last line again uses that $d(z)\geq d(0)$ on $J_{1,k}(\rho_0)$ and the definition of $b_{0,k}(z)$. For the first summand, we estimate
\begin{align*}
&\sup_{z\in [0,K]} \1_{J_{1,k}(\rho_0)}\left| d(z)-d(0)\right| \\
&\hspace{1cm} = \sup_{z\in [0,K]} \1_{J_{1,k}(\rho_0)}\left|\frac{2z\psi_n^\gamma}{\Delta_n^\gamma}\left( X_{k\Delta_n^\gamma}+X_{(k-1)\Delta_n^\gamma} -2\rho_0\right) + \frac{z^2(\psi_n^\gamma)^2}{\Delta_n^\gamma}\right| \\
&\hspace{1cm}\leq \1_{J_{1,k}(\rho_0)}\left( 2Kn^{-(1-\gamma)/2}\left|X_{k\Delta_n^\gamma}+X_{(k-1)\Delta_n^\gamma} -2\rho_0 \right| + \frac{K^2}{n}\right).
\end{align*} 
Then by an analogous computation as used in~\eqref{eq_proof:S_n11_2} with a subsequent application of Corollary~\ref{cor:bound_exp_X^2} we obtain
\begin{align}\label{eq_proof:Sn_12_2}
\E_0\left[ \sup_{z\in [0,K]}\1_{J_{1,k}(\rho_0)} \left| d(z)-d(0)\right|\exp\left(-\frac{1}{\alpha_0^2} d(0)\right) \right] \leq C_{\alpha_0,\beta_0}  \frac{1}{\sqrt{nk}}.
\end{align} 
As $ab\leq (a+b)^2$ for $ab>0$ we find
\[ \1_{J_{1,k}(\rho_0)}d(0) \leq \1_{J_{1,k}(\rho_0)}\frac{2}{\Delta_n^\gamma} \left( X_{k\Delta_n^\gamma}+X_{(k-1)\Delta_n^\gamma}-2\rho_0\right)^2\]
and then with an analogous reasoning as in~\eqref{eq_proof:S_n11_2} with a subsequent application of Corollary~\ref{cor:bound_exp_X^2} we find for the second summand in the right-hand side of~\eqref{eq_proof:Sn_12_5},
\begin{align}\label{eq_proof:Sn_12_3}
\E_0\left[ \sup_{z\in [0,K]}\1_{J_{1,k}(\rho_0)} d(0)|\left|b_{0,k}(z)-b_{0,k}(0)\right| \right] \leq  C_{\alpha_0,\beta_0}  \frac{1}{\sqrt{nk}}.
\end{align} 
Plugging in the estimates~\eqref{eq_proof:Sn_12_2} and~\eqref{eq_proof:Sn_12_3} into~\eqref{eq_proof:Sn_12} yields
\begin{align*}
&\E_0\left[\sup_{z\in [0,K]}\sup_{|\alpha_n-\alpha_0|\leq L/\sqrt{n}} \1_{J_{1,k}(\rho_0)}\left|b_{0,k}(z)-b_{0,k}(0) - b_{n,k}(z) + b_{n,k}(0)\right|\right]\leq \frac{C_{\alpha_0,\beta_0}(L)}{n\sqrt{k}} 
\end{align*}
such that because $|a_n|\leq C_{\alpha_0,\beta_0}$,
\begin{align}\label{eq_proof:Sn_12_final}
\E_0\left[ \sup_{z\in [0,K]}\sup_{|\alpha_n-\alpha_0|, |\beta_n-\beta_0|\leq L/\sqrt{n}}\left| S_{n,k}^{1,2}(z,\alpha_n,\beta_n)\right| \right] \leq  C_{\alpha_0,\beta_0}(L) \frac1n \frac{1}{\sqrt{k}}.
\end{align}
\smallskip

\noindent
$\mathbf{-\ S_{n,k}^{1,3}(z,\alpha_n,\beta_n)}$. Recall $d(z)$ given in~\eqref{eq_proof:Sn_12_4}. Inserting the the Taylor expansion~\eqref{eq_proof:Sn_12_Taylor} for $b_{n,k}(z) - b_{0,k}(z)$ and $ b_{n,k}(0) - b_{0,k}(0)$ and subsequently applying the triangle inequality reveals
\begin{align*}
&\left| b_{n,k}(z)b_{0,k}(0) - b_{0,k}(z)b_{n,k}(0) \right| \\
&\hspace{0.2cm} = \left| b_{0,k}(0)\left( b_{n,k}(z) - b_{0,k}(z)\right) - b_{0,k}(z)\left( b_{n,k}(0)-b_{0,k}(0)\right)\right| \\
&\hspace{0.2cm} \leq \left| \exp\left(-\frac{1}{\alpha_0^2}d(0)\right)\left[ \exp\left(-\frac{1}{\alpha_0^2}d(z)\right)d(z) -\exp\left(-\frac{1}{\alpha_0^2}d(0)\right)d(0)\right]\left(\frac{1}{\alpha_0^2}-\frac{1}{\alpha_n^2}\right)\right| \\
&\hspace{1cm} + \left|\exp\left(-\frac{1}{\alpha_0^2}d(0)\right)d(0) \left[ \exp\left(-\frac{1}{\alpha_0^2}d(0)\right)-\exp\left(-\frac{1}{\alpha_0^2}d(z)\right) \right]\left(\frac{1}{\alpha_0^2}-\frac{1}{\alpha_n^2}\right) \right| \\
&\hspace{1cm} + \left|\frac12 d(z)^2\exp\left(-\xi_1(z) d(z)\right)\exp\left(-\frac{1}{\alpha_0^2}d(0)\right) \left(\frac{1}{\alpha_0^2}-\frac{1}{\alpha_n^2}\right)^2 \right| \\
&\hspace{1cm}  +\left|\frac12 d(0)^2\exp\left(-\xi_2 d(0)\right)\exp\left(-\frac{1}{\alpha_0^2}d(z)\right) \left(\frac{1}{\alpha_0^2}-\frac{1}{\alpha_n^2}\right)^2\right|, 
\end{align*}
for two intermediate values $\xi_1(z),\xi_2$ between $\alpha_0^{-2}$ and $\alpha_n^{-2}$. Taking into account that $d(z)\geq 0$ and therefore $|\exp(-\alpha_0^{-2}d(z))|\leq 1$ on $J_{1,k}(\rho_0)$ for all $z\in\R$, the last two terms can be treated analogously as the remainder in $S_{n,k}^{1,2}$ and yield the same bound~\eqref{eq_proof:Sn_12_remainder}. The square bracket in the second summand equals $b_{0,k}(0)-b_{0,k}(z)$ and the square bracket within the first summand equals the left-hand side of~\eqref{eq_proof:Sn_12_5} and is in particular bounded as in~\eqref{eq_proof:Sn_12_5}. This means that in total the first two summands are bounded by 
\begin{align*}
&\left[  \left| d(z)-d(0)\right|\exp\left(-\frac{1}{\alpha_0^2} d(0)\right) + d(0)\left|b_{0,k}(z)-b_{0,k}(0)\right|\right] \\
&\hspace{1cm} \cdot \left(1+d(0)\right) \exp\left(-\frac{1}{\alpha_0^2}d(0)\right)  \left(\frac{1}{\alpha_0^2}-\frac{1}{\alpha_n^2}\right)
\end{align*}
Finally, we use $(1+d(0))\exp(-\alpha_0^{-2}d(0))\leq C_{\alpha_0}$, $|a_0a_n|\leq C_{\alpha_0,\beta_0}$ and the two uniform moment bounds~\eqref{eq_proof:Sn_12_2} and~\eqref{eq_proof:Sn_12_3} and arrive at
\begin{align}\label{eq_proof:Sn_13_final}
\E_0\left[ \sup_{z\in [0,K]}\sup_{|\alpha_n-\alpha_0|, |\beta_n-\beta_0|\leq L/\sqrt{n}}\left| S_{n,k}^{1,3}(z,\alpha_n,\beta_n)\right| \right] \leq  C_{\alpha_0,\beta_0}(L) \frac1n \frac{1}{\sqrt{k}}.
\end{align} 
\smallskip

Summing up the bounds~\eqref{eq_proof:Sn_11_final}, \eqref{eq_proof:Sn_12_final} and~\eqref{eq_proof:Sn_13_final}, one obtains
\begin{align*}
& \E_0\left[ \sup_{z\in [0,K]}\sup_{|\alpha_n-\alpha_0|, |\beta_n-\beta_0|\leq L/\sqrt{n}}\left| S_{n,k}^{1}(z,\alpha_n,\beta_n)\right| \right] \leq  C_{\alpha_0,\beta_0}(L) \frac1n \sum_{k=1}^n \frac{1}{\sqrt{k}}\rightarrow 0.
\end{align*}
\smallskip

\noindent
$\mathbf{\blacktriangleright S_n^2(z,\alpha_n,\beta_n)}$. Using $x/(1+x)\leq \log(1+x)\leq x$ for $x>-1$ we obtain (analogously to $S_{n,k}^1$)
\begin{align*}
&\left|\log\left(\frac{1-a_nb_{n,k}(z)}{1-a_0b_{0,k}(z)}\right) \right|\1_{J_{1,k}(\rho_0+z\psi_n^\gamma)}\\
&\hspace{1cm} \leq C_{\alpha_0,\beta_0}\left| a_0b_{0,k}(z) - a_nb_{n,k}(z)\right| \1_{J_{1,k}(\rho_0+z\psi_n^\gamma)}\\
&\hspace{1cm} \leq C_{\alpha_0,\beta_0} \left( |a_0-a_n||b_{0,k}(z)| + |a_n| |b_{0,k}(z)-b_{n,k}(z)|\right)\1_{J_{1,k}(\rho_0+z\psi_n^\gamma)}\\
&\hspace{1cm} \leq C_{\alpha_0,\beta_0} \left( |a_0-a_n| +  |b_{0,k}(z)-b_{n,k}(z)|\right)\1_{J_{1,k}(\rho_0+z\psi_n^\gamma)},
\end{align*} 
where the last step uses that $|b_{0,k}(z)|\leq 1$ on $J_{1,k}(\rho_0+z\psi_n^\gamma)\supset J_{1,k}(\rho_0)$ and uniform boundedness of $|a_n|$ for $|\alpha_n-\alpha_0|,|\beta_n-\beta_0|\leq L/\sqrt{n}$. It is clear by~\eqref{eq_proof:S_n11} that
\[ \sup_{|\alpha_n-\alpha_0|,|\beta_n-\beta_0|\leq L/\sqrt{n}} \left|a_0-a_n\right| \leq C_{\alpha_0,\beta_0}(L)\frac{1}{\sqrt{n}}\]
and by a first order Taylor expansion we find
\begin{align*}
\left|b_{0,k}(z)-b_{n,k}(z)\right| &= \left| \exp\left(-\frac{1}{\alpha_0^2}d(z)\right)-\exp\left(-\frac{1}{\alpha_n^2} d(z)\right)\right| \\
&= \left|\frac{1}{\alpha_0^2}-\frac{1}{\alpha_n^2}\right| d(z)\exp\left(-\xi(z) d(z)\right) \leq C_{\alpha_0,\beta_0}\left|\frac{1}{\alpha_0^2}-\frac{1}{\alpha_n^2}\right|,
\end{align*}
for some intermediate value $\xi(z)$ between $\alpha_0^{-2}$ and $\alpha_n^{-2}$. The last step uses boundedness of $x\mapsto x\exp(-x^2/c)$ the fact that $\xi(z)$ is uniformly (in $z$) bounded from below by a positive constant for $n$ large enough which we also tacitly assume in the sequel. Summing up these to estimates gives
\[ \sup_{z\in [0,K]}\sup_{|\alpha_n-\alpha_0|, |\beta_n-\beta_0|\leq L/\sqrt{n}}\1_{J_{1,k}(\rho_0+z\psi_n^\gamma)}\left|\log\left(\frac{1-a_nb_{n,k}(z)}{1-a_0b_{0,k}(z)}\right) \right| \leq C_{\alpha_0,\beta_0}(L)\frac{1}{\sqrt{n}}.\]
Then by definition of $S_{n}^2$ in~\eqref{eq_proof:def_S1andS2} we obtain
\begin{align*}
&\E_0\left[ \sup_{z\in [0,K]}\sup_{|\alpha_n-\alpha_0|, |\beta_n-\beta_0|\leq L/\sqrt{n}} \left| S_n^2(z,\alpha_n,\beta_n)\right|\right] \\
&\hspace{1cm} \leq C_{\alpha_0,\beta_0}(L)\frac{1}{\sqrt{n}} \sum_{k=1}^n \E_0\left[ \sup_{z\in [0,K]}\left( \1_{J_{1,k}(\rho_0+z\psi_n^\gamma)} - \1_{J_{1,k}(\rho_0)}\right)\right] \\
&\hspace{1cm} \leq C_{\alpha_0,\beta_0}(L) \frac1n\sum_{k=1}^n\frac{1}{\sqrt{k}}\rightarrow 0,
\end{align*}
where the last step follows as for $T_n^{1,1}$.
\end{itemize}

\medskip
\noindent
$\bullet\ \underline{m=3}:$ Here, we have the decomposition
\begin{align*}
T_n^3(z,\alpha_n,\beta_n) &= T_n^{3,1}(z,\alpha_n,\beta_n) + T_n^{3,2}(z,\alpha_n,\beta_n) + T_n^{3,3}(z,\alpha_n,\beta_n)+ T_n^{3,4}(z,\alpha_n,\beta_n),
\end{align*}
where
\begin{align*}
T_n^{3,1}(z,\alpha_n,\beta_n) &:= \log\left( \frac{(\alpha_n\beta_0(\alpha_0+\beta_0)}{\alpha_0\beta_n(\alpha_n+\beta_n)}\right) \sum_{k=1}^n \left( \1_{J_{3,k}(\rho_0+z\psi_n^\gamma)} - \1_{J_{3,k}(\rho_0)}\right), \\
T_n^{3,2}(z,\alpha_n,\beta_n) &:= \frac{1}{2\Delta_n^\gamma}\sum_{k=1}^n \left( \1_{J_{3,k}(\rho_0+z\psi_n^\gamma)} - \1_{J_{3,k}(\rho_0)}\right)\left( \left(\frac{(X_{k\Delta_n^\gamma}-\rho_0}{\beta_0} - \frac{X_{(k-1)\Delta_n^\gamma}-\rho_0}{\alpha_0}\right)^2 \right. \\
&\hspace{3cm} \left. - \left(\frac{(X_{k\Delta_n^\gamma}-\rho_0}{\beta_n} - \frac{X_{(k-1)\Delta_n^\gamma}-\rho_0}{\alpha_n}\right)^2 \right),\\
T_n^{3,3}(z,\alpha_n,\beta_n) &:= \frac{z\psi_n^\gamma}{\Delta_n^\gamma}\sum_{k=1}^n \left( \left(\frac{1}{\alpha_0}-\frac{1}{\beta_0}\right)\left( \frac{X_{k\Delta_n^\gamma}-\rho_0}{\beta_0}-\frac{X_{(k-1)\Delta_n^\gamma}-\rho_0}{\alpha_0}\right) \right.\\
&\hspace{1.5cm} - \left. \left(\frac{1}{\alpha_n}-\frac{1}{\beta_n}\right)\left( \frac{X_{k\Delta_n^\gamma}-\rho_0}{\beta_n}-\frac{X_{(k-1)\Delta_n^\gamma}-\rho_0}{\alpha_n}\right) \right) \1_{J_{3,k}(\rho_0+z\psi_n^\gamma)}, \\
T_n^{3,4}(z,\alpha_n,\beta_n) &:= \frac{(z\psi_n^\gamma)^2}{\Delta_n^\gamma} \left( \left(\frac{1}{\alpha_0}-\frac{1}{\beta_0}\right)^2 - \left(\frac{1}{\alpha_n}-\frac{1}{\beta_n}\right)^2\right) \sum_{k=1}^n \1_{J_{3,k}(\rho_0+z\psi_n^\gamma)}.
\end{align*}

First of all, we note that
\begin{align}\label{eq_proof:bound_indicators_k=3}
\begin{split}
&\sup_{z\in [0,K]}\left|\1_{J_{3,k}(\rho_0+z\psi_n^\gamma)} - \1_{J_{3,k}(\rho_0)} \right| \\
&\hspace{1cm}\leq \1_{\{\rho_0\leq X_{(k-1)\Delta_n^\gamma} \leq \rho_0+K\psi_n^\gamma, X_{(k-1)\Delta_n^\gamma}\leq X_{k\Delta_n^\gamma} \}} + \1_{\{\rho_0\leq X_{k\Delta_n^\gamma} \leq \rho_0+K\psi_n^\gamma, X_{(k-1)\Delta_n^\gamma}\leq X_{k\Delta_n^\gamma}\}}.
\end{split}
\end{align}
As in~\eqref{eq_proof:bound_expectation_indicator_T1}, we obtain from this
\[ \E_{\alpha_0,\beta_0,\rho_0}\left[\sup_{z\in [0,K]}\left|\1_{J_{3,k}(\rho_0+z\psi_n^\gamma)} - \1_{J_{3,k}(\rho_0)} \right|\right] \leq C_{\alpha_0,\beta_0}KL\frac{1}{\sqrt{kn}}. \]
\begin{itemize}
\item[$\mathbf{-\ T_n^{3,1}}$.] By a Taylor expansion, we obtain 
\[ \sup_{|\alpha_n-\alpha_0|, |\beta_n-\beta_0|\leq L/\sqrt{n}} \left|\log\left( \frac{(\alpha_n\beta_0(\alpha_0+\beta_0)}{\alpha_0\beta_n(\alpha_n+\beta_n)}\right)\right| \leq C_{\alpha_0,\beta_0}L\frac{1}{\sqrt{n}}.\]
Together with~\eqref{eq_proof:bound_indicators_k=3}, we then obtain
\begin{align*}
&\E_0\left[ \sup_{z\in [-K,K]}  \sup_{|\alpha_n-\alpha_0|, |\beta_n-\beta_0|\leq L/\sqrt{n}} \left| T_n^{3,1}(z,\alpha_n,\beta_n) \right|\right] \leq C_{\alpha_0,\beta_0}LK\frac{1}{n}\sum_{k=1}^n \frac{1}{\sqrt{k}}\rightarrow 0.
\end{align*}

\smallskip
\item[$\mathbf{-\ T_n^{3,2}}$.] We first bound
\[ \left|T_n^{3,2}(z,\alpha_n,\beta_n)\right|\leq R_n^1(z,\beta_n) + R_n^2(z,\alpha_n,\beta_n) + R_n^3(z,\alpha_n), \]
where
\begin{align*}
R_n^1(z,\beta_n) &= \frac{1}{2\Delta_n^\gamma} \sum_{k=1}^n \left(X_{k\Delta_n^\gamma}-\rho_0\right)^2\left|\frac{1}{\beta_0^2}-\frac{1}{\beta_n^2}\right|\left| \1_{J_{3,k}(\rho_0+z\psi_n^\gamma)} - \1_{J_{3,k}(\rho_0)}\right| \\
R_n^2(z,\alpha_n,\beta_n) &= \frac{1}{2\Delta_n^\gamma} \sum_{k=1}^n\left|X_{k\Delta_n^\gamma}-\rho_0\right|\left|X_{(k-1)\Delta_n^\gamma}-\rho_0\right|\left|\frac{1}{\alpha_0\beta_0}-\frac{1}{\alpha_n\beta_n}\right| \\
&\hspace{3cm} \cdot\left| \1_{J_{3,k}(\rho_0+z\psi_n^\gamma)} - \1_{J_{3,k}(\rho_0)}\right| \\
R_n^3(z,\alpha_n) &= \frac{1}{2\Delta_n^\gamma} \sum_{k=1}^n \left(X_{(k-1)\Delta_n^\gamma}-\rho_0\right)^2\left|\frac{1}{\alpha_0^2}-\frac{1}{\alpha_n^2}\right|\left| \1_{J_{3,k}(\rho_0+z\psi_n^\gamma)} - \1_{J_{3,k}(\rho_0)}\right|.
\end{align*}
By a Taylor argument and using the bound~\eqref{eq_proof:bound_indicators_k=3}, we obtain
\begin{align*}
&\sup_{z\in [-K,K]}\sup_{|\beta_n-\beta_0|\leq L/\sqrt{n}}\left|R_n^1(z,\beta_n) \right| \\
&\hspace{1cm}\leq \frac{C_{\beta_0}L}{2\Delta_n^\gamma\sqrt{n}} \sum_{k=1}^n \left(X_{k\Delta_n^\gamma}-\rho_0\right)^2\1_{\{\rho_0\leq X_{(k-1)\Delta_n^\gamma} \leq \rho_0+K\psi_n^\gamma, X_{(k-1)\Delta_n^\gamma}\leq X_{k\Delta_n^\gamma} \}} \\
&\hspace{1.5cm} + \frac{C_{\beta_0}L}{2\Delta_n^\gamma\sqrt{n}} \sum_{k=1}^n \left(X_{k\Delta_n^\gamma}-\rho_0\right)^2 \1_{\{\rho_0\leq X_{k\Delta_n^\gamma} \leq \rho_0+K\psi_n^\gamma, X_{(k-1)\Delta_n^\gamma}\leq X_{k\Delta_n^\gamma}\}}.
\end{align*}
Using the condition of the indicator, the second summand can be directly bounded as
\[ \frac{C_{\beta_0}L}{2\Delta_n^\gamma\sqrt{n}} \sum_{k=1}^n \left(X_{k\Delta_n^\gamma}-\rho_0\right)^2 \1_{\{\rho_0\leq X_{k\Delta_n^\gamma} \leq \rho_0+K\psi_n^\gamma, X_{(k-1)\Delta_n^\gamma}\leq X_{k\Delta_n^\gamma}\}} \leq C_{\beta_0}LK^2 \frac{\sqrt{n}(\psi_n^\gamma)^2}{\Delta_n^\gamma},\]
which converges to zero as $\sqrt{n}(\psi_n^\gamma)^2(\Delta_n^\gamma)^{-1} =n^{-1/2}$. For the first one, we estimate its expectation. Using the bound~\eqref{eq:bound_transition_density} on the transition density, this gives first for a single summand
\begin{align}\label{eq_proof:Tn_32}
\begin{split}
&\E_0\left[ \frac{(X_{k\Delta_n^\gamma}-\rho_0´)^2}{\Delta_n^\gamma}\1_{\{\rho_0\leq X_{(k-1)\Delta_n^\gamma} \leq \rho_0+K\psi_n^\gamma, X_{(k-1)\Delta_n^\gamma}\leq X_{k\Delta_n^\gamma} \}} \right] \\
&\hspace{0.5cm} \leq C_{\alpha_0,\beta_0} \E_0\left[\1_{\{\rho_0\leq X_{(k-1)\Delta_n^\gamma} \leq \rho_0+K\psi_n^\gamma\}} \int_{X_{(k-1)\Delta_n^\gamma}}^\infty \frac{(y-\rho_0)^2}{\Delta_n^\gamma} \frac{1}{\sqrt{\Delta_n^\gamma}} \right.\\
&\hspace{8cm}\left. \cdot \exp\left(-\frac{(y-X_{(k-1)\Delta_n^\gamma})^2}{2\Delta_n^\gamma\max\{\alpha_0^2,\beta_0^2\}}\right) dy \right] \\
&\hspace{0.5cm} \leq C_{\alpha_0,\beta_0} \E_0\left[\1_{\{\rho_0\leq X_{(k-1)\Delta_n^\gamma} \leq \rho_0+K\psi_n^\gamma\}} \left(\frac{(X_{(k-1)\Delta_n^\gamma}-\rho_0)^2}{\Delta_n^\gamma}+1\right)  \right]\\
&\hspace{3cm} \cdot \int_0^\infty \left(1+y^2\right) \exp\left(-\frac{y^2}{2}\right) dy \\
&\hspace{0.5cm} \leq C_{\alpha_0,\beta_0} \left( K^2 \frac{(\psi_n^\gamma)^2}{\Delta_n^\gamma} + \frac{1}{\sqrt{nk}}\right),
\end{split}
\end{align}
where the last step uses the bound~\eqref{eq_proof:bound_expectation_indicator_T1}. Hence, we also have
\begin{align*}
&\E_0\left[ \frac{C_{\beta_0}L}{2\Delta_n^\gamma\sqrt{n}} \sum_{k=1}^n \left(X_{k\Delta_n^\gamma}-\rho_0\right)^2 \1_{\{\rho_0\leq X_{(k-1)\Delta_n^\gamma} \leq \rho_0+K\psi_n^\gamma, X_{(k-1)\Delta_n^\gamma}\leq X_{k\Delta_n^\gamma} \}}\right] \\
&\hspace{1cm} \leq C_{\alpha_0,\beta_0}L \left( K^2 \frac{\sqrt{n}(\psi_n^\gamma)^2}{\Delta_n^\gamma} + \frac{1}{\sqrt{n}}\sum_{k=1}^n \frac{1}{\sqrt{nk}}\right) \rightarrow 0
\end{align*}
and thus we have shown
\[ \E_0\left[ \sup_{z\in [-K,K]}  \sup_{|\beta_n-\beta_0|\leq L/\sqrt{n}} \left| R_n^{1}(z,\beta_n) \right|\right] \rightarrow 0. \]
The term $R_n^3$ can be dealt with analogously. For $R_n^2$ we again apply a Taylor argument and~\eqref{eq_proof:bound_indicators_k=3} to obtain
\begin{align*}
&\sup_{z\in [-K,K]}\sup_{|\alpha_n-\alpha_0|,|\beta_n-\beta_0|\leq L/\sqrt{n}}\left|R_n^2(z,\beta_n) \right| \\
&\hspace{0.5cm} \leq \frac{C_{\alpha_0,\beta_0}L}{2\Delta_n^\gamma\sqrt{n}} \sum_{k=1}^n \left|X_{k\Delta_n^\gamma}-\rho_0\right|\left|X_{(k-1)\Delta_n^\gamma}-\rho_0\right| \1_{\{\rho_0\leq X_{(k-1)\Delta_n^\gamma} \leq \rho_0+K\psi_n^\gamma, X_{(k-1)\Delta_n^\gamma}\leq X_{k\Delta_n^\gamma} \}} \\
&\hspace{1cm} + \frac{C_{\alpha_0,\beta_0}L}{2\Delta_n^\gamma\sqrt{n}} \sum_{k=1}^n\left|X_{k\Delta_n^\gamma}-\rho_0\right|\left|X_{(k-1)\Delta_n^\gamma}-\rho_0\right|\1_{\{\rho_0\leq X_{k\Delta_n^\gamma} \leq \rho_0+K\psi_n^\gamma, X_{(k-1)\Delta_n^\gamma}\leq X_{k\Delta_n^\gamma}\}} \\
&\hspace{0.5cm} \leq \frac{C_{\alpha_0,\beta_0}L\psi_n^\gamma}{2\Delta_n^\gamma\sqrt{n}} \sum_{k=1}^n \left|X_{k\Delta_n^\gamma}-\rho_0\right| \1_{\{\rho_0\leq X_{(k-1)\Delta_n^\gamma} \leq \rho_0+K\psi_n^\gamma, X_{(k-1)\Delta_n^\gamma}\leq X_{k\Delta_n^\gamma} \}} \\
&\hspace{1cm} + \frac{C_{\alpha_0,\beta_0}L\psi_n^\gamma}{2\Delta_n^\gamma\sqrt{n}} \sum_{k=1}^n \left|X_{(k-1)\Delta_n^\gamma}-\rho_0\right|\1_{\{\rho_0\leq X_{k\Delta_n^\gamma} \leq \rho_0+K\psi_n^\gamma, X_{(k-1)\Delta_n^\gamma}\leq X_{k\Delta_n^\gamma}\}}. 
\end{align*}
Next, we bound the expectation of a single summand using the bound~\eqref{eq:bound_transition_density} on the transition density. Analogously to~\eqref{eq_proof:Tn_32},
\begin{align*}
&\E_0\left[ \frac{|X_{k\Delta_n^\gamma}-\rho_0|}{\sqrt{\Delta_n^\gamma}}\1_{\{\rho_0\leq X_{(k-1)\Delta_n^\gamma} \leq \rho_0+K\psi_n^\gamma, X_{(k-1)\Delta_n^\gamma}\leq X_{k\Delta_n^\gamma} \}} \right]  \leq C_{\alpha_0,\beta_0}\left( \frac{\psi_n^\gamma}{\sqrt{\Delta_n^\gamma}} + \frac{1}{\sqrt{nk}}\right)
\end{align*}
and
\begin{align*}
&\E_0\left[\frac{|X_{(k-1)\Delta_n^\gamma}-\rho_0|}{\sqrt{\Delta_n^\gamma}}\1_{\{\rho_0\leq X_{k\Delta_n^\gamma} \leq \rho_0+K\psi_n^\gamma, X_{(k-1)\Delta_n^\gamma}\leq X_{k\Delta_n^\gamma}\}}\right] \\
&\hspace{0.5cm} \leq C_{\alpha_0,\beta_0}\E_0\left[\frac{|X_{(k-1)\Delta_n^\gamma}-\rho_0|}{\sqrt{\Delta_n^\gamma}} \1_{\{X_{(k-1)\Delta_n^\gamma}\leq \rho_0\}} \right.\\
&\hspace{5cm} \left. \cdot\int_{\rho_0}^{\rho_0+K\psi_n^\gamma} \frac{1}{\sqrt{\Delta_n^\gamma}}\exp\left(-\frac{(y-X_{(k-1)\Delta_n^\gamma})^2}{2\Delta_n^\gamma\max\{\alpha_0^2,\beta_0^2\}}\right)dy\right] \\
&\hspace{0.5cm} \leq C_{\alpha_0,\beta_0}\E_0\left[\frac{|X_{(k-1)\Delta_n^\gamma}-\rho_0|}{\sqrt{\Delta_n^\gamma}}\frac{K\psi_n^\gamma}{\sqrt{\Delta_n^\gamma}} \exp\left(-\frac{(X_{(k-1)\Delta_n^\gamma}-\rho_0)^2}{2\Delta_n^\gamma\max\{\alpha_0^2,\beta_0^2\}}\right) \right] \\
&\hspace{0.5cm} \leq C_{\alpha_0,\beta_0}\frac{K\psi_n^\gamma}{\sqrt{\Delta_n^\gamma}} \E_0\left[\exp\left(-\frac{(X_{(k-1)\Delta_n^\gamma}-\rho_0)^2}{4\Delta_n^\gamma\max\{\alpha_0^2,\beta_0^2\}}\right) \right] \\
&\hspace{0.5cm} \leq C_{\alpha_0,\beta_0}\frac{K\psi_n^\gamma}{\sqrt{\Delta_n^\gamma}} \frac{1}{\sqrt{k}},
\end{align*}
where the last step follows by Corollary~\ref{cor:bound_exp_X^2}. In total, we obtain
\begin{align*}
&\E_0\left[\sup_{z\in [-K,K]}\sup_{|\alpha_n-\alpha_0|,|\beta_n-\beta_0|\leq L/\sqrt{n}}\left|S_n^2(z,\beta_n) \right| \right] \\
&\hspace{1cm} \leq C_{\alpha_0,\beta_0} L \frac{\psi_n^\gamma}{\sqrt{n\Delta_n^\gamma}} \sum_{k=1}^n \left( \frac{\psi_n^\gamma}{\sqrt{\Delta_n^\gamma}} + \frac{1}{\sqrt{nk}}+\frac{K\psi_n^\gamma}{\sqrt{\Delta_n^\gamma}} \frac{1}{\sqrt{k}}\right)\\
&\hspace{1cm} \leq C_{\alpha_0,\beta_0}L\frac1n \left( \sqrt{n} + K+1\right)\rightarrow 0.
\end{align*}

\smallskip
\item[$\mathbf{-\ T_n^{3,3}}$.] This can be done analogously to the treatment of $R_n^2$ in the step $T_n^{3,2}$ before.

\smallskip
\item[$\mathbf{-\ T_n^{3,4}}$.] By a Taylor expansion, we find
\[ \sup_{|\alpha_n-\alpha_0|, |\beta_n-\beta_0|\leq L/\sqrt{n}}\left|\left( \left(\frac{1}{\alpha_0}-\frac{1}{\beta_0}\right)^2 - \left(\frac{1}{\alpha_n}-\frac{1}{\beta_n}\right)^2\right) \right| \leq C_{\alpha_0,\beta_0}L\frac{1}{\sqrt{n}}. \]
Using the explicit form of $\psi_n^\gamma$ and bounding each indicator by one yields
\begin{align*}
&\E_0\left[ \sup_{z\in [-K,K]}  \sup_{|\alpha_n-\alpha_0|, |\beta_n-\beta_0|\leq L/\sqrt{n}} \left| T_n^{3,4}(z,\alpha_n,\beta_n) \right|\right] \\
&\hspace{1cm}\leq C_{\alpha_0,\beta_0}LK^2 \frac{1}{\sqrt{n}}\frac{(\psi_n^\gamma)^2}{\Delta_n^\gamma} n  = C_{\alpha_0,\beta_0}LK^2\frac{1}{\sqrt{n}} \rightarrow 0.
\end{align*}
\end{itemize}

\section{Details on the proof of Proposition~\ref{prop:rate_of_triple_MLE}}\label{Section:triple_details}
This section is split into three parts. In Subsection~\ref{Subsection:tripleMLE_1}, we provide all details left out in the sketch of proof in Section~\ref{Section:Proof_sketch_triple}, except for Lemma~\ref{lemma:estimates_tripleMLE} which is proven in Subsection~\ref{Subsection:tripleMLE_2}. Finally, Subsection~\ref{Subsection:tripleMLE_3} contains additional auxiliary results that are used within Subsection~\ref{Subsection:tripleMLE_1}.

\subsection{Details for Section~\ref{Section:Proof_sketch_triple}}\label{Subsection:tripleMLE_1}

We start with a small remark that explains why~\eqref{eq_proof:order_remainder_ell1} is valid. Thereafter, the rest of the details left out in Section~\ref{Section:Proof_sketch_triple} is given.

\begin{remark}\label{Remark_LK=0}
Lemma~$4.1$ in~\cite{Brutsche/Rohde} remains valid for $K,L\geq 0$. In this case, we obtain the moment bounds
$\E[F_n^1(K,L)n^{-1/2}]\leq C_{\alpha_0,\beta_0}(L)<\infty$ and $\E[F_n^2(K,L)n^{-2/3}]\leq C_{\alpha_0,\beta_0}(L)<\infty$ for the remainders $F_n^1$ and $F_n^2$, where in case $L=0$ the constant $C_{\alpha_0,\beta_0}(L)$ is replaced by $C_{\alpha_0,\beta_0}$. The proof of this does not change at all. Here, we no longer have the statement that these moment bounds scale linearly in $L$, but the decomposition~\eqref{eq:decomposition_ell_1} does not involve $L$ and hence we do not need such a statement for the remainder in~\eqref{eq_proof:order_remainder_ell1}.
\end{remark}

In what follows, we are going to prove that for an appropriate choice of $K=(K_1,K_2,K_3)\in (0,\infty)^3$,
\begin{align}\label{eq_proof:introduction_proof_triple}
\limsup_{n\to\infty}\Pr_0\left( \sup_{(\theta_\rho,\theta_\alpha,\theta_\beta)\in \mathcal{T}_{n,\gamma}^j(K)} \left( \ell_{n,\gamma}^{\alpha_0,\beta_0}(\theta_\rho) + \ell_{n,\gamma}^2(\theta_\rho,\theta_\alpha,\theta_\beta)\right) \geq 0, B_n\cap E^\gamma\right) <\epsilon,
\end{align} 
for $j=1,2,3$ following the route of proof given in Section~\ref{Section:Proof_sketch_triple}. The event $E^\gamma$ is given in~\eqref{eq:event_Egamma} and $B_n$ is chosen to be the intersection of the events specified in Lemma~\ref{lemma:estimates_tripleMLE}, Lemma~\ref{lemma:bound_small_theta_beta} and Lemma~\ref{lemma:bound_L_sqrt(n)}, where the latter is replaced by Lemma~$4.1$ in~\citeSM{App:Brutsche/Rohde} in case $\gamma=1$. Throughout the proof, we use the notation
\[ f(x) = \log(x) - \frac12\left(x^2-1\right) \quad\textrm{ and }\quad g(x) = |\log(x)| + \frac12\left|x^2-1\right|.\]
Note that $f(x)\leq 0$ with equality if and only if $x=1$ and $f(x)\rightarrow -\infty$ for both $x\to 0$ and $x\to\infty$. During the proof, we will enumerate constants $C_j$ for an easier exposition. Moreover, we explicitly highlight their dependence on the parameters $K_1,K_2,K_3$ that have to be specified. Besides, the constants depend on $\alpha_0,\beta_0$ and the constants $\Gamma,\xi,\zeta$ appearing in~\eqref{eq_proof:setAn} which is suppressed in notation. \\
Although the constant $K=(K_1,K_2,K_3)$ has to be specified first for $j=3$ (which only depends on $K_3$), then for $j=2$ (that depends on $K_2,K_3$) and finally for $j=1$ (which depends on $K_1,K_2,K_3)$, we provide the derivation of~\eqref{eq_proof:introduction_proof_triple} in reversed order. The reason is that this order leads to a significantly more transparent proof.

\medskip
\noindent
\textbf{$\bullet\ \mathcal{T}_{n,\gamma}^1(K)$.} We assume $K_2,K_3$ to be given and construct $K_1$. By Lemma~\ref{lemma:estimates_tripleMLE}(a) with $L=K_2$ and Lemma~\ref{lemma:estimates_tripleMLE}(d) we have the bound
\[ \sup_{(\theta_\rho,\theta_\alpha,\theta_\beta)\in\mathcal{T}_{n,\gamma}^1} \ell_{n,\gamma}^2(\theta_\rho,\theta_\alpha,\theta_\beta)\1_{B_n} \leq C(K_2)\1_{B_n}\]
for $n$ sufficiently large, where we have used a Taylor expansion of $g$ in $x_0=1$ in the application of Lemma~\ref{lemma:estimates_tripleMLE}(d) which reveals $g(\alpha_0/(\alpha_0+\theta_\alpha)),g(\beta_0/(\beta_0+\theta_\beta))\leq C(K_2)/\sqrt{n}$ for $|\theta_\alpha|,|\theta_\beta|\leq K_2/\sqrt{n}$. Hence,
\begin{align*}
&\limsup_{n\to\infty} \Pr_0\left( \sup_{(\theta_\rho,\theta_\alpha,\theta_\beta)\in \mathcal{T}_{n,\gamma}^1(K)} \left( \ell_{n,\gamma}^{\alpha_0,\beta_0}(\theta_\rho) + \ell_{n,\gamma}^2(\theta_\rho,\theta_\alpha,\theta_\beta)\right) \geq 0, B_n\cap E^\gamma\right)\\
&\hspace{0.5cm} \leq \limsup_{n\to\infty} \Pr_0\left( \sup_{K_1n^{-(1+\gamma)/2}<|\theta_\rho|<K_3n^{-\gamma/2}}  \ell_{n,\gamma}^{\alpha_0,\beta_0}(\theta_\rho) \geq -C(K_2), B_n\cap E^\gamma\right).
\end{align*}
With this, we are back in the proof of consistency for known $\alpha_0,\beta_0$. In particular, showing that this term is bounded by $\epsilon$ for $K_1$ large enough is obtained for $\gamma=1$ by following the steps from~$(4.16)$ onwards in Subsection~$4.2$ of~\citeSM{App:Brutsche/Rohde}. For $\gamma\in(0,1)$ it works the same and is part of the proof of Lemma~\ref{lemma:bound_L_sqrt(n)}.

\medskip
\noindent
\textbf{$\bullet\ \mathcal{T}_{n,\gamma}^2(K)$.} We assume $K_3$ in $K=(K_1,K_2,K_3)\in (0,\infty)^3$ to be given and construct $K_2$. Note that $T_{n,\gamma}^2(K)$ does not depend on $K_1$. For constants $C',K'$ we first estimate
\begin{align}
&\Pr_0\Bigg( \sup_{(\theta_\rho,\theta_\alpha,\theta_\beta)\in\mathcal{T}_{n,\gamma}^2(K)}\ell_n(\theta_\rho,\theta_\alpha,\theta_\beta)\geq 0, B_n\cap E^\gamma \Bigg)\notag \\
&\hspace{0.5cm} \leq \Pr_0\Bigg( \sup_{K'n^{-(1+\gamma)/2}\leq |\theta_\rho| \leq K_3n^{-\gamma/2}}\ell_n^1(\theta_\rho)\geq 0, B_n\cap E^\gamma \Bigg) \label{eq_proof:P1}  \\
&\hspace{1cm} + \Pr_0\Bigg( \sup_{|\theta_\rho|\leq K'n^{-(1+\gamma)/2}} |\ell_n^1(\theta_\rho)| \geq C', E^\gamma\Bigg) \label{eq_proof:P2}\\
&\hspace{1cm} + 2\Pr_0\Bigg( \sup_{\substack{\max\{|\theta_\alpha|,|\theta_\beta|\} > K_2/\sqrt{n} \\|\theta_\rho|\leq K_3n^{-\gamma/2}}}\ell_n^2(\theta_\rho,\theta_\alpha,\theta_\beta) \geq -C', B_n\cap E^\gamma\Bigg). \label{eq_proof:P3}
\end{align}
To proceed, we first choose $K'$ large enough such that~\eqref{eq_proof:P1} is bounded by $\epsilon$. This is possible as already outlined in the discussion of $\mathcal{T}_{n,\gamma}^1$. For this $K'$, choose $C'$ large enough such that $\limsup_{n\to\infty}\eqref{eq_proof:P2}<\epsilon$. By Theorem~$13.2$ in~\citeSM{App:Billingsley_2}, this is possible as $(\ell_n^1(zn^{-(1+\gamma)/2}))_{[-K',K']}$ converges towards a tight limit in the Skorohod space $\mathcal{D}([-K',K'])$ by (1A) for $\gamma\in (0,1)$ and by Proposition~$1.2$ in~\citeSM{App:Brutsche/Rohde} for $\gamma=1$.
It is therefore sufficient to deal with~\eqref{eq_proof:P3} for given $C'$. By (b)-(d) of Lemma~\ref{lemma:estimates_tripleMLE} we find for any $R>0$ a constant $C_1=C_1(K_3,R)>0$ such that
\begin{align}\label{eq_proof:triple_consistency_3}
\begin{split}
\ell_{n,\gamma}^2(\theta_\rho,\theta_\alpha,\theta_\beta)\1_{B_n} &\leq \bigg(\overline{H}_{n,\gamma}^{\leq}(0,\theta_\alpha) + \overline{H}_{n,\gamma}^{>}(0,\theta_\beta) + nd_{\alpha_0,\beta_0}\left(\1_{\{\theta_\alpha>R\}}+\1_{\{\theta_\beta>R\}}\right) \\
&\hspace{1cm} \left. + C_1\left( g\left(\frac{\alpha_0}{\alpha_0+\theta_\alpha}\right) + g\left(\frac{\beta_0}{\beta_0+\theta_\beta}\right)\right)\sqrt{n}\right)\1_{B_n},
\end{split}
\end{align}
where the constant $d_{\alpha_0,\beta_0}>0$ is explicitly given in Lemma~\ref{lemma:estimates_tripleMLE}(d). Recalling $f(x)\leq 0$ and applying Lemma~\ref{lemma:estimate_counts_interval} for the parameters $a=\zeta\sqrt{n\Delta_n^\gamma}$ and $b=0$ yields
\begin{align}\label{eq_proof:triple_consistency_1}
\overline{H}_{n,\gamma}^{\leq}(0,\theta_\alpha)\1_{B_n} \leq \1_{B_n} C_2n f\left(\frac{\alpha_0}{\alpha_0+\theta_\alpha}\right)
\end{align} 
for $C_2=\xi\zeta/\max\{\alpha_0^2,\beta_0^2\}$ and with $a=0$, $b=\zeta\sqrt{n\Delta_n^\gamma}$ we also obtain
\begin{align}\label{eq_proof:triple_consistency_2}
\overline{H}_{n,\gamma}^{>}(0,\theta_\beta)\1_{B_n}\leq \1_{B_n} C_2n f\left(\frac{\beta_0}{\beta_0+\theta_\beta}\right). 
\end{align} 
This means that 
\[ \limsup_{n\to\infty}\Pr_0\Bigg( \sup_{\substack{\max\{|\theta_\alpha|,|\theta_\beta|\} > K_2/\sqrt{n} \\|\theta_\rho|\leq K_3n^{-\gamma/2}}}\ell_n^2(\theta_\rho,\theta_\alpha,\theta_\beta) \geq -C', B_n\cap E^\gamma\Bigg) =0 \]
follows by the estimates~\eqref{eq_proof:triple_consistency_3}, \eqref{eq_proof:triple_consistency_1} and~\eqref{eq_proof:triple_consistency_2} once we show that 
\begin{align}\label{eq_proof:triple_consistency_4}
\begin{split}
&\sup_{\substack{\max\{|\theta_\alpha|,|\theta_\beta|\} > K_2/\sqrt{n} \\|\theta_\rho|\leq K_3n^{-\gamma/2}}}  \left( C_2 n\left( f\left(\frac{\alpha_0}{\alpha_0+\theta_\alpha}\right) + f\left(\frac{\beta_0}{\beta_0+\theta_\beta}\right)\right)  \right. \\
&\hspace{0.5cm} \left. + nd_{\alpha_0,\beta_0}\left(\1_{\{\theta_\alpha>R\}}+\1_{\{\theta_\beta>R\}}\right)+ C_1\sqrt{n}\left( g\left(\frac{\alpha_0}{\alpha_0+\theta_\alpha}\right) + g\left(\frac{\beta_0}{\beta_0+\theta_\beta}\right)\right)\right) < -C'
\end{split}
\end{align}
for all $n\geq n_0$ with $n_0$ large enough, appropriate $R>0$ and $K_2$ sufficiently large. Note that the restriction for $\theta_\alpha$ and $\theta_\beta$ does not prevent the terms $g(\alpha_0/(\alpha_0+\theta_\alpha))$ and $g(\beta_0/(\beta_0+\theta_\beta))$ from exploding. However, these terms appear as a multiple of $\sqrt{n}$, compared to the leading terms involving the negative function $f$ that contain the factor $n$ and indeed turn out to compensate for the remainder involving $g$ for any $\theta_\alpha,\theta_\beta$.\\
We continue with proving~\eqref{eq_proof:triple_consistency_4}. Here, we first use $f(x)\rightarrow -\infty$ for $x\to 0$ in order to get rid of the term $nd_{\alpha_0,\beta_0}\1_{\{\theta_\alpha>R\}}$. To this aim, choose $0<r<1$ small enough such that for $x\leq r$ we have $f(x)<-2d_{\alpha_0,\beta_0}/C_2$. Then, 
\begin{align*}
C_2nf\left(\frac{\alpha_0}{\alpha_0+\theta_\alpha}\right) + nd_{\alpha_0,\beta_0} \leq \frac{C_2n}{2}f\left(\frac{\alpha_0}{\alpha_0+\theta_\alpha}\right) \quad\textrm{ for }\ \theta_\alpha > \frac{(1-r)\alpha_0}{r}.
\end{align*} 
An analogous estimate holds true when replacing $\alpha_0$ by $\beta_0$ and $\theta_\alpha$ by $\theta_\beta$ such that for the choice
\begin{align}\label{eq_proof:choiceR}
R := \frac{1-r}{r} \max\{\alpha_0,\beta_0\}
\end{align}
a sufficient condition for~\eqref{eq_proof:triple_consistency_4} is given by 
\begin{align*}
\sup_{\substack{\max\{|\theta_\alpha|,|\theta_\beta|\} > K_2/\sqrt{n} \\|\theta_\rho|\leq K_3n^{-\gamma/2}}} \sqrt{n} \left( F_{n,2C_1/C_2}\left( \frac{\alpha_0}{\alpha_0+\theta_\alpha}\right) + F_{n,2C_1/C_2}\left(\frac{\beta_0}{\beta_0+\theta_\beta}\right) \right) < -2C'/C_2,
\end{align*}
for all $n\geq n_0$ and $K_2$ sufficiently large, where the function $F_{n,C}$ is given as
\begin{align}\label{eq:function_FnC}
F_{n,C}(x) = \sqrt{n}\left( \log(x) - \frac12\left(x^2-1\right)\right) + C\left|\log(x)\right| + \frac{C}{2}\left|x^2-1\right|.
\end{align} 
Such a choice of $K_2$ is possible by properties (a) and (b) of $F_{n,C}$ given in Lemma~\ref{lemma:properties_F_nC}.

\medskip
\noindent
\textbf{$\bullet\ \mathcal{T}_{n,\gamma}^3(K)$.} As outlined in Section~\ref{Section:Proof_sketch_triple} we have to combine effects of $\ell_{n,\gamma}^{\alpha_0,\beta_0}$ and $\ell_{n,\gamma}^2$. One crucial point is to understand the effect of the term $T_{n,\gamma}$ in the decomposition~\eqref{eq:decomposition_ell_1+2}. The stochastic order of this term is determined by the number of indices $k$ with $\rho_0\leq X_{(k-1)\Delta_n^\gamma}\leq \rho_0+\theta_\rho$ and therefore we need to distinguish further cases. Moreover, there appears another special case: If $\theta_\rho$ is chosen such that the number of indices $k$ for which $X_{(k-1)\Delta_n^\gamma}\geq \rho_0+\theta_\rho$ is no longer of order $n$, the negative compensator of the term $g(\beta_0/(\beta_0+\theta_\beta))$ is missing. This term emerges from an application of Lemma~\ref{lemma:estimates_tripleMLE}(d) and explodes as $\theta_\beta\to -\beta_0$ and $\theta_\beta\to\infty$. For these boundary cases, we therefore no longer use the decomposition~\eqref{eq:decomposition_ell_1+2} together with Lemma~\ref{lemma:estimates_tripleMLE} as we did for $\mathcal{T}_{n,\gamma}^2(K)$. Instead, we will base our analysis for $\theta_\beta\approx -\beta$ and $\theta_\beta\to\infty$ on the decomposition in Lemma~\ref{lemma:bound_small_theta_beta}. 

To make the above mentioned heuristics precise, we fix $\delta_l, \delta_u>0$ and recall $\zeta>0$ from the definition of the event $A_n$ in~\eqref{eq_proof:setAn} that was chosen such that $(X_t)_{t\leq n\Delta_n^\gamma}$ has crossed the value $\zeta\sqrt{n\Delta_n^\gamma}$ at least once. We then decompose
\[ \mathcal{T}_{n,\gamma}^3(K) = \bigcup_{m=1}^5 \mathcal{T}_{n,\gamma}^{3,m}(K_3,\delta_l,\delta_u)\]
with the disjoint sets
\begin{align*}
\mathcal{T}_{n,\gamma}^{3,1}(K_3,\delta_l,\delta_u) &:= \left\{(\theta_\rho,\theta_\alpha,\theta_\beta): K_3n^{-\gamma/2} <\theta_\rho <n^{1/4-\gamma/2}  \right\}, \\
\mathcal{T}_{n,\gamma}^{3,2}(K_3,\delta_l,\delta_u) &:= \left\{(\theta_\rho,\theta_\alpha,\theta_\beta): n^{1/4-\gamma/2} \leq \theta_\rho <\zeta\sqrt{n\Delta_n^\gamma}/2  \right\}, \\
\mathcal{T}_{n,\gamma}^{3,3}(K_3,\delta_l,\delta_u) &:= \left\{(\theta_\rho,\theta_\alpha,\theta_\beta): \zeta\sqrt{n\Delta_n^\gamma}/2 \leq \theta_\rho <\infty, \beta_0+\theta_\beta < \delta_l  \right\}, \\
\mathcal{T}_{n,\gamma}^{3,4}(K_3,\delta_l,\delta_u) &:= \left\{(\theta_\rho,\theta_\alpha,\theta_\beta): \zeta\sqrt{n\Delta_n^\gamma}/2 \leq \theta_\rho <\infty,  \beta_0+\theta_\beta > \delta_u  \right\}, \\
\mathcal{T}_{n,\gamma}^{3,5}(K_3,\delta_l,\delta_u) &:= \left\{(\theta_\rho,\theta_\alpha,\theta_\beta): \zeta\sqrt{n\Delta_n^\gamma}/2 \leq \theta_\rho <\infty, \delta_l \leq \beta_0+\theta_\beta \leq \delta_u  \right\}.\\
\end{align*} 
In what follows we bound the probability
\[ \Pr_0\left( \sup_{(\theta_\rho,\theta_\alpha,\theta_\beta)\in \mathcal{T}_{n,\gamma}^{3,m}(K_3,\delta_l,\delta_u)} \left( \ell_{n,\gamma}^{\alpha_0,\beta_0}(\theta_\rho) + \ell_{n,\gamma}^2(\theta_\rho,\theta_\alpha,\theta_\beta)\right) \geq 0, B_n\cap E^\gamma\right)\]
by $\epsilon$ for $m=1,2,3,4,5$, $n\geq n_0$ and appropriate $K_3,\delta_l,\delta_u$. Hereby, it is important to choose $\delta_l,\delta_u$ such that the probability for $m=3,4$ satisfies the required bound and then afterwards prove that for this choice the term for $m=5$ can also be bounded by $\epsilon$ for $n$ sufficiently large.

\medskip
\noindent
\textbf{$\underline{m=1}$}: Parts (c) and (d) of Lemma~\ref{lemma:estimates_tripleMLE} together with the bound~\eqref{eq_proof:order_remainder_ell1} yield for any $R>0$ and a constant $C_3=C_3(R)$,
\begin{align*}
\sup_{\theta_\rho\in [K_3 n^{-\gamma/2}, n^{1/4-\gamma/2})} R_{n,\gamma}^{all}(\theta_\rho,\theta_\alpha,\theta_\beta)\1_{B_n} &\leq \1_{B_n}C_3\sqrt{n}\left(1+ g\left(\frac{\alpha_0}{\alpha_0+\theta_\alpha}\right) + g\left(\frac{\beta_0}{\beta_0+\theta_\beta}\right)\right) \\
&\hspace{1cm} +nd_{\alpha_0,\beta_0}\left(\1_{\{\theta_\alpha >R\}}+\1_{\{\theta_\beta>R\}}\right)
\end{align*} 
where the constant $d_{\alpha_0,\beta_0}$ is explicitly given in Lemma~\ref{lemma:estimates_tripleMLE}(d). Applying Lemma~\ref{lemma:estimate_counts_interval} to the parameters $a=-n^{1/4-\gamma/2}$, $b=\zeta\sqrt{n\Delta_n^\gamma}$ yields for a constant $C_4\leq C_2=\xi\zeta/\max\{\alpha_0^2,\beta_0^2\}$
\[ \sup_{\theta_\rho \in[Kn^{-\gamma/2}, n^{1/4-\gamma/2})}\overline{H}_{n,\gamma}^{>}(\theta_\rho,\theta_\alpha)\1_{B_n} \leq - \1_{B_n}C_4 n f\left(\frac{\beta_0}{\beta_0+\theta_\beta}\right).\]
Inserting these two estimates and~\eqref{eq_proof:triple_consistency_1} into the decomposition~\eqref{eq:decomposition_ell_1+2} then gives 
\begin{align}\label{eq_proof:T3_part1}
\begin{split}
&\left(\ell_{n,\gamma}^{\alpha_0,\beta_0}(\theta_\rho) + \ell_{n,\gamma}^2(\theta_\rho,\theta_\alpha,\theta_\beta)\right)\1_{B_n} \\
&\hspace{0.5cm} \leq \1_{B_n} C_4 n\left( f\left(\frac{\alpha_0}{\alpha_0+\theta_\alpha}\right) + f\left(\frac{\beta_0}{\beta_0+\theta_\beta}\right)\right) + T_{n,\gamma}(\theta_\rho,\theta_\alpha) + \1_{B_n}nd_{\alpha_0,\beta_0}\1_{\{\theta_\alpha>R\}} \\
&\hspace{1.5cm} +\1_{B_n} nd_{\alpha_0,\beta_0}\1_{\{\theta_\beta>R\}}+ \1_{B_n}C_3\sqrt{n}\left(1+ g\left(\frac{\alpha_0}{\alpha_0+\theta_\alpha}\right) + g\left(\frac{\beta_0}{\beta_0+\theta_\beta}\right)\right).
\end{split}
\end{align} 
We recall at this point that
\[ T_{n,\gamma}(\theta_\rho,\theta_\alpha) = H_{\alpha_0,\beta_0}(\theta_\alpha)\sum_{k=1}^n \1_{\{\rho_0\leq X_{(k-1)\Delta_n^\gamma} \leq \rho_0+\theta_\rho\}}\]
for the function
\[ H_{\alpha_0,\beta_0}(\theta_\alpha) =  \log\left(\frac{\beta_0}{\alpha_0}\right) - \frac12\left(\frac{\beta_0^2}{\alpha_0^2}-1\right) + \log\left(\frac{\alpha_0}{\alpha_0+\theta_\alpha}\right) - \frac12\left( \frac{\alpha_0^2}{(\alpha_0+\theta_\alpha)^2}-1\right)\frac{\beta_0^2}{\alpha_0^2}\]
for which by Lemma~\ref{lemma:function_negative} we have $H_{\alpha_0,\beta_0}(\theta_\alpha) \leq 0$ with equality if and only if $\theta_\alpha=\beta_0-\alpha_0$. Compared to~\eqref{eq_proof:triple_consistency_4}, the supremum is taken over $\mathcal{T}_{n,\gamma}^{3,1}$ instead of $\mathcal{T}_{n,\gamma}^1$ and we have the additional terms $T_{n,\gamma}$ and $C_3\sqrt{n}$, where the first one is negative and the second one an additional remainder. The crucial observation is now that $T_{n,\gamma}(\theta_\rho,\theta_\alpha)$ and $f(\alpha_0/(\alpha_0+\theta_\alpha))$ which are both negative cannot be close to zero at the same time as for the first expression this happens for $\theta_\alpha\approx \beta_0-\alpha_0$ and for the second one in case $\theta_\alpha\approx 0$. To make this precise, we first choose $R$ as in~\eqref{eq_proof:choiceR} and then further bound the right-hand side of~\eqref{eq_proof:T3_part1} (recall $f(x)\leq 0$) by
\begin{align*}
\1_{B_n}\left[ \frac{C_4\sqrt{n}}{2}\left( F_{n,2C_3/C_4}\left(\frac{\alpha_0}{\alpha_0+\theta_\alpha}\right) + F_{n,2C_3/C_4}\left(\frac{\beta_0}{\beta_0+\theta_\beta}\right)\right) + T_{n,\gamma}(\theta_\rho,\theta_\alpha) + C_3\sqrt{n}\right].
\end{align*}
We now distinguish two cases and abbreviate $\kappa=|\beta_0-\alpha_0|/4$:
\begin{itemize}
\item[$(1)\ \mathbf{\theta_\alpha\in B_\kappa(\beta_0-\alpha_0)}$.] Since in this case $|\theta_\alpha|>\kappa$, which means $\alpha_0/(\alpha_0+\theta_\alpha)\notin B_{c(\kappa)}(1)$ for some constant $c(\kappa)>0$. Then by Lemma~\ref{lemma:properties_F_nC}(a) and (c), and $T_{n,\gamma}\leq 0$, we obtain for some constant $C_5=C_5(\kappa)$,
\begin{align*}
&\left(\ell_{n,\gamma}^{\alpha_0,\beta_0}(\theta_\rho) + \ell_{n,\gamma}^2(\theta_\rho,\theta_\alpha,\theta_\beta)\right)\1_{B_n} \leq -\1_{B_n}C_5 n + \1_{B_n}C_3\sqrt{n},
\end{align*}
where the right-hand side is independent of $\theta_\rho,\theta_\alpha$ and $\theta_\beta$ and clearly converges to $-\infty$ for $n\to\infty$. \smallskip

\item[$(2)\ \mathbf{\theta_\alpha\notin B_\kappa(\beta_0-\alpha_0)}$.] In this case for $\theta_\alpha$, there exists a constant $C_6=C_6(\kappa)$ such that $H_{\alpha_0,\beta_0}(\theta_\alpha)\leq -C_6$. As by Lemma~\ref{lemma:properties_F_nC}(a) $\sqrt{n}F_{n,C}$ is bounded, we find some constant $C_7$ such that
\begin{align*}
&\sup_{K_3n^{-\gamma/2}\leq \theta_\rho <n^{1/4-\gamma/2}}\sup_{\theta_\alpha\notin B_{\delta(\beta_0-\alpha_0)}}\sup_{\theta_\beta}  \left(\ell_{n,\gamma}^{\alpha_0,\beta_0}(\theta_\rho) + \ell_{n,\gamma}^2(\theta_\rho,\theta_\alpha,\theta_\beta)\right)\1_{B_n} \\
&\hspace{1cm} \leq \1_{B_n}C_7\sqrt{n} - C_6 \1_{B_n}\sum_{k=1}^n \1_{\{\rho_0\leq X_{(k-1)\Delta_n^\gamma} \leq \rho_0+K_3n^{-\gamma/2}\}}.
\end{align*}
On the event
\[ B_n' := \left\{ \frac{C_6}{\sqrt{n}}\sum_{k=1}^n \1_{\{\rho_0\leq X_{(k-1)\Delta_n^\gamma} \leq \rho_0+K_3n^{-\gamma/2}\}} >2C_7\right\} \]
this upper bound converges to $-\infty$ for $n\to\infty$. Taking $K_3>0$ large enough ensures $\inf_{n\geq n_0}\Pr_0(B_n')>1-\epsilon$ for appropriate $n_0$ by Lemma~\ref{lemma:probability_A6}.
\end{itemize}

\medskip
\noindent
\textbf{$\underline{m=2}$}: Here, part (c) and (d) of Lemma~\ref{lemma:estimates_tripleMLE} together with the bound~\eqref{eq_proof:order_remainder_ell1} yield for any $R>0$ and a constant $C_8=C_8(R)$,
\begin{align}\label{eq_proof:R_all}
\begin{split}
\sup_{\theta_\rho>n^{1/4-\gamma/2}} R_{n,\gamma}^{all}(\theta_\rho,\theta_\alpha,\theta_\beta)\1_{B_n}  &\leq \1_{B_n}C_8 n^{1/2 + \gamma/6} \left(1+ g\left(\frac{\alpha_0}{\alpha_0+\theta_\alpha}\right) + g\left(\frac{\beta_0}{\beta_0+\theta_\beta}\right)\right)\\
&\hspace{1cm} + nd_{\alpha_0,\beta_0}\left(\1_{\{\theta_\alpha >R\}}+\1_{\{\theta_\beta>R\}}\right).
\end{split}
\end{align}
Next, Lemma~\ref{lemma:estimate_counts_interval} for the parameters $a=-\zeta\sqrt{n\Delta_n^\gamma}/2$ and $b=\zeta\sqrt{n\Delta_n^\gamma}$ yields for the constant $C_9=\xi\zeta/(2\max\{\alpha_0^2,\beta_0^2\})$
\[ \sup_{\theta_\rho \in(n^{1/4-\gamma/2},\zeta\sqrt{n\Delta_n^\gamma}/2)}\overline{H}_{n,\gamma}^{>}(\theta_\rho,\theta_\beta)\1_{B_n} \leq  \1_{B_n}C_9 n f\left(\frac{\beta_0}{\beta_0+\theta_\beta}\right).\]
Let $R$ be chosen as in~\eqref{eq_proof:choiceR}. By Lemma~\ref{lemma:properties_F_nC}(a) we obtain for a constant $C_{10}=C_{10}(R)$,
\begin{align}\label{eq_proof:estimate_f+g}
\begin{split}
&C_8n^{1/2+\gamma/6} g\left(\frac{\beta_0}{\beta_0+\theta_\beta}\right) + C_9nf\left(\frac{\beta_0}{\beta_0+\theta_\beta}\right) +nd_{\alpha_0,\beta_0}\1_{\{\theta_\beta>R\}}\\
&\hspace{1cm} \leq \frac{C_9\sqrt{n}}{2} F_{n,(2C_8/C_9)n^{\gamma/6}}\left(\frac{\beta_0}{\beta_0+\theta_\beta}\right) \leq C_{10} n^{\gamma/3}
\end{split}
\end{align}
and with~\eqref{eq_proof:triple_consistency_1} it follows from the decomposition~\eqref{eq:decomposition_ell_1+2} that
\begin{align*}
&\left(\ell_{n,\gamma}^{\alpha_0,\beta_0}(\theta_\rho) + \ell_{n,\gamma}^2(\theta_\rho,\theta_\alpha,\theta_\beta)\right)\1_{B_n} \\
&\hspace{1cm} \leq \1_{B_n}\left( C_2n f\left(\frac{\alpha_0}{\alpha_0+\theta_\alpha}\right) + C_8n^{1/2/+\gamma/6} g\left(\frac{\alpha_0}{\alpha_0+\theta_\alpha}\right) + (C_8+C_{10})n^{2/3} \right. \\
&\hspace{3cm} + T_{n,\gamma}(\theta_\rho,\theta_\alpha) + nd_{\alpha_0,\beta_0}\1_{\{\theta_\alpha>R\}} \bigg) \\
&\hspace{1cm} \leq \1_{B_n}\left( \frac{C_2\sqrt{n}}{2}F_{n,(2C_8/C_2)n^{\gamma/6}}\left(\frac{\alpha_0}{\alpha_0+\theta_\alpha}\right) + (C_8+C_{10})n^{2/3} + T_{n,\gamma}(\theta_\rho,\theta_\alpha) \right),
\end{align*}
where the last step uses again that $R$ was chosen as in~\eqref{eq_proof:choiceR}. We now distinguish cases, similar as we did in the study of $m=1$. We set $\kappa=|\beta_0-\alpha_0|/4$.
\begin{itemize}
\item[$(1)\ \mathbf{\theta_\alpha\in B_\kappa(\beta_0-\alpha_0)}$.]  Since $|\theta_\alpha|>\kappa$, which means $\alpha_0/(\alpha_0+\theta_\alpha)\notin B_{c(\kappa)}(1)$ for some constant $c(\kappa)>0$. Then by Lemma~\ref{lemma:properties_F_nC}(c), and $T_{n,\gamma}\leq 0$, we obtain for some constant $C_{11}=C_{11}(\kappa)$,
\begin{align*}
&\left(\ell_{n,\gamma}^{\alpha_0,\beta_0}(\theta_\rho) + \ell_{n,\gamma}^2(\theta_\rho,\theta_\alpha,\theta_\beta)\right)\1_{B_n} \leq \1_{B_n} \left( -C_{11}n + (C_8+C_{10})n^{2/3}\right),
\end{align*}
where the right-hand side is independent of $\theta_\rho,\theta_\alpha$ and $\theta_\beta$ and converges to zero for $n\to\infty$. \smallskip
\item[$(2)\ \mathbf{\theta_\alpha\notin B_\kappa(\beta_0-\alpha_0)}$.] As in (2) for $m=1$, there exists the constant $C_6=C_6(\kappa)$ such that $H_{\alpha_0,\beta_0}(\theta_\alpha)\leq -C_6$. Using Lemma~\ref{lemma:estimate_counts_interval} with $a=0$ and $b=n^{1/4-\gamma/2}$ gives
\[ \sup_{\theta_\rho >n^{1/4-\gamma/2}} \sup_{\theta_\alpha \notin B_\kappa(\beta_0-\alpha_0)} T_{n,\gamma}(\theta_\rho,\theta_\alpha)\1_{B_n} \leq -\1_{B_n}\frac{C_6\xi}{\max\{\alpha_0^2,\beta_0^2\}} n^{3/4}.\]
Repeating the estimate~\eqref{eq_proof:estimate_f+g} for $\alpha_0$ instead of $\beta_0$ and $C_2$ instead of $C_9$ then yields for another constant $C_{12}$,
\begin{align*}
&\sup_{\theta_\rho >n^{1/4-\gamma/2}}\sup_{\theta_\alpha\notin B_{\kappa(\beta_0-\alpha_0)}}\sup_{\theta_\beta}  \left(\ell_{n,\gamma}^{\alpha_0,\beta_0}(\theta_\rho) + \ell_{n,\gamma}^2(\theta_\rho,\theta_\alpha,\theta_\beta)\right)\1_{B_n} \\
&\hspace{1cm} \leq \1_{B_n}\left( (C_8+C_{10}+C_{12})n^{2/3} - \frac{C_6\xi}{\max\{\alpha_0^2,\beta_0^2\}}n^{3/4}\right),
\end{align*}
which converges to $-\infty$ for $n\to\infty$.
\end{itemize}

\medskip
\noindent
\textbf{$\underline{m=3}$}: We want to apply Lemma~\ref{lemma:bound_small_theta_beta} and denote the zwo constants on its right-hand side by $C_{13}$ and $C_{14}$. Let $r$ be small enough such that $f(x) < -2d_{\alpha_0,\beta_0}/C_{13}$ for $x\leq r$, then for $R= \alpha_0(1-r)/r$ we find
\begin{align*}
&\1_{B_n} \left(C_{13} n f\left(\frac{\alpha_0}{\alpha_0+\theta_\alpha}\right) + T_{n,\gamma}(\theta_\rho,\theta_\alpha) + C_{14} n^{1/2+\gamma/6} g\left(\frac{\alpha_0}{\alpha_0+\theta_\alpha}\right) +C_{14}n^{2/3} \right. \\
&\hspace{1.5cm} + nd_{\alpha_0,\beta_0}\1_{\{\theta_\alpha>R\}} \bigg)\\
&\hspace{0.5cm} \leq \1_{B_n}\left(\frac{C_{13}n}{2}f\left(\frac{\alpha_0}{\alpha_0+\theta_\alpha}\right) + T_{n,\gamma}(\theta_\rho,\theta_\alpha) + C_{14} n^{1/2+\gamma/6} g\left(\frac{\alpha_0}{\alpha_0+\theta_\alpha}\right) +C_{14}n^{2/3}\right) \\
&\hspace{0.5cm} = \1_{B_n}\left( \frac{C_{13}\sqrt{n}}{2} F_{n,(2C_{14}/C_{13})n^{\gamma/6}}\left(\frac{\alpha_0}{\alpha_0+\theta_\alpha}\right)+ T_{n,\gamma}(\theta_\rho,\theta_\alpha)+C_{14}n^{2/3}\right).
\end{align*}
Applying Lemma~\ref{lemma:estimate_counts_interval} to $a=0$ and $b=\zeta\sqrt{n\Delta_n^\gamma}/2$ yields for the already introduced constant $C_9=\xi\zeta/(2\max\{\alpha_0^2,\beta_0^2\})$,
\[ \sup_{\theta_\rho > \zeta\sqrt{n\Delta_n^\gamma}/2} T_{n,\gamma}(\theta_\rho,\theta_\alpha) \leq C_9 H_{\alpha_0,\beta_0}(\theta_\alpha) n.\]
Let $\kappa=|\beta_0-\alpha_0|/4$. Distinguishing cases for $\theta_\alpha$ as was already done for $m=2$ then yields $F_{n,(2C_{14}/C_{13})n^{\gamma/6}}(\alpha_0/(\alpha_0+\theta_\alpha)) \leq -C_{15}n$ in case $\theta_\alpha\in B_\kappa(\beta_0-\alpha_0)$ and $H_{\alpha_0,\beta_0}(\theta_\alpha)<-C_6$ for $\theta_\alpha\notin B_\kappa(\beta_0-\alpha_0)$ such that we have in total 
\[ \1_{B_n}\left( \frac{C_{13}\sqrt{n}}{2} F_{n,(2C_{14}/C_{13})n^{\gamma/6}}\left(\frac{\alpha_0}{\alpha_0+\theta_\alpha}\right)+ T_{n,\gamma}(\theta_\rho,\theta_\alpha)+C_{14}n^{2/3}\right) \leq -\1_{B_n}C_{16}n\]
for some constant $C_{13}>0$. Then, Lemma~\ref{lemma:bound_small_theta_beta}(i) yields
\begin{align*}
&\Pr_0\left( \sup_{\theta_\rho >\zeta\sqrt{n\Delta_n^\gamma}/2} \sup_{\theta_\alpha} \sup_{\theta_\beta\in (-\beta_0,-\beta_0+\delta_l)} \left( \ell_{n,\gamma}^{\alpha_0,\beta_0}(\theta_\rho) + \ell_{n,\gamma}^2(\theta_\rho,\theta_\alpha,\theta_\beta)\right) \geq 0\right) \\
&\hspace{0.5cm} \leq \Pr_0\left(\1_{B_n} \sup_{\theta_\rho>\zeta\sqrt{n\Delta_n}/2}\sup_{\theta_\beta\in (-\beta_0,-\beta_0+\delta_l)} A_1(\theta_\rho,\theta_\beta) > C_{16} n\right),
\end{align*}
where
\begin{align*}
A_1(\theta_\rho,\theta_\beta) &:= \sum_{k=1}^n  \1_{\{X_{(k-1)\Delta_n^\gamma}\geq\rho_0+\theta_\rho, X_{k\Delta_n^\gamma}> \rho_0+\theta_\rho\}}\\
&\hspace{1.5cm} \cdot\left[2\log\left(\frac{\beta_0}{\beta_0+\theta_\beta}\right) -\frac12\left(\frac{\beta_0^2}{(\beta_0+\theta_\beta)^2}-1\right)\frac{(X_{k\Delta_n^\gamma}-X_{(k-1)\Delta_n^\gamma})^2}{\Delta_n^\gamma\beta_0^2}\right].
\end{align*}
To bound the remaining probability we first derive a deterministic upper bound in order to get rid of both suprema and subsequently apply Markov's inequality. To this aim, define $S_j(\delta_l) = (-\beta_0+2^{-(j+1)}\delta_l, -\beta_0+2^{-j}\delta_l)$. For $x\geq 1$,
\[ 2\log(x)-\frac12\left(x^2-1\right)y \geq 0 \quad \Leftrightarrow\quad y\leq \frac{4\log(x)}{x^2-1},\]
where the upper bound is a decreasing function for $x\to\infty$ and should be read as $\infty$ in case $x=1$. Hence, we obtain for any $\delta_l<\beta_0$,
\begin{align}\label{eq_proof:choosing_delta_l}
\begin{split}
\sup_{\theta_\rho>\zeta\sqrt{n\Delta_n^\gamma}/2} \sup_{\theta_\beta\in (-\beta_0,-\beta_0+\delta_l]} \sum_{k=1}^n A_1(\theta_\rho,\theta_\beta) &\leq  \sum_{j\geq 0} \sup_{\theta_\rho} \sup_{\theta_\beta\in S_j(\delta_l)} A_1(\theta_\rho,\theta_\beta) \\
& \leq  \sum_{j\geq 0} \sum_{k=1}^n \log\left(\frac{\beta_0}{2^{-(j+1)}\delta_l}\right) \1_{A_{k,j}(\delta_l)},
\end{split}
\end{align}
where
\[ A_{k,j}(\delta_l) = \left\{ \frac{(X_{k\Delta_n^\gamma}-X_{(k-1)\Delta_n^\gamma})^2}{\Delta_n^\gamma \beta_0^2} \leq \frac{4\log(\beta_0 /(2^{-j}\delta_l))}{(\beta_0/(2^{-j}\delta_l))^2 -1}\right\}.\]
Note that this bound does not include a supremum over $\theta_\rho$ or $\theta_\beta$ any longer. As the upper bound in $A_{k,j}(\delta_l)$ is small for large values of $j$ and $(X_{k\Delta_n^\gamma}-X_{(k-1)\Delta_n^\gamma})^2/\Delta_n^\gamma $ has an expectation bounded away from zero, we expect the probability of $A_{k,j}(\delta_l)$ to be small for large $j$. Indeed, for any $a>0$, we obtain with the upper bound~\eqref{eq:bound_transition_density} on the transition density,
\begin{align}\label{eq_proof:probability_estimate_square}
\begin{split}
&\Pr_0\left( \frac{(X_{k\Delta_n^\gamma}-X_{(k-1)\Delta_n^\gamma})^2}{\Delta_n^\gamma \beta_0^2} \leq a\right) \\
&\hspace{0.5cm} = \Pr_0\left(X_{(k-1)\Delta_n^\gamma}- \sqrt{\beta_0^2\Delta_n^\gamma a} \leq X_{k\Delta_n^\gamma} \leq X_{(k-1)\Delta_n^\gamma}+ \sqrt{\beta_0^2\Delta_n^\gamma a}\right) \\
&\hspace{0.5cm} \leq C_{\alpha_0,\beta_0} \E_0\left[ \int_{X_{(k-1)\Delta_n^\gamma}- \sqrt{\beta_0^2\Delta_n^\gamma a}}^{X_{(k-1)\Delta_n^\gamma}+ \sqrt{\beta_0^2\Delta_n^\gamma a}} \frac{1}{\sqrt{\Delta_n^\gamma}} \exp\left(-\frac{(y-X_{(k-1)\Delta_n^\gamma})^2}{2\max\{\alpha_0^2,\beta_0^2\}\Delta_n^\gamma}\right) dy \right] \\
&\hspace{0.5cm} \leq C_{\alpha_0,\beta_0}\int_{- \sqrt{\beta_0^2 a}}^{\sqrt{\beta_0^2 a}}  \exp\left(-\frac{y^2}{2\max\{\alpha_0^2,\beta_0^2\}}\right) dy \leq C_{\alpha_0,\beta_0} \sqrt{a}.
\end{split}
\end{align}
By the bound~\eqref{eq_proof:choosing_delta_l}, Markov's inequality and~\eqref{eq_proof:probability_estimate_square}, we obtain for $\delta_l$ small enough
\begin{align*}
&\Pr_0\left(\1_{B_n} \sup_{\theta_\rho>\zeta\sqrt{n\Delta_n}/2}\sup_{\theta_\beta\in (-\beta_0,-\beta_0+\delta_l)} A_1(\theta_\rho,\theta_\beta) > C_{16} n\right) \\
&\hspace{1cm} \leq \frac{1}{C_{16}n}\E_0\left[  \sum_{j\geq 0} \sum_{k=1}^n \log\left(\frac{\beta_0}{2^{-(j+1)}\delta_l}\right) \1_{A_{k,j}(\delta_l)}\right] \\
&\hspace{1cm} \leq C_{\alpha_0,\beta_0}  \sum_{j\geq 0}\log\left(\frac{\beta_0}{2^{-(j+1)}\delta_l}\right) \sqrt{\frac{4\log(\beta_0 /(2^{-j}\delta_l))}{(\beta_0/(2^{-j}\delta_l))^2 -1}} \\
&\hspace{1cm} \leq C_{\alpha_0,\beta_0}  \left( \delta_l  + \delta_l \log(1/\delta_l)^{3/2}\right) \sum_{j>0}  j^{3/2} 2^{-j}.
\end{align*}
As the series $\sum_{j>0} j^{3/2}2^{-j}$ converges and $x\log(1/x)\rightarrow 0$ for $x\searrow 0$, the last expression is smaller than $\epsilon$ if $\delta_l$ is chosen small enough.

\medskip
\noindent
\textbf{$\underline{m=4}$}: Repeating the first steps of the case $m=3$ and using Lemma~\ref{lemma:bound_small_theta_beta}(ii) reveals
\begin{align*}
&\Pr_0\left( \sup_{\theta_\rho >\zeta\sqrt{n\Delta_n^\gamma}/2} \sup_{\theta_\alpha} \sup_{\theta_\beta > \delta_u-\beta_0} \left( \ell_{n,\gamma}^{\alpha_0,\beta_0}(\theta_\rho) + \ell_{n,\gamma}^2(\theta_\rho,\theta_\alpha,\theta_\beta)\right) \geq 0\right) \\
&\hspace{0.5cm} \leq \Pr_0\left(\sup_{\theta_\rho>\zeta\sqrt{n\Delta_n}/2}\sup_{\theta_\beta>\delta_u-\beta_0} A_2(\theta_\rho,\theta_\beta) \geq C_{16}n\right),
\end{align*}
where for $d_{\alpha_0,\beta_0}$ given explicitly in Lemma~\ref{lemma:estimates_tripleMLE}(d),
\begin{align*}
A_2(\theta_\rho,\theta_\beta) &= \sum_{k=1}^n  \1_{\{X_{(k-1)\Delta_n^\gamma}\geq\rho_0+\theta_\rho, X_{k\Delta_n^\gamma}> \rho_0+\theta_\rho\}}\\ &\hspace{1cm} \cdot\left[\log\left(\frac{\beta_0}{\beta_0+\theta_\beta}\right) +\frac{(X_{k\Delta_n^\gamma}-X_{(k-1)\Delta_n^\gamma})^2}{2\Delta_n^\gamma\beta_0^2} +d_{\alpha_0,\beta_0}\right].
\end{align*} 
For $\delta_u$ large enough such that $\log(\beta_0/\delta_u) +d_{\alpha_0,\beta_0} < \log(\beta_0/\delta_u)/2$, we obtain the estimates
\begin{align*}
&\sup_{\theta_\rho>\zeta\sqrt{n\Delta_n}/2}\sup_{\theta_\beta>\delta_u-\beta_0} A_2(\theta_\rho,\theta_\beta) \\
&\hspace{0.2cm} \leq \sup_{\theta_\rho>\zeta\sqrt{n\Delta_n}/2} \sum_{k=1}^n  \1_{\{X_{(k-1)\Delta_n^\gamma}\geq\rho_0+\theta_\rho, X_{k\Delta_n^\gamma}> \rho_0+\theta_\rho\}}\left[\frac12\log\left(\frac{\beta_0}{\delta_u}\right) +\frac{(X_{k\Delta_n^\gamma}-X_{(k-1)\Delta_n^\gamma})^2}{2\Delta_n^\gamma\beta_0^2}\right] \\
&\hspace{0.2cm} \leq  \sum_{k=1}^n \1_{E_k}\left[\frac12\log\left(\frac{\beta_0}{\delta_u}\right) +\frac{(X_{k\Delta_n^\gamma}-X_{(k-1)\Delta_n^\gamma})^2}{2\Delta_n^\gamma\beta_0^2}\right],
\end{align*}
where 
\[ E_k :=  \left\{ \frac{(X_{k\Delta_n^\gamma}-X_{(k-1)\Delta_n^\gamma})^2}{\Delta_n^\gamma\beta_0^2} > \log\left(\frac{\delta_u}{\beta_0}\right)\right\}.\]
Note that $E_k = \bigcup_{j=1}^\infty E_{k,j},$ where
\[ E_{k,j} = \left\{ j\log\left(\frac{\delta_u}{\beta_0}\right)< \frac{(X_{k\Delta_n^\gamma}-X_{(k-1)\Delta_n^\gamma})^2}{\Delta_n^\gamma\beta_0^2} \leq (j+1)\log\left(\frac{\delta_u}{\beta_0}\right)\right\}.\]
By Markov's inequality and Corollary~\ref{cor:increment_moments} we obtain
\begin{align*}
\Pr_0\left( E_{k,j}\right) &\leq \Pr_0\left( \frac{(X_{k\Delta_n^\gamma}-X_{(k-1)\Delta_n^\gamma})^2}{\Delta_n^\gamma\beta_0^2} \leq (j+1)\log\left(\frac{\delta_u}{\beta_0}\right)\right) \\
&\leq \left(\frac{1}{(j+1)\log(\delta_u/\beta_0)}\right)^3 \frac{1}{(\Delta_n^\gamma)^3\beta_0^6}\E_0\left[ \left( X_{k\Delta_n^\gamma}-X_{(k-1)\Delta_n^\gamma}\right)^6\right] \\
& \leq C_{\alpha_0,\beta_0} \left(\frac{1}{(j+1)\log(\delta_u/\beta_0)}\right)^3.
\end{align*}
Applying the union bound then reveals
\begin{align*}
&\Pr_0\left(\sup_{\theta_\rho>\zeta\sqrt{n\Delta_n}/2}\sup_{\theta_\beta>\delta_u-\beta_0} A_2(\theta_\rho,\theta_\beta) \geq C_{16}n\right) \\
&\hspace{0.5cm}\leq \Pr_0\left( \sum_{k=1}^n \1_{E_k}\left[\frac12\log\left(\frac{\beta_0}{\delta_u}\right) +\frac{(X_{k\Delta_n^\gamma}-X_{(k-1)\Delta_n^\gamma})^2}{2\Delta_n^\gamma\beta_0^2}\right] >C_{16}n\right) \\
&\hspace{0.5cm} \leq \sum_{j=1}^\infty \Pr_0\left( \sum_{k=1}^n \1_{E_{k,j}}\left[\log\left(\frac{\beta_0}{\delta_u}\right) +\frac{(X_{k\Delta_n^\gamma}-X_{(k-1)\Delta_n^\gamma})^2}{\Delta_n^\gamma\beta_0^2}\right] >2C_{16}n\right) \\
&\hspace{0.5cm} \leq \sum_{j=1}^\infty  \Pr_0\left( (j-1)\log\left(\frac{\delta_u}{\beta_0}\right) \sum_{k=1}^n \1_{E_{k,j}} >2C_{16}n\right) \\
&\hspace{0.5cm}\leq \sum_{j=1}^\infty \frac{j\log(\delta_u/\beta_0)}{2C_{16}n} \sum_{k=1}^n \E_0\left[\1_{E_{k,j+1}}\right] \\
&\hspace{0.5cm}\leq C_{\alpha_0,\beta_0} \sum_{j=1}^\infty \frac{j\log(\delta_u/\beta_0)}{2C_{16}n} \sum_{k=1}^n \left(\frac{1}{(j+2)\log(\delta_u/\beta_0)}\right)^3 \\
&\hspace{0.5cm} \leq  C_{\alpha_0,\beta_0} \frac{1}{2C_{16}\log(\delta_u/\beta_0)^2} \sum_{j=1}^\infty \frac{1}{j^2}.
\end{align*}
As the series converges, we can choose $\delta_u$ large enough such that the whole expression is bounded by $\epsilon$.

\medskip
\noindent
\textbf{$\underline{m=5}$}: We take $\delta_l,\delta_u$ as given from step $m=3,4$. First, note that $\overline{H}_{n,\gamma}^{>}(\theta_\rho,\theta_\beta)\leq 0$. Applying Lemma~\ref{lemma:estimates_tripleMLE}(d) with $R=\max\{\delta_u, \eqref{eq_proof:choiceR}\}$ gives the estimate~\eqref{eq_proof:R_all} on the remainder in~\eqref{eq:decomposition_ell_1+2} for which we additionally observe $g(\beta_0/(\beta_0+\theta_\beta))\leq C_{\alpha_0,\beta_0}(\delta_l,\delta_u)$. Moreover, we have~\eqref{eq_proof:triple_consistency_1} such that we obtain from decomposition~\eqref{eq:decomposition_ell_1+2} for some constant $C_{17}=C_{17}(\delta_l,\delta_u)$,
\begin{align*}
&\left(\ell_{n,\gamma}^{\alpha_0,\beta_0}(\theta_\rho) + \ell_{n,\gamma}^2(\theta_\rho,\theta_\alpha,\theta_\beta)\right)\1_{B_n} \\
&\hspace{1cm} \leq  \1_{B_n}\left( C_2nf\left(\frac{\alpha_0}{\alpha_0+\theta_\alpha}\right) + C_{17}n^{1/2+\gamma/6}g\left(\frac{\alpha_0}{\alpha_0+\theta_\alpha}\right) + C_{17}n^{2/3} \right.\\
&\hspace{3cm} + T_{n,\gamma}(\theta_\rho,\theta_\alpha)+nd_{\alpha_0,\beta_0}\1_{\{\theta_\alpha>R\}} \bigg) \\
&\hspace{1cm} \leq \1_{B_n}\left( \frac{2C_2\sqrt{n}}{2}F_{n,(2C_{17}/C_2)n^{\gamma/6}}\left(\frac{\alpha_0}{\alpha_0+\theta_\alpha}\right) + C_{17}n^{2/3} + T_{n,\gamma}(\theta_\rho,\theta_\alpha) \right),
\end{align*}
where the last step combines $C_2nf(\alpha_0/(\alpha_0+\theta_\alpha))$ and $nd_{\alpha_0,\beta_0}\1_{\{\theta_\alpha>R\}}$ as it was already done in step $m=2$. From the final expression, we then argue as in (1) and (2) in step $m=2$ and obtain that this term tends to zero as $n\to\infty$.

\subsection{Proof of Lemma~\ref{lemma:estimates_tripleMLE}}\label{Subsection:tripleMLE_2}
Let $(Y_n)_{n\in\N}$ be a sequence of random variables with $\E[|Y_n|]\leq C\psi_n$. Then by Markov's inequality, we can choose $C'$ larg enough such that
\[ \Pr\left( |Y_n| > C'\psi_n\right) \leq \frac{C\psi_n}{C'\psi_n} <\epsilon,\]
and for the set $B_n=\{|Y_n|\leq C'\psi_n\}$ we have $|Y_n|\1_{B_n}\leq C'\psi_n$ and $\Pr(B_n^c)<\epsilon$. In what follows, we will only construct moment bounds for all random variables in (a)-(d). The family of sets $(B_n)_{n\in\N}$ in the statement can then be constructed by the procedure mentioned before.

\begin{proof}[Proof of Lemma~\ref{lemma:estimates_tripleMLE}$(a)$]
First of all, we observe that
\begin{align*}
\sup_{0\leq \theta_\rho\leq Ln^{-\gamma/2}} \left| H_{n,\gamma}^{\leq}(\theta_\rho,\theta_\alpha)\right| & \leq \left[\left|\log\left(\frac{\alpha_0}{\alpha_0+\theta_\alpha}\right)\right| + \frac12\left|\frac{\alpha_0^2}{(\alpha_0+\theta_\alpha)^2}-1\right|\right] \\
&\hspace{0.7cm} \cdot \sum_{k=1}^n \left( 1+\frac{(X_{k\Delta_n^\gamma}-X_{(k-1)\Delta_n^\gamma})^2}{\alpha_0^2\Delta_n^\gamma}\right) \1_{\{\rho_0\leq X_{(k-1)\Delta_n^\gamma}\leq Ln^{-\gamma/2}\}}.
\end{align*}
By a Taylor expansion, we find for $n\geq n_0=n_0(L,\alpha_0)$ large enough that the square bracket is bounded by $C_{\alpha_0}(L)/\sqrt{n}$, uniformly in $|\theta_\alpha|\leq L/\sqrt{n}$. Moreover, by applying iterative conditioning and Corollary~\ref{cor:increment_moments} and subsequently the bound~\eqref{eq:bound_transition_density} on the transition density,
\begin{align*}
&\E_0\left[ \sum_{k=1}^n \left( 1+\frac{(X_{k\Delta_n^\gamma}-X_{(k-1)\Delta_n^\gamma})^2}{\alpha_0^2\Delta_n^\gamma}\right) \1_{\{\rho_0\leq X_{(k-1)\Delta_n^\gamma}\leq Ln^{-\gamma/2}\}} \right] \\
&\hspace{1cm} \leq C_{\alpha_0,\beta_0} \sum_{k=1}^n \E_0\left[\1_{\{\rho_0\leq X_{(k-1)\Delta_n^\gamma}\leq Ln^{-\gamma/2}\}} \right]\\
&\hspace{1cm} \leq C_{\alpha_0,\beta_0} \sum_{k=1}^n \frac{1}{\sqrt{k\Delta_n^\gamma}}\int_{\rho_0}^{\rho_0+Ln^{-\gamma/2}} \exp\left(-\frac{(y-x_0)^2}{2(k-1)\Delta_n^\gamma\max\{\alpha_0^2,\beta_0^2\}}\right) dy\\
&\hspace{1cm} \leq C_{\alpha_0,\beta_0} \sum_{k=1}^n \frac{1}{\sqrt{k}} \leq C_{\alpha_0,\beta_0} \sqrt{n}
\end{align*}
and putting together both bounds gives the claim. The estimate for $H_{n,\gamma}^{>}(\theta_\rho,\theta_\beta)$ follows analogously.
\end{proof}

\begin{proof}[Proof of Lemma~\ref{lemma:estimates_tripleMLE}$(b)$]
First, we obtain the uniform estimate
\begin{align*}
&\sup_{|\theta_\rho|\leq Ln^{-\gamma/2}}\left| H_{n,\gamma}^{\leq}(\theta_\rho,\theta_\alpha) - H_{n,\gamma}^{\leq}(0,\theta_\alpha)\right|  \leq \left|\log\left(\frac{\alpha_0}{\alpha_0+\theta_\alpha}\right)\right| S_{n,\gamma}^1 + \frac12 \left|\frac{\alpha_0^2}{(\alpha_0+\theta_\alpha)^2}-1\right| S_{n,\gamma}^2,
\end{align*}
with 
\begin{align*}
S_{n,\gamma}^1 &:= \sum_{k=1}^n \left(\1_{\{X_{(k-1)\Delta_n^\gamma},X_{k\Delta_n^\gamma}> \rho_0\}} - \1_{\{X_{(k-1)\Delta_n^\gamma},X_{k\Delta_n^\gamma}> \rho_0+Ln^{-\gamma/2}\}}\right), \\
S_{n,\gamma}^2 &:= \sum_{k=1}^n \frac{(X_{k\Delta_n^\gamma}-X_{(k-1)\Delta_n^\gamma})^2}{\alpha_0^2\Delta_n^\gamma}\left(\1_{\{X_{(k-1)\Delta_n^\gamma},X_{k\Delta_n^\gamma}> \rho_0\}} - \1_{\{X_{(k-1)\Delta_n^\gamma},X_{k\Delta_n^\gamma}> \rho_0+Ln^{-\gamma/2}\}}\right).
\end{align*} 
Note that $S_{n,\gamma}^j\geq 0$ for $j=1,2$. The claim then follows once we prove the moment bound $\E_0[S_{n,\gamma}^j] \leq C_{\alpha_0,\beta_0}L\sqrt{n}$ for $j=1,2$. To this aim, we first bound
\begin{align}\label{eq_proof:H1(b)}
\begin{split}
&\left(\1_{\{X_{(k-1)\Delta_n^\gamma},X_{k\Delta_n^\gamma}> \rho_0\}} - \1_{\{X_{(k-1)\Delta_n^\gamma},X_{k\Delta_n^\gamma}> \rho_0+Ln^{-\gamma/2}\}}\right) \\
&\hspace{1cm} \leq \1_{\{\rho_0 \leq X_{(k-1)\Delta_n^\gamma} \leq \rho_0+Ln^{-\gamma/2}\}} + \1_{\{\rho_0 \leq X_{k\Delta_n^\gamma} \leq \rho_0+Ln^{-\gamma/2}\}}\1_{\{X_{(k-1)\Delta_n^\gamma}\geq \rho_0\}}.
\end{split}
\end{align}
Replacing $K\psi_n^\gamma$ with $Ln^{-\gamma}$ in~\eqref{eq_proof:bound_expectation_indicator_T1} gives
\begin{align*}
&\E_0\left[\1_{\{\rho_0 \leq X_{(k-1)\Delta_n^\gamma} \leq \rho_0+Ln^{-\gamma/2}\}} \right] \leq  \frac{C_{\alpha_0,\beta_0}L}{\sqrt{k}}.
\end{align*}
The same bound holds true for $X_{(k-1)\Delta_n^\gamma}$ replaced by $X_{k\Delta_n^\gamma}$ such that in total we obtain
\[ \E_0\left[ S_{n,\gamma}^1\right] \leq C_{\alpha_0,\beta_0}L \sum_{k=1}^n \frac{1}{\sqrt{k}} \leq C_{\alpha_0,\beta_0}L\sqrt{n}.\]
To find the required moment bound for $S_{n,\gamma}^2$, one first applies~\eqref{eq_proof:H1(b)}. Replacing $K\psi_n^\gamma$ with $Ln^{-\gamma/2}$ in the estimates~\eqref{eq_proof:T12_bound_squared_increment} and~\eqref{eq_proof:T12_bound_squared_increment_indicator} then gives
\[ \E_0\left[ S_{n,\gamma}^2\right] \leq C_{\alpha_0,\beta_0}L \sum_{k=1}^n \frac{1}{\sqrt{k}} \leq C_{\alpha_0,\beta_0}K\sqrt{n}\]
and the first bound in Lemma~\ref{lemma:estimates_tripleMLE}(b) follows. The second one can be shown analogously.
\end{proof}

\begin{proof}[Proof of Lemma~\ref{lemma:estimates_tripleMLE}$(c)$]
The claim follows, once we establish
\begin{align}\label{eq_proof:Lemma_Hc_1}
\left|  \sum_{k=1}^n \left(\frac{(X_{k\Delta_n^\gamma} - X_{(k-1)\Delta_n^\gamma})^2}{\Delta_n^\gamma} - \alpha_0^2\right) \1_{\{X_{(k-1)\Delta_n^\gamma}\leq \rho_0\}} \right| =\mathcal{O}_{\Pr_0}\left(\sqrt{n}\right)
\end{align} 
and
\begin{align}\label{eq_proof:Lemma_Hc_2}
\left|  \sum_{k=1}^n \left(\frac{(X_{k\Delta_n^\gamma} - X_{(k-1)\Delta_n^\gamma})^2}{\Delta_n^\gamma} - \beta_0^2\right) \1_{\{X_{(k-1)\Delta_n^\gamma}\geq \rho_0\}} \right| =\mathcal{O}_{\Pr_0}\left(\sqrt{n}\right)
\end{align} 
In what follows, we only give a detailed proof for~\eqref{eq_proof:Lemma_Hc_1} as~\eqref{eq_proof:Lemma_Hc_2} follows analogously. By the triangle inequality,
\begin{align*}
\left|  \sum_{k=1}^n \left(\frac{(X_{k\Delta_n^\gamma} - X_{(k-1)\Delta_n^\gamma})^2}{\Delta_n^\gamma} - \alpha_0^2\right) \1_{\{X_{(k-1)\Delta_n^\gamma}\leq \rho_0\}} \right| \leq T_1 + T_2,
\end{align*}
where
\begin{align*}
T_1 &:= \left|  \sum_{k=1}^n \left(\frac{(X_{k\Delta_n^\gamma} - X_{(k-1)\Delta_n^\gamma})^2}{\Delta_n^\gamma} - \E_0\left[\left. \frac{(X_{k\Delta_n^\gamma} - X_{(k-1)\Delta_n^\gamma})^2}{\Delta_n^\gamma}\right| X_{(k-1)\Delta_n^\gamma}\right]\right) \1_{\{X_{(k-1)\Delta_n^\gamma}\leq \rho_0\}} \right|\\
T_2 &:=\left|  \sum_{k=1}^n \left(\E_0\left[\left. \frac{(X_{k\Delta_n^\gamma} - X_{(k-1)\Delta_n^\gamma})^2}{\Delta_n^\gamma} \right| X_{(k-1)\Delta_n^\gamma}\right]  - \alpha_0^2\right) \1_{\{X_{(k-1)\Delta_n^\gamma}\leq \rho_0\}} \right|.
\end{align*}
In what follows, we establish the desired stochastic order for each summand separately. For $T_1$ we first note that the sum inside the absolute value is a sum over martingale increments. Then using Markov's and Jensen's inequality,
\begin{align*}
&\Pr_0\left( T_1 >C\sqrt{n}\right) \\
&\hspace{0.2cm}\leq \frac{1}{C\sqrt{n}}\sqrt{\E_0[T_1^2]} \\
&\hspace{0.2cm}= \frac{1}{C\sqrt{n}} \left( \sum_{k=1}^n \E_0\left[\left(\frac{ (X_{k\Delta_n^\gamma} - X_{(k-1)\Delta_n^\gamma})^2}{\Delta_n^\gamma} - \E_0\left[\left. \frac{ (X_{k\Delta_n^\gamma} - X_{(k-1)\Delta_n^\gamma})^2}{\Delta_n^\gamma}\right| X_{(k-1)/n}\right]\right)^2 \right.\right. \\
&\hspace{2cm} \cdot\1_{\{X_{(k-1)\Delta_n^\gamma}\leq \rho_0\}} \Big] \bigg)^\frac12 \\
&\hspace{0.2cm} \leq \frac{1}{C\sqrt{n}} \left( \sum_{k=1}^n  \E_0\left[ \frac{(X_{k\Delta_n^\gamma} - X_{(k-1)\Delta_n^\gamma})^4}{(\Delta_n^\gamma)^2}\right] \right)^\frac12 \leq \frac{C_{\alpha_0,\beta_0}}{C} \longrightarrow 0
\end{align*}
for $C\to\infty$, where the last step uses Corollary~\ref{cor:increment_moments}. This proves that $T_1=\mathcal{O}_{\Pr_0}(\sqrt{n})$ and finishes the first part.
For the analysis of $T_2$, we note $\Pr_0(X_{(k-1)\Delta_n\gamma} =\rho_0)=0$ and only work on $B_k := \{X_{(k-1)\Delta_n^\gamma}<\rho_0\}$. Decomposing $1=\1_{\{X_{k\Delta_n^\gamma}\leq \rho_0\}}+\1_{\{X_{k\Delta_n^\gamma}>\rho_0\}}$ reveals by the triangle inequality
\begin{align}\label{eq_proof:E_squared_difference_1}
\begin{split}
&\left|\E_0\left[ \left. \frac{(X_{k\Delta_n^\gamma}-X_{(k-1)\Delta_n^\gamma})^2}{\Delta_n^\gamma}\right| X_{(k-1)\Delta_n^\gamma}\right]-\alpha_0^2\right|  \1_{B_k}  \\
&\hspace{1cm} \leq  \left|\E_0\left[ \left.\left(\frac{(X_{k\Delta_n^\gamma}-X_{(k-1)\Delta_n^\gamma})^2}{\Delta_n^\gamma}-\alpha_0^2 \right) \1_{B_k} \1_{\{X_{k\Delta_n^\gamma}\leq \rho_0\}} \right| X_{(k-1)\Delta_n^\gamma}\right] \right| \\
&\hspace{2cm} + \left|\E_0\left[ \left.\left(\frac{(X_{k\Delta_n^\gamma}-X_{(k-1)\Delta_n^\gamma})^2}{\Delta_n^\gamma} -\alpha_0^2\right) \1_{B_k} \1_{\{X_{k\Delta_n^\gamma}> \rho_0\}} \right| X_{(k-1)\Delta_n^\gamma}\right] \right|.
\end{split}
\end{align}
Subsequently, both summands will be bounded separately. The second expectation on the right-hand side of~\eqref{eq_proof:E_squared_difference_1} can be rewritten as
\begin{align*}
&\left| \E_0\left[ \left.\left(\frac{(X_{k\Delta_n^\gamma}-X_{(k-1)\Delta_n^\gamma})^2}{\Delta_n^\gamma} -\alpha_0^2\right) \1_{B_k}\1_{\{X_{k\Delta_n^\gamma}>\rho_0\}} \right| X_{(k-1)\Delta_n^\gamma}\right]\right| \\
&\hspace{0.1cm} = \1_{B_k}\left| \int_{\rho_0}^\infty \left(\frac{(y-X_{(k-1)\Delta_n^\gamma})^2}{\Delta_n^\gamma} -\alpha_0^2\right) \frac{1}{\sqrt{2\pi\Delta_n^\gamma}} \frac{2}{\alpha_0+\beta_0}\frac{\alpha_0}{\beta_0} \right.\\
&\hspace{3.5cm} \left. \cdot\exp\left(-\frac{1}{2\Delta_n^\gamma}\left(\frac{y-\rho_0}{\beta_0}-\frac{X_{(k-1)\Delta_n^\gamma}-\rho_0}{\alpha_0}\right)^2\right) dy \right| \\
&\hspace{0.1cm} =  \1_{B_k} \frac{2\alpha_0\beta_0^2}{\alpha_0+\beta_0} \int_{-(X_{(k-1)\Delta_n^\gamma}-\rho_0)/\sqrt{\alpha_0^2\Delta_n^\gamma}}^\infty \left[\bigg( y - \frac{(\alpha_0-\beta_0)(X_{(k-1)\Delta_n^\gamma}-\rho_0)}{\alpha_0\beta_0\sqrt{\Delta_n^\gamma}}\bigg)^2 \hspace{-0.1cm}-\frac{\alpha_0^2}{\beta_0^2}\right] \varphi(y) dy.
\end{align*}
where $\varphi(y)=\frac{1}{\sqrt{2\pi}}\exp(-y^2/2)$ is the Gaussian density. Now, using $x\mapsto (x^2+1)\exp(-x^2/4)$ is bounded together with a Gaussian-type tail inequality, we obtain
\begin{align*}
&\1_{B_k}\left| \int_{-(X_{(k-1)\Delta_n^\gamma}-\rho_0)/\sqrt{\alpha_0^2\Delta_n^\gamma}}^\infty \left(\left( y - \frac{(\alpha_0-\beta_0)(X_{(k-1)\Delta_n^\gamma}-\rho_0)}{\alpha_0\beta_0\sqrt{\Delta_n^\gamma}}\right)^2 - \frac{\alpha_0^2}{\beta_0^2}\right) \varphi(y) dy\right| \\
&\hspace{0.2cm} \leq C_{\alpha_0,\beta_0}\1_{B_k} \int_{-(X_{(k-1)\Delta_n^\gamma}-\rho_0)/\sqrt{\alpha_0^2\Delta_n^\gamma}}^\infty \left(y^2 +1\right) \exp\left(-\frac{y^2}{2}\right) dy  \\
&\hspace{1.2cm} +C_{\alpha_0,\beta_0}\1_{B_k} \frac{(X_{(k-1)/n}-\rho_0)^2}{\Delta_n^\gamma} \int_{-(X_{(k-1)\Delta_n^\gamma}-\rho_0)/\sqrt{\alpha_0^2\Delta_n^\gamma}}^\infty \exp\left(-\frac{y^2}{2}\right) dy \\
&\hspace{0.2cm} \leq C_{\alpha_0,\beta_0}\1_{B_k}  \int_{-(X_{(k-1)\Delta_n^\gamma}-\rho_0)/\sqrt{\alpha_0^2\Delta_n^\gamma}}^\infty \exp\left(-\frac{y^2}{4}\right) dy \\
&\hspace{1.2cm}  + C_{\alpha_0,\beta_0}\1_{B_k}\frac{(X_{(k-1)/n}-\rho_0)^2}{\Delta_n^\gamma} \exp\left(-\frac{(X_{(k-1)\Delta_n^\gamma}-\rho_0)^2}{2\alpha_0^2\Delta_n^\gamma}\right)\\
&\hspace{0.2cm} \leq C_{\alpha_0,\beta_0} \exp\left(-\frac{(X_{(k-1)\Delta_n^\gamma}-\rho_0)^2}{4\alpha_0^2\Delta_n^\gamma}\right).
\end{align*}
This bound will turn out to be sufficient for later purposes. For the first expectation on the right-hand side of~\eqref{eq_proof:E_squared_difference_1}, we find
\begin{align*}
& \left| \E_0\left[ \left.\left(\frac{(X_{k\Delta_n^\gamma}-X_{(k-1)\Delta_n^\gamma})^2}{\Delta_n^\gamma} -\alpha_0^2\right) \1_{B_k} \1_{\{X_{k\Delta_n^\gamma}\leq \rho_0\}} \right| X_{(k-1)\Delta_n^\gamma}\right] \right|\\
&\hspace{0.2cm} \leq  \1_{B_k}\left| \int_{-\infty}^{\rho_0} \left(\frac{(y-X_{(k-1)\Delta_n^\gamma})^2}{\Delta_n^\gamma} -\alpha_0^2\right) \frac{1}{\sqrt{2\pi\alpha_0^2\Delta_n^\gamma}}\exp\left(-\frac{(y-X_{(k-1)\Delta_n^\gamma})^2}{2\alpha_0^2\Delta_n^\gamma}\right) dy  \right|\\
&\hspace{0.5cm} +  \1_{B_k}\left| \int_{-\infty}^{\rho_0} \left(\frac{(y-X_{(k-1)\Delta_n^\gamma})^2}{\Delta_n^\gamma}-\alpha_0^2\right) \frac{1}{\sqrt{2\pi\alpha_0^2\Delta_n^\gamma}}\frac{\alpha_0-\beta_0}{\alpha_0+\beta_0} \right.\\
&\hspace{5.5cm} \left. \cdot\exp\left(-\frac{(y-2\rho_0+X_{(k-1)\Delta_n^\gamma})^2}{2\alpha_0^2\Delta_n^\gamma}\right) dy  \right|\\
&\hspace{0.2cm} = \1_{B_k}\alpha_0^2 \left|\int_{-\infty}^{(\rho_0-X_{(k-1)\Delta_n^\gamma})/\sqrt{\alpha_0^2\Delta_n^\gamma}} \left(y^2-1\right) \varphi(y) dy \right| \\
&\hspace{0.5cm} +  \1_{B_k} \alpha_0^2 \left| \frac{\alpha_0-\beta_0}{\alpha_0+\beta_0} \int_{-\infty}^{(X_{(k-1)\Delta_n^\gamma}-\rho_0)/\sqrt{\alpha_0^2\Delta_n^\gamma}} \left(\left(y-\frac{2(X_{(k-1)\Delta_n^\gamma}-\rho_0)}{\alpha_0\sqrt{\Delta_n^\gamma}}\right)^2-1\right) \varphi(y) dy  \right|,
\end{align*}
where again $\varphi$ denotes the Gaussian density. As $\int_\R y^2 \varphi(y) dy = \int_\R \varphi(y) dy$, we have
\begin{align*}
&\int_{-\infty}^{(\rho_0-X_{(k-1)\Delta_n^\gamma})/\sqrt{\alpha_0^2\Delta_n^\gamma}} \left(y^2-1\right)\varphi(y) dy  = -\int_{(\rho_0-X_{(k-1)\Delta_n^\gamma})/\sqrt{\beta_0^2/n}}^\infty \left(y^2-1\right) \varphi(y) dy .
\end{align*}
Then, as $\rho_0-X_{(k-1)\Delta_n^\gamma}\geq 0$, we can proceed as for the second summand in~\eqref{eq_proof:E_squared_difference_1}. Finally, 
\[ \left|\E_0\left[ \left. \frac{(X_{k\Delta_n^\gamma}-X_{(k-1)\Delta_n^\gamma})^2}{\Delta_n^\gamma} \right| X_{(k-1)\Delta_n^\gamma}\right]-\beta_0^2\right|  \1_{B_k}  \leq C_{\alpha_0,\beta_0} \exp\left(-\frac{(X_{(k-1)\Delta_n^\gamma}-\rho_0)^2}{4\alpha_0^2\Delta_n^\gamma}\right).\]
By Corollary~\ref{cor:bound_exp_X^2} we then obtain
\[ \sum_{k=1}^n \E_0\left[ \exp\left(-\frac{(X_{(k-1)\Delta_n^\gamma}-\rho_0)^2}{4\alpha_0^2\Delta_n^\gamma}\right)\right] \leq C_{\alpha_0,\beta_0}\sum_{k=1}^n \frac{1}{\sqrt{k}} \leq C_{\alpha_0,\beta_0}\sqrt{n} \]
which yields $T_2=\mathcal{O}_{\Pr_0}(1)$.
\end{proof}

\begin{proof}[Proof of Lemma~\ref{lemma:estimates_tripleMLE}$(d)$]
We prove the statement for $R_{n,\gamma}^{\leq}(\theta_\rho,\theta_\alpha,\theta_\beta)$, the other one follows analogously. First of all, we obtain the estimate
\[ R_{n,\gamma}^{\leq}(\theta_\rho,\theta_\alpha,\theta_\beta) = \sum_{j=1}^4 R_{n,\gamma}^{\leq, j}(\theta_\rho,\theta_\alpha,\theta_\beta),\]
where with the abbreviation $D_k=(X_{k\Delta_n^\gamma}-\rho_0-\theta_\rho)(X_{(k-1)\Delta_n^\gamma}-\rho_0-\theta_\rho)$ and suppression of the argument $(\theta_\rho,\theta_\alpha,\theta_\beta)$ for shorter displays,
\begin{align*}
R_{n,\gamma}^{\leq, 1}&:=  \sum_{k=1}^n \left[\log\left(\frac{\alpha_0}{\alpha_0+\theta_\alpha}\right) -\frac12\left(\frac{\alpha_0^2}{(\alpha_0+\theta_\alpha)^2}-1\right)\frac{(X_{k\Delta_n^\gamma}-X_{(k-1)\Delta_n^\gamma})^2}{\Delta_n^\gamma\alpha_0^2}\right] \\
&\hspace{1.2cm} \cdot\left( \1_{\{X_{(k-1)\Delta_n^\gamma}<\rho_0+\theta_\rho, X_{k\Delta_n^\gamma}\leq \rho_0+\theta_\rho\}}- \1_{\{X_{(k-1)\Delta_n^\gamma} \leq \rho_0+\theta_\rho\}} \right) \\
R_{n,\gamma}^{\leq, 2}&:=  \sum_{k=1}^n \log\left(\frac{1-\frac{\alpha_0+\theta_\alpha - (\beta_0+\theta_\beta)}{\alpha_0+\theta_\alpha+\beta_0+\theta_\beta}\exp\left(-\frac{2}{(\alpha_0+\theta_\alpha)^2\Delta_n^\gamma}D_k\right)}{1-\frac{\alpha_0 - \beta_0}{\alpha_0+\beta_0}\exp\left(-\frac{2}{\alpha_0^2\Delta_n^\gamma}D_k\right)}\right)\1_{\{X_{(k-1)\Delta_n^\gamma}<\rho_0+\theta_\rho, X_{k\Delta_n^\gamma}\leq \rho_0+\theta_\rho\}} \\
R_{n,\gamma}^{\leq, 3}&:=  \sum_{k=1}^n \1_{\{ X_{(k-1)\Delta_n^\gamma}<\rho_0+\theta<X_{k\Delta_n^\gamma}\}} \log\left(\frac{\alpha_0+\beta_0}{\alpha_0+\theta_\alpha+\beta_0+\theta_\beta}\frac{\alpha_0+\theta_\alpha}{\alpha_0}\frac{\beta_0}{\beta_0+\theta_\beta}\right)\\
R_{n,\gamma}^{\leq, 4} &:=  \sum_{k=1}^n \1_{\{ X_{(k-1)\Delta_n^\gamma}<\rho_0+\theta<X_{k\Delta_n^\gamma}\}} \frac{1}{2\Delta_n^\gamma} \left[ \left(\frac{X_{k\Delta_n^\gamma}-\rho_0-\theta_\rho}{\beta_0} -\frac{X_{(k-1)\Delta_n^\gamma}-\rho_0-\theta_\rho}{\alpha_0} \right)^2 \right. \\
&\hspace{4cm} \left. -\left(\frac{X_{k\Delta_n^\gamma}-\rho_0-\theta_\rho}{\beta_0+\theta_\beta} -\frac{X_{(k-1)\Delta_n^\gamma}-\rho_0-\theta_\rho}{\alpha_0+\theta_\alpha} \right)^2 \right].
\end{align*}
In what follows, we will prove
\begin{align}\label{eq_proof:triple_estimates_d}
R_{n,\gamma}^{\leq,j}(\theta_\rho,\theta_\alpha,\theta_\beta) \leq C_{\alpha_0,\beta_0}\left( g\left(\frac{\alpha_0}{\alpha_0+\theta_\alpha}\right) + g\left(\frac{\beta_0}{\beta_0+\theta_\beta}\right)\right) \Upsilon_{n,\gamma}^{\leq, j}(\theta_\rho),
\end{align} 
for $j=1,3,4$ with
\begin{align}\label{eq_proof:Upsilon}
\Upsilon_{n,\gamma}^{\leq, j}(\theta_\rho) = \sum_{k=1}^n \left( 1+ \frac{(X_{k\Delta_n^\gamma}-X_{(k-1)\Delta_n^\gamma})^2}{\Delta_n^\gamma}\right)\1_{\{X_{(k-1)\Delta_n^\gamma}\leq \rho_0+\theta_\rho <X_{k\Delta_n^\gamma}\}}.
\end{align} 
In case $j=2$ we obtain for $\theta_\alpha\leq R$ with $R>0$ the bound~\eqref{eq_proof:triple_estimates_d} with
\begin{align*}
\Upsilon_{n,\gamma}^{\leq, 2}(\theta_\rho) &= \sum_{k=1}^n \bigg( \1_{\{\rho_0+\theta_\rho-\sqrt{\Delta_n^\gamma} \leq X_{k\Delta_n^\gamma}\leq \rho_0+\theta_\rho\}} \\
&\hspace{1.5cm} \left.+ \exp\left(\frac{1}{(\alpha_0+R)^2\sqrt{\Delta_n^\gamma}}\left(X_{(k-1)\Delta_n^\gamma}-\rho-\theta_\rho\right)\right)\1_{\{X_{(k-1)\Delta_n^\gamma}\leq \rho_0+\theta_\rho\}}\right)
\end{align*} 
and the constant then additionally depends on $R$. In case $\theta_\alpha>R$, we obtain 
\[ R_{n,\gamma}^{\leq}(\theta_\rho,\theta_\alpha,\theta_\beta) \leq nd_{\alpha_0,\beta_0}. \]
The required stochastic order of $\Upsilon_{n,\gamma}^{\leq,j}$ for $j=1,3,4$ can then be derived analogously as for the remainder terms $\Xi_n^{(5,4)}$ and $\tilde{\Xi}_n^{(5,4)}$ in the proof of Lemma~$4.1$ in~\citeSM{App:Brutsche/Rohde} and that for $j=2$ following the lines of $\Xi_n^{(5,1)},\tilde{\Xi}_n^{(5,1)}, \Xi_n^{(5,3)}$ and $\tilde{\Xi}_n^{(5,3)}$ in the same reference, both times working on the set $A_n$ given in~\eqref{eq_proof:setAn}. We omit the details.
\begin{itemize}
\item[$\mathbf{\bullet}\ R_{n,\gamma}^{\leq,1}$.] First, we note that
\[ \left( \1_{\{X_{(k-1)\Delta_n^\gamma}<\rho_0+\theta_\rho, X_{k\Delta_n^\gamma}\leq \rho_0+\theta_\rho\}}- \1_{\{X_{(k-1)\Delta_n^\gamma} \leq \rho_0+\theta_\rho\}} \right) = -\1_{\{ X_{(k-1)\Delta_n^\gamma}\leq \rho_0+\theta_\rho<X_{k\Delta_n^\gamma}\}}.\]
This entails
\begin{align*}
R_{n,\gamma}^{\leq, 1}(\theta_\rho,\theta_\alpha,\theta_\beta) &\leq \left(\left|\log\left(\frac{\alpha_0}{\alpha_0+\theta_\alpha}\right) \right| + \frac12 \left| \frac{\alpha_0^2}{(\alpha_0+\theta_\alpha)^2}-1\right| \right) \\
&\hspace{1cm} \cdot \sum_{k=1}^n \left( 1+ \frac{(X_{k\Delta_n^\gamma}-X_{(k-1)\Delta_n^\gamma})^2}{\Delta_n^\gamma \alpha_0^2}\right)\1_{\{ X_{(k-1)\Delta_n^\gamma}\leq \rho_0+\theta_\rho<X_{k\Delta_n^\gamma}\}}.
\end{align*}

\smallskip
\item[$\mathbf{\bullet}\ R_{n,\gamma}^{\leq,2}$.] Recall the abbreviation
\[ D_k = (X_{k\Delta_n^\gamma}-\rho_0-\theta_\rho)(X_{(k-1)\Delta_n^\gamma}-\rho_0-\theta_\rho).\]
Using that on the event $\{X_{(k-1)\Delta_n^\gamma}<\rho_0+\theta_\rho, X_{k\Delta_n^\gamma}\leq \rho_0+\theta_\rho\}$
\[ 1-\frac{|\alpha_0-\beta_0|}{\alpha_0+\beta_0} \leq 1-\frac{\alpha_0 - \beta_0}{\alpha_0+\beta_0}\exp\left(-\frac{2}{\alpha_0^2\Delta_n^\gamma}D_k\right) \leq 1+\frac{|\alpha_0-\beta_0|}{\alpha_0+\beta_0}\]
and the inequality $\log(1+x)\leq x$ for $x>-1$ yields
\[ R_{n,\gamma}^{\leq, 2}(\theta_\rho,\theta_\alpha,\theta_\beta) \leq C_{\alpha_0,\beta_0} \left( R_{n,\gamma}^{\leq, 2,1}(\theta_\rho,\theta_\alpha,\theta_\beta) + R_{n,\gamma}^{\leq, 2,2}(\theta_\rho,\theta_\alpha,\theta_\beta)\right),\]
where
\begin{align}\label{eq_proof:d_R21}
\begin{split}
R_{n,\gamma}^{\leq, 2,1}(\theta_\rho,\theta_\alpha,\theta_\beta) &:= \left|\frac{\alpha_0+\theta_\alpha - (\beta_0+\theta_\beta)}{\alpha_0+\theta_\alpha+\beta_0+\theta_\beta}-\frac{\alpha_0-\beta_0}{\alpha_0+\beta_0}\right|\\
&\hspace{1cm} \cdot \sum_{k=1}^n \1_{\{X_{(k-1)\Delta_n^\gamma}<\rho_0+\theta_\rho, X_{k\Delta_n^\gamma}\leq \rho_0+\theta_\rho\}}\exp\left(-\frac{2}{\alpha_0^2\Delta_n^\gamma}D_k\right) \\
R_{n,\gamma}^{\leq, 2,2}(\theta_\rho,\theta_\alpha,\theta_\beta) &:= \left|\frac{\alpha_0+\theta_\alpha - (\beta_0+\theta_\beta)}{\alpha_0+\theta_\alpha+\beta_0+\theta_\beta}\right|\sum_{k=1}^n \1_{\{X_{(k-1)\Delta_n^\gamma}<\rho_0+\theta_\rho, X_{k\Delta_n^\gamma}\leq \rho_0+\theta_\rho\}} \\
&\hspace{1cm} \cdot \left|\exp\left(-\frac{2}{\alpha_0^2\Delta_n^\gamma}D_k\right)  -\exp\left(-\frac{2}{(\alpha_0+\theta_\alpha)^2\Delta_n^\gamma}D_k\right) \right|.
\end{split}
\end{align}
We are going to treat these terms separately. For the first one, a direct evaluation in the first step reveals
\begin{align*}
&\left|\frac{\alpha_0+\theta_\alpha - (\beta_0+\theta_\beta)}{\alpha_0+\theta_\alpha+\beta_0+\theta_\beta}-\frac{\alpha_0-\beta_0}{\alpha_0+\beta_0}\right|\\
&\hspace{1cm} = 2\left| \frac{\beta_0\theta_\alpha}{(\alpha_0+\theta_\alpha+\beta_0+\theta_\beta)(\alpha_0+\beta_0)} - \frac{\alpha_0\theta_\beta}{(\alpha_0+\theta_\alpha+\beta_0+\theta_\beta)(\alpha_0+\beta_0)} \right| \\
&\hspace{1cm} \leq \frac{2\beta_0 |\theta_\alpha|}{(\alpha_0+\theta_\alpha)(\alpha_0+\beta_0)} + \frac{2\alpha_0 |\theta_\beta|}{(\beta_0+\theta_\beta)(\alpha_0+\beta_0)} \\
&\hspace{1cm} \leq \frac{2|\theta_\alpha|}{\alpha_0+\theta_\alpha} + \frac{2|\theta_\beta|}{\beta_0+\theta_\beta}.
\end{align*}
Using that $|x-1|\leq |\log(x)|$ for $0<x\leq 1$ and $|x-1|\leq (x^2-1)/2$ for $x>1$, and recalling $g(x)=|\log(x)|+|x^2-1|/2$, we thus find
\begin{align}\label{eq_proof:R_21_d}
\begin{split}
&R_{n,\gamma}^{\leq, 2,1}(\theta_\rho,\theta_\alpha,\theta_\beta)\\
&\hspace{0.5cm}\leq 2\left( g\left(\frac{\alpha_0}{\alpha_0+\theta_\alpha}\right) + g\left(\frac{\beta_0}{\beta_0+\theta_\beta}\right)\right)\sum_{k=1}^n \1_{\{X_{(k-1)\Delta_n^\gamma}<\rho_0+\theta_\rho, X_{k\Delta_n^\gamma}\leq \rho_0+\theta_\rho\}} \\
&\hspace{2cm} \cdot  \exp\left(-\frac{2}{\alpha_0^2\Delta_n^\gamma}(X_{k\Delta_n^\gamma}-\rho_0-\theta_\rho)(X_{(k-1)\Delta_n^\gamma}-\rho_0-\theta_\rho)\right) \\
&\hspace{0.5cm}\leq 2\left( g\left(\frac{\alpha_0}{\alpha_0+\theta_\alpha}\right) + g\left(\frac{\beta_0}{\beta_0+\theta_\beta}\right)\right)\left[ \sum_{k=1}^n \1_{\{\rho_0+\theta_\rho-\sqrt{\Delta_n^\gamma} \leq X_{k\Delta_n^\gamma}\leq \rho_0+\theta_\rho\}} \right. \\
&\hspace{2cm} + \left. \sum_{k=1}^n \exp\left(\frac{2}{\alpha_0^2\sqrt{\Delta_n^\gamma}}(X_{(k-1)\Delta_n^\gamma}-\rho_0-\theta_\rho)\right)\1_{\{X_{(k-1)\Delta_n^\gamma} \leq \rho_0+\theta_\rho\}}\right].
\end{split}
\end{align}
For the discussion of $R_n^{\leq,2,2}$, let $R>0$ and assume that $\theta_\alpha\leq R<\infty$.
Then, we have for an intermediate value $\xi_k(n,\theta_\alpha)$ between $\alpha_0^{-2}$ and $(\alpha_0+\theta_\alpha)^{-2}$ which are both lower bounded by $(\alpha_0+R)^{-2}$,
\begin{align*}
&\left| \exp\left(-\frac{2}{\alpha_0^2\Delta_n^\gamma} D_k\right) - \exp\left(-\frac{2}{(\alpha_0+\theta_\alpha)^2\Delta_n^\gamma} D_k\right) \right| \\
&\hspace{1cm} = \left|\frac{1}{(\alpha_0+\theta_\alpha)^2}-\frac{1}{\alpha_0^2}\right| \frac{2 D_k}{\Delta_n^\gamma} \exp\left(-\frac{2\xi_k(n,\theta_\alpha)}{\Delta_n^\gamma} D_k\right) \\
&\hspace{1cm} \leq \frac{1}{\alpha_0^2}\left|\frac{\alpha_0^2}{(\alpha_0+\theta_\alpha)^2}-1\right| \frac{2 D_k}{\Delta_n^\gamma} \exp\left(-\frac{2}{(\alpha_0+R)^2\Delta_n^\gamma} D_k\right) \\
&\hspace{1cm} \leq C_{\alpha_0}(R)\left|\frac{\alpha_0^2}{(\alpha_0+\theta_\alpha)^2}-1\right|\exp\left(-\frac{1}{(\alpha_0+R)^2\Delta_n^\gamma} D_k\right),
\end{align*}
where the last step uses that $x\mapsto x\exp(-(\alpha_0+R)^{-2}x/2)$ is bounded. The required bound can then be obtained as in the second step of~\eqref{eq_proof:R_21_d}.

In case $\theta_\alpha>R$, we directly bound the logarithm in the definition of $R_{n,\gamma}^{\leq,2}$ using its monotonicity. On the event $\{X_{(k-1)\Delta_n^\gamma}<\rho_0+\theta_\rho, X_{k\Delta_n^\gamma}\leq \rho_0+\theta_\rho\}$ we have $D_k\geq 0$ and thus find on this event
\begin{align*}
\log\left(\frac{1-\frac{\alpha_0+\theta_\alpha - (\beta_0+\theta_\beta)}{\alpha_0+\theta_\alpha+\beta_0+\theta_\beta}\exp\left(-\frac{2}{(\alpha_0+\theta_\alpha)^2\Delta_n^\gamma}D_k\right)}{1-\frac{\alpha_0 - \beta_0}{\alpha_0+\beta_0}\exp\left(-\frac{2}{\alpha_0^2\Delta_n^\gamma}(X_{k\Delta_n^\gamma}-\rho_0-\theta_\rho)D_k\right)} \right) \leq \log\left( \frac{2}{1-\frac{|\alpha_0-\beta_0|}{\alpha_0+\beta_0}} \right) =d_{\alpha_0,\beta_0},
\end{align*}
which summed over $k=1,\dots, n$ directly gives the bound $nd_{\alpha_0,\beta_0}$.

\smallskip
\item[$\mathbf{\bullet}\ R_{n,\gamma}^{\leq,3}$.] By the mediant inequality, we have for $a,b,c,d>0$,
\[ \frac{a+b}{c+d} \leq \max\left\{\frac{a}{c},\frac{b}{d}\right\},\]
such that
\[ \log\left( \frac{\alpha_0+\beta_0}{\alpha_0+\theta_\alpha + \beta_0+\theta_\beta}\right) \leq \left| \log\left(\frac{\alpha_0}{\alpha_0+\theta_\alpha}\right)\right| + \left|\log\left(\frac{\beta_0}{\beta_0+\theta_\beta}\right)\right| \]
and in total
\[ \log\left(\frac{\alpha_0+\beta_0}{\alpha_0+\theta_\alpha+\beta_0+\theta_\beta}\frac{\alpha_0+\theta_\alpha}{\alpha_0}\frac{\beta_0}{\beta_0+\theta_\beta}\right)  \leq 2\left| \log\left(\frac{\alpha_0}{\alpha_0+\theta_\alpha}\right)\right| + 2\left|\log\left(\frac{\beta_0}{\beta_0+\theta_\beta}\right)\right|.\]
Hence, we obtain
\begin{align*}
R_{n,\gamma}^{\leq, 3}(\theta_\rho,\theta_\alpha,\theta_\beta)  &\leq \left( \left| \log\left(\frac{\alpha_0}{\alpha_0+\theta_\alpha}\right)\right| + \left|\log\left(\frac{\beta_0}{\beta_0+\theta_\beta}\right)\right| \right) \\
&\hspace{2cm} \cdot 2\sum_{k=1}^n\1_{\{ X_{(k-1)\Delta_n^\gamma}<\rho_0+\theta_\rho<X_{k\Delta_n^\gamma}\}}.
\end{align*}

\smallskip
\item[$\mathbf{\bullet}\ R_{n,\gamma}^{\leq,4}$.] First, we observe that 
\begin{align*}
&R_{n,\gamma}^{\leq, 4}(\theta_\rho,\theta_\alpha,\theta_\beta)\\
&\hspace{0.3cm}= \frac{1}{2\Delta_n^\gamma}\sum_{k=1}^n \1_{\{ X_{(k-1)\Delta_n^\gamma}<\rho_0+\theta_\rho<X_{k\Delta_n^\gamma}\}} \left(X_{k\Delta_n^\gamma}-\rho_0-\theta_\rho\right)^2 \left( \frac{1}{\beta_0^2}-\frac{1}{(\beta_0+\theta_\beta)^2}\right) \\
&\hspace{0.7cm} + \frac{1}{\Delta_n^\gamma}\sum_{k=1}^n \1_{\{ X_{(k-1)\Delta_n^\gamma}<\rho_0+\theta_\rho <X_{k\Delta_n^\gamma}\}} \left(X_{k\Delta_n^\gamma}-\rho_0-\theta_\rho\right)\left(X_{(k-1)\Delta_n^\gamma}-\rho_0-\theta_\rho\right) \\
&\hspace{3cm} \cdot\left(\frac{1}{(\alpha_0+\theta_\alpha)(\beta_0+\theta_\beta)^2}-\frac{1}{\alpha_0\beta_0}\right) \\
&\hspace{0.7cm} + \frac{1}{2\Delta_n^\gamma}\sum_{k=1}^n \1_{\{ X_{(k-1)\Delta_n^\gamma}<\rho_0+\theta_\rho <X_{k\Delta_n^\gamma}\}} \left(X_{(k-1)\Delta_n^\gamma}-\rho_0-\theta_\rho\right)^2 \left( \frac{1}{\alpha_0^2}-\frac{1}{(\alpha_0+\theta_\alpha)^2}\right).
\end{align*}
For $a,b\geq 0$ we have
\[ \left| 1-ab\right| \leq \left| 1-\min\{a,b\}^2\right| + \left| 1-\max\{a,b\}^2\right| = \left| 1-a^2\right| + \left| 1-b^2\right|,\]
such that we obtain
\begin{align*}
&R_{n,\gamma}^{\leq, 4}(\theta_\rho,\theta_\alpha,\theta_\beta)\\
&\hspace{0.5cm}\leq \frac12 \left( \left| \frac{\alpha_0^2}{(\alpha_0+\theta_\alpha)^2}-1\right| + \left|\frac{\beta_0^2}{(\beta_0+\theta_\beta)^2}-1\right|\right) \\
&\hspace{1.5cm} \cdot \left( \frac{1}{\alpha_0^2}+\frac{2}{\alpha_0\beta_0}+\frac{1}{\beta_0^2}\right) \sum_{k=1}^n  \1_{\{ X_{(k-1)\Delta_n^\gamma}<\rho_0+\theta_\rho<X_{k\Delta_n^\gamma}\}} \frac{(X_{k\Delta_n^\gamma}-X_{(k-1)\Delta_n^\gamma})^2}{\Delta_n^\gamma}.
\end{align*}
\end{itemize}
\end{proof}

\subsection{Proof of auxiliary results}\label{Subsection:tripleMLE_3}

\begin{lemma}\label{lemma:function_negative}
For all $x,y>0$ we have
\[ f(x,y):= \log(x) - \frac12 \left(x^2-1\right) y^2 + \log(y) - \frac12\left( y^2-1\right) \leq 0\]
with equality if and only if $y=1/x$.
\end{lemma}
\begin{proof}
It is clear that $f(x,1/x)=0$ by a direct calculation. Now fix $x>0$ and define $g_x(y)=f(x,y)$. Then,
\[ g_x'(y) = -x^2y + \frac{1}{y} = \frac{1}{y}\left( 1-x^2y^2\right).\]
From this it follows that $g_x'(y)>0$ for $y<1/x$ and $g_x'(y)<0$ for $y>1/x$, hence the statement of the lemma.
\end{proof}

The next result establishes important properties of the function $F_{n,C}$ given in~\eqref{eq:function_FnC}. In Figure~\ref{Fig_FnC}, this function is illustrated for different values of $n$.

\begin{center}
\begin{figure}[h]
\includegraphics[scale=0.25]{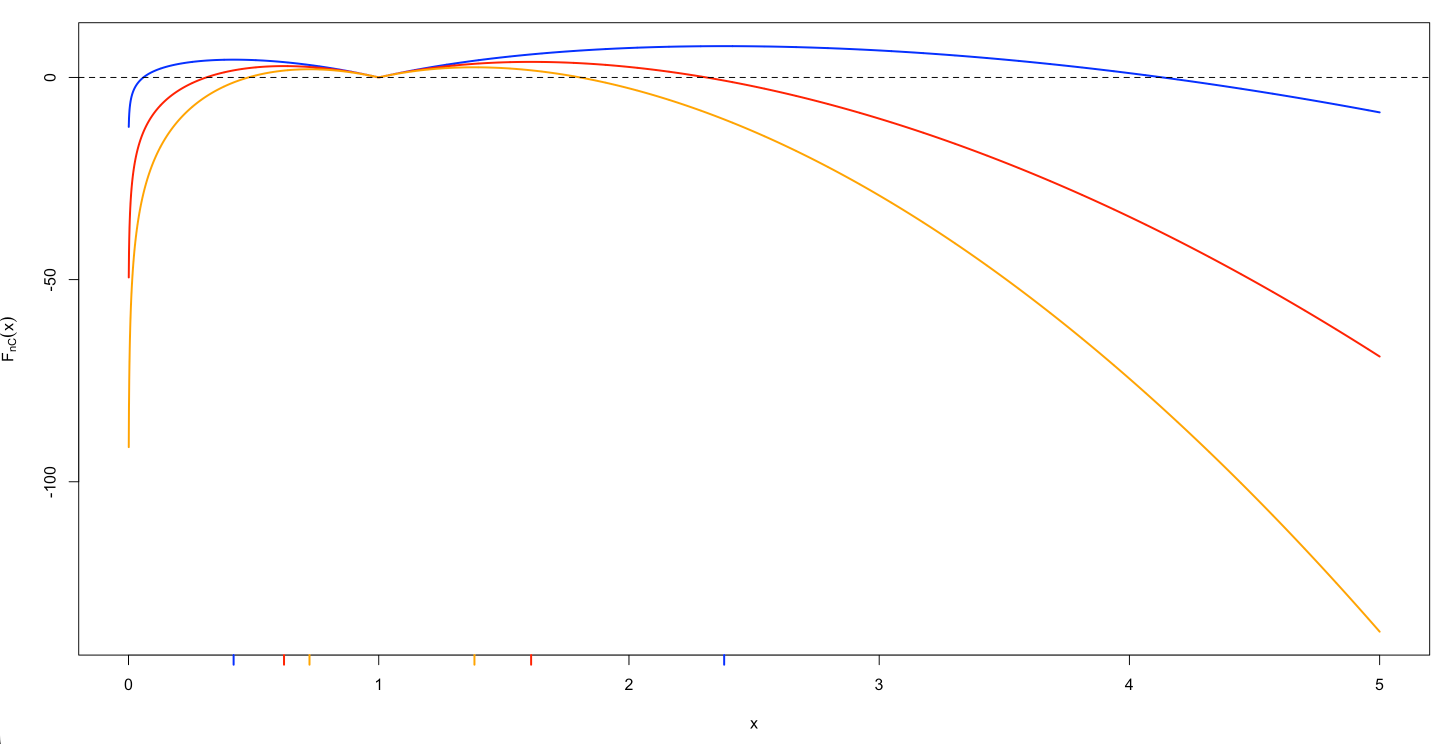}
\caption{\small{The function $F_{n_j,C}$ for $C=7$ and $n_1=100$ (blue), $n_2=250$ (red) and $n_3=500$ (orange). The location of the corresponding local maxima are depicted on the x-axis in the respective color.}}\label{Fig_FnC}
\end{figure}
\end{center}

\vspace{-1cm}

\begin{lemma}\label{lemma:properties_F_nC}
Let $C>0$ and 
\[ F_{n,C}(x) = \sqrt{n}\left( \log(x) - \frac12\left(x^2-1\right)\right) + C\left|\log(x)\right| + \frac{C}{2}\left|x^2-1\right|, \quad x>0.\]
Then, the following hold true:
\begin{enumerate}
\item[(a)] $\sup_{x>0} \sqrt{n}F_{n,C_n}(x) \leq \frac{2C_n^2}{1-1/\sqrt{2}}$ for every $C_n$ satisfying $C_n<\sqrt{n/2}$.
\item[(b)] Let $L>6C$ and $n$ large enough such that $(2CL+L^2+CL^2/\sqrt{n})/\sqrt{n} < C$. Then, we have
\[ \sqrt{n}F_{n,C}(x) \leq -\frac{CL}{9} \quad \textrm{ for every } x\notin\left[\frac{1}{1+L/\sqrt{n}},\frac{1}{1-L/\sqrt{n}}\right].\]
\item[(c)] For any $0<\delta<1$ and $C_n=o(\sqrt{n})$ there exists a constant $c=c(\delta,C)>0$ and $n_0\in\N$ such that for every $n\geq n_0$,
\[ F_{n,C_n}(x) \leq -c\sqrt{n} \quad \textrm{ for all } x\notin [1-\delta,1+\delta].\]
\end{enumerate}
\end{lemma}
\begin{proof}
We start with part (a) and first note that $F_{n,C_n}(1)=0$ and
\[ F_{n,C_n}(x) = \begin{cases}
(\sqrt{n}-C_n)\log(x) - (\sqrt{n}+C_n)(x^2-1)/2, & \textrm{ if } 0<x\leq 1, \\
(\sqrt{n}+C_n)\log(x) - (\sqrt{n}-C_n)(x^2-1)/2, & \textrm{ if } x> 1.
\end{cases}\]
From this, we directly get
\[ F_{n,C_n}'(x) = \begin{cases}
(\sqrt{n}-C_n)\frac1x - (\sqrt{n}+C_n)x, & \textrm{ if } 0<x\leq 1, \\
(\sqrt{n}+C_n)\frac1x - (\sqrt{n}-C_n)x, & \textrm{ if } x> 1,
\end{cases}\]
such that for $C_n<\sqrt{n}$, $F_{n,C_n}'(x)=0$ is satisfied for
\[ x_{\leq 1,n} = \sqrt{\frac{\sqrt{n}-C_n}{\sqrt{n}+C_n}} \quad \textrm{ and } \quad x_{>1,n} = \sqrt{\frac{\sqrt{n}+C_n}{\sqrt{n}-C_n}}.\]
One easily verifies that $F_{n,C_n}''(x) < 0$ for $C<\sqrt{n}$ in which case both $x_{\leq 1,n}$ and $x_{>1,n}$ are local maxima. As it is easy to see that for $C_n<\sqrt{n}$, we have $F_{n,C_n}(x)\rightarrow - \infty$ for both $x\to 0$ and $x\to\infty$ and $F_{n,C_n}(x)\rightarrow 0$ for $x\to 1$, so these are in fact global maxima. Inserting then yields with $\log(1+x)\leq x$
\begin{align*}
F_{n,C_n}(x_{\leq 1,n}) &= \frac{\sqrt{n}-C_n}{2} \log\left( \frac{\sqrt{n}-C_n}{\sqrt{n}+C_n}\right) - \frac{\sqrt{n}+C_n}{2} \left(\frac{\sqrt{n}-C_n}{\sqrt{n}+C_n} -1\right) \\
&= \frac{\sqrt{n}-C_n}{2} \log\left( \frac{\sqrt{n}-C_n}{\sqrt{n}+C_n}\right) +C_n \\
&\leq \frac{\sqrt{n}-C_n}{2} \left( \frac{\sqrt{n}-C_n}{\sqrt{n}+C_n}-1\right) +C_n \\
&= C_n\left( 1 - \frac{\sqrt{n}-C_n}{\sqrt{n}+C_n}\right).
\end{align*}
In particular, 
\[ \sqrt{n} F_{n,C_n}(x_{\leq 1,n}) \leq  \frac{2C_n^2 \sqrt{n}}{\sqrt{n} + C_n} \leq 2C_n^2.\]
Along these lines, we also obtain for $C_n<\sqrt{n/2}$
\[ \sqrt{n} F_{n,C_n}(x_{>1,n}) \leq \frac{2C_n^2 \sqrt{n}}{\sqrt{n}-C_n} \leq \frac{2C_n^2}{1-1/\sqrt{2}}\]
and finish the proof of part (a). For part (b), we first observe that
\[ \frac{1}{(1+L/\sqrt{n})^2} - x_{\leq 1,n}^2 = \frac{2(C-L)  + (2CL -L^2 +CL^2/\sqrt{n})/\sqrt{n}}{(1+L/\sqrt{n})^2 (\sqrt{n}+C)} < 0\]
under the conditions in part (b). Since $F_{n,C}'(x)\geq 0$ for $x\leq x_{\leq 1,n}$, meaning that $F_{n,C}(x)$ is monotonously increasing on $(0,x_{\leq 1,n}]$, we find for any $x\leq 1/(1+L/\sqrt{n})$ that
\begin{align*}
F_{n,C}(x) &\leq F_{n,C}\left(\frac{1}{1+L/\sqrt{n}}\right)\\
&= \left(\sqrt{n}-C\right) \left( \log\left(\frac{1}{1+L/\sqrt{n}}\right) - \frac12 \left( \frac{1}{(1+L/\sqrt{n})^2} - 1\right) \right) \\
&\hspace{1.5cm} -C\left( \frac{1}{(1+L/\sqrt{n})^2}-1\right)
\end{align*} 
A second order Taylor expansion of $x\mapsto\log(x)-(x^2-1)/2$ in $x_0=1$ with intermediate value $\xi$ between $1$ and $x$ yields
\[ \log(x)-\frac12\left(x^2-1\right) = -\frac12\left(1+\frac{1}{\xi^2}\right) (x-1)^2 \leq -\frac{(x-1)^2}{2}.\]
This entails
\begin{align*}
F_{n,C}(x)  &\leq -\frac{\sqrt{n}-C}{2}\left( \frac{1}{1+L/\sqrt{n}}-1\right)^2 - C\left( \frac{1}{(1+L/\sqrt{n})^2}-1\right) \\
&= - \frac{L}{\sqrt{n}} \frac{1}{(1+L/\sqrt{n})^2} \left( \frac{L}{2} - 2C - \frac{3CL}{2\sqrt{n}}\right) \leq - \frac{CL}{9\sqrt{n}},
\end{align*}
which proves the second assertion for $x<(1+L/\sqrt{n})^{-1}$. The claim for $x>(1-L/\sqrt{n})^{-1}$ follows along the same lines by first observing $x_{>1,n}^2-(1-L/\sqrt{n})^{-2}>0$ and then verifying $F_{n,C}(x)\leq -4CL/\sqrt{n}$ for $x>(1-L/\sqrt{n})^{-1}$. This finishes (b) and we start to prove (c). From the form of $F_{n,C_n}'(x)$, we obtain that $F_{n,C_n}(x)$ is monotonically increasing for $x\in (0,x_{\leq 1,n}] \supset (0,1-\delta)$ and monotonically decreasing for $x\in [x_{>1,n},\infty)\supset (1+\delta,\infty)$. The inclusions are valid for $n$ sufficiently large because $x_{\leq 1,n},x_{>1,n}\rightarrow 1$ as $n\to\infty$. Hence, proving $F_{n,C_n}(1\pm\delta)\leq -c\sqrt{n}$ for $n$ large enough suffices to establish (c). For $1-\delta$ this follows from
\[ F_{n,C_n}(1-\delta) = \sqrt{n} \left( \log(1-\delta) - \frac12\left((1-\delta)^2-1\right)\right) - C_n\left( \log(1-\delta)+\frac12\left((1-\delta)^2-1\right)\right)\]
together with the fact that $\log(1-\delta)-((1-\delta)^2-1)/2\leq -c'(\delta)$ for an appropriate constant $c'(\delta)>0$ and $C_n=o(\sqrt{n})$. For $1+\delta$, the claim follows analogously.
\end{proof}

The next Lemma provides two alternative estimates to the decomposition~\eqref{eq:decomposition_ell_1+2}, one that is useful for $\theta_\beta\approx -\beta_0$ and the other for $\theta_\beta\to\infty$. These are used in steps $m=3,4$ in the proof of Proposition~\ref{prop:rate_of_triple_MLE} in Subsection~\ref{Subsection:tripleMLE_1}.\\
Recall the notation of Section~\ref{Section:Proof_sketch_triple}, the functions $f(x)=\log(x)-(x^2-1)/2$ and $g(x)=|\log(x)|+|x^2-1|/2$, the constant $d_{\alpha_0,\beta_0}$ from Lemma~\ref{lemma:estimates_tripleMLE}(d) and the set $A_n$ given in~\eqref{eq_proof:setAn} together with the constant $\zeta$ in the definition of this set. Recall that all constants may depend on the globally chosen parameters $\zeta,\xi,\Gamma$ of this set $A_n$ and the true levels of volatility $\alpha_0,\beta_0$.

\begin{lemma}\label{lemma:bound_small_theta_beta}
Let $\epsilon,R>0$. Then there exists a sequence of sets $(B_n)_{n\in\N}$ with $B_n\subset A_n$ and $\Pr_0(B_n^c)<\epsilon$ together with constants $C_1,C_2=C_2(R)>0$ such that the following statements hold true:
\begin{itemize}
\item[(i)] There exists $\delta_l\in (0,1]$ such that for all $\theta_\beta < -\beta_0+\delta_l$ and $\theta_\rho>\zeta\sqrt{n\Delta_n^\gamma}/2$,
\begin{align*}
&\1_{B_n}\left(\ell_{n,\gamma}^{\alpha_0,\beta_0}(\theta_\rho)+\ell_{n,\gamma}^2(\theta_\rho,\theta_\alpha,\theta_\beta) \right)\\
&\hspace{0.2cm} \leq \1_{B_n} \left(C_1 n f\left(\frac{\alpha_0}{\alpha_0+\theta_\alpha}\right) + T_{n,\gamma}(\theta_\rho,\theta_\alpha) + C_2 n^{1/2+\gamma/6} g\left(\frac{\alpha_0}{\alpha_0+\theta_\alpha}\right)\right.\\
&\hspace{2cm} +C_2n^{2/3} + nd_{\alpha_0,\beta_0}\1_{\{\theta_\alpha>R\}}\bigg) \\
&\hspace{1cm} + \1_{B_n}\sum_{k=1}^n \left[2\log\left(\frac{\beta_0}{\beta_0+\theta_\beta}\right) -\frac12\left(\frac{\beta_0^2}{(\beta_0+\theta_\beta)^2}-1\right)\frac{(X_{k\Delta_n^\gamma}-X_{(k-1)\Delta_n^\gamma})^2}{\Delta_n^\gamma\beta_0^2}\right]\\
&\hspace{3cm} \cdot\1_{\{X_{(k-1)\Delta_n^\gamma}\geq\rho_0+\theta_\rho, X_{k\Delta_n^\gamma}> \rho_0+\theta_\rho\}}.
\end{align*}
\item[(ii)] There exists $\delta_u>0$ such that for $\theta_\beta > -\beta_0+\delta_u$ and $\theta_\rho>\zeta\sqrt{n\Delta_n^\gamma}/2$,
\begin{align*}
&\1_{B_n}\left(\ell_{n,\gamma}^{\alpha_0,\beta_0}(\theta_\rho)+\ell_{n,\gamma}^2(\theta_\rho,\theta_\alpha,\theta_\beta) \right)\\
&\hspace{0.2cm} \leq \1_{B_n} \left(C_1 n f\left(\frac{\alpha_0}{\alpha_0+\theta_\alpha}\right) + T_{n,\gamma}(\theta_\rho,\theta_\alpha) + C_2 n^{1/2+\gamma/6} g\left(\frac{\alpha_0}{\alpha_0+\theta_\alpha}\right) +C_2n^{2/3} \right.  \\
&\hspace{2cm} + nd_{\alpha_0,\beta_0}\1_{\{\theta_\alpha>R\}}\bigg) + \1_{B_n}\1_{B_n} d_{\alpha_0,\beta_0} \sum_{k=1}^n\1_{\{X_{(k-1)\Delta_n^\gamma}\geq\rho_0+\theta_\rho, X_{k\Delta_n^\gamma}> \rho_0+\theta_\rho\}}\\
&\hspace{1cm} + \1_{B_n}\sum_{k=1}^n \left[\log\left(\frac{\beta_0}{\beta_0+\theta_\beta}\right) -\frac12\left(\frac{\beta_0^2}{(\beta_0+\theta_\beta)^2}-1\right)\frac{(X_{k\Delta_n^\gamma}-X_{(k-1)\Delta_n^\gamma})^2}{\Delta_n^\gamma\beta_0^2}\right]\\
&\hspace{3cm} \cdot \1_{\{X_{(k-1)\Delta_n^\gamma}\geq\rho_0+\theta_\rho, X_{k\Delta_n^\gamma}> \rho_0+\theta_\rho\}}. 
\end{align*}
\end{itemize}
\end{lemma}
\begin{proof}
We begin the proof by decomposing
\[ \ell_{n,\gamma}^{\alpha_0,\beta_0}(\theta_\rho) + \ell_{n,\gamma}^2(\theta_\rho,\theta_\alpha,\theta_\beta) = \sum_{j=1}^{5} S_{n,\gamma}^j(\theta_\rho,\theta_\alpha,\theta_\beta),\]
where with the abbreviation $D_k=(X_{k\Delta_n^\gamma}-\rho_0-\theta_\rho)(X_{(k-1)\Delta_n^\gamma}-\rho_0-\theta_\rho)$ and suppressing the argument $(\theta_\rho,\theta_\alpha,\theta_\beta)$ for shorter displays,
\begin{align*}
S_{n,\gamma}^1 &:= \ell_{n,\gamma}^{\alpha_0,\beta_0}(\theta_\rho) + \sum_{k=1}^n \left[\log\left(\frac{\alpha_0}{\alpha_0+\theta_\alpha}\right) -\frac12\left(\frac{\alpha_0^2}{(\alpha_0+\theta_\alpha)^2}-1\right)\frac{(X_{k\Delta_n^\gamma}-X_{(k-1)\Delta_n^\gamma})^2}{\Delta_n^\gamma\alpha_0^2}\right] \\
&\hspace{3.5cm} \cdot\1_{\{X_{(k-1)\Delta_n^\gamma}<\rho_0+\theta_\rho, X_{k\Delta_n^\gamma}\leq \rho_0+\theta_\rho\}}, \\
S_{n,\gamma}^2 &:= \sum_{k=1}^n \log\left(\frac{1-\frac{\alpha_0+\theta_\alpha - (\beta_0+\theta_\beta)}{\alpha_0+\theta_\alpha+\beta_0+\theta_\beta}\exp\left(-\frac{2}{(\alpha_0+\theta_\alpha)^2\Delta_n^\gamma}D_k\right)}{1-\frac{\alpha_0 - \beta_0}{\alpha_0+\beta_0}\exp\left(-\frac{2}{\alpha_0^2\Delta_n^\gamma}D_k\right)}\right)\1_{\{X_{(k-1)\Delta_n^\gamma}<\rho_0+\theta_\rho, X_{k\Delta_n^\gamma}\leq \rho_0+\theta_\rho\}}, \\
S_{n,\gamma}^3 &:= \sum_{k=1}^n \left[\log\left(\frac{\beta_0}{\beta_0+\theta_\beta}\right) -\frac12\left(\frac{\beta_0^2}{(\beta_0+\theta_\beta)^2}-1\right)\frac{(X_{k\Delta_n^\gamma}-X_{(k-1)\Delta_n^\gamma})^2}{\Delta_n^\gamma\beta_0^2}\right] \\
&\hspace{1.2cm} \cdot\1_{\{X_{(k-1)\Delta_n^\gamma}\geq\rho_0+\theta_\rho, X_{k\Delta_n^\gamma}> \rho_0+\theta_\rho\}}, \\
S_{n,\gamma}^4 &:= \sum_{k=1}^n \log\left(\frac{1+\frac{\alpha_0+\theta_\alpha -( \beta_0+\theta_\beta)}{\alpha_0+\theta_\alpha+\beta_0+\theta_\beta}\exp\left(-\frac{2}{(\beta_0+\theta_\beta)^2\Delta_n^\gamma}D_k\right)}{1+\frac{\alpha_0 - \beta_0}{\alpha_0+\beta_0}\exp\left(-\frac{2}{\beta_0^2\Delta_n^\gamma}D_k\right)}\right) \1_{\{X_{(k-1)\Delta_n^\gamma}\geq\rho_0+\theta_\rho, X_{k\Delta_n^\gamma}> \rho_0+\theta_\rho\}}, \\
S_{n,\gamma}^5 &:= \sum_{k=1}^n \1_{\{ X_{(k-1)\Delta_n^\gamma}<\rho_0+\theta<X_{k\Delta_n^\gamma}\}} \left[\log\left(\frac{\alpha_0+\beta_0}{\alpha_0+\theta_\alpha+\beta_0+\theta_\beta}\frac{\alpha_0+\theta_\alpha}{\alpha_0}\frac{\beta_0}{\beta_0+\theta_\beta}\right) \right.\\
&\hspace{1cm} +\frac{1}{2\Delta_n^\gamma} \left( \left(\frac{X_{k\Delta_n^\gamma}-\rho_0-\theta_\rho}{\beta_0} -\frac{X_{(k-1)\Delta_n^\gamma}-\rho_0-\theta_\rho}{\alpha_0} \right)^2\right. \\
&\hspace{3cm} \left.\left. -\left(\frac{X_{k\Delta_n^\gamma}-\rho_0-\theta_\rho}{\beta_0+\theta_\beta} -\frac{X_{(k-1)\Delta_n^\gamma}-\rho_0-\theta_\rho}{\alpha_0+\theta_\alpha} \right)^2 \right) \right], \\
&\hspace{0.5cm}+ \sum_{k=1}^n \1_{\{ X_{k\Delta_n^\gamma}\leq\rho_0+\theta\leq X_{(k-1)\Delta_n^\gamma}\}} \left[\log\left(\frac{\alpha_0+\beta_0}{\alpha_0+\theta_\alpha+\beta_0+\theta_\beta}\frac{\beta_0+\theta_\beta}{\beta_0}\frac{\alpha_0}{\alpha_0+\theta_\alpha}\right) \right.\\
&\hspace{1cm} +\frac{1}{2\Delta_n^\gamma} \left( \left(\frac{X_{k\Delta_n^\gamma}-\rho_0-\theta_\rho}{\alpha_0} -\frac{X_{(k-1)\Delta_n^\gamma}-\rho_0-\theta_\rho}{\beta_0} \right)^2\right. \\
&\hspace{3cm} \left.\left. -\left(\frac{X_{k\Delta_n^\gamma}-\rho_0-\theta_\rho}{\alpha_0+\theta_\alpha} -\frac{X_{(k-1)\Delta_n^\gamma}-\rho_0-\theta_\rho}{\beta_0+\theta_\beta} \right)^2 \right) \right].
\end{align*}
In what follows, we treat these terms independently of one another. Let us now define $B_n$ as the intersection of the following three events: $B_{1,n}=$ the set for which the bound~\eqref{eq_proof:order_remainder_ell1} is valid, $B_{2,n}=$ the set in Lemma~\ref{lemma:estimates_tripleMLE} (also called $B_n$ in there) and the event
\[ B_{3,n}:=\left\{ \max_{0\leq t\leq n\Delta_n^\gamma} X_t - X_{n\Delta_n^\gamma} >1 \right\}.\]
Recall that $\limsup_{n}\Pr_0(B_{1,n}^c)=0$ as discussed right before~\eqref{eq_proof:order_remainder_ell1} and $\limsup_{n}\Pr_0(B_{2,n}^c)=0$ which is part of the statement of Lemma~\ref{lemma:estimates_tripleMLE}. To see that $\Pr_0(B_{3,n})\rightarrow 1$, we employ the Dambis--Dubins--Schwarz theorem (Theorem~$19.4$ in~\citeSM{App:Kallenberg}). Indeed, let be an independent Brownian motion $B$ (possibly on an extension of the original probability space, we also denote the measure on this extension with $\Pr_0$) such that $X_t = B_{\langle X\rangle_t}$ for all $t\geq 0$ almost surely. Then,
\[ \max_{0\leq t\leq n\Delta_n^\gamma} X_t - X_{n\Delta_n^\gamma} = \max_{0\leq t\leq \langle X\rangle_{n\Delta_n^\gamma}}B_t -  B_{\langle X\rangle_{n\Delta_n^\gamma}}. \]
Using the distributional identity $\max_{0\leq s\leq t}B_s - B_t \stackrel{\mathcal{D}}{=} |B_t|$ (Proposition~$14.13$ in~\citeSM{App:Kallenberg}), we obtain by conditioning
\begin{align*}
\Pr_0\left( \max_{0\leq t\leq n\Delta_n^\gamma} X_t - X_{n\Delta_n^\gamma} \leq 1 \right) &= \Pr_0\left( \max_{0\leq t\leq \langle X\rangle_{n\Delta_n^\gamma}}B_t -  B_{\langle X\rangle_{n\Delta_n^\gamma}} \leq 1 \right) \\
&= \E_0\left[ \Pr_0\left(\left. \max_{0\leq t\leq \langle X\rangle_{n\Delta_n^\gamma}}B_t -  B_{\langle X\rangle_{n\Delta_n^\gamma}}\leq 1 \right| \langle X\rangle_{n\Delta_n^\gamma}\right)\right] \\
&= \E_0 \left[ \Pr_0 \left(\left.\big| B_{\langle X\rangle_{n\Delta_n^\gamma}}\big| \leq 1 \right| \langle X\rangle_{n\Delta_n^\gamma}\right)\right] \\
&\leq \E_0\left[ \frac{2}{\sqrt{\langle X\rangle_{n\Delta_n^\gamma}}}\right] \leq \frac{2}{\min\{\alpha_0,\beta_0\}}\frac{1}{\sqrt{n\Delta_n^\gamma}}\rightarrow 0,
\end{align*}
where the last inequality used that 
\[ \langle X\rangle_{n\Delta_n^\gamma} = \int_0^{n\Delta_n^\gamma} \sigma_{\rho_0,\alpha_0,\beta_0}(X_t)^2 dt \geq \min\{\alpha_0^2,\beta_0^2\}n\Delta_n^\gamma.\]

We now prove assertion~(i) by estimating each $S_{n,\gamma}^j\1_{B_n}$ for $j=1,\dots, 5$ separately.
\begin{itemize}
\item[$\bullet\ \mathbf{S_{n,\gamma}^1}$.] We have the decomposition
\begin{align*}
S_{n,\gamma}^1(\theta_\rho,\theta_\alpha,\theta_\beta) &= \overline{H}_{n,\gamma}(0,\theta_\alpha) + T_{n,\gamma}(\theta_\rho,\theta_\alpha) + R_{n,\gamma}^1(\theta_\rho)\\
&\hspace{1cm} + \left( H_{n,\gamma}(\theta_\rho,\theta_\alpha) - \overline{H}_{n,\gamma}(\theta_\rho,\theta_\alpha)\right) + R_{n,\gamma}^{\leq, 1}(\theta_\rho,\theta_\alpha, \theta_\beta),
\end{align*} 
where $R_{n,\gamma}^{\leq,1}$ is defined in the proof of Lemma~\ref{lemma:estimates_tripleMLE}(d), does not depend on $\theta_\beta$ and is shown to satisfy
\[ \sup_{\theta_\rho>\zeta\sqrt{n\Delta_n^{\gamma}}/2}R_{n,\gamma}^{\leq,1}(\theta_\rho,\theta_\alpha)\1_{B_n} \leq \1_{B_n}Cg\left(\frac{\alpha_0}{\alpha_0+\theta_\alpha}\right) n^{1/2+\gamma/6}.\]
The same bound is valid for $\1_{B_n}(H_{n,\gamma}(\theta_\rho,\theta_\alpha) - \overline{H}_{n,\gamma}(\theta_\rho,\theta_\alpha)$ by Lemma~\ref{lemma:estimates_tripleMLE}(c). By Lemma~\ref{lemma:estimate_counts_interval} applied to $a=\zeta\sqrt{n\Delta_n^\gamma}/2$ and $b=0$, we obtain
\[ \overline{H}_{n,\gamma}(0,\theta_\alpha) \leq n \frac{\xi\zeta}{2\max\{\alpha_0^2,\beta_0^2\}} f\left(\frac{\alpha_0}{\alpha_0+\theta_\alpha}\right)\]
and $\sup_{\theta_\rho} R_{n,\gamma}^1(\theta_\rho)\1_{B_n} \leq Cn^{1/2+\gamma/6}$ according to~\eqref{eq_proof:order_remainder_ell1}. All these upper bounds appear on the right-hand side of the statement.\smallskip

\item[$\bullet\ \mathbf{S_{n,\gamma}^2}$.] This term is the same as $R_{n,\gamma}^{\leq,2}\leq C_{\alpha_0,\beta_0}(R_{n,\gamma}^{\leq,2,1}+R_{n,\gamma}^{\leq,2,2})$ in the proof of Lemma~\ref{lemma:estimates_tripleMLE}(d) where also the definition of $R_{n,\gamma}^{\leq,2,1},R_{n,\gamma}^{\leq,2,2}$ is given in~\eqref{eq_proof:d_R21}. To bound the first one, we apply the inequality
\begin{align}\label{eq_proof:small_delta_uniform_bound_factor}
\sup_{\theta_\alpha>-\alpha_0,\theta_\beta>-\beta_0} \left|\frac{\alpha_0+\theta_\alpha-(\beta_0+\theta_\beta)}{\alpha_0+\theta_\alpha+\beta_0+\theta_\beta} - \frac{\alpha_0-\beta_0}{\alpha_0+\beta_0}\right| \leq 2 
\end{align} 
and then follow the treatment of $R_{n,\gamma}^{\leq,2,1}$ in the proof of Lemma~\ref{lemma:estimates_tripleMLE}(d). The reasoning for $R_{n,\gamma}^{\leq,2,2}$ does not change at all compared to the proof of Lemma~\ref{lemma:estimates_tripleMLE}(d) and in total we obtain
\[ \sup_{\theta_\rho}S_{n,\gamma}^2(\theta_\rho,\theta_\alpha,\theta_\beta) \leq C(R)n^{1/2+\gamma/6}\left( 1+ g\left(\frac{\alpha_0}{\alpha_0+\theta_\alpha}\right)\right)\]
in case $\theta_\alpha\leq R$ and obtain the bound $nd_{\alpha_0,\beta_0}$ in case $\theta_\alpha>R$. \smallskip

\item[$\bullet\ \mathbf{S_{n,\gamma}^3}$.] This term is already present on the right-hand side of the assertion.\smallskip

\item[$\bullet\ \mathbf{S_{n,\gamma}^4}$.] As for $S_{n,\gamma}^2$ we obtain $S_{n,\gamma}^4\leq C_{\alpha_0,\beta_0}(S_{n,\gamma}^{4,1} + S_{n,\gamma}^{4,2})$, where the terms $S_{n,\gamma}^{4,1},S_{n,\gamma}^{4,2}$ replace $R_{n,\gamma}^{\leq,2,1}, R_{n,\gamma}^{\leq,2,2}$ from the proof of $S_{n,\gamma}^2$ and are given as
\begin{align*}
S_{n,\gamma}^{4,1}(\theta_\rho,\theta_\alpha,\theta_\beta) &= \left|\frac{\alpha_0+\theta_\alpha-(\beta_0+\theta_\beta)}{\alpha_0+\theta_\alpha+\beta_0+\theta_\beta} - \frac{\alpha_0-\beta_0}{\alpha_0+\beta_0}\right| \\
&\hspace{1.5cm} \cdot \sum_{k=1}^n\1_{\{X_{(k-1)\Delta_n^\gamma}\geq \rho_0+\theta_\rho,X_{k\Delta_n^\gamma}>\rho_0+\theta_0\}} \exp\left(-\frac{2}{\beta_0^2\Delta_n^\gamma} D_k\right)\\
S_{n,\gamma}^{4,2}(\theta_\rho,\theta_\alpha,\theta_\beta) &= \left|\frac{\alpha_0+\theta_\alpha-(\beta_0+\theta_\beta)}{\alpha_0+\theta_\alpha+\beta_0+\theta_\beta}\right|\sum_{k=1}^n\1_{\{X_{(k-1)\Delta_n^\gamma}\geq \rho_0+\theta_\rho,X_{k\Delta_n^\gamma}>\rho_0+\theta_0\}} \\
&\hspace{1.5cm} \cdot \left| \exp\left( -\frac{2}{\beta_0^2\Delta_n^\gamma}D_k\right)-\exp\left(-\frac{2}{(\beta_0+\theta_\beta)^2\Delta_n^\gamma} D_k\right)\right|,
\end{align*} 
where $D_k = (X_{(k-1)\Delta_n^\gamma}-\rho_0-\theta_\rho)(X_{k\Delta_n^\gamma}-\rho_0-\theta_\rho)$. The term $S_{n,\gamma}^{4,1}$ can then be dealt with as $R_{n,\gamma}^{\leq,2,1}$ in the proof of Lemma~\ref{lemma:estimates_tripleMLE}(d) by again applying bound~\eqref{eq_proof:small_delta_uniform_bound_factor}. For $S_{n,\gamma}^{4,2}$ we observe that the factor in front of the sum is uniformly in $\theta_\alpha,\theta_\beta$ bounded by one. Moreover, $(\beta_0+\theta_\beta)^{-2}\geq \delta_l^{-2}$ and by choosing $\delta_l\in (0,1]$ small enough such that $\delta_l^{-2}\geq \beta_0^{-2}$ we obtain by observing $D_k\geq 0$ on the event $\{X_{(k-1)\Delta_n^\gamma}\geq \rho_0+\theta_\rho,X_{k\Delta_n^\gamma}>\rho_0+\theta_0\}$
\[ S_{n,\gamma}^{4,2}(\theta_\rho,\theta_\alpha,\theta_\beta) \leq 2 \sum_{k=1}^n\1_{\{X_{(k-1)\Delta_n^\gamma}\geq \rho_0+\theta_\rho,X_{k\Delta_n^\gamma}>\rho_0+\theta_0\}}\exp\left( -\frac{2}{\beta_0^2\Delta_n^\gamma}D_k\right)\]
and we can then follow the follow the lines of $S_{n,\gamma}^{4,1}$. In total we then obtain
\[ S_{n,\gamma}^{4}(\theta_\rho,\theta_\alpha,\theta_\beta)\1_{B_n} \leq \1_{B_n} C_{\alpha_0,\beta_0}n^{1/2+\gamma/6}.\]

\item[$\bullet\ \mathbf{S_{n,\gamma}^5}$.] In this part, the event $B_{3,n}$ is used in a somewhat surprising way. First of all we drop negative summands and bound
\[ S_{n,\gamma}^5(\theta_\rho,\theta_\alpha,\theta_\beta) \leq S_{n,\gamma}^{5,1}(\theta_\rho,\theta_\alpha,\theta_\beta)+S_{n,\gamma}^{5,2}(\theta_\rho,\theta_\alpha,\theta_\beta)+S_{n,\gamma}^{5,3}(\theta_\rho,\theta_\beta)+S_{n,\gamma}^{5,4}(\theta_\rho,\theta_\beta), \]
where
\begin{align*}
S_{n,\gamma}^{5,1}(\theta_\rho,\theta_\alpha,\theta_\beta) &:=  \sum_{k=1}^n \1_{\{ X_{(k-1)\Delta_n^\gamma}<\rho_0+\theta_\rho<X_{k\Delta_n^\gamma}\}} \left[ \log\left( \frac{\alpha_0+\beta_0}{\alpha_0+\theta_\alpha+\beta_0+\theta_\beta}\frac{\alpha_0+\theta_\alpha}{\alpha_0}\right)\right.\\
&\hspace{2cm} \left. +\frac{1}{2\Delta_n^\gamma} \left(\frac{X_{k\Delta_n^\gamma}-\rho_0-\theta_\rho}{\beta_0} -\frac{X_{(k-1)\Delta_n^\gamma}-\rho_0-\theta_\rho}{\alpha_0} \right)^2\right],\\
S_{n,\gamma}^{5,2}(\theta_\rho,\theta_\alpha,\theta_\beta)&:= \sum_{k=1}^n \1_{\{ X_{k\Delta_n^\gamma}\leq\rho_0+\theta\leq X_{(k-1)\Delta_n^\gamma}\}} \left[\log\left(\frac{\alpha_0+\beta_0}{\alpha_0+\theta_\alpha+\beta_0+\theta_\beta}\frac{\alpha_0}{\alpha_0+\theta_\alpha}\right) \right.\\
&\hspace{1cm} +\frac{1}{2\Delta_n^\gamma} \left( \left(\frac{X_{k\Delta_n^\gamma}-\rho_0-\theta_\rho}{\alpha_0} -\frac{X_{(k-1)\Delta_n^\gamma}-\rho_0-\theta_\rho}{\beta_0} \right)^2\right],\\
S_{n,\gamma}^{5,3}(\theta_\rho,\theta_\beta) &:=  \sum_{k=1}^n \1_{\{ X_{(k-1)\Delta_n^\gamma}<\rho_0+\theta_\rho<X_{k\Delta_n^\gamma}\}} \log\left(\frac{\beta_0}{\beta_0+\theta_\beta}\right), \\
S_{n,\gamma}^{5,4}(\theta_\rho,\theta_\beta) &:=  \sum_{k=1}^n \1_{\{ X_{k\Delta_n^\gamma}<\rho_0+\theta_\rho<X_{(k-1)\Delta_n^\gamma}\}} \log\left(\frac{\beta_0+\theta_\beta}{\beta_0}\right).
\end{align*}
For $S_{n,\gamma}^{5,1}$, we note that for $\delta_l\leq 1$,
\[ \sup_{\theta_\alpha>-\alpha_0} \sup_{\theta_\beta\in (-\beta_0,-\beta_0+\delta_l)} \log\left( \frac{\alpha_0+\beta_0}{\alpha_0+\theta_\alpha+\beta_0+\theta_\beta}\frac{\alpha_0+\theta_\alpha}{\alpha_0}\right)\leq \log\left(\frac{\alpha_0+\beta_0}{\alpha_0}\right).\]
Next, we have
\begin{align*}
&\1_{\{ X_{(k-1)\Delta_n^\gamma}<\rho_0+\theta_\rho<X_{k\Delta_n^\gamma}\}} \left( X_{k\Delta_n^\gamma}-\rho_0-\theta_\rho\right)^2 \\
&\hspace{2cm}\leq \1_{\{ X_{(k-1)\Delta_n^\gamma}<\rho_0+\theta<X_{k\Delta_n^\gamma}\}}\left( X_{k\Delta_n^\gamma}-X_{(k-1)\Delta_n^\gamma}\right)^2
\end{align*} 
and the same with $k$ replaced by $(k-1)$. This means that in total
\[ S_{n,\gamma}^{5,1}(\theta_\rho,\theta_\alpha,\theta_\beta) \leq C_{\alpha_0,\beta_0} \Upsilon_{n,\gamma}^{\leq,1}(\theta_\rho,\theta_\alpha,\theta_\beta),\]
where $\Upsilon_{n,\gamma}^{\leq,1}$ is given in~\eqref{eq_proof:Upsilon} and satisfies $\Upsilon_{n,\gamma}^{\leq,1}\1_{B_n} \leq \1_{B_n}C_{\alpha_0,\beta_0}n^{1/2+\gamma/6}$. For $S_{n,\gamma}^{5,2}$ we note that
\[ \sup_{\theta_\beta\in (-\beta_0,-\beta_0+\delta_l)}\log\left( \frac{\alpha_0+\beta_0}{\alpha_0+\theta_\alpha+\beta_0+\theta_\beta}\frac{\alpha_0}{\alpha_0+\theta_\alpha}\right) \leq C_{\alpha_0,\beta_0}\left( 1+ g\left(\frac{\alpha_0}{\alpha_0+\theta_\alpha}\right)\right)\]
and then proceed as for $S_{n,\gamma}^{5,1}$ to obtain
\[ S_{n,\gamma}^{5,2}(\theta_\rho,\theta_\alpha,\theta_\beta)\1_{B_n} \leq \1_{B_n}C_{\alpha_0,\beta_0}\left( 1+ g\left(\frac{\alpha_0}{\alpha_0+\theta_\alpha}\right)\right) n^{1/2+\gamma/6}.\]
We now turn to $S_{n,\gamma}^{5,3}$ which is the most interesting part of the proof. This term is positive for $\theta_\beta<0$ and this positivity cannot be compensated within the same regime, i.e. with other indices $k$ for which $k\in I_{up} := \{1\leq k\leq n:X_{(k-1)\Delta_n^\gamma}<\rho_0+\theta_\rho\leq X_{k\Delta_n^\gamma}\}$. The idea to solve this issue is to identify possible (negative) compensators and there are two different types:\smallskip
\begin{itemize}
\item[$\bullet$] For each $k$ with $X_{k\Delta_n^\gamma}<\rho_0+\theta_\rho<X_{(k-1)\Delta_n^\gamma}$ we obtain the corresponding negative summand $-\log(\beta_0/(\beta_0+\theta_\beta))$ which is part of $S_{n,\gamma}^{5,4}$.
\item[$\bullet$] From $S_{n,\gamma}^3$, we already have the summand
\begin{align}\label{eq_proof:Sn5_delta_small}
\sum_{k=1}^n \log\left(\frac{\beta_0}{\beta_0+\theta_\beta}\right)\1_{\{X_{(k-1)\Delta_n^\gamma}\geq \rho_0+\theta_\rho, X_{k\Delta_n^\gamma}>\rho_0+\theta_\rho\}},
\end{align} 
which in an application of this lemma is seen to compensated for by the negative term
\[ -\frac12\left(\frac{\beta_0^2}{(\beta_0+\theta_\beta)^2}\right) \sum_{k=1}^n \frac{(X_{k\Delta_n^\gamma}-X_{(k-1)\Delta_n^\gamma})^2}{\Delta_n^\gamma\beta_0^2}\1_{\{X_{(k-1)\Delta_n^\gamma}\geq \rho_0+\theta_\rho, X_{k\Delta_n^\gamma}>\rho_0+\theta_\rho\}} \]
that is also part of $S_{n,\gamma}^3$. In principle, adding this term and two times~\eqref{eq_proof:Sn5_delta_small} would also result in a negative term with high probability.
\end{itemize}\smallskip
Now, the crucial observation is that on the event $B_n$, we can construct for each path an injective mapping
\begin{align}\label{eq_proof:injection}
\begin{split}
I_{up} \rightarrow \left\{ k:X_{k\Delta_n^\gamma}<\rho_0+\theta_\rho<X_{(k-1)\Delta_n^\gamma} \right\} \cup \left\{ k:X_{(k-1)\Delta_n^\gamma}\geq \rho_0+\theta_\rho, X_{k\Delta_n^\gamma}>\rho_0+\theta_\rho\right\},
\end{split}
\end{align}
where especially $B_{3,n}$ plays an important role: For each $k\leq n-1$ with $k\in I_{up}$ we obviously have either $X_{(k+1)\Delta_n^\gamma}>\rho_0+\theta_\rho$ or $X_{(k+1)\Delta_n^\gamma}\leq\rho_0+\theta_\rho$ and map $k\mapsto k+1$. For $k=n$ we have no successor and argue with the path properties of $(X_t)_{0\leq t\leq n\Delta_n^\gamma}$ and our particular choice of $B_n$. On this event, the difference of the running maximum of $(X_t)_{0\leq t\leq n\Delta_n^\gamma}$ and $X_{n\Delta_n^\gamma}$ is at least one. Since $|X_{(k-1)\Delta_n^\gamma}-X_{k\Delta_n^\gamma}|\leq (\Delta_n^\gamma)^{4/9}$ on $A_n\supset B_n$ for two consecutive observations, this implies that for $n$ large enough, there exists a smallest index $k_0\leq n-2$ such that three consecutive observations lie above $\rho_0+\theta_\rho$, i.e. $X_{k_0\Delta_n^\gamma}, X_{(k_0+1)\Delta_n^\gamma}, X_{(k_0+2)\Delta_n^\gamma}>\rho_0+\theta_\rho$. In particular, $k_0+1\notin I_{up}$ such that we can map $n\mapsto k_0+2$ (see Figure~\ref{Fig_cancelling} for an illustration).

\vspace{-0.5cm}

\begin{center}
\begin{figure}[h]
\begin{tikzpicture}
  \draw[->] (0,0.2) -- (12.5,0.2) node[right] {$x$};
  \draw[->] (0,0.2) -- (0,5) node[left] {$y$};
  \draw[-] (0,2) -- (12,2);
  
  \draw[black,thick] (0,1) -- (1,1.5);
  \draw[red,thick] (1,1.5) -- (2,3);
  \draw[blue,thick] (2,3) -- (3,1.3);
  \draw[black,thick] (3,1.3) -- (4,1.8);
  \draw[black,thick] (4,1.8) -- (5,0.6);
  \draw[black,thick] (5,0.6) -- (6,0.9);
  \draw[red,thick] (6,0.9) -- (7,2.2);
  \draw[blue,thick] (7,2.2) -- (8, 2.8); 
  \draw[blue,thick] (8,2.8) -- (9,4.5);
  \draw[black,thick] (9,4.5) -- (10, 2.4);
  \draw[black,thick] (10,2.4) -- (11,1.6);
  \draw[red,thick] (11,1.6) -- (12,2.4);
  
  \fill (0,1) circle (1.5pt);
  \fill (1,1.5) circle (1.5pt);
  \fill (2,3) circle (1.5pt);
  \fill (3,1.3) circle (1.5pt);
  \fill (4,1.8) circle (1.5pt);
  \fill (5,0.6) circle (1.5pt);
  \fill (6,0.9) circle (1.5pt);
  \fill (7,2.2) circle (1.5pt);
  \fill (8,2.8) circle (1.5pt);
  \fill (9,4.5) circle (1.5pt);
  \fill (10,2.4) circle (1.5pt);
  \fill (11,1.6) circle (1.5pt);
  \fill (12,2.4) circle (1.5pt);

  \foreach \z in {1,2,3,4,5,6,7,8,9,10,11,12} {
    \draw[black,thick] (\z,0.1) -- (\z,0.3) node[below, yshift=-5pt] {\z};
  }

\end{tikzpicture}
\caption{\small{An illustration of discrete realization, linearly interpolated. To construct the injection~\eqref{eq_proof:injection}, each upcrossing (red) is mapped to its successor, except of the last segment. Here, we chose one of the earlier segments above the threshold that has no preceding upcrossing. Such a segment always exists if there are three consecutive points above the threshold.}}\label{Fig_cancelling}
\end{figure}
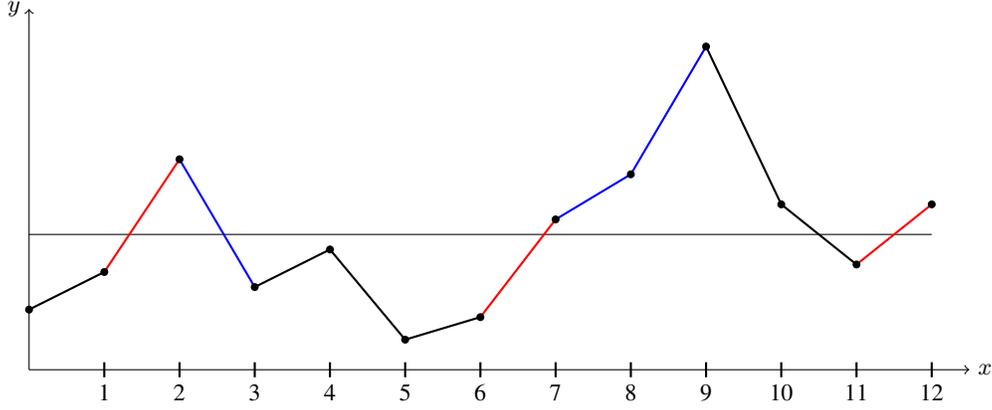
\end{center}

\vspace{-1cm}

Using the above construction of an injective mapping in~\eqref{eq_proof:injection} reveals
\begin{align*}
S_{n,\gamma}^{5,3}(\theta_\rho,\theta_\beta) &= \sum_{k=1}^{n} \1_{\{ X_{(k-1)\Delta_n^\gamma}<\rho_0+\theta_\rho<X_{k\Delta_n^\gamma}\}} \log\left(\frac{\beta_0}{\beta_0+\theta_\beta}\right)\\
&\leq \log\left(\frac{\beta_0}{\beta_0+\theta_\beta}\right)\sum_{k=1}^{n}\1_{\{ \rho_0+\theta_\rho<X_{(k-1)\Delta_n^\gamma}, X_{k\Delta_n^\gamma}\}}  \\
&\hspace{1cm} + \log\left(\frac{\beta_0}{\beta_0+\theta_\beta}\right)\sum_{k=1}^{n}\1_{\{ X_{k\Delta_n^\gamma}\leq\rho_0+\theta_\rho<X_{(k-1)\Delta_n^\gamma}\}}\\
& = -S_{n,\gamma}^{5,4}(\theta_\rho,\theta_\beta) + \log\left(\frac{\beta_0}{\beta_0+\theta_\beta}\right)\sum_{k=1}^{n}\1_{\{ \rho_0+\theta_\rho<X_{(k-1)\Delta_n^\gamma}, X_{k\Delta_n^\gamma}\}}.
\end{align*}
Now it is clear that $S_{n,\gamma}^{5,3}+S_{n,\gamma}^{5,4}$ appears on the right-hand side of the statement and the proof of part (i) is finished.
\end{itemize}

We finally turn to part (ii) that is actually easier to prove than (i). Here, $S_{n,\gamma}^j$ can be dealt with the same as in (i) for $j=1,2,3$ and we discuss the remaining two cases.
\begin{itemize}
\item[$\bullet\ \mathbf{S_{n,\gamma}^4}$.] We observe that as $D_k=(X_{k\Delta_n^\gamma}-\rho_0-\theta_\rho)(X_{(k-1)\Delta_n^\gamma}-\rho_0-\theta_\rho)\geq 0$ on the event that $X_{(k-1)\Delta_n^\gamma},X_{k\Delta_n^\gamma}\geq \rho_0+\theta_\rho$,
\begin{align*}
\log\left(\frac{1-\frac{\alpha_0+\theta_\alpha - (\beta_0+\theta_\beta)}{\alpha_0+\theta_\alpha+\beta_0+\theta_\beta}\exp\left(-\frac{2}{(\beta_0+\theta_\beta)^2\Delta_n^\gamma}D_k\right)}{1-\frac{\alpha_0 - \beta_0}{\alpha_0+\beta_0}\exp\left(-\frac{2}{\beta_0^2\Delta_n^\gamma}(X_{k\Delta_n^\gamma}-\rho_0-\theta_\rho)D_k\right)} \right) \leq \log\left( \frac{2}{1-\frac{|\alpha_0-\beta_0|}{\alpha_0+\beta_0}} \right) =d_{\alpha_0,\beta_0},
\end{align*}
such that
\[ S_{n,\gamma}^4(\theta_\rho,\theta_\alpha,\theta_\beta) \leq d_{\alpha_0,\beta_0} \sum_{k=1}^n\1_{\{X_{(k-1)\Delta_n^\gamma}\geq\rho_0+\theta_\rho, X_{k\Delta_n^\gamma}> \rho_0+\theta_\rho\}},\]
which appear on the right-hand side of the statement. \smallskip

\item[$\bullet\ \mathbf{S_{n,\gamma}^5}$.] Recall $S_{n,\gamma}^{5,j}$, $j=1,2,3,4$, from the treatment of $S_{n,\gamma}^5$ in~(i). Nothing has to be changed in the discussion of $S_{n,\gamma}^{5,1}$ and $S_{n,\gamma}^{5,3}\leq 0$ for $\delta_u>0$. Adding $S_{n,\gamma}^{5,2}+S_{n,\gamma}^{5,4}$ and observing that
\begin{align*}
&\sup_{\theta_\alpha}\sup_{\theta_\beta>\delta_u-\beta_0} \log\left( \frac{\alpha_0+\beta_0}{\alpha_0+\theta_\alpha+\beta_0+\theta_\beta}\frac{\alpha_0}{\alpha_0+\theta_\alpha}\frac{\beta_0+\theta_\beta}{\beta_0}\right) \leq C_{\alpha_0,\beta_0}\left( 1+ g\left(\frac{\alpha_0}{\alpha_0+\theta_\alpha}\right)\right),
\end{align*} 
once can proceed as for $S_{n,\gamma}^{5,2}$ in~(i). In total, this gives
\[ S_{n,\gamma}^{5}(\theta_\rho,\theta_\alpha,\theta_\beta)\1_{B_n} \leq \1_{B_n}C_{\alpha_0,\beta_0}\left( 1+ g\left(\frac{\alpha_0}{\alpha_0+\theta_\alpha}\right)\right) n^{1/2+\gamma/6},\]
which appears on the right-hand side of the statement.
\end{itemize}
\end{proof}


\bibliographystyleSM{imsart-nameyear} 
\bibliographySM{Bibliography_Appendix}       

\end{document}